\documentclass{article} 
\usepackage{iclr2027_conference,times}

\usepackage{amsmath,amsfonts,bm}

\def\eqref#1{equation~\ref{#1}}

\def\1{\bm{1}}

\DeclareMathAlphabet{\mathsfit}{\encodingdefault}{\sfdefault}{m}{sl}
\SetMathAlphabet{\mathsfit}{bold}{\encodingdefault}{\sfdefault}{bx}{n}

\newcommand{\E}{\mathbb{E}}
\newcommand{\Ls}{\mathcal{L}}

\DeclareMathOperator*{\argmin}{arg\,min}

\usepackage{amsthm}
\usepackage{amssymb}
\usepackage{mathtools}
\usepackage{booktabs}
\usepackage{url}

\usepackage{xcolor}
\usepackage{colortbl}
\usepackage{color}
\usepackage{graphicx}
\usepackage{multirow}
\usepackage{pifont}
\usepackage[many]{tcolorbox}

\usepackage{caption} 

\usepackage{algorithm}
\usepackage{algorithmic}

\usepackage{verbatim}
\usepackage{xspace} %

\usepackage{enumerate}
\usepackage{enumitem}

\newtheorem{theorem}{Theorem}
\newtheorem{lemma}[theorem]{Lemma}
\newtheorem{proposition}[theorem]{Proposition}
\newtheorem{corollary}[theorem]{Corollary}
\newtheorem{assumption}[theorem]{Assumption}

\newtheorem{example}[theorem]{Example}
\theoremstyle{definition}

\newtheorem{remark}[theorem]{Remark}

\definecolor{theoremframe}{HTML}{2F6FA3}
\definecolor{theoremback}{HTML}{F0F6FA}
\definecolor{lemmaframe}{HTML}{7A4E9D}
\definecolor{lemmaback}{HTML}{F7F1FA}
\definecolor{propositionframe}{HTML}{00858C}
\definecolor{propositionback}{HTML}{EFF8F8}
\definecolor{corollaryframe}{HTML}{788A25}
\definecolor{corollaryback}{HTML}{F7F8ED}
\definecolor{assumptionframe}{HTML}{287A46}
\definecolor{assumptionback}{HTML}{F0F8F2}
\definecolor{definitionframe}{HTML}{606873}
\definecolor{definitionback}{HTML}{F3F4F5}
\definecolor{remarkframe}{HTML}{B06A00}
\definecolor{remarkback}{HTML}{FCF6EB}
\definecolor{exampleframe}{HTML}{A94E3D}
\definecolor{exampleback}{HTML}{FCF2EF}

\tcbset{
  paperboxbase/.style={
    enhanced,
    breakable,
    boxsep=0pt,
    left=5pt,
    right=5pt,
    top=1.5pt,
    bottom=1.5pt,
    before skip=3pt,
    after skip=3pt,
    before upper={\setlength{\topsep}{0pt}\setlength{\partopsep}{0pt}}
  },
  theoremboxstyle/.style={
    paperboxbase,
    colback=theoremback,
    colframe=theoremframe,
    boxrule=0.65pt,
    arc=1.5pt,
    outer arc=1.5pt,
    top=2.5pt,
    bottom=2.5pt,
    before skip=6pt,
    after skip=6pt
  },
  lemmaboxstyle/.style={
    paperboxbase,
    colback=lemmaback,
    frame hidden,
    borderline west={2pt}{0pt}{lemmaframe},
    borderline east={0.4pt}{0pt}{lemmaframe},
    sharp corners,
    top=1pt,
    bottom=1pt,
    before skip=1pt,
    after skip=1pt
  },
  propositionboxstyle/.style={
    paperboxbase,
    colback=propositionback,
    frame hidden,
    borderline north={1.2pt}{0pt}{propositionframe},
    borderline south={0.45pt}{0pt}{propositionframe},
    sharp corners
  },
  corollaryboxstyle/.style={
    paperboxbase,
    colback=corollaryback,
    frame hidden,
    borderline west={1.8pt}{0pt}{corollaryframe},
    borderline south={0.45pt}{0pt}{corollaryframe},
    sharp corners
  },
  assumptionboxstyle/.style={
    paperboxbase,
    colback=assumptionback,
    frame hidden,
    borderline west={2.4pt}{0pt}{assumptionframe},
    sharp corners,
    top=1pt,
    bottom=1pt,
    before skip=1pt,
    after skip=1pt
  },
  definitionboxstyle/.style={
    paperboxbase,
    colback=definitionback,
    colframe=definitionframe,
    boxrule=0.4pt,
    leftrule=1.8pt,
    sharp corners
  },
  remarkboxstyle/.style={
    paperboxbase,
    colback=remarkback,
    colframe=remarkframe,
    boxrule=0.35pt,
    toprule=1.5pt,
    arc=1pt,
    outer arc=1pt,
    top=1pt,
    bottom=1pt,
    before skip=1pt,
    after skip=1pt
  },
  exampleboxstyle/.style={
    paperboxbase,
    colback=exampleback,
    colframe=exampleframe,
    boxrule=0.4pt,
    leftrule=2pt,
    arc=3pt,
    outer arc=3pt,
    breakable=false
  }
}

\tcolorboxenvironment{theorem}{theoremboxstyle}
\tcolorboxenvironment{nonuniformtheorem}{theoremboxstyle}
\tcolorboxenvironment{lemma}{lemmaboxstyle}
\tcolorboxenvironment{proposition}{propositionboxstyle}
\tcolorboxenvironment{corollary}{corollaryboxstyle}
\tcolorboxenvironment{assumption}{assumptionboxstyle}
\tcolorboxenvironment{coordinateassumption}{assumptionboxstyle}
\tcolorboxenvironment{definition}{definitionboxstyle}
\tcolorboxenvironment{remark}{remarkboxstyle}
\tcolorboxenvironment{example}{exampleboxstyle}

\newenvironment{mainresultbox}{}{}
\newenvironment{assumptionbox}{}{}
\newenvironment{lemmabox}{}{}
\newtcolorbox{remarkbox}{remarkboxstyle}
\newtcolorbox{examplebox}{exampleboxstyle}

\newcommand{\Obound}[1]{\mathcal{O}\left( #1 \right)}
\newcommand{\OboundTilde}[1]{\tilde{\mathcal{O}}\left( #1 \right)}

\newcommand{\F}{F}

\providecommand{\argmin}{\operatorname*{arg\,min}}
\providecommand{\E}{}
\renewcommand{\E}{\mathbb{E}}
\newcommand{\inner}[2]{\left\langle #1,#2\right\rangle}
\newcommand{\norm}[1]{\left\lVert #1\right\rVert}

\definecolor{PineGreen}{HTML}{008B72}
\newcommand{\greencheck}{\color{PineGreen}\ding{51}}
\newcommand{\redx}{\color{red} \ding{55}}

\newcommand{\circledOne}{\text{\ding{172}}}
\newcommand{\circledTwo}{\text{\ding{173}}}
\newcommand{\circledThree}{\text{\ding{174}}}
\newcommand{\circledFour}{\text{\ding{175}}}
\newcommand{\circledFive}{\text{\ding{176}}}
\newcommand{\circledSix}{\text{\ding{177}}}
\newcommand{\circledSeven}{\text{\ding{178}}}
\newcommand{\circledEight}{\text{\ding{179}}}

\definecolor{revisionblue}{RGB}{0,0,0}
\DeclareCaptionFont{revisionblue}{\color{revisionblue}}
\DeclareRobustCommand{\grev}[1]{\textcolor{revisionblue}{#1}}
\newenvironment{greenrevision}{\begingroup\color{revisionblue}}{\endgroup}
\DeclareRobustCommand{\rev}[1]{\textcolor{revisionblue}{#1}}
\newenvironment{revision}{\par\color{revisionblue}\tcbset{coltext=revisionblue}}{\par}
\definecolor{currentrevisiongreen}{RGB}{0,0,0}
\DeclareRobustCommand{\newrev}[1]{\textcolor{currentrevisiongreen}{#1}}
\newenvironment{currentrevision}{\begingroup\color{currentrevisiongreen}}{\endgroup}
\definecolor{currentrevisionorange}{RGB}{0,0,0}
\DeclareRobustCommand{\orrev}[1]{\textcolor{currentrevisionorange}{#1}}

\definecolor{revisionpurple}{RGB}{0,0,0}
\DeclareRobustCommand{\purrev}[1]{\textcolor{revisionpurple}{#1}}
\DeclareCaptionFont{revisionpurple}{\color{revisionpurple}}
\newenvironment{purplerevision}{\begingroup\color{revisionpurple}}{\endgroup}
\definecolor{experimentred}{RGB}{0,0,0}

\DeclareCaptionFont{experimentred}{\color{experimentred}}
\newif\ifexperimentrevision
\AtBeginEnvironment{tabular}{\ifexperimentrevision\color{experimentred}\fi}
\newcommand{\submittednumber}[2]{\expandafter\def\csname submittednumber@#1\endcsname{#2}}
\submittednumber{alg:restart-meta}{1}
\submittednumber{alg:full-gradient-inner}{2}
\submittednumber{alg:coordinate-inner}{3}
\submittednumber{alg:nonuniform-coordinate-inner}{4}
\submittednumber{thm:nonuniform-coordinate-main}{8}
\submittednumber{sec:discussion-limitations}{6}
\submittednumber{def:inexact-segment-relaxation}{12}
\submittednumber{lem:inexact-segment-potential}{13}
\submittednumber{cor:inexact-phase-schedule}{14}
\submittednumber{app:inexact-relaxation}{G.1}
\submittednumber{app:adaptive-phase-lengths}{G.2}
\newcommand{\experimentref}[1]{%
  \ifcsname submittednumber@#1\endcsname
    \csname submittednumber@#1\endcsname
  \else
    \liveexperimentref{#1}%
  \fi
}
\newenvironment{experimentrevision}{%
  \par\experimentrevisiontrue\color{experimentred}\captionsetup{font+=experimentred}%
  \let\liveexperimentref\ref\let\ref\experimentref
}{\par}
\newcommand{\revisionbibliography}[1]{%
  \begingroup
  \let\originalbibitem\bibitem
  \renewcommand{\bibitem}[2][]{%
    \originalbibitem[##1]{##2}%
    \ifcsname revisionbib@##2\endcsname\color{revisionblue}\else\color{black}\fi
  }%
  \bibliography{#1}%
  \endgroup
}
\expandafter\def\csname revisionbib@lobanov2026restarted\endcsname{}
\expandafter\def\csname revisionbib@elbakkali2026acceleration\endcsname{}
\expandafter\def\csname revisionbib@rudomanov2026everest\endcsname{}
\expandafter\def\csname revisionbib@fanaee2013bike\endcsname{}
\expandafter\def\csname revisionbib@hebrail2006household\endcsname{}

\usepackage{hyperref}
\usepackage{url}

\title{A Few Accelerated Algorithms for Convex Optimization under $(H_0,H_1)$-Smoothness\thanks{Preprint. December 3, 2026 (First version: August 5, 2026).}}

\iclrfinalcopy

\author{Aleksandr Lobanov \\
MSU AI Center, IAI MSU, HSE\\
\texttt{lobanovav@my.msu.ru}
}

\begin{document}

\vspace*{-1em}
\maketitle
\thispagestyle{empty}

\vspace{-1.3em}

\begin{abstract}
\vspace{-1em}
We develop accelerated algorithms for convex \((H_0,H_1)\)-smooth optimization, where \(\norm{\nabla^2 f(x)}\le H_0+H_1(f(x)-f^*)\) \rev{and \(f^*:=\min_x f(x)\)}. This class generalizes standard smoothness and \rev{strictly} contains the \((L_0,L_1)\)-smooth class. Combining \rev{Nesterov acceleration with a two-regime weight schedule and a realized curvature correction}, we obtain an asymptotic bound of \rev{\(O(R\sqrt{H_0/\varepsilon}+R\sqrt{H_1}\log(\Delta_0/\varepsilon))\)} full-gradient queries, \rev{where \(R\ge\norm{x^0-x^*}\), \(\Delta_0\ge f(x^0)-f^*\) is a known upper bound, and \(\varepsilon\) is the desired accuracy}. \rev{The approach also yields randomized coordinate methods: curvature-dependent terms carry \(d\) under uniform sampling and are governed by} \(S_{1/2}^{(j)}=\sum_i\sqrt{H_{j,i}}\) \rev{under non-uniform sampling}. These results provide, to our knowledge, the first accelerated full-gradient and coordinate guarantees for this convex class. \rev{Moreover, for full-gradient linear-span and uniformly sampled coordinate methods, lower bounds show that the dependence on \(H_1\) is optimal up to logarithmic factors.} \rev{Experiments on two real-data generalized linear models compare our method and its closest accelerated counterparts in iteration count, oracle cost, and wall-clock performance.} \end{abstract}

\vspace{-0.5em}
\section{Introduction}

Gradient descent is one of the basic computational primitives of modern
optimization.  Its stochastic and adaptive descendants, including SGD,
momentum, AdaGrad, Adam, AdamW, and more recent matrix-valued optimizers such as
Muon, are the default tools behind large-scale machine learning
\citep{bottou2018optimization,duchi2011adagrad,kingma2015adam,loshchilov2019adamw,liu2025muon}.
Even when these methods differ in normalization, adaptivity, momentum, or
matrix preconditioning, they share the same first-order nature: each iteration
uses local descent information and turns it into a parameter update.  This is
why a sharper understanding of gradient-based optimization assumptions often
translates directly into sharper theory for practical training algorithms.

The importance of this question is particularly visible in deep learning and
large language model training.  Modern architectures are trained almost entirely
with first-order methods and their adaptive variants, from Transformers
\citep{vaswani2017attention} to large-scale language models
\citep{brown2020language}.  At the same time, successful training recipes rely
on mechanisms that are not naturally explained by the classical globally smooth
model: gradient clipping is used to stabilize training with large or transient
gradients \citep{pascanu2013difficulty,zhang2019gradientclipping}, learning-rate
warmup is standard in large-batch and Transformer training
\citep{goyal2017largebatch,vaswani2017attention}, and clipped or normalized
updates are particularly effective in overparameterized regimes
\citep{lobanov2026clipped}.  These empirical practices suggest that
relevant curvature is not a single global constant.  It changes along the
trajectory and often becomes much smaller as training~approaches~a~good~solution.

Classical accelerated theory gives optimal guarantees for convex optimization
under a global Lipschitz bound on the gradient
\citep{nesterov1983method,nesterov2018lectures}.  This assumption is sharp as a
worst-case abstraction, but it can be too conservative for objectives whose
curvature decreases during training.  This gap between practice and theory
motivated generalized smoothness assumptions.  The most studied example is
\((L_0,L_1)\)-smoothness,
$
  \norm{\nabla^2 f(x)}
  \le
  L_0+L_1\norm{\nabla f(x)},
$
which was introduced to model the observed dependence of local smoothness on
the gradient norm in deep learning and was later developed as a framework for
clipping, normalization, and convex optimization under generalized smoothness
\citep{zhang2019gradientclipping,koloskova2023clipping,li2023generalized,lobanov2024linear,vankov2024l0l1,tyurin2024unified,gorbunov2024l0l1,lobanov2025power,gaash2025clipped,tyurin2025near, borodich2026nesterov,lobanov2026clipped}.
More recently, gap-dependent smoothness assumptions, in particular
\((H_0,H_1)\)-smoothness assumptions, were proposed, where the curvature is
controlled by the current optimality gap rather than by the gradient norm.
Following this line of work
\citep{vaswani2025armijo,liu2025warmup,alimisis2025warmup,lobanov2026clipped,vaswani2026steepestadam},
in this paper we focus on convex objectives satisfying the
\((H_0,H_1)\)-smoothness condition
\begin{equation}
\label{eq:h0h1-smoothness}
  \norm{\nabla^2 f(x)}
  \le
  H_0+H_1\bigl(f(x)-f^*\bigr),
  \qquad H_0,H_1\ge 0 .
\end{equation}
When \(H_1=0\), this is the standard \(H_0\)-smooth convex setting, while
\(H_1>0\) allows the curvature to be large far from the optimum and to decrease
with the optimization gap.  Importantly, this model is not disjoint from the
\((L_0,L_1)\)-smooth literature: it covers the \((L_0,L_1)\)-smooth class and
is strictly broader, since there are \((H_0,H_1)\)-smooth functions that are
not \((L_0,L_1)\)-smooth; see
\rev{Appendix~\ref{app:l0l1-implication}}.

The strict inclusion makes acceleration under \((H_0,H_1)\)-smoothness
distinct from the existing \((L_0,L_1)\)-smooth theory.  For the latter class,
\citet{vankov2024l0l1} obtain
\(\widetilde O(\sqrt{L_0R^2/\varepsilon}
+(L_1R)^{2/3}\log(F_0/\varepsilon))\),
where \(F_0:=f(x^0)-f^*\).
Another accelerated bound is exponential in \(L_1R\) \citep{gorbunov2024l0l1}.
Neither result yields acceleration in terms of \(H_0\), \(H_1\), leading to
our central question:
\vspace{-1em}
\begin{center}
\emph{Can one obtain accelerated rates for the broader convex\\
\((H_0,H_1)\)-smooth class, and how should the complexity depend on \(H_0\)
and \(H_1\)?}
\end{center}

A naive answer would be to freeze the curvature
\rev{using an upper bound \(\Delta_0\) on the initial gap}.  Classical
acceleration \rev{then} gives
\[
  O\left(
    \rev{R}\sqrt{\frac{H_0+H_1\rev{\Delta_0}}{\varepsilon}}
  \right).
\]
This bound ignores the fact that the \(H_1F(x)\) part \rev{of the curvature} should decrease as the
method approaches the solution set, where \(F(x):=f(x)-f^*\).  Our goal is to separate the baseline
curvature \(H_0\) from the gap-dependent curvature \(H_1F(x)\) and obtain
the asymptotic bound
\[
  \rev{O}\left(
    \rev{R}\left(
      \sqrt{\frac{H_0}{\varepsilon}}
      +
      \sqrt{H_1}\log\frac{\Delta_0}{\varepsilon}
    \right)
  \right).
\]
This bound has the classical accelerated dependence on the baseline curvature
\(H_0\).  The \(H_1\)-term also has an accelerated, strongly-convex-like
form: it contributes \(\rev{R}\sqrt{H_1}\log(\Delta_0/\varepsilon)\), namely a
logarithmic dependence on the desired accuracy.  In particular, the rate avoids the
\(\sqrt{\rev{\Delta_0}/\varepsilon}\) penalty that would arise from freezing the curvature
\rev{using the initial-gap upper bound}.

We also study the coordinate version of this question.  Coordinate descent
methods are designed for large-scale problems in which a full gradient is too
expensive, while a single partial derivative or a small block update can be
computed cheaply and can exploit sparsity or separable structure.  This
granularity is central in huge-scale convex optimization and has led to a rich
accelerated theory under global coordinate smoothness
\citep{nesterov2012coordinate,richtarik2014iteration,fercoq2015accelerated,allenzhu2016evenfaster}.
However, these guarantees rely on fixed coordinate Lipschitz constants.  To the
best of our knowledge, accelerated results for convex \((L_0,L_1)\)-smooth
optimization were not previously available for coordinate methods.  Since the
\((H_0,H_1)\)-smooth model strictly contains the \((L_0,L_1)\)-smooth class,
the coordinate case poses the sharper challenge of obtaining acceleration in
the broader gap-dependent smoothness regime, namely under
\((H_0,H_1)\)-smoothness:
\begin{center}
\emph{Can one obtain coordinate acceleration without replacing the problem by
pessimistic \\coordinate smoothness constants
\rev{based on the initial gap upper bound}?}
\end{center}

In answering this question, we aim for the coordinate analogue of the
full-gradient accelerated rate: the same \((H_0,H_1)\)-dependence, up to the
standard factor \(d\) for using one partial derivative per iteration.  Thus a
randomized coordinate method should use
\begin{equation*}
  \rev{O}\left(
    d\rev{R}\left(
      \sqrt{\frac{H_0}{\varepsilon}}
      +
      \sqrt{H_1}\log\frac{\Delta_0}{\varepsilon}
    \right)
  \right)
\end{equation*}
coordinate-gradient calls, preserving acceleration in \(H_0\) and the
strongly-convex-like logarithmic \(H_1\)-term.


\vspace{-0.5em}
\section{Related Work and Main Contributions}
\label{sec:related-work-contributions}
\vspace{-0.5em}

\begin{table*}
    \captionsetup{skip=2pt}
    \caption{\small Comparison of the iteration complexity ($\# N$) of state-of-the-art full-gradient and coordinate algorithms. Notation: $d$ is the problem dimension; $S_{1/2}= \sum_{i=1}^d \sqrt{L_i}$ ($L_i$ are Lipschitz constants  for partial derivatives); $\varepsilon$ is the desired solution accuracy; $R \rev{\ge} \norm{\rev{x^0} - x^*}$, \(F_0:=f(x^0)-f^*\), and \(\Delta_0\ge F_0\) is a known upper bound; FM stands for a full-gradient method; CM stands for a coordinate method.}
    \label{tab:table_compare}
    \centering
    \begingroup
    \renewcommand{\arraystretch}{1.18}
    \resizebox{\linewidth}{!}{%
    \begin{tabular}{||l | c | c c | c c c||}
    \hline
    \hline
    \rule{0pt}{2.6ex}\multirow{2}{*}{Reference}
    & \multirow{2}{*}{$\# N$}
    & \multicolumn{2}{c|}{Method}
    & \multirow{2}{*}{$(L_0,L_1)$?}
    & \multirow{2}{*}{$(H_0,H_1)$?}
    & \multirow{2}{*}{Acceleration?}
    \\ \cline{3-4}
    &
    &
    FM?
    & CM?
    &
    &
    &
    \\ \midrule
    \citep{Nesterov_2017}
    & $\Obound{\sqrt{\frac{S_{1/2}^2 \rev{R}^2}{\varepsilon}}}$
    & \redx
    & \greencheck
    & \redx
    & \redx
    & \greencheck
    \\
    \citep{lobanov2024linear}
    & $\Obound{\max\left\{\frac{d L_0 R^2}{\varepsilon}, dL_1R \log\frac{F_0}{\varepsilon} \right\}}$
    & \redx
    & \greencheck
    & \greencheck
    & \redx
    & \redx
    \\
    \citep{lobanov2024linear} $\& $ \citet{vankov2024l0l1}
    & $\Obound{\frac{L_0 R^2}{\varepsilon} + L_1 R \log\frac{F_0}{\varepsilon}}$
    & \greencheck
    & \redx
    & \greencheck
    & \redx
    & \redx
    \\
    \citet{vankov2024l0l1}
    & $\OboundTilde{\sqrt{\frac{L_0 R^2}{\varepsilon}} + \left( L_1R \right)^{2/3}\log\frac{F_0}{\varepsilon}}$
    & \greencheck
    & \redx
    & \greencheck
    & \redx
    & \greencheck
    \\
    \citep{lobanov2026clipped}
    & $\Obound{\frac{H_0 R^2}{\varepsilon} + H_1 R^2\log\frac{F_0}{\varepsilon}}$
    & \greencheck
    & \redx
    & \greencheck
    & \greencheck
    & \redx
    \\ \midrule
    \multirow{2}{*}{This paper}
    & $\rev{O}\left(\rev{R}\sqrt{\frac{H_0}{\varepsilon}} + \rev{R}\sqrt{H_1}\log\frac{\Delta_0}{\varepsilon}\right)$
    & \greencheck
    & \redx
    & \greencheck
    & \greencheck
    & \greencheck
    \\
    & $\rev{O}\left(d\rev{R}\sqrt{\frac{H_0}{\varepsilon}} + d\rev{R}\sqrt{H_1}\log\frac{\Delta_0}{\varepsilon}\right)$
    & \redx
    & \greencheck
    & \greencheck
    & \greencheck
    & \greencheck
    \\
    \hline
    \hline
    \end{tabular}%
    }
    \endgroup
    \vspace{-0.9em}
\end{table*}
The literature on gradient-descent variants under generalized smoothness is
already extensive and remains very active.  We discuss below the contributions
most relevant to our work and display a representative selection in
Table~\ref{tab:table_compare}.
\vspace{-0.5em}
\paragraph{Full-gradient methods.}
\vspace{-0.5em}

Although our results are for convex objectives, the generalized smoothness line was first
motivated by non-convex training.  In particular,
\citet{zhang2019gradientclipping} introduced \((L_0,L_1)\)-smoothness for twice
differentiable objectives,
\(\norm{\nabla^2 f(x)}\le L_0+L_1\norm{\nabla f(x)}\), as a relaxation of global
\(L\)-smoothness motivated by the empirical correlation between local curvature
and gradient norm during neural-network training.  This condition was later
reformulated in a first-order form by \citet{zhang2020improvedclipping},
extending the \((L_0,L_1)\)-smooth framework to differentiable functions that
need not be twice differentiable; subsequent work developed this direction in
non-convex and stochastic settings
\citep{crawshaw2022signsgd,chen2023generalizedsmooth,wang2023adagradrelaxed,koloskova2023clipping,li2023generalized,li2024adamrelaxed,hubler2024parameteragnostic,vankov2024l0l1}.
For convex \((L_0,L_1)\)-smooth objectives, the non-accelerated theory has
progressed from results that still relied on an additional global \(L\)-smoothness
assumption \citep{koloskova2023clipping,takezawa2024parameter}, to bounds without
global smoothness \citep{gorbunov2024l0l1,lobanov2024linear,vankov2024l0l1}.
\citet{gorbunov2024l0l1} obtained
\(\Obound{L_0R^2/\varepsilon+L_1^2R^2}\), where the second term does not depend on
the desired accuracy.  This was sharpened by
\citet{lobanov2024linear,vankov2024l0l1} to the mixed rate
\(\Obound{L_0R^2/\varepsilon+L_1R\log(F_0/\varepsilon)}\).
Accelerated methods under \((L_0,L_1)\)-smoothness were studied by
\citet{li2023generalized,gorbunov2024l0l1,vankov2024l0l1}.  The first two
bounds couple \(L_0\) and \(L_1\) or involve exponential or level-set factors,
whereas \citet{vankov2024l0l1} separates the two contributions and obtains
\(\widetilde O(\sqrt{L_0R^2/\varepsilon}
+(L_1R)^{2/3}\log(F_0/\varepsilon))\).  These results characterize
acceleration in the native \((L_0,L_1)\) parametrization and do not directly
provide a rate in terms of \(H_0\) and \(H_1\).  The gap-dependent
\((H_0,H_1)\)-smoothness model,
\(\norm{\nabla^2 f(x)}\le H_0+H_1(f(x)-f^*)\), was recently studied in
\citep{vaswani2025armijo,liu2025warmup,alimisis2025warmup} and is weaker than
the \((L_0,L_1)\)-smooth model.  Existing convex guarantees in this regime are
non-accelerated: \citet{lobanov2026clipped} obtains
\(\Obound{H_0R^2/\varepsilon+H_1R^2\log(F_0/\varepsilon)}\).
To the best of our knowledge, no accelerated full-gradient rate was previously
known for convex \((H_0,H_1)\)-smooth objectives.  Our method achieves the asymptotic rate
\(\rev{O}\bigl(\sqrt{H_0\rev{R}^2/\varepsilon}
+\sqrt{H_1\rev{R}^2}\log(\Delta_0/\varepsilon)\bigr)\)
\rev{without global \(L\)-smoothness.}
\rev{Moreover, we show that for full-gradient linear-span methods (including
standard first-order algorithms), the \(H_1\)-term is optimal up to logarithmic
factors.}
\vspace{-1em}
\paragraph{Coordinate methods.}
Coordinate descent methods are a natural alternative to full-gradient methods in
large-scale problems, where reading all coordinates of the gradient at every
iteration may be substantially more expensive than querying a single partial
derivative.  Their classical theory is well developed under fixed coordinate
smoothness constants \(L_i\): randomized coordinate descent and its accelerated,
parallel, proximal, and non-uniformly sampled variants achieve rates governed by
coordinate smoothness summaries such as \(S_{1/2}=\sum_i\sqrt{L_i}\)
\citep{nesterov2012coordinate,richtarik2014iteration,fercoq2015accelerated,Nesterov_2017,allenzhu2016evenfaster}.
These guarantees are the coordinate analogue of classical smooth optimization:
they are sharp in the global model, but they require global coordinate
Lipschitz constants and therefore inherit the same pessimism when local
curvature is better described by generalized smoothness.  Much less is known in
this regime.  For convex \((L_0,L_1)\)-smooth objectives,
\citet{lobanov2024linear} obtained a coordinate method with complexity
\(\Obound{\max\{dL_0R^2/\varepsilon,dL_1R\log(F_0/\varepsilon)\}}\), but this is
a non-accelerated rate.  To the best of our knowledge, no previous coordinate
method is both accelerated and applicable to convex \((L_0,L_1)\)-smooth
objectives, and no accelerated coordinate result was known for the broader
\((H_0,H_1)\)-smooth class.  Our coordinate result fills this gap: it gives an
accelerated randomized coordinate method under convex \((H_0,H_1)\)-smoothness,
applies in particular to the \((L_0,L_1)\)-smooth subclass, and pays only the
natural additional factor \(d\) in coordinate-gradient calls relative to the
full-gradient rate.
\vspace{-0.5em}

\section{Problem setup}
\label{sec:problem-setup}
\vspace{-0.5em}
We consider the unconstrained minimization problem
\begin{equation}
\label{eq:problem}
  \min_{x\in\mathbb{R}^d} f(x),
\end{equation}
where the objective \(f:\mathbb{R}^d\to\mathbb{R}\) maps a parameter vector to a
scalar loss value.  Throughout the paper we use the Euclidean norm for vectors
and the induced spectral norm for matrices.
We assume that the solution set
\(X^*:=\argmin_{x\in\mathbb{R}^d}f(x)\) is nonempty\rev{, and we} fix
an~arbitrary~minimizer~${x^*\in X^*}$\rev{ and set \(f^*:=f(x^*)\)}.

\begin{assumptionbox}
\begin{assumption}[Convexity]
\label{ass:convexity}
The function \(f\) is differentiable and convex, that is,
\[
  f(y)
  \ge
  f(x)+\inner{\nabla f(x)}{y-x},
  \qquad
  \forall x,y\in\mathbb{R}^d .
\]
\end{assumption}
\end{assumptionbox}

As discussed above, this paper focuses on \((H_0,H_1)\)-smoothness
\citep{liu2025warmup,alimisis2025warmup,lobanov2026clipped}.

\begin{assumptionbox}
\begin{assumption}[\((H_0,H_1)\)-smoothness]
\label{ass:h0h1-smoothness}
There exist constants \(H_0\rev{>}0\) and \(H_1\rev{\ge}0\).  Let
\(\Ls(\Delta):=2(H_0+H_1\Delta)\) and
\(r_1:=1/((2\sqrt3+\sqrt6)\sqrt{H_1})\).  For all
\(x,y\in\mathbb{R}^d\) satisfying \(\norm{y-x}\le r_1\),
\[
  f(y)
  \le
  f(x)+\inner{\nabla f(x)}{y-x}
  +\frac{\Ls(\F(x))}{2}\norm{y-x}^2 .
\]
\end{assumption}
\end{assumptionbox}

In Assumption~\ref{ass:h0h1-smoothness}, the limiting case \(H_1=0\) is the
usual globally smooth convex setting: one can take \(\Ls(\Delta)=2H_0\) and
\(r_1=\infty\).  The pure \(H_1\)-driven regime
\(H_0=0\) is covered when a minimizer exists.  For positive objectives where a
solution may fail to exist, one should instead use an uncentered positive-function
variant, with the curvature controlled by \(f(x)\) rather than \(f(x)-f^*\), as
in \citet{vaswani2025armijo,vaswani2026steepestadam}.
\rev{In this paper, we focus on the regime \(H_0>0\).}

Assumption~\ref{ass:h0h1-smoothness} is a first-order local-model formulation
and does not require second derivatives.  For twice continuously differentiable
convex functions, it corresponds, up to absolute constants, to
\(\norm{\nabla^2 f(x)}\le H_0+H_1(f(x)-f^*)\); see
\citet[Lemma~2 and Appendix~C.2]{liu2025warmup}.

Moreover, Assumption~\ref{ass:h0h1-smoothness} is strictly weaker than the
\((L_0,L_1)\)-smooth framework.  First,
\newrev{the conversion in Appendix~\ref{app:l0l1-inclusion} shows} that every convex
\rev{\((L_0,L_1)\)-smooth function} satisfies
Assumption~\ref{ass:h0h1-smoothness} with
\newrev{\(H_0=2L_0\) and \(H_1=4L_1^2\) when \(L_0>0\); if \(L_0=0\),
any \(H_0>0\) suffices}.  Second,
Examples~\ref{ex:one-dimensional-separation}
and~\ref{ex:multidimensional-separation} are convex \((H_0,H_1)\)-smooth
functions but fail to satisfy any \((L_0,L_1)\)-smoothness condition with finite
constants.  \rev{See Appendix~\ref{app:l0l1-implication} for the proofs.}

\begin{example}[One-dimensional]
\label{ex:one-dimensional-separation}
Fix \(p\in(1,2]\).  Let
\(q_p(t):=1+(1+t^2)^{p/2}\exp(-(1+t^2)^p\sin^2(\pi t))\), and define
\[
  \phi_p(x)
  =
  \int_0^x\int_0^s q_p(t)\,dt\,ds .
\]
Then \(\phi_p\) is convex, \((H_0,H_1)\)-smooth  (\(H_0\text{=}H_1\text{=}2\)),
but not \((L_0,L_1)\)-smooth~for~finite~constants.
\end{example}

\begin{example}[Multidimensional]
\label{ex:multidimensional-separation}
For \(d\ge2\), let
\[
  \Phi_p(x_1,\ldots,x_d)
  =
  \phi_p(x_1)+\frac12\sum_{i=2}^d x_i^2,
\]
where \(\phi_p\) is the function from
Example~\ref{ex:one-dimensional-separation}.  Then \(\Phi_p\) is strongly
convex and \((H_0,H_1)\)-smooth with \(H_0=H_1=2\), but not
\((L_0,L_1)\)-smooth for finite constants.
\end{example}

Examples~\ref{ex:one-dimensional-separation}
and~\ref{ex:multidimensional-separation} contain narrow high-curvature peaks.
The Hessian can be large, while the gradient, which integrates curvature over a
neighborhood, remains much smaller.  This separates
\((H_0,H_1)\)-smoothness from assumptions that force curvature to scale with
\(\norm{\nabla f(x)}\).  In machine learning tasks, where losses may be highly
anisotropic and locally sensitive along only a few parameter directions, this
distinction is useful because local sensitivity and accumulated first-order
signal need not vary at the same rate.
\vspace{-0.5em}
\section{Accelerated Algorithms}
\label{sec:accelerated-algorithms}
\vspace{-0.5em}

We now begin the presentation of our main algorithmic results.
\rev{We consider three oracle settings and present a separate accelerated
algorithm for each: the full-gradient setting, the coordinate setting with
uniform sampling, and the coordinate setting with non-uniform sampling.}
\rev{To prove the guarantees for these three algorithms without repeating the
same argument, in Appendix~\ref{app:unified-acceleration} we introduce and
analyze a unified block-coordinate method that covers all three settings; see
Algorithm~\ref{alg:accelerated-block} and
Theorem~\ref{thm:common-block}. The full-gradient, uniformly sampled
coordinate, and non-uniformly sampled coordinate guarantees are stated, in
the same order, in Theorems~\ref{thm:full-gradient-main},
\ref{thm:coordinate-main}, and~\ref{thm:nonuniform-coordinate-main}, and
proved in Appendices~\ref{app:full-gradient-proof},
\ref{app:coordinate-proof}, and~\ref{app:nonuniform-coordinate-proof}. These
proofs reduce the unified result to its single-block, uniformly sampled
singleton-block, and non-uniformly sampled singleton-block specializations,
respectively.}

\vspace{-0.5em}
\subsection{Full-Gradient Setup}
\label{sec:full-gradient}

In this subsection, we \rev{consider} the full-gradient oracle setting, where
the algorithm can query \(\nabla f(x)\) at each iterate. Our construction
\rev{builds on the classical estimate-sequence coupling of Nesterov's
accelerated full-gradient method
\citep{nesterov1983method,nesterov2018lectures} and uses the two-regime weight
principle of \citet{elbakkali2026acceleration}.} The resulting full-gradient
\rev{method} is presented in Algorithm~\ref{alg:full-gradient-inner}.

\begingroup
\setlength{\intextsep}{4pt}
\begin{algorithm}[H]
\caption{\rev{Accelerated Gradient Method}}
\label{alg:full-gradient-inner}
\color{revisionblue}
\begin{algorithmic}[1]
\renewcommand{\algorithmicrequire}{\textbf{Input:}}
\renewcommand{\algorithmicensure}{\textbf{Return:}}
\REQUIRE Objective \(f\), initial point \(x^0\), iteration budget \(N\),
 initial weight \(A_0>0\), and prescribed rules for the
 \newrev{query weight \(\widehat b_k>0\), update weight \(b_k>0\),} and
 \newrev{\(c_k\ge0\)}
\STATE Initialize \(x_0\gets x^0\), \(v_0\gets x^0\)
\FOR{\(k=0,\ldots,N-1\)}
  \STATE Set \(y_k\gets\dfrac{A_kx_k+\newrev{\widehat b_k}v_k}
  {A_k+\newrev{\widehat b_k}}\)
  \STATE Update \(v_{k+1}\gets v_k-b_k\nabla f(y_k)\)
  \STATE Set
  \(a_k\gets A_k+b_k-
  \dfrac{c_kb_k^2}{A_k+b_k}\norm{v_{k+1}-v_k}^2\)
  \STATE Set \(z_k\gets
  \newrev{\dfrac{A_kx_k+b_kv_k}{A_k+b_k}}
  -\dfrac{b_k^2}{a_k}\nabla f(y_k)\)
  \STATE Update \(A_{k+1}\gets\max\{A_0,a_k\}\)
  \STATE Update
  \(x_{k+1}\gets\dfrac{a_kz_k+(A_{k+1}-a_k)x^0}{A_{k+1}}\)
\ENDFOR
\ENSURE \(x_N\)
\end{algorithmic}
\end{algorithm}
\endgroup
\vspace{-0.6em}

\rev{Algorithm~\ref{alg:full-gradient-inner} maintains the primal sequence
\(x_k\), the auxiliary accelerated sequence \(v_k\), and the accepted
acceleration weight \(A_k\). The
point \(y_k\) is the classical Nesterov coupling point: it combines \(x_k\) and
\(v_k\) so that the same gradient drives both the primal and auxiliary updates.
The weight \(b_k\) controls this coupling; the rule specified in
Theorem~\ref{thm:full-gradient-main} limits weight growth in the regime governed
by \(H_1\) and recovers the classical accelerated square-root growth governed
by \(H_0\). The update of
\(v_{k+1}\) accumulates the gradient information used for acceleration. Because
the local curvature contains the gap-dependent term \(H_1\F(y_k)\), the
realized displacement \(v_{k+1}-v_k\) produces an additional curvature
contribution. The definition of \(a_k\) subtracts a controlled multiple of this
contribution from the unadjusted weight \(A_k+b_k\), with \(c_k\) determining
the size of the correction. The step size \(b_k^2/a_k\) in the definition of
\(z_k\) is chosen so that the gradient term produced by this step matches the
corresponding term from the \(v_{k+1}\)-update in the potential analysis.
Finally, \(A_{k+1}=\max\{A_0,a_k\}\) keeps the accepted weight at least \(A_0\):
if \(a_k\ge A_0\), then \(x_{k+1}=z_k\); otherwise, \(z_k\) is averaged with
\(x^0\). The method uses one full-gradient query per iteration and returns
\(x_N\).}

\rev{Algorithm~\ref{alg:full-gradient-inner} combines the estimate-sequence
coupling of classical accelerated full-gradient methods
\citep{nesterov1983method,nesterov2018lectures} with the two-regime weight
principle of \citet{elbakkali2026acceleration}. Our new ingredient is the
realized correction that replaces \(A_k+b_k\) with \(a_k\) using the actual
displacement \(v_{k+1}-v_k\), together with the accepted-weight safeguard
\(A_{k+1}\ge A_0\) and the corresponding averaging with \(x^0\). This yields a
method that accounts for gap-dependent curvature using one full-gradient query
per iteration, without function-value evaluations, line search, or restarts.}

We now state the convergence guarantee for the full-gradient method in
Algorithm~\ref{alg:full-gradient-inner}.

\begin{mainresultbox}
\begin{theorem}[Full-gradient complexity]
\label{thm:full-gradient-main}
Suppose Assumptions~\ref{ass:convexity} and
\ref{ass:h0h1-smoothness} hold, \rev{\(\norm{x^0-x^*}\le R\) for some
\(R>0\), \(\Delta_0\ge f(x^0)-f^*\), and
\(0<\varepsilon\le\Delta_0\), where \(\varepsilon\) denotes the desired
accuracy. Set
\(A_0=R^2/(2\Delta_0)\) and choose the parameters in
Algorithm~\ref{alg:full-gradient-inner} as follows:}
\begin{revision}
\[
 c_k=
 \begin{cases}
 4\newrev{H_1},
 &\displaystyle \max_{0\le j\le k}A_j<
 \max\left\{A_0,
 \min\left\{
 \frac{(1+128R\sqrt{H_1})^2}{4096H_0},
 \frac{4eR^2}{\varepsilon}
 \right\}\right\},\\[0.9em]
 4\newrev{H_1}\left(1+\log\frac{2\Delta_0}{\varepsilon}\right),
 &\displaystyle \max_{0\le j\le k}A_j\ge
 \max\left\{A_0,
 \min\left\{
 \frac{(1+128R\sqrt{H_1})^2}{4096H_0},
 \frac{4eR^2}{\varepsilon}
 \right\}\right\},
 \end{cases}
\]
\[
 \newrev{\widehat b_k=}b_k=\frac{A_k}
 {1+\newrev{64R\sqrt{c_k}}+64\sqrt{H_0A_k}};
\]
\end{revision}
then Algorithm~\ref{alg:full-gradient-inner} returns a point
\rev{\(x_N\)} satisfying
\rev{\(f(x_N)-f^*\le\varepsilon\)} after
\begin{revision}
\[
 \boxed{\displaystyle
 N=O\left(
 \sqrt{\frac{H_0R^2}{\varepsilon}}
 +\left(1+\sqrt{H_1R^2}\right)
 \left(1+\log\frac{\Delta_0}{\varepsilon}\right)
 \right).}
\]
\end{revision}
\end{theorem}
\end{mainresultbox}

\rev{Algorithm~\ref{alg:full-gradient-inner} evaluates
\(\nabla f(y_k)\) once per iteration and reuses it in all updates. Therefore,
the iteration bound in Theorem~\ref{thm:full-gradient-main} is also its
full-gradient oracle complexity.}
When \(H_1=0\), the first term matches the classical Nesterov dependence
\citep{nesterov1983method}. \rev{Moreover, the lower bound
(see Theorem~\ref{thm:growth-lower-bounds}) shows that, for standard
full-gradient first-order algorithms satisfying the linear-span condition,
including Algorithm~\ref{alg:full-gradient-inner}, the \(\sqrt{H_1R^2}\)
dependence cannot be improved beyond logarithmic factors.}

\vspace{-1em}
\paragraph{Proof sketch.}
\rev{We specialize the unified method to a single block, so the block sampling
becomes deterministic. The common argument uses the potential
\(\Phi_k:=A_k\F(x_k)+\frac12\norm{v_k-x^*}^2\). It first verifies that
\(a_k>0\) and that \(z_k-y_k\) satisfies the radius condition in
Assumption~\ref{ass:h0h1-smoothness}. The auxiliary update, convexity, and
Assumption~\ref{ass:h0h1-smoothness} then give a one-step potential estimate,
in which the realized correction defining \(a_k\) accounts for the additional
\(H_1\)-dependent curvature. After normalizing the potential, the two choices
of \(c_k\), together with the rule for \(b_k\), control the initial
high-curvature regime and recover the classical accelerated weight growth in
the \(H_0\)-dominated regime. Finally, a self-bounding estimate for
\(\norm{\nabla f(x)}\) shows that the projection of the unified method is
inactive in the full-gradient specialization, so its updates coincide with
Algorithm~\ref{alg:full-gradient-inner}. The resulting weight growth gives
\(f(x_N)-f^*\le\varepsilon\) with the stated iteration complexity.}

\begin{currentrevision}
With the conversion \(H_0=2L_0\), \(H_1=4L_1^2\) from
Appendix~\ref{app:l0l1-inclusion}, Theorem~\ref{thm:full-gradient-main} gives
\(O(\sqrt{L_0R^2/\varepsilon}+(1+L_1R)
(1+\log(\Delta_0/\varepsilon)))\); its linear \(L_1R\)-term misses the sharper
\((L_1R)^{2/3}\) dependence. A direct specialization of
Algorithm~\ref{alg:full-gradient-inner} recovers this dependence.

\vspace{0.5em}
\begin{remarkbox}
\paragraph{Remark (\(L_0,L_1\)-smoothness).}
With the parameters of Proposition~\ref{prop:l0l1-full-gradient-rate} in
Appendix~\ref{app:l0l1-full-gradient}, Algorithm~\ref{alg:full-gradient-inner}
reaches \(f(x_N)-f^*\le\varepsilon\) with iteration complexity
\[
 N=O\left(
 1+\sqrt{\frac{L_0R^2}{\varepsilon}}
 +(L_1R)^{2/3}\log^{2/3}\!\frac{2\Delta_0}{\varepsilon}
 \right).
\]
\end{remarkbox}
\end{currentrevision}
\subsection{Coordinate Setup}
\label{sec:coordinate}
\vspace{-0.5em}

Forming a full gradient may be prohibitively expensive in high-dimensional
problems. \rev{We therefore replace each full-gradient query with one uniformly
sampled coordinate derivative \(\nabla_i f(x)\) per iteration}, which can
reduce the oracle cost whenever a partial derivative is substantially cheaper
to evaluate than the full gradient. Such savings are possible even for dense
data: for structured problems, \citet{Nesterov_2017} show that both the
coordinate-oracle cost and the remaining algorithmic operations can scale
linearly with the problem dimension; on such instances, their accelerated
coordinate method can outperform standard fast gradient methods.

The resulting \rev{uniform coordinate} method is presented in
Algorithm~\ref{alg:coordinate-inner}. \rev{It retains the coupling point and
weight-correction mechanism of
Algorithm~\ref{alg:full-gradient-inner}, but replaces the full-gradient
auxiliary update with a projected coordinate update. At each iteration, the
same sampled partial derivative is used in the updates of \(v_{k+1}\) and
\(z_k\), while the factor \(d\) is the inverse of the uniform sampling
probability.}
It is also related to the accelerated coordinate method of
\citet{Nesterov_2017}: both use
one random partial derivative to update the primal and auxiliary sequences.
\rev{Their method uses fixed coordinate smoothness constants, whereas ours
uses two-regime weights and corrects the accepted weight using the realized
projected displacement to account for gap-dependent curvature.}
Section~\ref{sec:nonuniform-sampling} extends this construction to non-uniform
sampling based on coordinate-specific curvature.

\begingroup
\setlength{\intextsep}{4pt}
\begin{algorithm}[H]
\caption{\rev{Accelerated Projected Coordinate Method}}
\label{alg:coordinate-inner}
\color{revisionblue}
\begin{algorithmic}[1]
\renewcommand{\algorithmicrequire}{\textbf{Input:}}
\renewcommand{\algorithmicensure}{\textbf{Return:}}
\REQUIRE Objective \(f\), initial point \(x^0\), radius \(R\), iteration budget
 \(N\), initial weight \(A_0>0\), and prescribed rules for the
 \newrev{query weight \(\widehat b_k>0\), update weight \(b_k>0\),} and
 \newrev{\(c_k\ge0\)}
\STATE Initialize \(x_0\gets x^0\), \(v_0\gets x^0\)
\FOR{\(k=0,\ldots,N-1\)}
  \STATE Set \(y_k\gets\dfrac{A_kx_k+\newrev{\widehat b_k}v_k}
  {A_k+\newrev{\widehat b_k}}\)
  \STATE Sample \(i_k\sim\operatorname{Unif}\{1,\ldots,d\}\)
  \STATE Update
  \(v_{k+1}\gets
  P_{B(x^0,3R)}\!\left(v_k-db_k\nabla_{i_k}f(y_k)e_{i_k}\right)\)
  \STATE Set
  \(a_k\gets A_k+b_k-
  \dfrac{c_kb_k^2}{A_k+b_k}
  \left(e_{i_k}^{\top}(v_{k+1}-v_k)\right)^2\)
  \STATE Set
  \(z_k\gets\newrev{\dfrac{A_kx_k+b_kv_k}{A_k+b_k}}+
  \dfrac{db_k}{a_k}e_{i_k}^{\top}(v_{k+1}-v_k)e_{i_k}\)
  \STATE Update \(A_{k+1}\gets\max\{A_0,a_k\}\)
  \STATE Update
  \(x_{k+1}\gets\dfrac{a_kz_k+(A_{k+1}-a_k)x^0}{A_{k+1}}\)
\ENDFOR
\ENSURE \(x_N\)
\end{algorithmic}
\end{algorithm}
\endgroup
\vspace{-0.6em}

\rev{The sequences} \(x_k\), \(v_k\), and \(A_k\), \rev{and the points}
\(y_k\), \rev{\(z_k\),} have the same roles as in
Algorithm~\ref{alg:full-gradient-inner}. \rev{The vector} \(e_i\) is the
\(i\)-th standard basis vector, and \(i_k\) is sampled uniformly. \rev{Here
\(B(x^0,3R):=\{x\in\mathbb R^d:\norm{x-x^0}\le3R\}\), and
\(P_{B(x^0,3R)}\) denotes the Euclidean projection onto this ball. The
projection keeps the auxiliary sequence inside \(B(x^0,3R)\); it is retained
because, unlike in the full-gradient case, the unprojected coordinate update
need not remain in this ball. Since each coordinate is selected with
probability \(1/d\), the factor \(d\) is the inverse-probability scaling of the
sampled derivative, and \newrev{its square is absorbed into the
curvature-valued parameter \(c_k\).} The
realized selected-coordinate displacement
\(e_{i_k}^{\top}(v_{k+1}-v_k)\) determines both the corrected weight \(a_k\)
and the point \(z_k\).} The single queried derivative
\(\nabla_{i_k}f(y_k)\) is reused \rev{throughout the iteration}. The method
returns \(x_N\).

We now state the convergence guarantee for the coordinate method
\rev{in Algorithm~\ref{alg:coordinate-inner}}.

\begin{mainresultbox}
\begin{theorem}[Coordinate complexity]
\label{thm:coordinate-main}
Suppose Assumptions~\ref{ass:convexity} and
\ref{ass:h0h1-smoothness} hold, \rev{\(\norm{x^0-x^*}\le R\) for some
\(R>0\), \(\Delta_0\ge f(x^0)-f^*\), and
\(0<\varepsilon\le\Delta_0\), where \(\varepsilon\) denotes the desired
accuracy. Set \(A_0=R^2/(2\Delta_0)\) and choose the parameters in
Algorithm~\ref{alg:coordinate-inner} as follows:}
\begin{revision}
\begingroup
\small
\[
 c_k=
 \begin{cases}
 4\newrev{d^2H_1},
 &\displaystyle \max_{0\le j\le k}A_j<
 \max\left\{A_0,
 \min\left\{
 \frac{(1+128dR\sqrt{H_1})^2}{4096d^2H_0},
 \frac{4eR^2}{\varepsilon}
 \right\}\right\},\\[0.9em]
 4\newrev{d^2H_1}\left(1+\log\frac{2\Delta_0}{\varepsilon}\right),
 &\displaystyle \max_{0\le j\le k}A_j\ge
 \max\left\{A_0,
 \min\left\{
 \frac{(1+128dR\sqrt{H_1})^2}{4096d^2H_0},
 \frac{4eR^2}{\varepsilon}
 \right\}\right\},
 \end{cases}
\]
\endgroup
\[
 \newrev{\widehat b_k=}b_k=\frac{A_k}
 {1+\newrev{64R\sqrt{c_k}}+64d\sqrt{H_0A_k}};
\]
\end{revision}
then Algorithm~\ref{alg:coordinate-inner} returns a point
\rev{\(x_N\)} satisfying
\rev{\(\E[f(x_N)-f^*]\le\varepsilon\)} after
\begin{revision}
\[
 \boxed{\displaystyle
 N=O\left(
 d\sqrt{\frac{H_0R^2}{\varepsilon}}
 +\left(1+d\sqrt{H_1R^2}\right)
 \left(1+\log\frac{\Delta_0}{\varepsilon}\right)
 \right).}
\]
\end{revision}
\end{theorem}
\end{mainresultbox}

\rev{Algorithm~\ref{alg:coordinate-inner} queries one partial derivative
\(\nabla_{i_k}f(y_k)\) per iteration and reuses it in all updates. Therefore,
the iteration bound in Theorem~\ref{thm:coordinate-main} is also its coordinate
oracle complexity.} Relative to
Theorem~\ref{thm:full-gradient-main}, \rev{the curvature-dependent terms carry} the standard
factor \(d\). When
\(H_1=0\), \rev{the \(H_0\)-dependent term} matches the classical uniform-sampling accelerated
coordinate rate \citep{nesterov2012coordinate} and agrees with
\citet{Nesterov_2017} when all coordinate smoothness constants coincide.
\rev{Moreover, the lower bound in Theorem~\ref{thm:growth-lower-bounds} shows
that, for uniformly sampled coordinate methods, including
Algorithm~\ref{alg:coordinate-inner}, the \(d\sqrt{H_1R^2}\) dependence cannot be
improved beyond logarithmic factors.}

To the best of our knowledge, Theorem~\ref{thm:coordinate-main} is the first
accelerated coordinate guarantee for the general class of convex
\((H_0,H_1)\)-smooth objectives.  Recent coordinate analyses under related
non-uniform smoothness conditions concern normalized Gauss--Southwell updates
and special losses \citep{vaswani2026steepestadam}, rather than accelerated
complexity for this general convex class.

\vspace{-1em}
\paragraph{Proof sketch.}
\rev{The deterministic weight-growth argument is the same as in the
full-gradient case; the only additional issue is the random coordinate
update. We represent the \(d\) coordinates as singleton blocks and sample each
block with probability \(1/d\).} Let \(\mathcal F_k\) be the
\(\sigma\)-algebra generated by all randomness revealed before sampling
\(i_k\). Uniform sampling gives
\(\E[d\nabla_{i_k}f(y_k)e_{i_k}\mid\mathcal F_k]=\nabla f(y_k)\), \rev{so the
sampling error in the potential estimate vanishes after taking the conditional
expectation. The optimality condition for the projection controls the
auxiliary update, while the correction defining \(a_k\) uses the realized
projected displacement and absorbs the corresponding \(H_1\)-curvature term.
Consequently, the inverse-probability factor \(d\) in the coordinate update
multiplies both curvature terms by \(d^2\). Since the accelerated bound depends
on the square root of these terms, this produces the factor \(d\) in the
curvature-dependent terms of the stated complexity and yields
\(\E[f(x_N)-f^*]\le\varepsilon\).}

\begin{revision}
\vspace{-0.5em}
\section{Lower Bounds for the \texorpdfstring{\(H_1\)}{H1}-Dependence}
\label{sec:growth-lower-bounds}
\vspace{-0.5em}

The upper bounds in Theorems~\ref{thm:full-gradient-main}
and~\ref{thm:coordinate-main} consist of an \(H_0\)-dependent term and a
second term governed by \(H_1\). When \(H_1=0\), the first term recovers the
classical optimal smooth-convex dependence for full-gradient methods and the
standard accelerated dependence for uniformly sampled coordinate methods
\citep{nesterov1983method,nesterov2012coordinate,Nesterov_2017}. We therefore
focus on the less-understood \(H_1\)-dependent term and ask whether the factors
\(\sqrt{H_1R^2}\) in the full-gradient setting and
\(d\sqrt{H_1R^2}\) in the uniformly sampled coordinate setting can be
improved. The following theorem addresses this question for full-gradient
linear-span methods and uniformly sampled coordinate methods.

\vspace{-1em}
\paragraph{Method classes.}
We consider full-gradient linear-span methods and coordinate methods that
sample coordinates uniformly and update only previously revealed coordinate
directions. These classes contain Algorithms~\ref{alg:full-gradient-inner}
and~\ref{alg:coordinate-inner}, respectively; their precise definitions are
given in Appendix~\ref{app:growth-lower-bounds}.

\begin{mainresultbox}
\begin{theorem}[Lower bounds for the \(H_1\)-dependent term]
\label{thm:growth-lower-bounds}
Fix \(R>0\), \(\varepsilon>0\), and \(H_1>0\), and set
\(\kappa:=H_1R^2\) and \(H_0:=\varepsilon/R^2\). Suppose that \(\kappa\) is
sufficiently large and
\(d\ge\sqrt\kappa/(4[\log(8\kappa)]^{3/2})\). Then there exist a convex
differentiable function \(f:\mathbb R^d\to\mathbb R\), an initial point
\(x^0\), and a minimizer \(x^*\) satisfying
Assumptions~\ref{ass:convexity} and~\ref{ass:h0h1-smoothness} with constants
\(H_0,H_1\), such that
\[
 \norm{x^0-x^*}\le R,
 \qquad
 3\varepsilon\le f(x^0)-f^*\le4\varepsilon.
\]
Every full-gradient linear-span method attaining
\(f(x_{\rm out})-f^*\le\varepsilon\), and every uniformly sampled coordinate
method attaining \(\E[f(x_{\rm out})-f^*]\le\varepsilon\), must use,
respectively,
\[
 N=\Omega\left(
 \frac{\sqrt{H_1R^2}}{[\log(8H_1R^2)]^{3/2}}
 \right),
 \qquad
 Q=\Omega\left(
 \frac{d\sqrt{H_1R^2}}{[\log(8H_1R^2)]^{3/2}}
 \right)
\]
full-gradient and coordinate queries.
\end{theorem}
\end{mainresultbox}

See Appendix~\ref{app:growth-lower-bounds} for the proof.
The choice \(H_0=\varepsilon/R^2\) makes the classical
\(H_0\)-dependent term constant and therefore isolates the dependence on
\(H_1\).
The polynomial dependence of the second terms in
Theorems~\ref{thm:full-gradient-main} and~\ref{thm:coordinate-main} is
therefore optimal up to logarithmic factors for the two~method~classes~above.
\end{revision}

\vspace{-0.5em}
\section{Extension to Non-Uniform Sampling}
\label{sec:nonuniform-sampling}
\vspace{-0.5em}

Uniform sampling ignores heterogeneity across coordinates.  Under classical
coordinate smoothness, sampling coordinate \(i\) proportionally to the square
root of its smoothness constant can replace a worst-coordinate dependence by a
sum of coordinate contributions
\citep{Nesterov_2017,allenzhu2016evenfaster}.  Under \((H_0,H_1)\)-smoothness,
both the baseline and gap-dependent curvature may be heterogeneous.  We now
adapt this importance-sampling principle to that setting.

\begin{assumptionbox}
\begin{assumption}[Coordinate \((H_0,H_1)\)-smoothness]
\label{ass:coordinate-h0h1-smoothness}
For every \(i\in\{1,\ldots,d\}\), there exist constants
\(H_{0,i}\ge0\), \rev{\(H_{1,i}\ge0\)}.  Let
\(\Ls_i(\Delta):=2(H_{0,i}+H_{1,i}\Delta)\) and
\(r_{1,i}:=1/((2\sqrt3+\sqrt6)\sqrt{H_{1,i}})\).  For every
\(x\in\mathbb R^d\) and \(h\in\mathbb R\) \rev{satisfying}
\(\lvert h\rvert\le r_{1,i}\),
\[
  f(x+he_i)
  \le
  f(x)+h\nabla_i f(x)
  +\frac{\Ls_i(\F(x))}{2}h^2.
\]
\end{assumption}
\end{assumptionbox}

The case \(H_{1,i}=0\) is covered by \(\Ls_i(\Delta)=2H_{0,i}\) and
\(r_{1,i}=\infty\).  Assumption~\ref{ass:h0h1-smoothness} implies
Assumption~\ref{ass:coordinate-h0h1-smoothness} with
\(H_{j,i}=H_j\) for all \(i\) and \(j\in\{0,1\}\); the coordinate condition is
imposed only along \(x+he_i\) and allows coordinate-specific constants and
radii.  Let
\(S_{1/2}^{(0)}:=\sum_{i=1}^d H_{0,i}^{1/2}\) and
\(S_{1/2}^{(1)}:=\sum_{i=1}^d H_{1,i}^{1/2}\).  If
\(H_{j,i}\le\rho_iH_j\) for \(j\in\{0,1\}\) and \(\rho_i\in[0,1]\), then
\(S_{1/2}^{(j)}\le\sqrt{H_j}\sum_i\sqrt{\rho_i}\).  Hence, if curvature is
concentrated on \(m\ll d\) coordinates and
\(\sum_i\sqrt{\rho_i}=O(m)\), non-uniform sampling replaces the uniform factor
\(d\) by \(m\) \rev{in the curvature-dependent terms}, as in the classical heterogeneous-smoothness setting
\citep{Nesterov_2017,allenzhu2016evenfaster}.

This assumption differs from the coordinate \((L_0,L_1)\)-smoothness condition
used by \citet[Assumption~1.3]{lobanov2024linear}, which bounds
\(\lvert\nabla_i f(x+he_i)-\nabla_i f(x)\rvert\) by
\((L_{0,i}+L_{1,i}\lvert\nabla_i f(x)\rvert)\lvert h\rvert\) inside a common
coordinate radius.  Their condition is gradient-dependent, whereas
Assumption~\ref{ass:coordinate-h0h1-smoothness} is gap-dependent and stated
directly as the local upper model required by the algorithm.  Both recover
standard coordinate smoothness when the growth constants vanish.

We assume known valid upper bounds on \(\{H_{0,i},H_{1,i}\}_{i=1}^d\), as in
classical non-uniform accelerated coordinate methods
\citep{Nesterov_2017,allenzhu2016evenfaster}.  \rev{Coordinates with
\(H_{0,i}=0\) are constant and may be omitted; see
Remark~\ref{rem:constant-block}. If no coordinates remain, the initial point
is optimal. Otherwise, we relabel the remaining coordinates so that
\(H_{0,i}>0\) for every \(i\), and choose the fixed sampling probabilities
as follows. If \(S_{1/2}^{(1)}>0\), set
\(p_i:=\frac12\frac{\sqrt{H_{0,i}}}{S_{1/2}^{(0)}}
+\frac12\frac{\sqrt{H_{1,i}}}{S_{1/2}^{(1)}}\); if
\(S_{1/2}^{(1)}=0\), set
\(p_i:=\frac{\sqrt{H_{0,i}}}{S_{1/2}^{(0)}}\).
These probabilities are positive, sum to one, and remain fixed throughout
the run. This leads to the non-uniform coordinate method in
Algorithm~\ref{alg:nonuniform-coordinate-inner}.}

\begingroup
\setlength{\intextsep}{4pt}
\begin{algorithm}[H]
\caption{\rev{Accelerated Projected Coordinate Method with Non-Uniform Sampling}}
\label{alg:nonuniform-coordinate-inner}
\color{revisionblue}
\begin{algorithmic}[1]
\renewcommand{\algorithmicrequire}{\textbf{Input:}}
\renewcommand{\algorithmicensure}{\textbf{Return:}}
\REQUIRE Objective \(f\), initial point \(x^0\), radius \(R\), iteration
budget \(N\), initial weight \(A_0>0\), sampling probabilities
 \(p_1,\ldots,p_d>0\), \newrev{and rules for
 \(\widehat b_k,b_k>0\) and \(c_k\ge0\)}
\STATE Initialize \(x_0\gets x^0\), \(v_0\gets x^0\)
\FOR{\(k=0,\ldots,N-1\)}
  \STATE Set \(y_k\gets\dfrac{A_kx_k+\newrev{\widehat b_k}v_k}
  {A_k+\newrev{\widehat b_k}}\)
  \STATE Sample \(i_k\sim p\), that is, \(\Pr(i_k=i)=p_i\)
  \STATE Update
  \(v_{k+1}\gets
  P_{B(x^0,3R)}\!\left(
  v_k-\dfrac{b_k}{p_{i_k}}\nabla_{i_k}f(y_k)e_{i_k}
  \right)\)
  \STATE Set
  \(a_k\gets A_k+b_k-
  \dfrac{c_kb_k^2}{A_k+b_k}
  \left(e_{i_k}^{\top}(v_{k+1}-v_k)\right)^2\)
  \STATE Set
  \(z_k\gets\newrev{\dfrac{A_kx_k+b_kv_k}{A_k+b_k}}+
  \dfrac{b_k}{p_{i_k}a_k}
  e_{i_k}^{\top}(v_{k+1}-v_k)e_{i_k}\)
  \STATE Update \(A_{k+1}\gets\max\{A_0,a_k\}\)
  \STATE Update
  \(x_{k+1}\gets
  \dfrac{a_kz_k+(A_{k+1}-a_k)x^0}{A_{k+1}}\)
\ENDFOR
\ENSURE \(x_N\)
\end{algorithmic}
\end{algorithm}
\endgroup
\vspace{-0.6em}

\rev{Algorithm~\ref{alg:nonuniform-coordinate-inner} uses the same sequences
and coupling points as Algorithm~\ref{alg:coordinate-inner}, but samples
\(i_k\) from the fixed distribution \(p\). The inverse-probability factor
\(1/p_{i_k}\) replaces \(d\) in the auxiliary and \(z_k\)-updates, and
\newrev{the worst-case effective curvature \(K_1\) enters the correction
through \(c_k\)}. The two components of \(p_i\)
control the baseline and gap-dependent coordinate curvatures, respectively.}
If \(H_{0,i}=H_0\) and \(H_{1,i}=H_1\) for all \(i\), then
\rev{\(p_i=1/d\)}, and
Algorithm~\ref{alg:nonuniform-coordinate-inner} reduces to
Algorithm~\ref{alg:coordinate-inner}.  When all \(H_{1,i}=0\), the probabilities
become proportional to \(\sqrt{H_{0,i}}\), matching the classical
square-root importance sampling of \citet{Nesterov_2017}.
\rev{Unlike this classical rule, the mixture above also accounts for
heterogeneous gap-dependent curvature while remaining fixed throughout the
run.}

We next state the non-uniform coordinate guarantee.

\begin{mainresultbox}
\begin{theorem}[Non-uniform coordinate complexity]
\label{thm:nonuniform-coordinate-main}
Suppose Assumptions~\ref{ass:convexity} and
\ref{ass:coordinate-h0h1-smoothness} hold, \rev{\(\norm{x^0-x^*}\le R\)
for some \(R>0\), \(\Delta_0\ge f(x^0)-f^*\), and
\(0<\varepsilon\le\Delta_0\), where \(\varepsilon\) denotes the desired
accuracy. Let \(p_1,\ldots,p_d\) be the probabilities defined above and set}
\begin{revision}
\[
 K_0:=\max_{1\le i\le d}\frac{H_{0,i}}{p_i^2},
 \qquad
 K_1:=\max_{1\le i\le d}\frac{H_{1,i}}{p_i^2}.
\]
Set \(A_0=R^2/(2\Delta_0)\) and choose the parameters in
Algorithm~\ref{alg:nonuniform-coordinate-inner} as follows:
\[
 c_k=
 \begin{cases}
 4\newrev{K_1},
 &\displaystyle \max_{0\le j\le k}A_j<
 \max\left\{A_0,
 \min\left\{
 \frac{(1+128R\sqrt{K_1})^2}{4096K_0},
 \frac{4eR^2}{\varepsilon}
 \right\}\right\},\\[0.9em]
 4\newrev{K_1}\left(1+\log\frac{2\Delta_0}{\varepsilon}\right),
 &\displaystyle \max_{0\le j\le k}A_j\ge
 \max\left\{A_0,
 \min\left\{
 \frac{(1+128R\sqrt{K_1})^2}{4096K_0},
 \frac{4eR^2}{\varepsilon}
 \right\}\right\},
 \end{cases}
\]
\[
 \newrev{\widehat b_k=}b_k=\frac{A_k}
 {1+\newrev{64R\sqrt{c_k}}+64\sqrt{K_0A_k}};
\]
\end{revision}
then Algorithm~\ref{alg:nonuniform-coordinate-inner} returns a point
\rev{\(x_N\)} satisfying
\rev{\(\E[f(x_N)-f^*]\le\varepsilon\)} after
\begin{revision}
\[
 \boxed{\displaystyle
 N=O\left(
 \frac{RS_{1/2}^{(0)}}{\sqrt\varepsilon}
 +\left(1+RS_{1/2}^{(1)}\right)
 \left(1+\log\frac{\Delta_0}{\varepsilon}\right)
 \right).}
\]
\end{revision}
\end{theorem}
\end{mainresultbox}

\rev{Algorithm~\ref{alg:nonuniform-coordinate-inner} queries one partial
derivative \(\nabla_{i_k}f(y_k)\) per iteration and reuses it in all updates.
Therefore, the iteration bound in
Theorem~\ref{thm:nonuniform-coordinate-main} is also its coordinate oracle
complexity.} In the uniform case,
\(S_{1/2}^{(0)}=d\sqrt{H_0}\) and
\(S_{1/2}^{(1)}=d\sqrt{H_1}\), recovering
Theorem~\ref{thm:coordinate-main}; heterogeneous coordinate curvatures can make
both sums substantially smaller.  When \(H_{1,i}=0\) for all \(i\), the first
term recovers the standard non-uniform accelerated coordinate dependence on
\(\sum_i\sqrt{H_{0,i}}\) \citep{Nesterov_2017,allenzhu2016evenfaster}.

To the best of our knowledge, Theorem~\ref{thm:nonuniform-coordinate-main} is
the first accelerated non-uniform coordinate guarantee under coordinate-wise
\((H_0,H_1)\)-smoothness.  In contrast, the coordinate
\((L_0,L_1)\)-smooth analysis of \citet{lobanov2024linear} uses uniform sampling
and gives a non-accelerated rate.

\vspace{-1em}
\paragraph{Proof sketch.}
\rev{Remark~\ref{rem:constant-block} allows us first to remove every
coordinate with \(H_{0,i}=0\). On the remaining coordinates, take singleton
blocks \(U_i=e_i\) and sample block \(i\) with probability \(p_i\); the
resulting block updates are exactly those of
Algorithm~\ref{alg:nonuniform-coordinate-inner}. Suppose first that
\(S_{1/2}^{(1)}>0\). The two components in \(p_i\) are probability
distributions, so their equally weighted mixture is also a probability
distribution. Retaining only component \(j\in\{0,1\}\) gives
\(p_i\ge\sqrt{H_{j,i}}/(2S_{1/2}^{(j)})\), and therefore
\(\sqrt{H_{j,i}}/p_i\le2S_{1/2}^{(j)}\). Taking the maximum over \(i\) and
squaring yields \(K_j\le4(S_{1/2}^{(j)})^2\). If
\(S_{1/2}^{(1)}=0\), then every \(H_{1,i}=0\), while
\(p_i=\sqrt{H_{0,i}}/S_{1/2}^{(0)}\) gives the exact identities
\(K_0=(S_{1/2}^{(0)})^2\) and \(K_1=0\). Thus the quantities
\(d^2H_0\) and \(d^2H_1\) from the uniform proof are replaced by bounds in
terms of \((S_{1/2}^{(0)})^2\) and \((S_{1/2}^{(1)})^2\). Substituting these
bounds into Theorem~\ref{thm:common-block} gives the stated complexity,
including the additive \(O(1+\log(\Delta_0/\varepsilon))\) term.}
\vspace{-0.8em}
\section{Numerical experiments}
\label{sec:experiments}
\vspace{-0.8em}

\rev{In this section, we evaluate the proposed full-gradient algorithm in
practice on two real-data generalized linear models. Our practical parameter
choices follow the same regime-dependent weight principles as the theoretical
rules and retain the realized-curvature correction, but are less conservative
and adapt to observed gradient norms and local secant curvature.
\purrev{Practical AGM and APCM require \emph{no prescribed smoothness
constants}. APCM initializes its step scale, derivative memory, and
non-uniform sampling probabilities using \(2d\) partial-derivative queries.} The
coordinate experiments, including the comparison of uniform and non-uniform
sampling, are reported in Appendix~\ref{app:coordinate-experiment}.}

\vspace{-0.5em}
\rev{We compare our Accelerated Gradient Method (AGM) with AGMsDR
\citep{vankov2024l0l1}, \(\ell\)-AGD \citep{tyurin2025near}, and Accelerated
GRAAL \citep{borodich2026nesterov}.
\purrev{AGM uses a practical two-weight rule based on secant curvature and a
gradient-norm cap, with \(c_k=4\) and
\(a_k\ge0.91(A_k+b_k)>0\). It needs one gradient per iteration and no function
values or inner search. The full rules and baseline settings appear in
Appendix~\ref{app:experimental-implementation}; theoretical rates concern the
prescribed schedules.} AGMsDR uses golden-section search with
tolerance \(10^{-12}\) and at most \(64\) iterations. We use
\(x^0=(b^0,0)\), \(b^0\in\{0,2,4,6,8\}\), and stop at
\(\F(x_k)/\F(x^0)\le10^{-6}\), without dataset-specific tuning. The reference
values \(f^*\) use primal-dual bounds matching to machine precision. We report
gradient and function-value calls separately using one OpenBLAS thread.
\orrev{Only algorithmic calls are counted; times include gap monitoring.
Both figures use \(b^0=4\), the middle of the tested grid.}
\purrev{Appendix~\ref{app:parameter-choices} compares the theoretical and
practical AGM settings, and Appendix~\ref{app:search-tolerance} examines
AGMsDR's sensitivity to the search tolerance.}}

\vspace{-1em}
\paragraph{\rev{Bike Sharing.}}
\rev{Count prediction from heterogeneous tabular features is a standard
supervised-learning problem. In convex Poisson regression with a log link, the
loss curvature grows with the predicted rate, making bike-demand forecasting a
natural real-data benchmark for generalized smoothness. We use the hourly UCI
Bike Sharing data \citep{fanaee2013bike}, comprising
\(n=17{,}379\) observations with weather, calendar, and time-of-day features.
After centering, rescaling, and one-hot encoding, the model has \(51\)
parameters including an intercept. Writing \(s_i=b+a_i^\top w\), and denoting
the mean rental count by \(\bar y\), we minimize the \(\ell_2\)-regularized
Poisson regression loss
\[
 f(b,w)=\frac1n\sum_{i=1}^n
 \left[\exp(s_i)-\frac{y_i}{\bar y}s_i\right]
 +\frac{0.01}{2}\norm{w}^2.
\]
The intercept \(b\) is not regularized. The feature normalization yields the
global certificates \((L_0,L_1)=(2.01,2)\) and
\((H_0,H_1)=(0.01+4\log 2,4)\approx(2.783,4)\). We use the common protocol
above, \orrev{use the same \(x^0\) and supply constants only where required}, and
report median times over seven runs.}

\vspace{-1em}
\begin{figure}[H]
  \centering
  \includegraphics[width=0.96\linewidth]{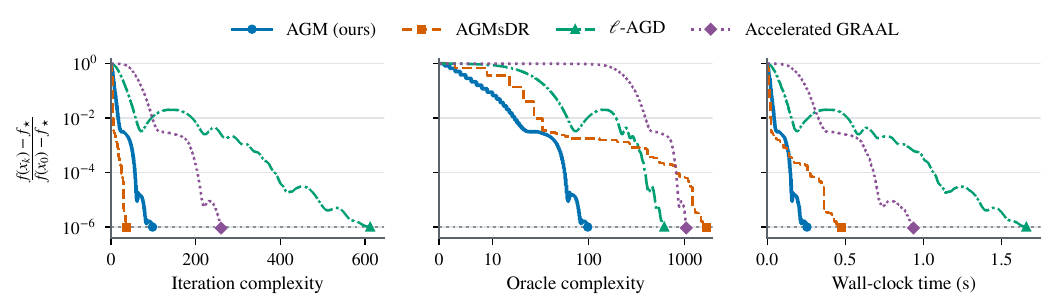}
  \caption{\small\rev{Convergence on Bike Sharing for \(b^0=4\). The panels
  show the relative objective gap against iterations, oracle calls, and
  wall-clock time \orrev{(median over seven runs)}; the dashed line marks the target \(10^{-6}\).}}
  \label{fig:bike-sharing-full-gradient}
\end{figure}
\vspace{-1em}

\rev{Figure~\ref{fig:bike-sharing-full-gradient} shows the representative run
\(b^0=4\). AGMsDR has the lowest iteration count, but each iteration solves an
inner one-dimensional segment problem; consequently, its \(37\) outer
iterations require \(1683\) oracle calls, more than every other method. In
contrast, AGM uses one gradient call per iteration and no inner search, so its
\orrev{\(98\)} iterations require only \orrev{\(98\)} oracle calls and \orrev{\(0.255\)} seconds,
compared with \(614\) calls for \(\ell\)-AGD and \(1040\) for Accelerated
GRAAL. Across all five initializations, AGM requires \orrev{\(91\)--\(152\)} calls and
is the fastest method in every case. Complete results for both datasets and all
initializations are reported in Table~\ref{tab:real-data-full-results} of
Appendix~\ref{app:real-data-full-results}.}

\vspace{-1em}
\paragraph{\rev{Household Power.}}
\rev{Short-horizon load forecasting from lagged sensor and calendar features
is a standard supervised-learning task. Positive, heteroscedastic targets are
naturally modeled by Gamma regression, whose convex log-link loss provides a
second real-data benchmark with state-dependent curvature. We use the UCI Individual
Household Electric Power Consumption data
\citep{hebrail2006household}. From the minute-level series, we select
\(262{,}144\) training observations and \(65{,}536\) later observations for
validation, and predict power \(15\) minutes ahead from lagged measurements
and calendar features. The model has \(44\) parameters including an
unregularized intercept. With \(s_i=b+a_i^\top w\) and \(\bar y\) equal to the
mean training power, we minimize
\[
 f(b,w)=\frac1n\sum_{i=1}^n
 \left[\frac{y_i}{\bar y}\exp(s_i)-s_i\right]
 +\frac{0.01}{2}\norm{w}^2.
\]
The normalized design gives \((L_0,L_1)=(2.01,2)\) and
\((H_0,H_1)\approx(4.638,4)\). We use the same protocol \orrev{and supply these constants where required}
and report median times over three runs.}

\vspace{-1em}
\begin{figure}[H]
  \centering
  \includegraphics[width=0.96\linewidth]{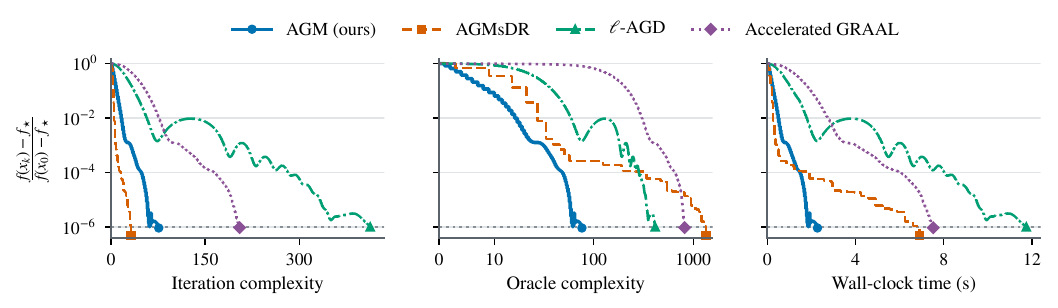}
  \caption{\small\rev{Convergence on Household Power for \(b^0=4\). The panels
  show the relative objective gap against iterations, oracle calls, and
  wall-clock time \orrev{(median over three runs)}; the dashed line marks the target \(10^{-6}\).}}
  \label{fig:household-power-full-gradient}
\end{figure}

\vspace{-1.5em}
\rev{Figure~\ref{fig:household-power-full-gradient} shows
essentially the same pattern as Bike Sharing. For \(b^0=4\), AGM reaches the
target in \orrev{\(76\)} gradient calls. AGMsDR again uses fewer outer iterations, but
its line search raises the cost to \(1348\) oracle calls, \orrev{compared with}
the \(820\) calls required by Accelerated GRAAL. Across all initializations,
AGM uses the fewest oracle calls and remains the fastest method in four of five
cases.}

\vspace{-0.5em}
\section{Discussion and Future Work}
\label{sec:discussion-limitations}
\vspace{-0.5em}

\begin{revision}
Table~\ref{tab:table_compare} is representative rather than exhaustive: it
focuses on rates stated directly in the native \((H_0,H_1)\) parametrization
and on their full-gradient or coordinate counterparts.  Among accelerated
methods for \((L_0,L_1)\)-smooth convex optimization, we include AGMsDR
\citep{vankov2024l0l1} because, to our knowledge, it provides the best-known
separated dependence on both \(L_0\) and \(L_1\).  This is not a claim that
AGMsDR is the best end-to-end method: every outer iteration solves a
one-dimensional segment subproblem.  Figures~\ref{fig:bike-sharing-full-gradient}
and~\ref{fig:household-power-full-gradient} make this distinction explicit.
AGMsDR reaches the target \(10^{-6}\) in the fewest outer iterations, but its
line search entails more oracle calls and longer wall-clock time than
Algorithm~\ref{alg:full-gradient-inner} in the displayed runs.  The
\(\ell\)-AGD method of \citet{tyurin2025near} removes the segment minimization,
while Accelerated GRAAL \citep{borodich2026nesterov} additionally adapts to
local curvature without prescribed smoothness constants.  Their guarantees,
however, have a less favorable dependence on \(L_1\) than the specialized
AGMsDR bound; in the displayed experiments, both also trail our method in
iterations, oracle calls, and wall-clock time.

Section~\ref{sec:problem-setup} and Appendix~\ref{app:l0l1-implication} show
that the \((H_0,H_1)\)-smooth class strictly contains the
\((L_0,L_1)\)-smooth class, and therefore also the standard globally
\(L\)-smooth class.  This hierarchy motivates our focus on these three
classes.  By contrast, \citet{tyurin2025near} studies the more general
gradient-dependent condition
\(\norm{\nabla^2 f(x)}\le\ell(\norm{\nabla f(x)})\).
Appendix~\ref{app:ell-incomparability} shows that this general
\(\ell\)-smooth class and the \((H_0,H_1)\)-smooth class are incomparable.
For a common experimental certificate, we therefore instantiate
\(\ell(s)=L_0+L_1s\), which is valid for the tested \((L_0,L_1)\)-smooth
objectives.

Our work focuses on the setting in which an upper bound
\(\Delta_0\ge F_0\) on the initial gap is known.
Theorems~\ref{thm:full-gradient-main} and~\ref{thm:coordinate-main} give, to our
knowledge, the first accelerated full-gradient and coordinate rates for the convex
\((H_0,H_1)\)-smooth class. For fixed \(\Delta_0/\varepsilon\), the lower bound in
Theorem~\ref{thm:growth-lower-bounds} shows that the \(\sqrt{H_1R^2}\) dependence is
optimal up to logarithmic factors for full-gradient linear-span and uniformly sampled coordinate methods.
When no \(\Delta_0\) is given, a bound based on the initial distance \(R\) must
allow for instances where \(F_0/\varepsilon\) grows exponentially in a power of \(H_1\).
The factor \(\log(\Delta_0/\varepsilon)\) can thus contribute an additional power of \(H_1\).
To determine the optimal dependence on \(H_1\) without a gap bound, we study a variant
of our full-gradient method in Appendix~\ref{app:radius-full-gradient}. It attains,
to our knowledge, the first gap-independent rate with an \((H_1R^2)^{3/5}\) term
(Theorem~\ref{thm:rad-upper}). Theorem~\ref{radlb:theorem} matches this term up to
absolute constants for deterministic first-order methods in growing dimension.

\newrev{For \((L_0,L_1)\)-smooth objectives, our full-gradient
methods attain the best-known \(L_1\)-dependence without solving one-dimensional
subproblems, both with a known \(\Delta_0\) (Algorithm~\ref{alg:full-gradient-inner},
Proposition~\ref{prop:l0l1-full-gradient-rate}) and without one
(Algorithms~\ref{alg:radius-normalized} and~\ref{alg:radius-full-gradient},
Proposition~\ref{prop:rad-l-upper}). Both guarantees follow from a direct
\((L_0,L_1)\) analysis. Applying the conversion in Appendix~\ref{app:l0l1-inclusion}
to the \((H_0,H_1)\) rates gives a larger power of \(L_1R\) in both cases.
For fixed \(\Delta_0/\varepsilon\), the power increases from \(2/3\) to \(1\)
(Theorem~\ref{thm:full-gradient-main}), and without a gap bound it increases
from \(1\) to \(6/5\) (Theorem~\ref{thm:rad-upper}). This gap leaves open
whether a sharper conversion can recover the direct bounds in both settings.
A further direction is
to develop algorithms adaptive to unknown smoothness parameters for both the
\((H_0,H_1)\)- and \((L_0,L_1)\)-smooth classes without rate degradation.}
\end{revision}

\vspace{-0.5em}
\section{Conclusion}
\label{sec:conclusion}
\vspace{-0.5em}

We developed a \rev{unified acceleration framework} for convex
\((H_0,H_1)\)-smooth optimization, \rev{yielding} full-gradient,
\rev{uniformly sampled coordinate}, and \rev{non-uniformly sampled coordinate}
methods.  \rev{Using} only local upper models within their validity radii,
rather than a global quadratic model, the resulting \rev{oracle complexity}
bounds separate the \(H_0\) and \(H_1\) contributions, \rev{recover classical
accelerated rates} when \(H_1=0\), and exploit coordinate heterogeneity
\rev{under} non-uniform sampling.  We complement these guarantees with
\rev{lower bounds showing that the dependence on \(H_1\) is
optimal up to logarithmic factors for full-gradient linear-span and uniformly
sampled coordinate methods}.  \rev{Experiments on two real-data generalized
linear models compare the full-gradient method with its closest accelerated
counterparts in oracle cost and wall-clock time.  Overall, these results extend
accelerated full-gradient and coordinate optimization beyond the globally
smooth setting.}

\section*{Acknowledgments}
AL thanks Yuriy Dorn for his support.


\bibliography{3_references}
\bibliographystyle{iclr2027_conference}

\onecolumn
\appendix
\setcounter{tocdepth}{2}

\begin{center}
  {\LARGE\bfseries APPENDIX\\[0.4em]
  A Few Accelerated Algorithms for Convex Optimization under
  \((H_0,H_1)\)-Smoothness\par}
  \vspace{-0.2em}
\end{center}

\tableofcontents

\begin{revision}
\section{\texorpdfstring{\((H_0,H_1)\)}{(H0,H1)}-Smoothness Is Strictly
Weaker Than \texorpdfstring{\((L_0,L_1)\)}{(L0,L1)}-Smoothness}
\label{app:l0l1-implication}

We first show that every convex
\((L_0,L_1)\)-smooth function satisfies our \((H_0,H_1)\)-smoothness
assumption.  We then prove that the converse fails using the two separating
examples.

\subsection{Inclusion of the \texorpdfstring{\((L_0,L_1)\)}{(L0,L1)}-Smooth
Class}
\label{app:l0l1-inclusion}

\begin{currentrevision}
We distinguish the Hessian convention used in the introduction from the local
upper model used in \citet[Appendix~A, equation~(8)]{lobanov2024linear}.  The
latter states that, for some \(L_0,L_1\ge0\),
\begin{equation}
\label{eq:local-l0l1-function-model}
 f(y)
 \le
 f(x)+\inner{\nabla f(x)}{y-x}
 +\frac{L_0+L_1\norm{\nabla f(x)}}{2}\norm{y-x}^2
\end{equation}
for every \(x,y\in\mathbb R^d\) satisfying
\(\norm{y-x}\le1/L_1\).  When \(L_1=0\), this radius is understood as
\(\infty\).

For convex twice continuously differentiable functions, the two conventions
describe the same class up to absolute constants, but not necessarily with the
same constants.  Indeed, the local model~\eqref{eq:local-l0l1-function-model}
implies
\(\norm{\nabla^2 f(x)}\le L_0+L_1\norm{\nabla f(x)}\) by applying it to
\(y=x+th\), dividing by \(t^2/2\), and letting \(t\to0\).  Conversely, suppose
that the Hessian condition holds with constants \((L_0,L_1)\).  For
\(h=y-x\), \(\norm{h}\le1/(2L_1)\), Gr\"onwall's inequality gives
\[
 L_0+L_1\norm{\nabla f(x+th)}
 \le
 e^{L_1t\norm{h}}\bigl(L_0+L_1\norm{\nabla f(x)}\bigr)
 \le
 2\bigl(L_0+L_1\norm{\nabla f(x)}\bigr),
 \qquad 0\le t\le1.
\]
Taylor's integral formula therefore yields
\[
 f(y)
 \le
 f(x)+\inner{\nabla f(x)}{y-x}
 +\frac{2L_0+2L_1\norm{\nabla f(x)}}{2}\norm{y-x}^2.
\]
Thus the Hessian condition with \((L_0,L_1)\) implies the local model with
\((2L_0,2L_1)\).
\end{currentrevision}

\begin{mainresultbox}
\begin{proposition}[\newrev{From the local \((L_0,L_1)\)-model to
\((H_0,H_1)\)-smoothness}]
\label{prop:l0l1-implies-local-h0h1}
\begin{currentrevision}
Let \(f:\mathbb R^d\to\mathbb R\) be differentiable, convex,
bounded from below, and satisfy \eqref{eq:local-l0l1-function-model}, with
\(f^*:=\inf_{x\in\mathbb R^d}f(x)>-\infty\).  Then
Assumption~\ref{ass:h0h1-smoothness} holds with
\(H_0=L_0\) and \(H_1=L_1^2\) when \(L_0>0\).  When \(L_0=0\), one may take
any \(H_0>0\) and \(H_1=L_1^2\).
\end{currentrevision}
\end{proposition}
\end{mainresultbox}

\begin{proof}
\begin{currentrevision}
Fix \(w\in\mathbb R^d\).  The self-bounding estimate of
\citet[Lemma~B.1]{lobanov2024linear}, which follows from the local
model~\eqref{eq:local-l0l1-function-model}, gives
\begin{equation}
\label{eq:l0l1-gradient-self-bound}
 \norm{\nabla f(w)}^2
 \le
 2\bigl(L_0+L_1\norm{\nabla f(w)}\bigr)\F(w).
\end{equation}
When \(L_1>0\), solving this quadratic inequality for its nonnegative root
and then bounding the square root gives
\begin{align}
 \norm{\nabla f(w)}
 &\overset{\circledOne}{\le}
 L_1\F(w)+\sqrt{L_1^2\F(w)^2+2L_0\F(w)}
 \nonumber\\
 &\overset{\circledTwo}{\le}
 2L_1\F(w)+\frac{L_0}{L_1}.
\label{eq:l0l1-gradient-gap-bound}
\end{align}
The first inequality is the quadratic formula applied to
\eqref{eq:l0l1-gradient-self-bound}.  For the second, we used
\[
 \sqrt{L_1^2\F(w)^2+2L_0\F(w)}
 \le L_1\F(w)+\frac{L_0}{L_1},
\]
which follows by squaring both nonnegative sides, since
\[
 \left(L_1\F(w)+\frac{L_0}{L_1}\right)^2
 =L_1^2\F(w)^2+2L_0\F(w)+\frac{L_0^2}{L_1^2}.
\]
Multiplying \eqref{eq:l0l1-gradient-gap-bound} by \(L_1\), adding \(L_0\),
and dividing by two yields
\begin{equation}
\label{eq:l0l1-curvature-gap-bound}
 \frac{L_0+L_1\norm{\nabla f(w)}}{2}
 \le L_0+L_1^2\F(w).
\end{equation}
If \(L_1=0\), the same inequality follows directly.

It remains to verify that the local model is valid on the radius required by
Assumption~\ref{ass:h0h1-smoothness}.  For \(H_1=L_1^2\) and \(L_1>0\),
\[
 r_1
 =\frac{1}{(2\sqrt3+\sqrt6)L_1}
 <\frac1{L_1}.
\]
For \(L_1=0\), both radii are infinite.  Hence, whenever
\(\norm{y-x}\le r_1\), we may combine
\eqref{eq:local-l0l1-function-model} with
\eqref{eq:l0l1-curvature-gap-bound} at \(w=x\) to obtain
\[
 f(y)
 \le
 f(x)+\inner{\nabla f(x)}{y-x}
 +\bigl(L_0+L_1^2\F(x)\bigr)\norm{y-x}^2.
\]
This is Assumption~\ref{ass:h0h1-smoothness} with
\(H_0=L_0\) and \(H_1=L_1^2\).  If \(L_0=0\), replacing it by any
\(H_0>0\) only enlarges the right-hand side.
\end{currentrevision}
\end{proof}
\end{revision}

\begin{revision}
\subsection{Strictness of the Inclusion: Separating Examples}
\end{revision}

\begin{examplebox}
\noindent\textbf{Example~\ref{ex:one-dimensional-separation}
(One-dimensional).}\enspace
{\itshape Fix \(p\in(1,2]\).  Let
\(q_p(t):=1+(1+t^2)^{p/2}\exp(-(1+t^2)^p\sin^2(\pi t))\), and define
\[
  \phi_p(x)
  =
  \int_0^x\int_0^s q_p(t)\,dt\,ds .
\]
Then \(\phi_p\) is convex, \((H_0,H_1)\)-smooth
\((H_0\text{=}H_1\text{=}2)\), but not
\((L_0,L_1)\)-smooth~for~finite~constants.\par}
\end{examplebox}

\begin{proof}
The function \(q_p\) is infinitely differentiable because it is a composition
of infinitely differentiable functions and \(1+t^2\) is strictly positive.
The fundamental theorem of calculus therefore gives
\(\phi_p'(x)=\int_0^x q_p(t)\,dt\) and \(\phi_p''(x)=q_p(x)\).  In particular,
\(\phi_p(0)=\phi_p'(0)=0\).  Since the exponential term in \(q_p\) is
nonnegative, \(q_p(x)\ge1\) for every \(x\).  Taylor's formula with integral
remainder consequently yields
  \begin{equation}
  \label{eq:appendix-phi-lower}
  \phi_p(x)
  =
  x^2\int_0^1(1-s)\phi_p''(sx)\,ds
  =
  x^2\int_0^1(1-s)q_p(sx)\,ds
  \ge
  x^2\int_0^1(1-s)\,ds
  =
  \frac{x^2}{2}.
  \end{equation}
The first equality uses \(\phi_p(0)=\phi_p'(0)=0\), the second uses
\(\phi_p''=q_p\), and the inequality uses \(q_p\ge1\).  Hence \(\phi_p\) is
1-strongly convex, and its unique minimizer is \(x^*=0\), with
\(\phi_p^*=\phi_p(0)=0\).

For every \(x\in\mathbb R\),
\begin{equation}
\label{eq:appendix-one-dimensional-hessian}
  \phi_p''(x)
  =q_p(x)
  \le 1+(1+x^2)^{p/2}
  \le 2+x^2
  \le 2+2\phi_p(x).
\end{equation}
The first inequality follows from \(\exp(-u)\le1\) for \(u\ge0\).  The second
uses \(p/2\le1\) and \(1+x^2\ge1\), which imply
\((1+x^2)^{p/2}\le1+x^2\).  The last inequality follows from
\eqref{eq:appendix-phi-lower}.  Thus
\(\lvert\phi_p''(x)\rvert\le2+2(\phi_p(x)-\phi_p^*)\) for all \(x\).
Applying the local-model implication of
\citet[Lemma~2]{liu2025warmup} with \(\rho=1\) therefore shows that \(\phi_p\)
satisfies Assumption~\ref{ass:h0h1-smoothness} with \(H_0=H_1=2\).

Suppose, toward a contradiction, that \(\phi_p\) is
\((L_0,L_1)\)-smooth for some finite \(L_0,L_1\ge0\).  Since \(\phi_p\) is
twice continuously differentiable, the original Hessian formulation of
\((L_0,L_1)\)-smoothness introduced by
\citet{zhang2019gradientclipping} requires the pointwise condition
\begin{equation}
\label{eq:appendix-one-dimensional-l0l1-necessary}
  \lvert\phi_p''(x)\rvert
  \le
  L_0+L_1\lvert\phi_p'(x)\rvert,
  \qquad x\in\mathbb R.
\end{equation}
The same necessary condition follows from the first-order local-gradient
formulation of \citet{zhang2020improvedclipping}: divide its inequality by
\(\lvert y-x\rvert\) and let \(y\to x\), so that the difference quotient
converges to \(\lvert\phi_p''(x)\rvert\).
At every integer \(n\ge1\), \(\sin(\pi n)=0\), and hence
\begin{equation}
\label{eq:appendix-phi-hessian-integers}
  \phi_p''(n)
  =q_p(n)
  =1+(1+n^2)^{p/2}
  \ge n^p,
\end{equation}
where the inequality follows from \(1+n^2\ge n^2\).

We next bound \(\phi_p'(n)\).  Set
\(a(t):=(1+t^2)^{p/2}\),
\(I_m:=[m-\frac12,m+\frac12]\), and
\(A_m:=(1+m^2)^{p/2}\).  If \(m\ge1\) and \(t\in I_m\), then
\(m/2\le t\le3m/2\), so
\[
  \frac{1+m^2}{4}
  \le 1+t^2
  \le \frac94(1+m^2).
\]
Raising these inequalities to the positive power \(p/2\) and using
\(p\le2\) gives
\begin{equation}
\label{eq:appendix-am-bounds}
  \frac{A_m}{4}
  \le
  2^{-p}A_m
  \le
  a(t)
  \le
  \left(\frac32\right)^p A_m
  \le
  \frac94A_m.
\end{equation}
For \(u:=t-m\in[-\frac12,\frac12]\), concavity of \(\sin(\pi u)\) on
\([0,\frac12]\), together with symmetry, gives
\(\lvert\sin(\pi t)\rvert=\lvert\sin(\pi u)\rvert\ge2\lvert u\rvert\).
Combining
\(\lvert\sin(\pi t)\rvert\ge2\lvert u\rvert\) with the lower bound on
\(a(t)\) in \eqref{eq:appendix-am-bounds}, we obtain
\(a(t)^2\sin^2(\pi t)\ge A_m^2u^2/4\).  Therefore,
\begin{equation}
\label{eq:appendix-interval-integral}
\begin{aligned}
  &\int_{I_m}a(t)\exp\!\left(-a(t)^2\sin^2(\pi t)\right)dt \\
  &\quad\le
  \frac94A_m\int_{-1/2}^{1/2}
  \exp\!\left(-\frac{A_m^2u^2}{4}\right)du
  \le
  \frac94A_m\int_{-\infty}^{\infty}
  \exp\!\left(-\frac{A_m^2u^2}{4}\right)du
  =
  \frac92\sqrt{\pi}.
\end{aligned}
\end{equation}
The first inequality uses the upper bound \(a(t)\le9A_m/4\) and the lower
bound on the exponent, the second enlarges the integration domain, and the
last equality is the Gaussian integral after the substitution
\(v=A_mu/2\).  For \(m=0\), \(t\in I_0\) implies
\(a(t)\le(5/4)^{p/2}\le5/4\), so the same integral over \(I_0\) is at most
\(5/4<9\sqrt\pi/2\).  The intervals \(I_0,\ldots,I_n\) cover \([0,n]\) up
to endpoints.  Since the integrand is nonnegative,
\eqref{eq:appendix-interval-integral} and the \(m=0\) bound give, for every
integer \(n\ge1\),
\begin{equation}
\label{eq:appendix-phi-gradient-integers}
\begin{aligned}
  \phi_p'(n)
  &=\int_0^n q_p(t)\,dt \\
  &=n+\int_0^n a(t)
    \exp\!\left(-a(t)^2\sin^2(\pi t)\right)dt \\
  &\le n+\sum_{m=0}^n\int_{I_m}a(t)
    \exp\!\left(-a(t)^2\sin^2(\pi t)\right)dt \\
  &\le n+\frac92\sqrt\pi(n+1)
  \le (1+9\sqrt\pi)n.
\end{aligned}
\end{equation}
The last inequality uses \(n+1\le2n\), valid for \(n\ge1\).  Substituting
\eqref{eq:appendix-phi-hessian-integers} and
\eqref{eq:appendix-phi-gradient-integers} into the necessary
\((L_0,L_1)\)-smoothness condition
\eqref{eq:appendix-one-dimensional-l0l1-necessary} gives
\[
  n^p
  \le
  L_0+L_1(1+9\sqrt\pi)n.
\]
After division by \(n\), the left-hand side is \(n^{p-1}\to\infty\) because
\(p>1\), whereas the right-hand side is at most
\(L_0+L_1(1+9\sqrt\pi)\).  This contradiction proves that no finite pair
\((L_0,L_1)\) can satisfy the generalized smoothness condition for \(\phi_p\).
\end{proof}

\begin{examplebox}
\noindent\textbf{Example~\ref{ex:multidimensional-separation}
(Multidimensional).}\enspace
{\itshape For \(d\ge2\), let
\[
  \Phi_p(x_1,\ldots,x_d)
  =
  \phi_p(x_1)+\frac12\sum_{i=2}^d x_i^2,
\]
where \(\phi_p\) is the function from
Example~\ref{ex:one-dimensional-separation}.  Then \(\Phi_p\) is strongly
convex, \((H_0,H_1)\)-smooth \((H_0\text{=}H_1\text{=}2)\), but not
\((L_0,L_1)\)-smooth~for~finite~constants.\par}
\end{examplebox}

\begin{proof}
Differentiating the definition of \(\Phi_p\) gives
\[
  \nabla\Phi_p(x)
  =
  \bigl(\phi_p'(x_1),x_2,\ldots,x_d\bigr)
\]
and
\[
  \nabla^2\Phi_p(x)
  =
  \operatorname{diag}\bigl(\phi_p''(x_1),1,\ldots,1\bigr).
\]
The proof of Example~\ref{ex:one-dimensional-separation} shows that
\(\phi_p''(x_1)\ge1\).  Hence \(\nabla^2\Phi_p(x)\succeq I_d\), so
\(\Phi_p\) is 1-strongly convex.  Moreover,
\(\nabla\Phi_p(\mathbf 0)=\mathbf 0\) and \(\Phi_p(\mathbf 0)=0\); therefore
\(\mathbf 0\in\mathbb R^d\) is the unique minimizer and \(\Phi_p^*=0\).

Since \(\phi_p''(x_1)\ge1\), the largest diagonal entry of the Hessian is
\(\phi_p''(x_1)\).  Consequently,
\[
  \norm{\nabla^2\Phi_p(x)}
  =
  \phi_p''(x_1)
  \le
  2+2\phi_p(x_1)
  \le
  2+2\Phi_p(x).
\]
Here the first inequality follows from
\eqref{eq:appendix-one-dimensional-hessian}, and the second follows from
\(\Phi_p(x)=\phi_p(x_1)+\frac12\sum_{i=2}^d x_i^2\ge\phi_p(x_1)\).
Thus
\(\norm{\nabla^2\Phi_p(x)}\le2+2(\Phi_p(x)-\Phi_p^*)\), and the
local-model implication of \citet[Lemma~2]{liu2025warmup} with \(\rho=1\)
shows that \(\Phi_p\) satisfies Assumption~\ref{ass:h0h1-smoothness} with
\(H_0=H_1=2\).

Suppose, toward a contradiction, that \(\Phi_p\) is
\((L_0,L_1)\)-smooth for some finite \(L_0,L_1\ge0\).  Because \(\Phi_p\) is
twice continuously differentiable, the Hessian formulation of
\citet{zhang2019gradientclipping} requires
\begin{equation}
\label{eq:appendix-multidimensional-l0l1-necessary}
  \norm{\nabla^2\Phi_p(x)}
  \le
  L_0+L_1\norm{\nabla\Phi_p(x)},
  \qquad x\in\mathbb R^d.
\end{equation}
Equivalently, \eqref{eq:appendix-multidimensional-l0l1-necessary} follows
from the first-order formulation of
\citet{zhang2020improvedclipping} by taking \(y=x+tv\) for an arbitrary unit
vector \(v\), dividing by \(\lvert t\rvert\), letting \(t\to0\), and then
taking the supremum over all such \(v\).
For every integer \(n\ge1\), set \(x^{(n)}:=(n,0,\ldots,0)\).  The estimates
in \eqref{eq:appendix-phi-hessian-integers} and
\eqref{eq:appendix-phi-gradient-integers} give
\[
  \norm{\nabla^2\Phi_p(x^{(n)})}
  =\phi_p''(n)
  \ge n^p
\]
and
\[
  \norm{\nabla\Phi_p(x^{(n)})}
  =\lvert\phi_p'(n)\rvert
  \le(1+9\sqrt\pi)n.
\]
Substitution into
\eqref{eq:appendix-multidimensional-l0l1-necessary} yields
\[
  n^p
  \le
  L_0+L_1(1+9\sqrt\pi)n.
\]
Dividing by \(n\) gives
\(n^{p-1}\le L_0/n+L_1(1+9\sqrt\pi)\).  The right-hand side remains bounded
as \(n\to\infty\), whereas the left-hand side diverges because \(p>1\).
This contradiction proves that \(\Phi_p\) is not \((L_0,L_1)\)-smooth for
any finite \(L_0,L_1\).
\end{proof}

\begin{revision}
\section{Missing Proofs: The Upper Bounds}
\label{app:unified-acceleration}

To streamline the presentation and avoid repeating the same analysis, we
first formulate a unified algorithm that covers all three oracle settings
considered in this paper and establish a general upper bound for it. After
proving this general result, we show how the unified algorithm and its analysis
specialize to Algorithm~\ref{alg:full-gradient-inner} in the
full-gradient setting~(Appendix~\ref{app:full-gradient-proof}), to
Algorithm~\ref{alg:coordinate-inner} with uniform coordinate
sampling~(Appendix~\ref{app:coordinate-proof}), and to
Algorithm~\ref{alg:nonuniform-coordinate-inner} with non-uniform coordinate
sampling~(Appendix~\ref{app:nonuniform-coordinate-proof}).

Motivated by the classical block-coordinate framework of
\citet{nesterov2012coordinate}, we adopt a block-coordinate method as the
unified algorithm, since this framework contains all three oracle settings as
special cases. Before presenting the algorithm itself, we introduce the
notation needed to state it.

\vspace{-0.5em}
\paragraph{Block-coordinate notation.}

Let the \(d\) coordinates be partitioned into \(n\) blocks, and let
\(d_\xi\ge1\) denote the size of block \(\xi\), so that
\(\sum_{\xi=1}^n d_\xi=d\). For each \(\xi=1,\ldots,n\), let
\(U_\xi\in\mathbb R^{d\times d_\xi}\) collect the columns of the
\(d\times d\) identity matrix \(I_d\) corresponding to block \(\xi\).
Then every
\(x\in\mathbb R^d\) has the unique representation
\(x=\sum_{\xi=1}^n U_\xi x^{(\xi)}\), where
\(x^{(\xi)}:=U_\xi^\top x\in\mathbb R^{d_\xi}\) denotes its \(\xi\)-th
block. For \(h\in\mathbb R^{d_\xi}\), the vector
\(U_\xi h\in\mathbb R^d\) places the entries of \(h\) in block \(\xi\)
and sets all remaining coordinates to zero. Conversely, \(U_\xi^\top x\)
extracts the \(\xi\)-th block of \(x\). We use the identities
\begin{equation}
\label{eq:block-coordinate-identities}
 \nabla_\xi f(x):=U_\xi^\top\nabla f(x),\qquad
 \norm{U_\xi h}=\norm h,\qquad
 \sum_{\xi=1}^n U_\xi U_\xi^\top=I_d.
\end{equation}
The first identity defines the block gradient. The second follows because
the columns of \(U_\xi\) are distinct coordinate vectors, so
\(U_\xi^\top U_\xi=I_{d_\xi}\); the last follows because the blocks
partition all \(d\) coordinates. Throughout the analysis,
\(\xi\) denotes a block index, whereas \(i\) is reserved for individual
coordinates.

Using the notation above, the full-gradient and coordinate
\((H_0,H_1)\)-smoothness assumptions can both be recovered as special cases
of the following block assumption.

\begin{assumptionbox}
\begin{assumption}[Block \((H_0,H_1)\)-smoothness]
\label{ass:block-h0h1-smoothness}
For every \(\xi\in\{1,\ldots,n\}\), there exist constants
\(H_{0,\xi},H_{1,\xi}\ge0\) such that for every
\(x\in\mathbb R^d\) and \(h\in\mathbb R^{d_\xi}\), if
\(\norm{h}\le r_{1,\xi}\), then
\[
 f(x+U_\xi h)
 \le f(x)+\inner{\nabla_\xi f(x)}{h}
     +(H_{0,\xi}+H_{1,\xi}\F(x))\norm{h}^2,
\]
where \(r_{1,\xi}=1/((2\sqrt3+\sqrt6)\sqrt{H_{1,\xi}})\) for
\(H_{1,\xi}>0\) and \(r_{1,\xi}=\infty\) otherwise.
\end{assumption}
\end{assumptionbox}

The full-gradient smoothness assumption is recovered by taking a single
block with \(n=1\), \(d_1=d\), \(U_1=I_d\), \(H_{0,1}=H_0\), and
\(H_{1,1}=H_1\). In this case, \(\nabla_1 f(x)=\nabla f(x)\), and
Assumption~\ref{ass:block-h0h1-smoothness} coincides with
Assumption~\ref{ass:h0h1-smoothness}. The coordinate smoothness assumption
is recovered by taking \(n=d\) singleton blocks and indexing them by
\(i=1,\ldots,d\), with \(d_i=1\) and \(U_i=e_i\). Then \(h\) is a scalar and
the block gradient \(\nabla_i f(x)\) is the partial derivative, so
Assumption~\ref{ass:block-h0h1-smoothness} coincides with Assumption~\ref{ass:coordinate-h0h1-smoothness}.

To formulate the unified algorithm for all three oracle settings, let
\(p=(p_1,\ldots,p_n)\) be a probability distribution over the blocks, with
\(p_\xi>0\) and \(\sum_{\xi=1}^n p_\xi=1\), and define
\begin{equation}
\label{eq:block-effective-curvatures}
 \boxed{\displaystyle
 K_0:=\max_{1\le\xi\le n}\frac{H_{0,\xi}}{p_\xi^2},
 \qquad
 K_1:=\max_{1\le\xi\le n}\frac{H_{1,\xi}}{p_\xi^2}}.
\end{equation}
These quantities combine the block curvatures with the sampling probabilities
and are the only curvature parameters that enter the unified upper bound. In
the full-gradient setting, \(K_0=H_0\) and \(K_1=H_1\); with uniform
coordinate sampling, \(K_0=d^2H_0\) and \(K_1=d^2H_1\); and with the
non-uniform probabilities defined in
Appendix~\ref{app:nonuniform-coordinate-proof},
\(K_0\le4(S_{1/2}^{(0)})^2\) and
\(K_1\le4(S_{1/2}^{(1)})^2\), where
\(S_{1/2}^{(j)}=\sum_i\sqrt{H_{j,i}}\). The corresponding derivations are
given in Appendices~\ref{app:full-gradient-proof},
\ref{app:coordinate-proof}, and~\ref{app:nonuniform-coordinate-proof},
respectively.

We can now state the unified algorithm
(see Algorithm~\ref{alg:accelerated-block}).

\setlength{\intextsep}{2pt}
\setlength{\textfloatsep}{2pt}
\begin{algorithm}[!ht]
\caption{\rev{Accelerated Projected Block-Coordinate Method}}
\label{alg:accelerated-block}
\color{revisionblue}
\begin{algorithmic}[1]
\renewcommand{\algorithmicrequire}{\textbf{Input:}}
\renewcommand{\algorithmicensure}{\textbf{Return:}}
\REQUIRE Objective \(f\), initial point \(x^0\), radius \(R\), iteration budget \(N\), initial weight \(A_0>0\), and prescribed rules for the \newrev{query weight \(\widehat b_k>0\), update weight \(b_k>0\),} and \newrev{\(c_k\ge0\)}
\STATE Initialize \(x_0\gets x^0\), \(v_0\gets x^0\)
\FOR{\(k=0,\ldots,N-1\)}
 \STATE Compute \(y_k\gets\dfrac{A_kx_k+\newrev{\widehat b_k}v_k}{A_k+\newrev{\widehat b_k}}\) \label{line:block-y}
 \STATE Sample a block index \(\xi_k\in\{1,\ldots,n\}\) from \(p\), so that
 \orrev{\(\Pr(\xi_k=\xi\mid\mathcal F_k)=p_\xi\) for every \(\xi\), where
 \(\mathcal F_k\) is the history before this sampling} \label{line:block-sample}
 \STATE Update \(v_{k+1}\gets P_{B(x^0,3R)}\!\left(v_k-\dfrac{b_k}{p_{\xi_k}}U_{\xi_k}\nabla_{\xi_k}f(y_k)\right)\) \label{line:block-v-update}
 \STATE Set \(\Delta v_k\gets v_{k+1}-v_k\) \label{line:block-delta}
 \STATE Compute \(a_k\gets A_k+b_k-
 \dfrac{c_k\newrev{b_k^2\norm{(\Delta v_k)^{(\xi_k)}}^2}}{A_k+b_k}\)
 \label{line:block-a}
 \STATE Compute \(z_k\gets\newrev{\dfrac{A_kx_k+b_kv_k}{A_k+b_k}}+\dfrac{b_k}{p_{\xi_k}a_k}U_{\xi_k}(\Delta v_k)^{(\xi_k)}\) \label{line:block-z}
 \STATE Update \(A_{k+1}\gets\max\{A_0,a_k\}\) \label{line:block-A-update}
 \STATE Update \(x_{k+1}\gets\dfrac{a_kz_k+(A_{k+1}-a_k)x^0}{A_{k+1}}\) \label{line:block-x-update}
\ENDFOR
\ENSURE \(x_N\)
\end{algorithmic}
\end{algorithm}

Algorithm~\ref{alg:accelerated-block} uses the ball
\(B(x^0,3R):=\{x\in\mathbb R^d:\norm{x-x^0}\le3R\}\) and the Euclidean
projection
\(P_Q(u):=\argmin_{w\in Q}\frac12\norm{w-u}^2\) of \(u\) onto a nonempty
closed convex set \(Q\). The sequence
\(x_k\) is the primal sequence, \(v_k\) is the auxiliary accelerated
sequence, and \(A_k\) is the current acceleration weight; the additional
weight \(b_k\) determines the weighted average \(y_k\) at which the block
gradient is evaluated. Before projection, the factor \(1/p_{\xi_k}\) makes
the embedded sampled block gradient unbiased, while the projection keeps the
auxiliary point inside \(B(x^0,3R)\). The vector \(\Delta v_k\) is the
resulting projected displacement, and \((\Delta v_k)^{(\xi_k)}\) is its
component in the sampled block.
The point \(z_k\) is obtained by moving \(y_k\) only in the sampled block.
Finally, if
\(a_k\ge A_0\), then \(A_{k+1}=a_k\) and \(x_{k+1}=z_k\); otherwise,
\(z_k\) is averaged with \(x^0\) to keep the accepted weight at least
\(A_0\). The method returns \(x_N\).

Algorithm~\ref{alg:accelerated-block} retains from the accelerated
block-coordinate methods of \citet{nesterov2012coordinate} and
\citet{Nesterov_2017} the coupling point \(y_k\), the primal and auxiliary
sequences \(x_k\) and \(v_k\), and the inverse-probability scaling by
\(p_{\xi_k}\) of the sampled block gradient. The choice of \(b_k\) follows
the same two-regime principle as the one-phase weight schedule of
\citet{elbakkali2026acceleration}: it restricts weight growth in the
high-curvature regime governed by \(K_1\) and recovers the classical
square-root growth governed by \(K_0\). The new ingredients are the
projection in the \(v_{k+1}\)-update, the actual projected displacement
\(\Delta v_k\), and the realized correction \newrev{in the \(a_k\)-update}, which transforms the
provisional weight \(A_k+b_k\) into the corrected weight \(a_k\) after the
block \(\xi_k\) is sampled. These modifications enforce the local radius
condition and account for the gap-dependent block curvature.

We are now ready to state the theoretical guarantee for the unified method.

\begin{mainresultbox}
\begin{theorem}[Block-gradient complexity]
\label{thm:common-block}
Suppose Assumptions~\ref{ass:convexity} and
\ref{ass:block-h0h1-smoothness} hold for the block matrices
\(U_1,\ldots,U_n\) and the constants
\(H_{0,\xi},H_{1,\xi}\), \(\xi=1,\ldots,n\), \grev{and assume that
\(\norm{x^0-x^*}\le R\), where \(R>0\). Let
\(\Delta_0\ge f(x^0)-f^*\) be a deterministic upper bound.}
Fix \orrev{an accuracy \(\varepsilon>0\) and} sampling probabilities
\(p_\xi>0\), \(\xi=1,\ldots,n\), with
\(\sum_{\xi=1}^n p_\xi=1\), and let
\(K_0=\max_{1\le\xi\le n}H_{0,\xi}/p_\xi^2\) and
\(K_1=\max_{1\le\xi\le n}H_{1,\xi}/p_\xi^2\).
\orrev{If \(\Delta_0\le\varepsilon\), or if \(K_0=0\) and hence \(f\) is
constant by Remark~\ref{rem:constant-block}, take \(N=0\) and return \(x^0\).
Otherwise, set} \grev{\(A_0=R^2/(2\Delta_0)\) and choose
the parameters as follows:}
\begin{equation}
\label{eq:block-c-rule}
 c_k=
 \begin{cases}
 \newrev{4K_1},
 &\displaystyle \max_{0\le j\le k}A_j<
 \max\left\{A_0,
 \min\left\{
 \frac{(1+128R\sqrt{K_1})^2}{4096K_0},
 \frac{4eR^2}{\varepsilon}
 \right\}\right\},\\[0.8em]
 \newrev{4K_1\left(1+\log\dfrac{2\Delta_0}{\varepsilon}\right)},
 &\displaystyle \max_{0\le j\le k}A_j\ge
 \max\left\{A_0,
 \min\left\{
 \frac{(1+128R\sqrt{K_1})^2}{4096K_0},
 \frac{4eR^2}{\varepsilon}
 \right\}\right\},
 \end{cases},
\end{equation}
and
\[
 \newrev{\widehat b_k=}b_k=\frac{A_k}
 {1+64R\sqrt{\newrev{c_k}}+64\sqrt{K_0A_k}}.
\]
Then Algorithm~\ref{alg:accelerated-block} returns \(x_N\) such that
\(\E_{\xi_0,\ldots,\xi_{N-1}}[f(x_N)-f^*]\le\varepsilon\) after
\[
 \boxed{\displaystyle
 N=O\left(
 R\sqrt{\frac{K_0}{\varepsilon}}
 +\left(1+R\sqrt{K_1}\right)
 \left(1+\log\frac{\Delta_0}{\varepsilon}\right)
 \right)}
\]
iterations, using one block-gradient query and no function-value evaluations
per iteration.
\end{theorem}
\end{mainresultbox}

\begin{proof}[Proof of Theorem~\ref{thm:common-block}]
\begin{currentrevision}
Following \citet{Nesterov_2017}, we condition on the previously sampled blocks
and organize the proof in six steps. First, we define the potential and bound
its initial value. Second, we display the three terms created by block
smoothness, verify that the corrected weight and displacement make the
smoothness inequality valid, and bound each curvature term separately. Third,
we derive the projection and convexity estimates, add them to the smoothness
estimate, and pass from the candidate point to the point accepted by the
algorithm. Fourth, we normalize the resulting one-step estimate to obtain a
contraction in both regimes of \(c_k\). Fifth, we apply the contraction before
and after the switching time and bound the two resulting events. Finally, we
bound the pre-switch horizon \(N_1\) and the post-switch horizon \(N_2\), and
set the total iteration budget to \(N=N_1+N_2\).
\end{currentrevision}

\begin{greenrevision}
We may restrict attention to a nontrivial run. Indeed,
\(\Delta_0\le\varepsilon\) already certifies \(x^0\), while \(K_0=0\)
implies that every block is constant by
Remark~\ref{rem:constant-block}. Henceforth
\(\Delta_0>\varepsilon\), \(R>0\), and \(K_0>0\).
\end{greenrevision}
Fix the problem data. Throughout the analysis, choose \(A_0\), \(b_k\), and
\(c_k\) as in Theorem~\ref{thm:common-block}, and denote the threshold in the
two-regime choice of \(c_k\) by
\begin{equation}
\label{eq:block-switch-threshold}
 A_{\rm sw}:=
 \max\left\{A_0,
 \min\left\{
 \frac{(1+128R\sqrt{K_1})^2}{4096K_0},
 \frac{4eR^2}{\varepsilon}
 \right\}\right\}.
\end{equation}
Our choice of \(b_k\) is
\begin{equation}
\label{eq:block-b-rule}
 b_k=\frac{A_k}
 {1+64R\sqrt{\newrev{c_k}}+64\sqrt{K_0A_k}}.
\end{equation}
\begin{currentrevision}
The first branch of \eqref{eq:block-c-rule} gives \(c_k=4K_1\). In the
second branch, \(\Delta_0>\varepsilon\) gives
\(1+\log(2\Delta_0/\varepsilon)>1\), so
\(c_k=4K_1(1+\log(2\Delta_0/\varepsilon))\ge4K_1\). Hence, in both
branches,
\begin{equation}
\label{eq:block-c-lower-bound}
 c_k\ge4K_1\ge0.
\end{equation}
\end{currentrevision}
The quantities introduced only to organize the proof, such as
\(A_{\rm sw}\), \(\gamma\), \(t_k\), \(T\), \(N_1\), \(N_2\), and \(s_j\),
are analytical auxiliaries rather than additional method parameters.

\paragraph{Step 1: Potential and initialization.}
\begin{greenrevision}
\begin{currentrevision}
Remark~\ref{rem:constant-block}, proved after
Theorem~\ref{thm:common-block}, shows that every block with
\(H_{0,\xi}=0\) is a constant direction and may be omitted.
\end{currentrevision}

Let
\(\mathcal F_k:=\sigma(\xi_0,\ldots,\xi_{k-1})\) be the information generated
by the blocks sampled before iteration \(k\), with
\(\mathcal F_0=\{\varnothing,\Omega\}\). Conditional on \(\mathcal F_k\), the points
\(x_k,v_k,y_k\) and the parameters \(A_k,b_k,c_k\) are fixed; only the new
block \(\xi_k\) is random. This is the only probabilistic fact needed in the
one-step analysis.
\end{greenrevision}
\begin{currentrevision}
Accordingly, \(\E[\,\cdot\mid\mathcal F_k]\) averages only over the new block
\(\xi_k\), while quantities such as \(A_k,b_k\), and \(y_k\) are treated as
fixed in the one-step conditional expectation.
\end{currentrevision}

To turn growth of \(A_k\) into decay of the objective gap, define
\begin{equation}
\label{eq:block-potential}
 \Phi_k:=A_k\F(x_k)+\frac12\norm{v_k-x^*}^2.
\end{equation}
\begin{currentrevision}
This potential contains the two quantities controlled by the two updates of
the algorithm. The primal update controls \(\F(x_k)\), while the projection
update naturally produces the difference of squared distances
\(\norm{v_{k+1}-x^*}^2-\norm{v_k-x^*}^2\). Moreover, both terms in
\(\Phi_k\) are nonnegative, so the definition immediately implies
\[
 A_k\F(x_k)\le\Phi_k,
 \qquad
 \F(x_k)\le\frac{\Phi_k}{A_k}.
\]
Thus controlling \(\Phi_k\) and proving growth of \(A_k\) together produce
decay of the objective gap. This is why we track \(\Phi_k\) instead of
\(\F(x_k)\) alone.
\end{currentrevision}
By construction,
\begin{greenrevision}
\begin{equation}
\label{eq:block-adaptedness}
 A_k,\ b_k,\ c_k,\ x_k,\ v_k,\ y_k,\ \Phi_k
 \quad\text{are }\mathcal F_k\text{-measurable}.
\end{equation}
\begin{currentrevision}
Indeed, \(x_0,v_0,A_0\) are deterministic. Given
\(\mathcal F_k\), the parameter rules \eqref{eq:block-b-rule} and
\eqref{eq:block-c-rule} determine \(b_k,c_k\), and
Line~\ref{line:block-y} determines \(y_k\) before the new block is sampled
in Line~\ref{line:block-sample}. Lines~\ref{line:block-v-update}--%
\ref{line:block-x-update} then determine the state at iteration \(k+1\).
Induction gives \eqref{eq:block-adaptedness}.
\end{currentrevision}
The two terms in \(\Phi_k\) are nonnegative whenever \(A_k>0\), because
\(\F\ge0\).
\end{greenrevision}
At initialization,
\begin{equation}
\label{eq:block-initial-potential}
 \Phi_0
 \overset{\circledOne}{=}
 A_0\F(x^0)+\frac12\norm{x^0-x^*}^2
 \overset{\circledTwo}{\le}
 A_0\Delta_0+\frac12R^2
 \overset{\circledThree}{=}
 R^2,
\end{equation}
where in \(\circledOne\), we used \(x_0=v_0=x^0\); in \(\circledTwo\), we used
\(\F(x^0)\le\Delta_0\) and \(\norm{x^0-x^*}\le R\); in
\(\circledThree\), we substituted \(A_0=R^2/(2\Delta_0)\).
\begin{currentrevision}
The choice \(A_0=R^2/(2\Delta_0)\) balances the two initial contributions:
\[
 A_0\F(x^0)\le A_0\Delta_0=\frac{R^2}{2},
 \qquad
 \frac12\norm{x^0-x^*}^2\le\frac{R^2}{2}.
\]
Their sum is at most \(R^2\), which gives the deterministic starting bound
used in all later expectation estimates.
\end{currentrevision}
\begin{greenrevision}
Since \(v_0=x^0\in B(x^0,3R)\) and every auxiliary update in
Line~\ref{line:block-v-update} is projected onto this ball, induction gives
\begin{equation}
\label{eq:block-ball-membership}
 v_k\in B(x^0,3R)\quad\text{for every }k,
 \qquad x^*\in B(x^0,3R).
\end{equation}
The second inclusion follows from
\(\norm{x^0-x^*}\le R<3R\). These two facts will be used in Steps~2--4.
\end{greenrevision}

\paragraph{Step 2: \newrev{The candidate smoothness estimate.}}
Fix an iteration \(k\). Conditional on \(\mathcal F_k\), the quantities
\(A_k,b_k,c_k,x_k,v_k,y_k\) are fixed, whereas \(\xi_k\),
\(a_k\), \(z_k\), and the updated points are random.
\begin{currentrevision}
We begin with the only place where the local curvature assumption is used.
Line~\ref{line:block-z} can be written as
\[
 z_k
 =y_k+U_{\xi_k}
 \left(
 \frac{b_k}{p_{\xi_k}a_k}(\Delta v_k)^{(\xi_k)}
 \right).
\]
We will apply Assumption~\ref{ass:block-h0h1-smoothness} with the following
three explicit choices:
\[
 x=y_k,
 \qquad
 \xi=\xi_k,
 \qquad
 h=\frac{b_k}{p_{\xi_k}a_k}(\Delta v_k)^{(\xi_k)}.
\]
With these choices,
\[
 x+U_\xi h
 =y_k+U_{\xi_k}
 \left(
 \frac{b_k}{p_{\xi_k}a_k}(\Delta v_k)^{(\xi_k)}
 \right)
 =z_k.
\]
Once \eqref{eq:block-weight-positive} verifies that \(a_k>0\) and
\eqref{eq:block-radius-valid} verifies that
\(\norm{h}=\norm{z_k-y_k}\le r_{1,\xi_k}\),
Assumption~\ref{ass:block-h0h1-smoothness} yields
\[
\begin{aligned}
 f(z_k)
 &\le f(y_k)
 +\inner{\nabla_{\xi_k}f(y_k)}{
 \frac{b_k}{p_{\xi_k}a_k}(\Delta v_k)^{(\xi_k)}}\\
 &\quad+
 \bigl(H_{0,\xi_k}+H_{1,\xi_k}\F(y_k)\bigr)
 \norm{
 \frac{b_k}{p_{\xi_k}a_k}(\Delta v_k)^{(\xi_k)}}^2.
\end{aligned}
\]
Subtracting \(f^*\) from both sides replaces \(f(z_k)\) and \(f(y_k)\)
by \(\F(z_k)\) and \(\F(y_k)\), respectively. Pulling the scalar out of
the inner product and the squared norm then yields
\[
\begin{aligned}
 \F(z_k)
 &\le \F(y_k)
 +\frac{b_k}{p_{\xi_k}a_k}
 \inner{\nabla_{\xi_k}f(y_k)}{(\Delta v_k)^{(\xi_k)}}\\
 &\quad+
 \frac{b_k^2}{p_{\xi_k}^2a_k^2}
 \bigl(H_{0,\xi_k}+H_{1,\xi_k}\F(y_k)\bigr)
 \norm{(\Delta v_k)^{(\xi_k)}}^2.
\end{aligned}
\]
Because \eqref{eq:block-weight-positive} proves \(a_k>0\), we may multiply this
inequality by \(a_k\) without changing its direction. For the linear term,
one factor \(a_k\) cancels the denominator \(a_k\):
\[
 a_k\frac{b_k}{p_{\xi_k}a_k}
 \inner{\nabla_{\xi_k}f(y_k)}{(\Delta v_k)^{(\xi_k)}}
 =\frac{b_k}{p_{\xi_k}}
 \inner{\nabla_{\xi_k}f(y_k)}{(\Delta v_k)^{(\xi_k)}}.
\]
For the quadratic term, one of the two factors \(a_k\) in the denominator
is cancelled:
\[
 a_k\frac{b_k^2}{p_{\xi_k}^2a_k^2}
 \bigl(H_{0,\xi_k}+H_{1,\xi_k}\F(y_k)\bigr)
 \norm{(\Delta v_k)^{(\xi_k)}}^2
 =\frac{b_k^2}{p_{\xi_k}^2a_k}
 \bigl(H_{0,\xi_k}+H_{1,\xi_k}\F(y_k)\bigr)
 \norm{(\Delta v_k)^{(\xi_k)}}^2.
\]
Finally, distributing the last parenthesis separates the baseline and
gap-dependent curvature contributions. Hence the resulting inequality is
\begin{equation}
\label{eq:block-smoothness-roadmap}
\begin{aligned}
 a_k\F(z_k)
 &\le a_k\F(y_k)
 +\underbrace{
 \frac{b_k}{p_{\xi_k}}
 \inner{\nabla_{\xi_k}f(y_k)}{(\Delta v_k)^{(\xi_k)}}
 }_{T_{\rm lin}:\ \text{linear term}}\\
 &\quad+
 \underbrace{
 \frac{H_{0,\xi_k}b_k^2}{p_{\xi_k}^2a_k}
 \norm{(\Delta v_k)^{(\xi_k)}}^2
 }_{T_0:\ \text{baseline-curvature term}}
 +
 \underbrace{
 \frac{H_{1,\xi_k}b_k^2}{p_{\xi_k}^2a_k}
 \F(y_k)\norm{(\Delta v_k)^{(\xi_k)}}^2
 }_{T_1:\ \text{gap-dependent-curvature term}}.
\end{aligned}
\end{equation}
At this point, the derivation is conditional on the two checks just stated.
We prove them next, after which \eqref{eq:block-smoothness-roadmap} becomes a
valid inequality. We will then show that \(T_0\) is controlled by part of the
descent in the projection step, while \(T_1\) is absorbed by the correction
from \(A_k+b_k\) to \(a_k\). The term \(T_{\rm lin}\) will cancel exactly
with a term of the opposite sign from the projection step.
\end{currentrevision}
For this one-step calculation only, write
\begin{equation}
\label{eq:block-r-rho}
 r_k:=\newrev{\frac{b_k^2\norm{(\Delta v_k)^{(\xi_k)}}^2}
 {(A_k+b_k)^2}},
 \qquad
 \rho_k:=\frac{a_k}{A_k+b_k}=1-c_kr_k.
\end{equation}
\begin{currentrevision}
The quantity \(r_k\) is the realized squared block displacement normalized by
the squared provisional weight \((A_k+b_k)^2\). The quantity \(\rho_k\) is the fraction of the provisional weight
\(A_k+b_k\) that remains after the correction. Thus proving \(a_k>0\) is
equivalent to proving \(c_kr_k<1\); the estimate
\eqref{eq:block-weight-positive} establishes the stronger bound
\(c_kr_k<1/4\).
\end{currentrevision}
Before applying Assumption~\ref{ass:block-h0h1-smoothness} at \(y_k\) to the
displacement \(z_k-y_k\), we verify two facts.

\noindent\fbox{\strut\textit{2.1. Positivity of \(a_k\).}}
By \eqref{eq:block-ball-membership}, both \(v_k\) and \(v_{k+1}\) belong to
\(B(x^0,3R)\), so the triangle
inequality gives \(\norm{\Delta v_k}\le6R\). We first bound
\(b_k/(A_k+b_k)\) in order to control \(c_kr_k\). Since
\(A_k+b_k\ge A_k\), \eqref{eq:block-b-rule} gives
\begin{equation}
\label{eq:block-b-ratio}
 \frac {b_k}{A_k+b_k}
 \overset{\circledOne}{\le}\frac{b_k}{A_k}
 \overset{\circledTwo}{=}
 \frac{1}{1+64R\sqrt{\newrev{c_k}}+64\sqrt{K_0A_k}}
 \overset{\circledThree}{\le}
 \frac{1}{1+64R\sqrt{\newrev{c_k}}},
\end{equation}
where in \(\circledOne\), we used \(A_k+b_k\ge A_k\); in
\(\circledTwo\), we substituted \eqref{eq:block-b-rule}; in
\(\circledThree\), we dropped
\(64\sqrt{K_0A_k}\ge0\).
Using the definition of \(r_k\) in \eqref{eq:block-r-rho} and
\(\norm{(\Delta v_k)^{(\xi_k)}}\le\norm{\Delta v_k}\), we obtain
\begin{equation}
\label{eq:block-weight-positive}
 0\overset{\circledOne}{\le}c_kr_k
 \overset{\circledTwo}{\le}\newrev{\frac{36c_kR^2}
 {(1+64R\sqrt{c_k})^2}}
 \overset{\circledThree}{\le}\frac{36}{64^2}
 \overset{\circledFour}{<}\frac14,\qquad
 \frac34(A_k+b_k)
 \overset{\circledFive}{\le}a_k
 \overset{\circledSix}{\le}A_k+b_k,
\end{equation}
where in \(\circledOne\), we used \newrev{\(c_k\ge0\) from
\eqref{eq:block-c-lower-bound}} and \(r_k\ge0\),
which follows from \eqref{eq:block-r-rho}; in
\(\circledTwo\), we substituted the definition of \(r_k\) and used
\(\norm{(\Delta v_k)^{(\xi_k)}}\le6R\), and
\eqref{eq:block-b-ratio}; in
\(\circledThree\), we used
\newrev{\(1+64R\sqrt{c_k}\ge64R\sqrt{c_k}\) when \(c_k>0\).
When \(c_k=0\), the bound is immediate}; in
\(\circledFour\), we used \(36/64^2<1/4\). Finally,
\(a_k=(A_k+b_k)(1-c_kr_k)\); the bounds
\(3/4\le1-c_kr_k\le1\) give \(\circledFive\) and \(\circledSix\).
Hence \(a_k>0\), so the candidate point \(z_k\) is well defined. Since
\(A_k+b_k\ge A_k\ge A_0\), we also have
\(A_0\le A_{k+1}=\max\{A_0,a_k\}\le A_k+b_k\).
\begin{currentrevision}
This calculation explains the term \(64R\sqrt{c_k}\) in the denominator of
\(b_k\). The largest possible displacement contributes \(36R^2\), while
squaring \(64R\sqrt{c_k}\) produces \(64^2R^2c_k\), which cancels the
factor \(c_kR^2\) and keeps the relative correction uniformly below one.
\end{currentrevision}

\noindent\fbox{\strut\textit{2.2. Radius condition.}}
For \(H_{1,\xi_k}>0\), we have
\begin{equation}
\label{eq:block-radius-valid}
\begin{aligned}
 \norm{z_k-y_k}
  &\overset{\circledOne}{=}
 \frac{b_k\norm{(\Delta v_k)^{(\xi_k)}}}{p_{\xi_k}a_k}
 \overset{\circledTwo}{\le}
 \frac{1}{1+64R\sqrt{c_k}}
 \frac43\frac{6R}{p_{\xi_k}}
 \overset{\circledThree}{=}
 \newrev{\frac{8R}{p_{\xi_k}(1+64R\sqrt{c_k})}}
 \overset{\circledFour}{\le}
 \newrev{\frac{1}{8p_{\xi_k}\sqrt{c_k}}}\\
 &
 \overset{\circledFive}{\le}
 \newrev{\frac{1}{16\sqrt{H_{1,\xi_k}}}}
 \overset{\circledSix}{<}
 \frac{1}{(2\sqrt3+\sqrt6)\sqrt{H_{1,\xi_k}}}
 \newrev{=r_{1,\xi_k}}.
\end{aligned}
\end{equation}
where in \(\circledOne\), we used Line~\ref{line:block-z} and
\newrev{\eqref{eq:block-coordinate-identities}}; in \(\circledTwo\), we used the bound
\((A_k+b_k)/a_k\le4/3\) from \eqref{eq:block-weight-positive},
\eqref{eq:block-b-ratio}, and
\newrev{\(\norm{(\Delta v_k)^{(\xi_k)}}\le6R\)}; in
\(\circledThree\), we multiplied the constants \((4/3)6R=8R\); in
\(\circledFour\), we used
\newrev{\(1+64R\sqrt{c_k}\ge64R\sqrt{c_k}\)}; in
\begin{currentrevision}
\(\circledFive\), \eqref{eq:block-effective-curvatures} gives
\(H_{1,\xi_k}/p_{\xi_k}^2\le K_1\), while
\eqref{eq:block-c-lower-bound} gives \(c_k\ge4K_1\). Thus
\(p_{\xi_k}^2c_k\ge4H_{1,\xi_k}\), which implies
\(p_{\xi_k}\sqrt{c_k}\ge2\sqrt{H_{1,\xi_k}}\);
\end{currentrevision}
in
\(\circledSix\), we used \(2\sqrt3+\sqrt6<16\) and the definition of \(r_{1,\xi_k}\) in
Assumption~\ref{ass:block-h0h1-smoothness}. When
\(H_{1,\xi_k}=0\), the radius is infinite.
Thus the radius condition holds, completing both checks.

\noindent\fbox{\strut\textit{\newrev{2.3. Applying smoothness and controlling
the curvature terms.}}}
\begin{currentrevision}
The application of Assumption~\ref{ass:block-h0h1-smoothness}, including its
three arguments and every algebraic substitution, was written explicitly in
\eqref{eq:block-smoothness-roadmap}. It remained conditional on two
hypotheses. The first is now verified by
\eqref{eq:block-weight-positive}, which gives \(a_k>0\). The second is
verified by \eqref{eq:block-radius-valid}, which gives
\(\norm{z_k-y_k}<r_{1,\xi_k}\). Therefore the displayed application of the
assumption is valid, and multiplication by the positive number \(a_k\)
preserves the direction of the inequality.
\end{currentrevision}

We first bound the coefficient of the baseline-curvature term \(T_0\):
\begin{equation}
\label{eq:block-baseline-small}
\begin{aligned}
 \frac{H_{0,\xi_k}b_k^2}{p_{\xi_k}^2a_k}
 &\overset{\circledOne}{\le}
 \frac43\frac{H_{0,\xi_k}}{p_{\xi_k}^2}\frac{b_k^2}{A_k}\\
 &\overset{\circledTwo}{=}
 \frac43\frac{H_{0,\xi_k}}{p_{\xi_k}^2}
 \frac{A_k}{(1+64R\sqrt{c_k}+64\sqrt{K_0A_k})^2}\\
 &\overset{\circledThree}{\le}\newrev{\frac43\frac{K_0 A_k}
 {(1+64R\sqrt{c_k}+64\sqrt{K_0A_k})^2}}
 \overset{\circledFour}{\le}\frac4{3\cdot64^2}
 \overset{\circledFive}{<}\frac18.
\end{aligned}
\end{equation}
In \(\circledOne\), we used
\(a_k\ge3(A_k+b_k)/4\ge3A_k/4\), which implies
\(1/a_k\le4/(3A_k)\). In \(\circledTwo\), we substituted
\[
 b_k=\frac{A_k}{1+64R\sqrt{c_k}+64\sqrt{K_0A_k}}
\]
and simplified \(b_k^2/A_k\).
\newrev{In \(\circledThree\), we used
\eqref{eq:block-effective-curvatures}, which gives
\(H_{0,\xi_k}/p_{\xi_k}^2\le K_0\).}
In \(\circledFour\), we dropped the
first two nonnegative terms in the denominator and used
\((64\sqrt{K_0A_k})^2=64^2K_0A_k\) when \(K_0>0\); when \(K_0=0\), the
numerator is zero. In \(\circledFive\), we used
\(4/(3\cdot64^2)<1/8\). Hence
\begin{equation}
\label{eq:block-T0-bound}
 T_0\le\frac18\norm{\Delta v_k}^2,
\end{equation}
where we also used
\(\norm{(\Delta v_k)^{(\xi_k)}}\le\norm{\Delta v_k}\).
\begin{currentrevision}
Thus, the term \(64\sqrt{K_0A_k}\) in the denominator of \(b_k\) keeps the
baseline-curvature coefficient at most
\(4/(3\cdot64^2)\), independently of the current value of \(A_k\).
\end{currentrevision}

Starting from the now valid inequality
\eqref{eq:block-smoothness-roadmap}, we obtain
\begin{align}
 a_k\F(z_k)
 &\overset{\circledOne}{\le}a_k\F(y_k)+T_{\rm lin}+T_0+T_1\\
 &\overset{\circledTwo}{\le}\newrev{\left(a_k+
      \frac{H_{1,\xi_k}(A_k+b_k)^2r_k}{p_{\xi_k}^2a_k}\right)}\F(y_k)
      +T_{\rm lin}+\frac18\norm{\Delta v_k}^2.
\label{eq:block-smoothness-with-terms}
\end{align}
In \(\circledOne\), we used \eqref{eq:block-smoothness-roadmap}. In
\(\circledTwo\), we applied \eqref{eq:block-T0-bound} to \(T_0\) and used
\newrev{\((A_k+b_k)^2r_k=
b_k^2\norm{(\Delta v_k)^{(\xi_k)}}^2\)} to rewrite \(T_1\).

\begin{currentrevision}
It remains to verify that the correction in \(a_k\) absorbs \(T_1\).
Define the unused curvature slack by
\begin{equation}
\label{eq:block-curvature-slack}
 L_k:=A_k+b_k-a_k-
 \frac{H_{1,\xi_k}(A_k+b_k)^2r_k}{p_{\xi_k}^2a_k}.
\end{equation}
The difference \(A_k+b_k-a_k\) is the amount subtracted from the provisional
weight, while the final fraction is the additional coefficient of
\(\F(y_k)\) created by the \(H_1\)-curvature. Therefore, \(L_k\ge0\) means
that the correction is sufficient. Direct substitution gives
\begin{equation}
\label{eq:block-curvature-slack-positive}
\begin{aligned}
 L_k
 &\overset{\circledOne}{=}
 (A_k+b_k)c_kr_k-
 \frac{H_{1,\xi_k}(A_k+b_k)^2r_k}
 {p_{\xi_k}^2\rho_k(A_k+b_k)}\\
 &\overset{\circledTwo}{=}
 (A_k+b_k)r_k
 \left(c_k-\frac{H_{1,\xi_k}}{p_{\xi_k}^2\rho_k}\right)
 \overset{\circledThree}{\ge}0.
\end{aligned}
\end{equation}
In \(\circledOne\), we used
\(A_k+b_k-a_k=(A_k+b_k)(1-\rho_k)=(A_k+b_k)c_kr_k\),
and \(a_k=\rho_k(A_k+b_k)\). In \(\circledTwo\), we cancelled the common
factor \(A_k+b_k\). In \(\circledThree\), the factors
\(A_k+b_k\) and \(r_k\) are nonnegative, while
\[
 \frac{H_{1,\xi_k}}{p_{\xi_k}^2\rho_k}
 \le \frac43\frac{H_{1,\xi_k}}{p_{\xi_k}^2}
 \le \frac43K_1
 \le c_k.
\]
The first inequality follows from \(\rho_k\ge3/4\) in
\eqref{eq:block-weight-positive}. The second follows from
\(H_{1,\xi_k}/p_{\xi_k}^2\le K_1\) in
\eqref{eq:block-effective-curvatures}. For the last inequality,
\eqref{eq:block-c-lower-bound} gives \(c_k\ge4K_1\ge(4/3)K_1\).
The stronger lower bound \(4K_1\) is also used in
\(\circledFive\) of \eqref{eq:block-radius-valid}. Moreover, the two exact
values in \eqref{eq:block-c-rule} yield
\eqref{eq:block-gamma-curvature}, which is used in the normalized-potential
estimate.

Rearranging \eqref{eq:block-curvature-slack} gives
\[
 a_k+\frac{H_{1,\xi_k}(A_k+b_k)^2r_k}{p_{\xi_k}^2a_k}
 =A_k+b_k-L_k.
\]
Substitution into \eqref{eq:block-smoothness-with-terms} yields
\begin{equation}
\label{eq:block-trial-descent}
 a_k\F(z_k)
 \le (A_k+b_k-L_k)\F(y_k)
 +T_{\rm lin}+\frac18\norm{\Delta v_k}^2.
\end{equation}
Thus \(T_0\) has been bounded by \(\norm{\Delta v_k}^2/8\), \(T_1\) has
been absorbed into \(A_k+b_k-L_k\), and only \(T_{\rm lin}\) remains to be
cancelled by the projection step.
\end{currentrevision}

\paragraph{Step 3: \newrev{Projection, coupling, and the one-step potential estimate.}}
\begin{currentrevision}
We now derive the two estimates that will be combined with
\eqref{eq:block-trial-descent}. The projection estimate supplies the opposite
of \(T_{\rm lin}\) and a negative squared-displacement term. The convexity
coupling then changes the coefficient \(A_k+b_k\) of \(\F(y_k)\) into the
old potential \(\Phi_k\). The line numbers in the headings refer to
Algorithm~\ref{alg:accelerated-block}.
\end{currentrevision}

\noindent\fbox{\strut\textit{3.1. Projection step
(Lines~\ref{line:block-v-update}--\ref{line:block-delta}).}}
The first-order condition for the Euclidean projection controls the change in
the auxiliary sequence. By definition,
\(v_{k+1}\) minimizes
\[
 \frac12\norm{
 u-v_k+\frac{b_k}{p_{\xi_k}}U_{\xi_k}\nabla_{\xi_k}f(y_k)
 }^2
 \quad\text{over}\quad u\in B(x^0,3R).
\]
Therefore, its first-order condition is
\[
 \inner{\Delta v_k+(b_k/p_{\xi_k})U_{\xi_k}
 \nabla_{\xi_k}f(y_k)}{u-v_{k+1}}\ge0,
 \qquad u\in B(x^0,3R).
\]
The first vector in the inner product is the gradient of the squared-distance
objective at \(v_{k+1}\).
\begin{currentrevision}
The choice \(u=x^*\) is admissible by
\eqref{eq:block-ball-membership}. Substituting this choice into the
first-order condition gives
\[
 \inner{\Delta v_k+(b_k/p_{\xi_k})U_{\xi_k}
 \nabla_{\xi_k}f(y_k)}{x^*-v_{k+1}}\ge0.
\]
Multiplying by \(-1\), separating the two inner products, and using
\end{currentrevision}
\[
 \inner{\Delta v_k}{v_{k+1}-x^*}
 \overset{\circledOne}{=}
 \frac12\norm{v_{k+1}-x^*}^2
  -\frac12\norm{v_k-x^*}^2
  +\frac12\norm{\Delta v_k}^2,
\]
where in \(\circledOne\), we expanded
\(\norm{v_k+\Delta v_k-x^*}^2\). Substituting this identity into the
first-order condition gives
\begin{equation}
\label{eq:block-projection-opt}
 \begin{aligned}
 &\frac12\norm{v_{k+1}-x^*}^2
  -\frac12\norm{v_k-x^*}^2\\
 &\quad\overset{\circledOne}{\le}
  -\frac {b_k}{p_{\xi_k}}
   \inner{\nabla_{\xi_k}f(y_k)}{U_{\xi_k}^\top(v_k-x^*)}
  -\underbrace{\frac {b_k}{p_{\xi_k}}
   \inner{\nabla_{\xi_k}f(y_k)}{(\Delta v_k)^{(\xi_k)}}}_{T_{\rm lin}}
   -\frac12\norm{\Delta v_k}^2,
 \end{aligned}
\end{equation}
where in \(\circledOne\), we used the first-order condition with \(u=x^*\) and
\(U_{\xi_k}^\top(v_{k+1}-v_k)=(\Delta v_k)^{(\xi_k)}\).
\begin{currentrevision}
The underbraced term is exactly the negative of the linear term in
\eqref{eq:block-trial-descent}. Hence the two occurrences of
\(T_{\rm lin}\) will cancel without any inequality or loss of constants.
\end{currentrevision}
For Step~4 we also need a form without \(x^*\). Taking \(u=v_k\) gives
\begin{equation}
\label{eq:block-projection-vk}
 \frac {b_k}{p_{\xi_k}}
 \inner{\nabla_{\xi_k}f(y_k)}{(\Delta v_k)^{(\xi_k)}}
 \overset{\circledOne}{\le}-\norm{\Delta v_k}^2,
\end{equation}
where in \(\circledOne\), we used
\(v_k-v_{k+1}=-\Delta v_k\).

\noindent\fbox{\strut\textit{3.2. Convexity coupling
(Line~\ref{line:block-y}).}}
Convexity at \(y_k\) with comparison points \(x_k\) and \(x^*\) gives
\(\F(y_k)-\F(x_k)\le\inner{\nabla f(y_k)}{y_k-x_k}\) and
\(\F(y_k)\le\inner{\nabla f(y_k)}{y_k-x^*}\).
Multiply the first by \(A_k\), the second by \(b_k\), and add. Substitution
of \((A_k+b_k)y_k=A_kx_k+b_kv_k\) gives
\begin{equation}
\label{eq:block-coupling}
 (A_k+b_k)\F(y_k)-A_k\F(x_k)
 \overset{\circledOne}{\le}
 \inner{\nabla f(y_k)}{(A_k+b_k)y_k-A_kx_k-b_kx^*}
 \overset{\circledTwo}{=}
 b_k\inner{\nabla f(y_k)}{v_k-x^*},
\end{equation}
where in \(\circledOne\), we added the two convexity inequalities with weights
\(A_k\) and \(b_k\); in \(\circledTwo\), we used
\((A_k+b_k)y_k=A_kx_k+b_kv_k\).
\begin{currentrevision}
Equivalently,
\[
 (A_k+b_k)\F(y_k)
 \le A_k\F(x_k)+b_k\inner{\nabla f(y_k)}{v_k-x^*}.
\]
This explains the particular definition of \(y_k\): it converts the
\((A_k+b_k)\F(y_k)\) term from the smoothness estimate into the old
objective term \(A_k\F(x_k)\) in \(\Phi_k\), plus one full-gradient inner
product.
\end{currentrevision}

\noindent\fbox{\strut\textit{\newrev{3.3. Coupled potential and sampling error
(Line~\ref{line:block-sample}).}}}
Introduce the candidate potential
\begin{equation}
\label{eq:block-candidate-potential}
 \widetilde\Phi_{k+1}
 :=a_k\F(z_k)+\frac12\norm{v_{k+1}-x^*}^2.
\end{equation}
\begin{currentrevision}
Adding \eqref{eq:block-projection-opt} to
\eqref{eq:block-trial-descent} cancels \(T_{\rm lin}\). The two remaining
squared-displacement terms combine as
\[
 -\frac12\norm{\Delta v_k}^2+\frac18\norm{\Delta v_k}^2
 =-\frac38\norm{\Delta v_k}^2.
\]
We then use \eqref{eq:block-coupling}. These operations give
\end{currentrevision}
\begin{equation}
\label{eq:block-raw-coupled}
 \begin{aligned}
 \widetilde\Phi_{k+1}
 &\overset{\circledOne}{\le}
 \frac12\norm{v_k-x^*}^2
 +(A_k+b_k-L_k)\F(y_k)
 -\frac {b_k}{p_{\xi_k}}
 \inner{\nabla_{\xi_k}f(y_k)}{U_{\xi_k}^\top(v_k-x^*)}
 -\frac38\norm{\Delta v_k}^2\\
 &\overset{\circledTwo}{\le}
 \Phi_k
 +\underbrace{
 b_k\inner{\nabla f(y_k)}{v_k-x^*}
 -\frac {b_k}{p_{\xi_k}}
 \inner{\nabla_{\xi_k}f(y_k)}{U_{\xi_k}^\top(v_k-x^*)}
 }_{Z_k}\\
 &\quad
 -L_k\F(y_k)
 -\frac38\norm{\Delta v_k}^2\\
 &
 \overset{\circledThree}{=}
 \Phi_k+Z_k-L_k\F(y_k)-\frac38\norm{\Delta v_k}^2,
 \end{aligned}
\end{equation}
where in \(\circledOne\), the two occurrences of \(T_{\rm lin}\) sum to
zero and the quadratic coefficient is \(1/8-1/2=-3/8\); in
\(\circledTwo\), we applied \eqref{eq:block-coupling}, used the definition
of \(\Phi_k\), and grouped the difference of the full-gradient and sampled
block inner products into \(Z_k\); in \(\circledThree\), we used this
definition. The remaining curvature term has the favorable sign because
\(L_k\ge0\) by \eqref{eq:block-curvature-slack-positive}.
\begin{currentrevision}
The term \(Z_k\) is the error caused by replacing the full gradient in the
coupling inequality with one inverse-probability-scaled block gradient. Its
sampled vector has the following conditional expectation, where all
quantities except \(\xi_k\) are fixed conditional on \(\mathcal F_k\):
\begin{equation}
\label{eq:block-sampled-gradient-unbiased}
\begin{aligned}
 &\E\!\left[\left.
 \frac{1}{p_{\xi_k}}U_{\xi_k}\nabla_{\xi_k}f(y_k)
 \right|\mathcal F_k\right]\\
 &\quad=\sum_{\xi=1}^n p_\xi
 \frac{1}{p_\xi}U_\xi\nabla_\xi f(y_k)
 =\sum_{\xi=1}^nU_\xi U_\xi^\top\nabla f(y_k)
 =\nabla f(y_k).
\end{aligned}
\end{equation}
Here the first equality is the definition of conditional expectation for the
discrete draw \(\Pr(\xi_k=\xi)=p_\xi\), the second cancels \(p_\xi\), and
the third and fourth use the first and third identities in
\eqref{eq:block-coordinate-identities}, respectively. Substituting this equality into the
definition of \(Z_k\) gives zero:
\end{currentrevision}
\begin{equation}
\label{eq:block-sampling-error-zero}
 \begin{aligned}
 \E[Z_k\mid\mathcal F_k]
 &\overset{\circledOne}{=}
 b_k\inner{\nabla f(y_k)}{v_k-x^*}\\
 &\quad
 -b_k\inner{
 \sum_{\xi=1}^nU_\xi U_\xi^\top\nabla f(y_k)
 }{v_k-x^*}\\
 &\overset{\circledTwo}{=}0,
 \end{aligned}
\end{equation}
where in \(\circledOne\), we averaged the sampled term in \(Z_k\) using
\eqref{eq:block-sampled-gradient-unbiased}; in
\(\circledTwo\), we used the third identity in
\eqref{eq:block-coordinate-identities}.

\begin{currentrevision}
For Step~4 we also need a form that does not contain \(Z_k\). Substituting
\end{currentrevision}
\eqref{eq:block-projection-vk} into \eqref{eq:block-trial-descent} and adding
\(\frac12\norm{v_{k+1}-x^*}^2\) gives
\begin{equation}
\label{eq:block-raw-uncoupled}
 \widetilde\Phi_{k+1}
 \overset{\circledOne}{\le}(A_k+b_k-L_k)\F(y_k)
     +\frac12\norm{v_{k+1}-x^*}^2
     -\frac78\norm{\Delta v_k}^2,
\end{equation}
where in \(\circledOne\), we used \eqref{eq:block-trial-descent} and
\(-7/8=-1+1/8\): inequality~\eqref{eq:block-projection-vk} contributes
\(-\norm{\Delta v_k}^2\), while the baseline-curvature term contributes at
most \(\norm{\Delta v_k}^2/8\).

\noindent\fbox{\strut\textit{\newrev{3.4. Accepted primal update
(Lines~\ref{line:block-A-update}--\ref{line:block-x-update}).}}}
Estimate~\eqref{eq:block-raw-uncoupled} bounds the candidate potential
\(\widetilde\Phi_{k+1}\), which is defined using \((z_k,a_k)\). The algorithm,
however, sets \(A_{k+1}=\max\{A_0,a_k\}\) and defines \(x_{k+1}\) using this
accepted weight. We now bound the possible increase in the potential caused by
this replacement. Define the corresponding acceptance correction by
\begin{equation}
\label{eq:block-acceptance-correction}
 \eta_k:=
 \begin{cases}
  (A_0-a_k)\Delta_0, & a_k<A_0,\\
  0, & a_k\ge A_0.
 \end{cases}
\end{equation}
\begin{currentrevision}
Here \(\Delta_0\) is the assumed upper bound on the initial gap:
\(\F(x^0)\le\Delta_0\). If \(a_k\ge A_0\), then
\(A_{k+1}=a_k\) and \(x_{k+1}=z_k\), so \(\eta_k=0\). If \(a_k<A_0\),
then \(A_{k+1}=A_0\) and
\[
 x_{k+1}=\frac{a_k}{A_0}z_k+\frac{A_0-a_k}{A_0}x^0.
\]
Thus the accepted point is a convex combination of the candidate point and
the initial point. The term \(\eta_k=(A_0-a_k)\Delta_0\) is the possible
increase caused by the contribution of \(x^0\).
\end{currentrevision}
The update of \(x_{k+1}\) and convexity give
\begin{equation}
\label{eq:block-accepted-potential}
 \begin{aligned}
 \Phi_{k+1}
 &\overset{\circledOne}{=}
 A_{k+1}\F\!\left(
 \frac{a_k}{A_{k+1}}z_k
 +\frac{A_{k+1}-a_k}{A_{k+1}}x^0
 \right)
 +\frac12\norm{v_{k+1}-x^*}^2\\
 &\overset{\circledTwo}{\le}
 a_k\F(z_k)+(A_{k+1}-a_k)\F(x^0)
 +\frac12\norm{v_{k+1}-x^*}^2\\
 &\overset{\circledThree}{\le}
 \widetilde\Phi_{k+1}+(A_{k+1}-a_k)\Delta_0
 \overset{\circledFour}{=}
 \widetilde\Phi_{k+1}+\eta_k,
 \end{aligned}
\end{equation}
where in \(\circledOne\), we substituted \eqref{eq:block-potential} and
Line~\ref{line:block-x-update}; in \(\circledTwo\), we used convexity, since
\(a_k/A_{k+1}\) and \((A_{k+1}-a_k)/A_{k+1}\) are nonnegative and sum to
one; in \(\circledThree\), we used \(\F(x^0)\le\Delta_0\) and
\eqref{eq:block-candidate-potential}; in \(\circledFour\), we used
Line~\ref{line:block-A-update} and \eqref{eq:block-acceptance-correction}.
If \(a_k<A_0\), then
\(A_k+b_k\le4a_k/3<4A_0/3\) by
\eqref{eq:block-weight-positive}, and
\(A_0-a_k\le A_k+b_k-a_k=(A_k+b_k)c_kr_k\). If \(a_k\ge A_0\), the
term \(\eta_k\) is zero.
\begin{currentrevision}
For completeness, the two inequalities in the case \(a_k<A_0\) follow as
follows. First, \(a_k\ge3(A_k+b_k)/4\) implies
\(A_k+b_k\le4a_k/3<4A_0/3\). Second, the update rule always gives
\(A_k\ge A_0\), and hence \(A_k+b_k\ge A_0\). Subtracting \(a_k\) from
both sides yields
\[
 A_0-a_k\le A_k+b_k-a_k=(A_k+b_k)c_kr_k.
\]
Multiplying these bounds by \(\Delta_0\) gives
\[
 \eta_k=(A_0-a_k)\Delta_0
 \le\frac43A_0\Delta_0\,c_kr_k.
\]
\end{currentrevision}

\begin{revision}
To convert the factor \(c_kr_k\) into a squared displacement, use
\eqref{eq:block-r-rho} and \eqref{eq:block-b-ratio}:
\begin{equation}
\label{eq:block-curvature-correction-bound}
 c_kr_k
 \overset{\circledOne}{\le}
 \newrev{\frac{c_k\norm{\Delta v_k}^2}
 {(1+64R\sqrt{c_k})^2}}
 \overset{\circledTwo}{\le}
 \frac{\norm{\Delta v_k}^2}{64^2 R^2},
\end{equation}
where in \(\circledOne\), we substituted \(r_k\) from
\eqref{eq:block-r-rho} \newrev{and bounded}
\(b_k/(A_k+b_k)\) by \eqref{eq:block-b-ratio}; in \(\circledTwo\), we dropped the positive constant
\(1\) from the denominator when \newrev{\(c_k>0\). If \(c_k=0\), then
\(c_kr_k=0\), so \eqref{eq:block-curvature-correction-bound} holds
trivially}.

We can now bound the acceptance correction:
\begin{equation}
\label{eq:block-averaging-cost}
 \eta_k
 \overset{\circledOne}{\le}\frac43 A_0\Delta_0\,c_kr_k
 \overset{\circledTwo}{=}\frac23R^2c_kr_k
 \overset{\circledThree}{\le}
 \frac{2}{3\cdot64^2}\norm{\Delta v_k}^2,
\end{equation}
where in \(\circledOne\), if \(a_k<A_0\), we used the definition of
\(\eta_k\) in \eqref{eq:block-acceptance-correction}, the identity
\(A_k+b_k-a_k=(A_k+b_k)c_kr_k\) from \eqref{eq:block-r-rho}, and the bound
\(A_k+b_k<4A_0/3\) from \eqref{eq:block-weight-positive}; if
\(a_k\ge A_0\), then \(\eta_k=0\) by
\eqref{eq:block-acceptance-correction}; in \(\circledTwo\), we
substituted \(A_0\Delta_0=R^2/2\); in \(\circledThree\), we applied
\eqref{eq:block-curvature-correction-bound}.
\end{revision}

We use \eqref{eq:block-averaging-cost} to absorb the possible increase \(\eta_k\)
into the negative squared-norm term in \eqref{eq:block-raw-coupled}. Indeed,
\(2/(3\cdot64^2)\le1/8\), while \(L_k\F(y_k)\ge0\). Therefore,
\begin{equation}
\label{eq:block-unscaled-pathwise}
 \begin{aligned}
 \Phi_{k+1}
 &\overset{\circledOne}{\le}\widetilde\Phi_{k+1}+\eta_k\\
 &\overset{\circledTwo}{\le}
 \Phi_k+Z_k-\frac38\norm{\Delta v_k}^2-L_k\F(y_k)+\eta_k\\
 &\overset{\circledThree}{\le}
 \Phi_k+Z_k-\frac14\norm{\Delta v_k}^2,
 \end{aligned}
\end{equation}
where in \(\circledOne\), we used \eqref{eq:block-accepted-potential}; in
\(\circledTwo\), we used \eqref{eq:block-raw-coupled}; in
\(\circledThree\), we used \eqref{eq:block-averaging-cost},
\(2/(3\cdot64^2)\le1/8\), and \(L_k\F(y_k)\ge0\).
Taking the conditional expectation of
\eqref{eq:block-unscaled-pathwise} yields
\begin{equation}
\label{eq:block-unscaled}
 \E[\Phi_{k+1}\mid\mathcal F_k]
 \overset{\circledOne}{\le}
 \Phi_k-\frac14\E[\norm{\Delta v_k}^2\mid\mathcal F_k],
\end{equation}
where in \(\circledOne\), we used \(\E[Z_k\mid\mathcal F_k]=0\).
\begin{currentrevision}
More explicitly, conditional expectation preserves the pathwise inequality
\eqref{eq:block-unscaled-pathwise}. Since \(\Phi_k\) is
\(\mathcal F_k\)-measurable by \eqref{eq:block-adaptedness},
\(\E[\Phi_k\mid\mathcal F_k]=\Phi_k\). Linearity then gives
\[
 \E[\Phi_{k+1}\mid\mathcal F_k]
 \le \Phi_k+\E[Z_k\mid\mathcal F_k]
 -\frac14\E[\norm{\Delta v_k}^2\mid\mathcal F_k],
\]
and \eqref{eq:block-sampling-error-zero} removes the middle term.
\end{currentrevision}
Thus the potential decreases in conditional
expectation; the negative term
\newrev{will be used to control the normalization term in the next step.}

\paragraph{Step 4: A normalized one-step estimate.}
\begin{currentrevision}
The estimate \eqref{eq:block-unscaled} proves that the potential does not
increase in conditional expectation. By itself, however, this is not yet an
accelerated convergence statement. Since the first term in the potential is
\(A_k\F(x_k)\), nonnegativity of the distance term gives
\[
 \F(x_k)\le\frac{\Phi_k}{A_k}.
\]
Thus convergence comes from two facts used together: \(\Phi_k\) must remain
controlled, and \(A_k\) must grow. We therefore study a potential divided by
a power of \(A_k\). The exponent is denoted by \(\gamma\); it is smaller
than one because the realized correction makes \(a_k\) slightly smaller
than the provisional weight \(A_k+b_k\). We derive the normalized estimate
in four parts and explain how the negative squared-displacement term in
\eqref{eq:block-unscaled} controls the additional terms introduced by this
normalization.
\end{currentrevision}

\noindent\fbox{\strut\textit{\grev{4.1. Normalized potential.}}}
We normalize the potential to convert growth of \(A_k\) into convergence.
For the analysis, define
\begin{equation}
\label{eq:block-gamma-range}
 \gamma:=
 \begin{cases}
  \dfrac12, & \displaystyle\max_{0\le j\le k}A_j<A_{\rm sw},\\[0.6em]
  \displaystyle
  1-\dfrac{1}{2(1+\log(2\Delta_0/\varepsilon))},
  & \displaystyle\max_{0\le j\le k}A_j\ge A_{\rm sw},
 \end{cases}
 \qquad \gamma\in[1/2,1).
\end{equation}
\begin{currentrevision}
Combining the two branches of \eqref{eq:block-c-rule} with the corresponding
two branches of \eqref{eq:block-gamma-range} gives
\end{currentrevision}
\begin{equation}
\label{eq:block-gamma-curvature}
 \newrev{(1-\gamma)c_k=2K_1.}
\end{equation}
\begin{currentrevision}
Before the switch, the first branches of \eqref{eq:block-c-rule} and
\eqref{eq:block-gamma-range} give \(c_k=4K_1\) and \(\gamma=1/2\), so
\((1-\gamma)c_k=(1/2)4K_1=2K_1\). After the switch,
\[
 1-\gamma=\frac{1}{2(1+\log(2\Delta_0/\varepsilon))},
\]
and the second branch of \eqref{eq:block-c-rule} gives
\[
 c_k=4K_1(1+\log(2\Delta_0/\varepsilon))
\]
which again gives \((1-\gamma)c_k=2K_1\). The coefficient \(c_k\) is therefore
chosen together with the normalization exponent, rather than independently.
This equality will make the coefficient of \(\F(y_k)\) nonpositive in
Part~4.2.
\end{currentrevision}
To account for the correction from \(A_k+b_k\) to \(a_k\), define the
analytical factor
\begin{equation}
\label{eq:block-t-definition}
 t_k:=\rho_k^{-\gamma}-1\ge0.
\end{equation}
Indeed, \eqref{eq:block-gamma-range} gives \(\gamma>0\), while
\eqref{eq:block-r-rho} and \eqref{eq:block-weight-positive} give
\(0<\rho_k\le1\); hence \(\rho_k^{-\gamma}\ge1\).
We next compare the normalized accepted potential with the candidate
potential and use \eqref{eq:block-raw-coupled} and
\eqref{eq:block-raw-uncoupled}:
\begin{equation}
\label{eq:block-normalized-expanded}
 \begin{aligned}
 (A_k+b_k)^\gamma\frac{\Phi_{k+1}}{A_{k+1}^\gamma}
 &\overset{\circledOne}{\le}
 \left(\frac{A_k+b_k}{A_{k+1}}\right)^\gamma
 (\widetilde\Phi_{k+1}+\eta_k)\\
 &\overset{\circledTwo}{\le}
 \left(\frac{A_k+b_k}{a_k}\right)^\gamma
 (\widetilde\Phi_{k+1}+\eta_k)\\
 &\overset{\circledThree}{=}
 \rho_k^{-\gamma}(\widetilde\Phi_{k+1}+\eta_k)
 \overset{\circledFour}{=}
 (1+t_k)(\widetilde\Phi_{k+1}+\eta_k)\\
 &\overset{\circledFive}{=}
 \widetilde\Phi_{k+1}+t_k\widetilde\Phi_{k+1}
 +(1+t_k)\eta_k\\
 &\overset{\circledSix}{\le}
 \Phi_k+Z_k
 +\underbrace{\bigl(t_k(A_k+b_k-L_k)-L_k\bigr)}_{
     \text{coefficient of }\F(y_k)}\F(y_k)\\
 &\quad
 -\frac38\norm{\Delta v_k}^2
 +\frac{t_k}{2}\norm{v_{k+1}-x^*}^2
 -\frac{7t_k}{8}\norm{\Delta v_k}^2
 +(1+t_k)\eta_k,
 \end{aligned}
\end{equation}
where in \(\circledOne\), we used
\eqref{eq:block-accepted-potential}; in \(\circledTwo\), we used
Line~\ref{line:block-A-update} and \eqref{eq:block-r-rho} to obtain
\(A_{k+1}\ge a_k=\rho_k(A_k+b_k)>0\). We also used
\eqref{eq:block-candidate-potential},
\eqref{eq:block-acceptance-correction}, and
\eqref{eq:block-weight-positive}, together with
\(\F(z_k)\ge0\) and \(\Delta_0>0\), to obtain
\(\widetilde\Phi_{k+1}+\eta_k\ge0\); in \(\circledThree\), we used
\eqref{eq:block-r-rho}; in \(\circledFour\), we used
\eqref{eq:block-t-definition}; in \(\circledFive\), we expanded the product; in
\(\circledSix\), we
used \eqref{eq:block-raw-coupled} for \(\widetilde\Phi_{k+1}\) and
\eqref{eq:block-raw-uncoupled} for
\(t_k\widetilde\Phi_{k+1}\), which is valid because \(t_k\ge0\). In
particular, only the first estimate contains \(Z_k\), so its coefficient
remains \(1\).

\par\noindent\fbox{\strut\textit{4.2. Coefficient of \(\F(y_k)\).}}
It remains to show that the coefficient of \(\F(y_k)\) in
\eqref{eq:block-normalized-expanded} is nonpositive. First,
\begin{equation}
\label{eq:block-coefficient-rewrite}
 \begin{aligned}
 A_k+b_k-L_k
 &\overset{\circledOne}{=}
 \newrev{a_k+\frac{H_{1,\xi_k}(A_k+b_k)^2r_k}{p_{\xi_k}^2a_k}}\\
 &\overset{\circledTwo}{=}
 \rho_k(A_k+b_k)
 \newrev{+\frac{H_{1,\xi_k}r_k(A_k+b_k)^2}
 {p_{\xi_k}^2\rho_k(A_k+b_k)}}\\
 &\overset{\circledThree}{=}
 \newrev{(A_k+b_k)\left(\rho_k+
 \frac{H_{1,\xi_k}r_k}{p_{\xi_k}^2\rho_k}\right)},
 \end{aligned}
\end{equation}
where in \(\circledOne\), we rearranged the definition of \(L_k\) in
\eqref{eq:block-curvature-slack}; in \(\circledTwo\), we used
\eqref{eq:block-r-rho}; in \(\circledThree\), we simplified the common factor
\(A_k+b_k\). Because \(A_k\ge A_0>0\) and \(b_k>0\), we have
\(A_k+b_k>0\). We may therefore divide the coefficient by \(A_k+b_k\)
without changing its sign:
\begin{equation}
\label{eq:block-normalized-coefficient-identity}
 \begin{aligned}
 &\frac{t_k(A_k+b_k-L_k)-L_k}{A_k+b_k}\\
 &\overset{\circledOne}{=}
 \frac{(1+t_k)(A_k+b_k-L_k)-(A_k+b_k)}{A_k+b_k}\\
 &\overset{\circledTwo}{=}
 \newrev{(1+t_k)\left(\rho_k+
 \frac{H_{1,\xi_k}r_k}{p_{\xi_k}^2\rho_k}\right)-1}\\
 &\overset{\circledThree}{=}
 \newrev{\rho_k^{-\gamma}\left(\rho_k+
 \frac{H_{1,\xi_k}r_k}{p_{\xi_k}^2\rho_k}\right)-1}\\
 &\overset{\circledFour}{=}
 \newrev{\rho_k^{1-\gamma}+
 \frac{H_{1,\xi_k}r_k}{p_{\xi_k}^2\rho_k^{1+\gamma}}-1},
 \end{aligned}
\end{equation}
where in \(\circledOne\), we used
\(-L_k=(A_k+b_k-L_k)-(A_k+b_k)\); in \(\circledTwo\), we used
\eqref{eq:block-coefficient-rewrite}; in \(\circledThree\), we used
\(1+t_k=\rho_k^{-\gamma}\); in \(\circledFour\), we distributed
\(\rho_k^{-\gamma}\).

To bound the last expression, we estimate its two positive terms separately.
For the first term, the function
\(s\mapsto s^{1-\gamma}\) is concave on \((0,\infty)\), because
\eqref{eq:block-gamma-range} gives \(0<1-\gamma\le1/2\) and
\[
 \frac{d^2}{ds^2}s^{1-\gamma}
 =-\gamma(1-\gamma)s^{-1-\gamma}\le0,
 \qquad s>0.
\]
Every differentiable concave function lies below its tangent line. Applying
this fact at \(s=1\) gives
\begin{equation}
\label{eq:block-concave-power-bound}
 \begin{aligned}
 \rho_k^{1-\gamma}
 &\overset{\circledOne}{\le}
 1+(1-\gamma)(\rho_k-1)\\
 &\overset{\circledTwo}{=}
 1-(1-\gamma)(1-\rho_k),
 \end{aligned}
\end{equation}
where in \(\circledOne\), we used the tangent line at \(s=1\), whose value
is \(1\) and whose slope is \(1-\gamma\); in \(\circledTwo\), we used
\(\rho_k-1=-(1-\rho_k)\).

The second term satisfies
\begin{equation}
\label{eq:block-reciprocal-power-bound}
 \begin{aligned}
 \rho_k^{1+\gamma}
 &\overset{\circledOne}{\ge}\rho_k^2
 \overset{\circledTwo}{\ge}\left(\frac34\right)^2
 =\frac9{16},\\
 \newrev{\frac{H_{1,\xi_k}r_k}
 {p_{\xi_k}^2\rho_k^{1+\gamma}}}
 &\overset{\circledThree}{\le}
 \newrev{\frac{16H_{1,\xi_k}r_k}{9p_{\xi_k}^2}},
 \end{aligned}
\end{equation}
where in \(\circledOne\), \eqref{eq:block-gamma-range} gives
\(1+\gamma\le2\), while \eqref{eq:block-r-rho} and
\eqref{eq:block-weight-positive} give \(0<\rho_k\le1\); for a base in this
interval, increasing the exponent decreases the power; in
\(\circledTwo\), \eqref{eq:block-r-rho} and
\eqref{eq:block-weight-positive} give \(\rho_k\ge3/4\); in
\(\circledThree\), we took reciprocals of the positive quantities and
multiplied by \newrev{\(H_{1,\xi_k}r_k/p_{\xi_k}^2\ge0\)}, where
\newrev{\(r_k\ge0\) follows directly from \eqref{eq:block-r-rho}}.

Combining \eqref{eq:block-normalized-coefficient-identity},
\eqref{eq:block-concave-power-bound}, and
\eqref{eq:block-reciprocal-power-bound} gives:
\begin{equation}
\label{eq:block-coefficient-nonpositive}
 \begin{aligned}
 &\frac{t_k(A_k+b_k-L_k)-L_k}{A_k+b_k}\\
 &\overset{\circledOne}{=}
 \newrev{\rho_k^{1-\gamma}+
 \frac{H_{1,\xi_k}r_k}{p_{\xi_k}^2\rho_k^{1+\gamma}}-1}\\
 &\overset{\circledTwo}{\le}
 \left[1-(1-\gamma)(1-\rho_k)\right]
 \newrev{+\frac{16H_{1,\xi_k}r_k}{9p_{\xi_k}^2}}-1\\
 &\overset{\circledThree}{=}
 \newrev{-(1-\gamma)(1-\rho_k)+
 \frac{16H_{1,\xi_k}r_k}{9p_{\xi_k}^2}}\\
 &\overset{\circledFour}{=}
 \newrev{-2K_1r_k
 +\frac{16H_{1,\xi_k}r_k}{9p_{\xi_k}^2}}\\
 &\overset{\circledFive}{\le}0,
 \end{aligned}
\end{equation}
where in \(\circledOne\), we used
\eqref{eq:block-normalized-coefficient-identity}; in \(\circledTwo\), we applied
\eqref{eq:block-concave-power-bound} and
\eqref{eq:block-reciprocal-power-bound}; in \(\circledThree\), the two
constant terms \(1\) and \(-1\) sum to zero; in \(\circledFour\),
\eqref{eq:block-r-rho} gives \(1-\rho_k=c_kr_k\), while
 \newrev{\eqref{eq:block-gamma-curvature} gives
 \((1-\gamma)c_k=2K_1\)}; in
\begin{currentrevision}
 In \(\circledFive\), \eqref{eq:block-effective-curvatures} gives
 \(H_{1,\xi_k}/p_{\xi_k}^2\le K_1\),
 \eqref{eq:block-r-rho} gives \(r_k\ge0\), and \(16/9<2\).
\end{currentrevision}
 Thus the underbraced coefficient is
nonpositive.

\noindent\fbox{\strut\textit{4.3. Remaining positive terms.}}
It remains to control the two positive terms in
\eqref{eq:block-normalized-expanded}. We first bound the distance in the
first term:
\begin{equation}
\label{eq:block-auxiliary-distance}
 \begin{aligned}
 \frac12\norm{v_{k+1}-x^*}^2
 &\overset{\circledOne}{\le}
 \frac12\left(\norm{v_{k+1}-x^0}+\norm{x^0-x^*}\right)^2\\
 &\overset{\circledTwo}{\le}
 \frac12(3R+R)^2
 \overset{\circledThree}{=}8R^2,
 \end{aligned}
\end{equation}
where in \(\circledOne\), we used the triangle inequality; in
\(\circledTwo\), we used \eqref{eq:block-ball-membership} and
\(\norm{x^0-x^*}\le R\); in \(\circledThree\), we simplified the
right-hand side.

We next bound the two factors \(t_k\) and \(1+t_k\):
\begin{equation}
\label{eq:block-t-bounds}
 \begin{aligned}
 t_k
 &\overset{\circledOne}{=}\rho_k^{-\gamma}-1
 \overset{\circledTwo}{\le}\rho_k^{-1}-1\\
 &
 \overset{\circledThree}{=}\frac{1-\rho_k}{\rho_k}
 \overset{\circledFour}{=}\frac{c_kr_k}{\rho_k}
 \overset{\circledFive}{\le}\frac43c_kr_k,\\
 1+t_k
 &\overset{\circledSix}{=}\rho_k^{-\gamma}
 \overset{\circledSeven}{\le}\rho_k^{-1}
 \overset{\circledEight}{\le}\frac43,
 \end{aligned}
\end{equation}
where in \(\circledOne\) and \(\circledSix\), we used
\eqref{eq:block-t-definition}; in \(\circledTwo\) and
\(\circledSeven\), we used \eqref{eq:block-gamma-range} and the bound
\(0<\rho_k\le1\) from \eqref{eq:block-r-rho} and
\eqref{eq:block-weight-positive}; in \(\circledThree\), we combined the
two fractions; in \(\circledFour\), we used
\(1-\rho_k=c_kr_k\) from \eqref{eq:block-r-rho}; in \(\circledFive\) and
\(\circledEight\), we used \(\rho_k\ge3/4\), which follows from
\eqref{eq:block-r-rho} and \eqref{eq:block-weight-positive}.

We now combine \eqref{eq:block-auxiliary-distance},
\eqref{eq:block-t-bounds}, \eqref{eq:block-averaging-cost}, and
\eqref{eq:block-curvature-correction-bound}:
\begin{equation}
\label{eq:block-positive-terms-bound}
 \begin{aligned}
 &\frac{t_k}{2}\norm{v_{k+1}-x^*}^2+(1+t_k)\eta_k\\
 &\overset{\circledOne}{\le}
 \frac43c_kr_k\cdot8R^2
 +\frac43\cdot\frac23R^2c_kr_k\\
 &\overset{\circledTwo}{=}
 \left(\frac{32}{3}+\frac89\right)c_kr_kR^2
 \overset{\circledThree}{=}
 \frac{104}{9}c_kr_kR^2\\
 &\overset{\circledFour}{\le}
 \frac{104}{9\cdot64^2}\norm{\Delta v_k}^2
 \overset{\circledFive}{\le}
 \frac18\norm{\Delta v_k}^2,
 \end{aligned}
\end{equation}
where in \(\circledOne\), we used \eqref{eq:block-auxiliary-distance},
\eqref{eq:block-t-bounds}, and the middle bound
\(\eta_k\le(2/3)R^2c_kr_k\) in
\eqref{eq:block-averaging-cost}; in \(\circledTwo\), we multiplied the
constants; in \(\circledThree\), we used \(32/3+8/9=104/9\); in
\(\circledFour\), we used
\eqref{eq:block-curvature-correction-bound}; in \(\circledFive\), we used
\(104/(9\cdot64^2)\le1/8\).

\noindent\fbox{\strut\textit{4.4. Conditional expectation and the two regimes.}}
We first combine \eqref{eq:block-coefficient-nonpositive} with
\eqref{eq:block-positive-terms-bound}:
\begin{equation}
\label{eq:block-normalized-before-expectation}
 \begin{aligned}
 &(A_k+b_k)^\gamma\frac{\Phi_{k+1}}{A_{k+1}^\gamma}\\
 &\overset{\circledOne}{\le}
 \Phi_k+Z_k-\frac38\norm{\Delta v_k}^2
 -\frac{7t_k}{8}\norm{\Delta v_k}^2\\
 &\quad
 +\frac{t_k}{2}\norm{v_{k+1}-x^*}^2+(1+t_k)\eta_k\\
 &\overset{\circledTwo}{\le}
 \Phi_k+Z_k-\frac14\norm{\Delta v_k}^2
 -\frac{7t_k}{8}\norm{\Delta v_k}^2
 \overset{\circledThree}{\le}\Phi_k+Z_k,
 \end{aligned}
\end{equation}
where in \(\circledOne\), we used
\eqref{eq:block-normalized-expanded},
\eqref{eq:block-coefficient-nonpositive}, and \(\F(y_k)\ge0\); in
\(\circledTwo\), we applied \eqref{eq:block-positive-terms-bound} and used
\(-3/8+1/8=-1/4\); in \(\circledThree\), we used
\(t_k\ge0\) from \eqref{eq:block-t-definition}.

Taking the conditional expectation of
\eqref{eq:block-normalized-before-expectation} gives
\begin{equation}
\label{eq:block-normalized}
 \begin{aligned}
 \E\left[\left.
 \frac{\Phi_{k+1}}{A_{k+1}^{\gamma}}
 \right|\mathcal F_k\right]
 &\overset{\circledOne}{\le}
 \frac{\Phi_k+\E[Z_k\mid\mathcal F_k]}
 {(A_k+b_k)^\gamma}\\
 &\overset{\circledTwo}{=}
 \frac{\Phi_k}{(A_k+b_k)^\gamma},
 \end{aligned}
\end{equation}
where in \(\circledOne\), we used that
\(A_k,b_k,c_k\) are \(\mathcal F_k\)-measurable by
\eqref{eq:block-adaptedness}, and \(\gamma\) depends only on the
\(\mathcal F_k\)-measurable switching condition. Moreover,
\(A_k+b_k>0\) follows from
\(A_k\ge A_0>0\) in Line~\ref{line:block-A-update} and
\(b_k>0\) in \eqref{eq:block-b-rule}; in \(\circledTwo\), we used
\eqref{eq:block-sampling-error-zero}. Thus growth of \(A_k\) contracts the
normalized potential.

The two values of \(\gamma\) in \eqref{eq:block-gamma-range}
depend only on \(A_0,\ldots,A_k\), so \(\gamma\) is fixed conditional on
\(\mathcal F_k\).
Therefore,
\eqref{eq:block-normalized} applies in both regimes, including the iteration
at which \(A_k\) first reaches \(A_{\rm sw}\).

\paragraph{Step 5: Applying the estimates in the two regimes.}
\begin{currentrevision}
The parameter rule changes when the accumulated weight first reaches
\(A_{\rm sw}\). There are therefore two possible histories. The method may
remain in the first regime throughout the first horizon \(N_1\); in that
case, the pre-switch contraction alone must prove the target accuracy. Or it
may reach \(A_{\rm sw}\) no later than \(N_1\); in that case, we start the
post-switch contraction at the crossing time and run it for an additional
horizon \(N_2\). The following stopping time separates these two events.
\end{currentrevision}
Define the switching time by
\begin{equation}
\label{eq:block-switch-time}
 T:=\inf\{k\ge0:A_k\ge A_{\rm sw}\},
\end{equation}
with the convention \(T=\infty\) if the set is empty. We treat the two
events \(\{T>N_1\}\) and \(\{T\le N_1\}\) separately in four parts. For
an event \(E\), let \(\mathbb{I}_E\) denote its indicator.
\begin{currentrevision}
Explicitly, \(\mathbb I_E=1\) on outcomes belonging to \(E\) and
\(\mathbb I_E=0\) otherwise. Hence \(X\mathbb I_E=X\) on \(E\) and
\(X\mathbb I_E=0\) on its complement. We will
also repeatedly use the following conditional-expectation identity. If
\(E\in\mathcal F_k\), then \(\mathbb I_E\) is already known at time \(k\),
and for every integrable random variable \(X\),
\[
 \E[X\mathbb I_E]
 =\E\!\left[\E[X\mathbb I_E\mid\mathcal F_k]\right]
 =\E\!\left[\mathbb I_E\E[X\mid\mathcal F_k]\right].
\]
The first equality is the tower property, and the second pulls the known
factor \(\mathbb I_E\) outside the conditional expectation. Each later use
of an indicator follows this identity.
\end{currentrevision}

\noindent\fbox{\strut\textit{5.1. Pre-switch contraction up to \(N_1\).}}
Define
\begin{equation}
\label{eq:block-first-horizon}
 N_1:=
 \begin{cases}
  0, & A_{\rm sw}=A_0,\\[0.3em]
  \displaystyle
  \left\lceil
  12(1+128R\sqrt{K_1})
  \left(1+\log\frac{2\Delta_0}{\varepsilon}\right)
  \right\rceil,
  & A_{\rm sw}>A_0.
 \end{cases}
\end{equation}
\begin{currentrevision}
The value of \(N_1\) is chosen to ensure the desired contraction. Before
the switch, Part~5.1 will show that one iteration multiplies the normalized
potential by at most
\[
 \left(1+\frac{1}{2(1+128R\sqrt{K_1})}\right)^{-1/2}.
\]
After \(N_1\) iterations, this factor is raised to the power \(N_1\). The
coefficient \(12(1+128R\sqrt{K_1})\) is large enough to convert this
geometric decay into
\(\exp(-2(1+\log(2\Delta_0/\varepsilon)))\); the calculation is written
explicitly below in \eqref{eq:block-unswitched-potential}.
\end{currentrevision}
If \(A_{\rm sw}=A_0\), then \eqref{eq:block-switch-time} gives \(T=0\).
Suppose now that \(A_{\rm sw}>A_0\) and \(k<T\). Then
\begin{equation}
\label{eq:block-pre-switch-curvature}
 \begin{aligned}
 A_k<A_{\rm sw}
 &\overset{\circledOne}{\le}
 \frac{(1+128R\sqrt{K_1})^2}{4096K_0},\\
 64\sqrt{K_0A_k}
 &\overset{\circledTwo}{\le}1+128R\sqrt{K_1},
 \end{aligned}
\end{equation}
where in \(\circledOne\), we used \eqref{eq:block-switch-threshold} and
\(A_{\rm sw}>A_0\), so the outer maximum selects its second argument; in
\(\circledTwo\), we multiplied the first inequality by \(4096K_0\) and
took square roots.

\begin{currentrevision}
For \(k<T\), \eqref{eq:block-switch-time} gives
\(A_j<A_{\rm sw}\) for every \(0\le j\le k\), so the first branch of
\eqref{eq:block-c-rule} gives \(c_k=4K_1\). Substituting this value into
\eqref{eq:block-b-rule} and using
\eqref{eq:block-pre-switch-curvature} gives
\end{currentrevision}
\begin{equation}
\label{eq:block-pre-switch-growth}
 \begin{aligned}
 A_k+b_k
 &\overset{\circledOne}{=}
 A_k\left(
 1+\frac{1}
 {1+128R\sqrt{K_1}+64\sqrt{K_0A_k}}
 \right)\\
 &\overset{\circledTwo}{\ge}
 A_k\left(
 1+\frac{1}{2(1+128R\sqrt{K_1})}
 \right),
 \end{aligned}
\end{equation}
where in \(\circledOne\), \newrev{we used \eqref{eq:block-b-rule} and the first
branch of \eqref{eq:block-c-rule}};
in \(\circledTwo\), we used
\eqref{eq:block-pre-switch-curvature} to upper-bound the denominator.

The event that no switch has occurred by iteration \(k\) can be written as
\begin{equation}
\label{eq:block-pre-switch-event}
 \{T>k\}
 =\bigcap_{j=0}^k\{A_j<A_{\rm sw}\}
 \in\mathcal F_k.
\end{equation}
The equality follows from \eqref{eq:block-switch-time}. The membership follows
from \eqref{eq:block-adaptedness} and
\(\mathcal F_j\subseteq\mathcal F_k\) for \(j\le k\). On this event,
\eqref{eq:block-gamma-range} gives \(\gamma=1/2\). Since
\(\mathbb{I}_{\{T>k\}}\) is \(\mathcal F_k\)-measurable, the tower property
gives
\begin{equation}
\label{eq:block-indicator-tower}
 \begin{aligned}
 &\E\!\left[
 \frac{\Phi_{k+1}}{\sqrt{A_{k+1}}}
 \mathbb{I}_{\{T>k\}}\right]\\
 &\quad=
 \E\!\left[
 \mathbb{I}_{\{T>k\}}
 \E\!\left[\left.
 \frac{\Phi_{k+1}}{\sqrt{A_{k+1}}}
 \right|\mathcal F_k\right]\right].
 \end{aligned}
\end{equation}
Indeed, the tower property first inserts the conditional expectation with
respect to \(\mathcal F_k\), and measurability of the indicator then allows
it to be taken outside the inner conditional expectation. We now obtain
\begin{equation}
\label{eq:block-pre-switch-contraction}
 \begin{aligned}
 &\E\!\left[
 \frac{\Phi_{k+1}}{\sqrt{A_{k+1}}}
 \mathbb{I}_{\{T>k+1\}}\right]\\
 &\overset{\circledOne}{\le}
 \E\!\left[
 \frac{\Phi_{k+1}}{\sqrt{A_{k+1}}}
 \mathbb{I}_{\{T>k\}}\right]\\
 &\overset{\circledTwo}{=}
 \E\!\left[
 \mathbb{I}_{\{T>k\}}
 \E\!\left[\left.
 \frac{\Phi_{k+1}}{\sqrt{A_{k+1}}}
 \right|\mathcal F_k\right]\right]\\
 &\overset{\circledThree}{\le}
 \E\!\left[
 \frac{\Phi_k}{\sqrt{A_k+b_k}}
 \mathbb{I}_{\{T>k\}}\right]\\
 &\overset{\circledFour}{\le}
 \left(1+\frac{1}{2(1+128R\sqrt{K_1})}\right)^{-1/2}
 \E\!\left[
 \frac{\Phi_k}{\sqrt{A_k}}\mathbb{I}_{\{T>k\}}\right],
 \end{aligned}
\end{equation}
where in \(\circledOne\), \eqref{eq:block-potential},
Line~\ref{line:block-A-update}, and \(\F\ge0\) give
\(\Phi_{k+1}/\sqrt{A_{k+1}}\ge0\). Moreover,
since \(k+1>k\), every outcome satisfying \(T>k+1\) also satisfies
\(T>k\). Thus \(\{T>k+1\}\subseteq\{T>k\}\), and hence
\(\mathbb{I}_{\{T>k+1\}}\le\mathbb{I}_{\{T>k\}}\).
Multiplying this pointwise inequality by the nonnegative random variable
\(\Phi_{k+1}/\sqrt{A_{k+1}}\) and taking expectations gives
\(\circledOne\); in \(\circledTwo\),
\eqref{eq:block-indicator-tower} gives the stated equality; in
\(\circledThree\), we applied \eqref{eq:block-normalized} with
\(\gamma=1/2\); in \(\circledFour\), we used
\eqref{eq:block-pre-switch-growth}.

We now propagate the one-step contraction
\eqref{eq:block-pre-switch-contraction} from iteration \(0\) to iteration
\(k\). This bounds the normalized potential on the paths that have not
switched by time \(k\):
\begin{equation}
\label{eq:block-pre-switch-iteration}
 \begin{aligned}
 \E\!\left[
 \frac{\Phi_k}{\sqrt{A_k}}\mathbb{I}_{\{T>k\}}\right]
 &\overset{\circledOne}{\le}
 \left(1+\frac{1}{2(1+128R\sqrt{K_1})}\right)^{-k/2}
 \E\!\left[
 \frac{\Phi_0}{\sqrt{A_0}}\mathbb{I}_{\{T>0\}}\right]\\
 &\overset{\circledTwo}{\le}
 \frac{\Phi_0}{\sqrt{A_0}}
 \left(1+\frac{1}{2(1+128R\sqrt{K_1})}\right)^{-k/2}\\
 &\overset{\circledThree}{\le}
 \frac{R^2}{\sqrt{A_0}}
 \left(1+\frac{1}{2(1+128R\sqrt{K_1})}\right)^{-k/2},
 \end{aligned}
\end{equation}
where in \(\circledOne\), we iterated
\eqref{eq:block-pre-switch-contraction} for
\(0,\ldots,k-1\); each iteration contributes the factor with exponent
\(-1/2\), so their product has exponent \(-k/2\); in
\(\circledTwo\), we used \(\mathbb{I}_{\{T>0\}}\le1\) and the fact that
\(\Phi_0/\sqrt{A_0}\) is deterministic and nonnegative; in
\(\circledThree\), we applied \eqref{eq:block-initial-potential}.

On \(\{T>N_1\}\), we have \(A_{N_1}<A_{\rm sw}\). Therefore,
\begin{equation}
\label{eq:block-unswitched-potential}
 \begin{aligned}
 \E[\Phi_{N_1}\mathbb{I}_{\{T>N_1\}}]
 &\overset{\circledOne}{\le}
 \sqrt{A_{\rm sw}}\,
 \E\!\left[
 \frac{\Phi_{N_1}}{\sqrt{A_{N_1}}}
 \mathbb{I}_{\{T>N_1\}}\right]\\
 &\overset{\circledTwo}{\le}
 R^2\sqrt{\frac{A_{\rm sw}}{A_0}}
 \left(1+\frac{1}{2(1+128R\sqrt{K_1})}\right)^{-N_1/2}\\
 &\overset{\circledThree}{\le}
 R^2\sqrt{\frac{A_{\rm sw}}{A_0}}\,
 e^{-2(1+\log(2\Delta_0/\varepsilon))},
 \end{aligned}
\end{equation}
where in \(\circledOne\), \eqref{eq:block-switch-time} gives
\(A_{N_1}<A_{\rm sw}\) on the indicated event; in \(\circledTwo\), we applied
\eqref{eq:block-pre-switch-iteration}. To justify \(\circledThree\), we use
the following chain:
\[
 \begin{aligned}
 &\left(
 1+\frac{1}{2(1+128R\sqrt{K_1})}
 \right)^{-N_1/2}\\
 &\overset{\circledOne}{=}
 \exp\left\{
 -\frac{N_1}{2}
 \log\left(
 1+\frac{1}{2(1+128R\sqrt{K_1})}
 \right)\right\}\\
 &\overset{\circledTwo}{\le}
 \exp\left\{
 -\frac{N_1}{6(1+128R\sqrt{K_1})}
 \right\}\\
 &\overset{\circledThree}{\le}
 \exp\left\{
 -2\left(1+\log\frac{2\Delta_0}{\varepsilon}\right)
 \right\},
 \end{aligned}
\]
where in \(\circledOne\), we used \(a^{-r}=\exp(-r\log a)\); in
\(\circledTwo\), the bound \(\log(1+x)\ge x/(1+x)\) at
\(x=1/(2(1+128R\sqrt{K_1}))\) gives
\(\frac12\log(1+1/(2(1+128R\sqrt{K_1})))
\ge1/(6(1+128R\sqrt{K_1}))\), where we also used
\(2(1+128R\sqrt{K_1})+1
\le3(1+128R\sqrt{K_1})\); in \(\circledThree\), we used
\eqref{eq:block-first-horizon}, which gives
\(N_1\ge12(1+128R\sqrt{K_1})
(1+\log(2\Delta_0/\varepsilon))\). If \(N_1=0\), then \(T=0\), so the
event \(\{T>N_1\}\) is empty and \eqref{eq:block-unswitched-potential}
holds trivially.
\begin{currentrevision}
For completeness, the logarithmic inequality used in \(\circledTwo\) of the
chain following \eqref{eq:block-unswitched-potential} follows directly
from integration. For \(x\ge0\), the function \(t\mapsto1/(1+t)\) is
decreasing, and hence
\[
 \log(1+x)
 =\int_0^x\frac{dt}{1+t}
 \ge\int_0^x\frac{dt}{1+x}
 =\frac{x}{1+x}.
\]
This proves the logarithmic inequality used in \(\circledTwo\).
\end{currentrevision}

\noindent\fbox{\strut\textit{5.2. Final contribution of unswitched paths.}}
The estimate~\eqref{eq:block-unswitched-potential} bounds the potential at the
intermediate time \(N_1\) on the event \(\{T>N_1\}\). We now propagate this
bound to the final iteration \(N\) and
convert it into an objective-gap estimate. We first record an upper bound on
the switching threshold:
\begin{equation}
\label{eq:block-switch-threshold-upper}
 \frac{A_{\rm sw}}{A_0}
 \overset{\circledOne}{\le}
 \frac{4eR^2/\varepsilon}{A_0}
 \overset{\circledTwo}{=}
 \frac{8e\Delta_0}{\varepsilon},
\end{equation}
where in \(\circledOne\), we used
\eqref{eq:block-switch-threshold} and \(\Delta_0>\varepsilon\), which imply
that both arguments of the outer maximum are at most
\(4eR^2/\varepsilon\). Explicitly, the inner minimum is at most its second
argument \(4eR^2/\varepsilon\), while
\[
 A_0=\frac{R^2}{2\Delta_0}
 <\frac{R^2}{2\varepsilon}
 <\frac{4eR^2}{\varepsilon}.
\]
Therefore their maximum is also at most \(4eR^2/\varepsilon\). In
\(\circledTwo\), we substituted
\(A_0=R^2/(2\Delta_0)\).

\begin{revision}
By \eqref{eq:block-pre-switch-event}, the event \(\{T>N_1\}\) belongs to
\(\mathcal F_{N_1}\), and hence to \(\mathcal F_k\) for every
\(k\ge N_1\). Therefore, for such \(k\),
\begin{equation}
\label{eq:block-unswitched-propagation}
 \begin{aligned}
 \E[\Phi_{k+1}\mathbb{I}_{\{T>N_1\}}]
 &\overset{\circledOne}{=}
 \E\!\left[
 \mathbb{I}_{\{T>N_1\}}
 \E[\Phi_{k+1}\mid\mathcal F_k]\right]\\
 &\overset{\circledTwo}{\le}
 \E[\Phi_k\mathbb{I}_{\{T>N_1\}}],
 \end{aligned}
\end{equation}
where in \(\circledOne\), we used the tower property and the
\(\mathcal F_k\)-measurability of the indicator; in \(\circledTwo\), we used
\eqref{eq:block-unscaled}. Iterating
\eqref{eq:block-unswitched-propagation} from \(N_1\) to \(N-1\) gives
\begin{equation}
\label{eq:block-unswitched-iteration}
 \E[\Phi_N\mathbb{I}_{\{T>N_1\}}]
 \le\E[\Phi_{N_1}\mathbb{I}_{\{T>N_1\}}].
\end{equation}
Combining \eqref{eq:block-unswitched-iteration} with
\(\F(x_N)\le\Phi_N/A_N\le\Phi_N/A_0\) gives
\end{revision}
\begin{align}
 \E[\F(x_N)\mathbb{I}_{\{T>N_1\}}]
 &\overset{\circledOne}{\le}
 \frac{1}{A_0}
 \E[\Phi_{N_1}\mathbb{I}_{\{T>N_1\}}]\nonumber\\
 &\overset{\circledTwo}{\le}
 \frac{R^2}{A_0}\sqrt{\frac{A_{\rm sw}}{A_0}}\,
 e^{-2(1+\log(2\Delta_0/\varepsilon))}\nonumber\\
 &\overset{\circledThree}{\le}
 2\Delta_0\sqrt{\frac{8e\Delta_0}{\varepsilon}}\,
 e^{-2(1+\log(2\Delta_0/\varepsilon))}\nonumber\\
 &\overset{\circledFour}{=}
 2\Delta_0\sqrt{\frac{8e\Delta_0}{\varepsilon}}\,
 e^{-2}\left(\frac{2\Delta_0}{\varepsilon}\right)^{-2}
 \nonumber\\
 &\overset{\circledFive}{=}
 \sqrt2\,e^{-3/2}\sqrt{\frac{\varepsilon}{\Delta_0}}\,
 \varepsilon
 \overset{\circledSix}{<}\frac{\varepsilon}{2},
 \label{eq:block-late}
\end{align}
where in \(\circledOne\), \eqref{eq:block-potential} and
Line~\ref{line:block-A-update} give
\(\F(x_N)\le\Phi_N/A_N\le\Phi_N/A_0\), while
\eqref{eq:block-unswitched-propagation} propagates the potential bound from
\(N_1\) to \(N\); in \(\circledTwo\), we
used \eqref{eq:block-unswitched-potential}; in \(\circledThree\), we
substituted \(R^2/A_0=2\Delta_0\) and used
\eqref{eq:block-switch-threshold-upper}; in \(\circledFour\), we used
\(e^{-2(1+\log u)}=e^{-2}u^{-2}\) with
\(u=2\Delta_0/\varepsilon\); in \(\circledFive\), we used
\((2\Delta_0/\varepsilon)^{-2}
=\varepsilon^2/(4\Delta_0^2)\),
\(\sqrt{8e\Delta_0/\varepsilon}
=\sqrt{8e}\sqrt{\Delta_0/\varepsilon}\), and
\(\sqrt{8e}/2=\sqrt2\,e^{1/2}\), and then collected the remaining powers;
in \(\circledSix\), we used \(\Delta_0>\varepsilon\) and \(e^3>8\).

\noindent\fbox{\strut\textit{5.3. Post-switch contraction.}}
On \(\{T\le N_1\}\), the weight at the switching time satisfies
\begin{equation}
\label{eq:block-switch-weight}
 A_{\rm sw}\le A_T\le2A_{\rm sw}.
\end{equation}
The lower bound follows from \eqref{eq:block-switch-time}. If \(T=0\), then
\eqref{eq:block-switch-threshold} and \eqref{eq:block-switch-time} give
\(A_T=A_0=A_{\rm sw}\). If \(T>0\), then
\[
 A_T
 \overset{\circledOne}{\le}A_{T-1}+b_{T-1}
 \overset{\circledTwo}{\le}2A_{T-1}
 \overset{\circledThree}{<}2A_{\rm sw},
\]
where in \(\circledOne\), we used Line~\ref{line:block-A-update} and
\eqref{eq:block-weight-positive}. More precisely,
Line~\ref{line:block-A-update} gives
\(A_T=\max\{A_0,a_{T-1}\}\), while
\eqref{eq:block-weight-positive} gives
\(a_{T-1}\le A_{T-1}+b_{T-1}\). The other argument of the maximum also
satisfies \(A_0\le A_{T-1}<A_{T-1}+b_{T-1}\), because
Line~\ref{line:block-A-update} keeps \(A_{T-1}\ge A_0\) and
\eqref{eq:block-b-rule} gives \(b_{T-1}>0\). Thus both arguments of the
maximum are at most \(A_{T-1}+b_{T-1}\); in \(\circledTwo\),
\eqref{eq:block-b-rule} gives \(b_{T-1}\le A_{T-1}\), because its
denominator is at least one; in \(\circledThree\), \(T>0\)
and \eqref{eq:block-switch-time} states that \(T\) is the first index with
\(A_T\ge A_{\rm sw}\); hence \(A_{T-1}<A_{\rm sw}\).

\begin{currentrevision}
For \(k\ge T\), \eqref{eq:block-switch-time} gives
\(\max_{0\le j\le k}A_j\ge A_T\ge A_{\rm sw}\), so the second branch of
\eqref{eq:block-c-rule} fixes
\(c_k=4K_1(1+\log(2\Delta_0/\varepsilon))\), while
\eqref{eq:block-gamma-range} gives the corresponding value of
\(\gamma\). The weight update then satisfies
\end{currentrevision}
\begin{equation}
\label{eq:block-post-switch-weight-recurrence}
 \begin{aligned}
 A_{k+1}
 &\overset{\circledOne}{\le}A_k+b_k\\
 &\overset{\circledTwo}{=}
 A_k+
 \frac{A_k}
 {1+128R\sqrt{K_1(1+\log(2\Delta_0/\varepsilon))}
  +64\sqrt{K_0A_k}},
 \end{aligned}
\end{equation}
where in \(\circledOne\), we used Line~\ref{line:block-A-update} and
\eqref{eq:block-weight-positive}; in \(\circledTwo\), we used
\newrev{\eqref{eq:block-b-rule} with the second branch of
\eqref{eq:block-c-rule}}.

We need a deterministic upper bound for the random weights
\(A_T,A_{T+1},\ldots\). The bound
\eqref{eq:block-switch-weight} gives the
initial bound \(A_T\le2A_{\rm sw}\). Starting from this bound and applying
the same right-hand side as in
\eqref{eq:block-post-switch-weight-recurrence} defines
\begin{equation}
\label{eq:block-envelope-sequence}
 s_0:=2A_{\rm sw},\qquad
 s_{j+1}:=s_j+
 \frac{s_j}
 {1+128R\sqrt{K_1(1+\log(2\Delta_0/\varepsilon))}
  +64\sqrt{K_0s_j}} .
\end{equation}
The update on the right-hand side is increasing in \(s\), because the
derivative of its second term is
\begin{equation}
\label{eq:block-envelope-monotonicity}
 \frac{1+128R\sqrt{K_1(1+\log(2\Delta_0/\varepsilon))}
       +32\sqrt{K_0s}}
      {\left(1+128R\sqrt{K_1(1+\log(2\Delta_0/\varepsilon))}
       +64\sqrt{K_0s}\right)^2}>0.
\end{equation}
Using \eqref{eq:block-switch-weight} as the base case, induction gives
\begin{equation}
\label{eq:block-weight-envelope}
 A_{T+j}\le s_j,\qquad j\ge0,
 {\color{revisionblue}\quad\text{on }\{T\le N_1\}}.
\end{equation}
Indeed, if \(A_{T+j}\le s_j\), then
\[
 A_{T+j+1}
 \overset{\circledOne}{\le}A_{T+j}+b_{T+j}
 \overset{\circledTwo}{\le}s_{j+1},
\]
where in \(\circledOne\), we used Line~\ref{line:block-A-update} and
\eqref{eq:block-weight-positive}; in \(\circledTwo\), we used
\eqref{eq:block-b-rule}, \eqref{eq:block-envelope-monotonicity}, and
\eqref{eq:block-envelope-sequence}.

The ratio appears because the right-hand side of
\eqref{eq:block-normalized} can be written as
\[
 \frac{\Phi_k}{(A_k+b_k)^\gamma}
 =
 \frac{\Phi_k}{A_k^\gamma}
 \left(\frac{A_k}{A_k+b_k}\right)^\gamma.
\]
For \(k=T+j\), we bound this ratio as follows:
\begin{equation}
\label{eq:block-envelope-ratio}
 \begin{aligned}
 \frac{A_{T+j}}{A_{T+j}+b_{T+j}}
 &\overset{\circledOne}{=}
 \left(1+\frac{1}
 {1+128R\sqrt{K_1(1+\log(2\Delta_0/\varepsilon))}
  +64\sqrt{K_0A_{T+j}}}\right)^{-1}\\
 &\overset{\circledTwo}{\le}
 \left(1+\frac{1}
 {1+128R\sqrt{K_1(1+\log(2\Delta_0/\varepsilon))}
  +64\sqrt{K_0s_j}}\right)^{-1}\\
 &\overset{\circledThree}{=}
 \frac{s_j}{
 s_j+
 \dfrac{s_j}
 {1+128R\sqrt{K_1(1+\log(2\Delta_0/\varepsilon))}
  +64\sqrt{K_0s_j}}}\\
 &\overset{\circledFour}{=}\frac{s_j}{s_{j+1}},
 \end{aligned}
\end{equation}
\begin{greenrevision}
where in \(\circledOne\), we used the expression for
\(A_k+b_k\) in \eqref{eq:block-post-switch-weight-recurrence} at
\(k=T+j\) and divided numerator and denominator by
\(A_{T+j}\ge A_0>0\), where the lower bound follows from
Line~\ref{line:block-A-update}; in
\(\circledTwo\), we used \eqref{eq:block-weight-envelope}; the displayed expression is
nondecreasing in its final argument because both
\(s\mapsto1+128R\sqrt{K_1(1+\log(2\Delta_0/\varepsilon))}
+64\sqrt{K_0s}\) and \(u\mapsto(1+1/u)^{-1}\) are nondecreasing; in
\(\circledThree\), we multiplied the numerator and denominator by
\(s_j>0\); in \(\circledFour\), the denominator is exactly
\(s_{j+1}\) by \eqref{eq:block-envelope-sequence}.
\end{greenrevision}

\begin{greenrevision}
Because the switching time is random, we iterate the post-switch estimate
separately for each possible value of \(T\). Fix \(t\le N_1\). By
\eqref{eq:block-switch-time},
\begin{equation}
\label{eq:block-switch-event}
 \{T=t\}
 =
 \left(\bigcap_{m=0}^{t-1}\{A_m<A_{\rm sw}\}\right)
 \cap\{A_t\ge A_{\rm sw}\}
 \in\mathcal F_t,
\end{equation}
where the intersection is the whole sample space when \(t=0\), and the
membership follows from \eqref{eq:block-adaptedness} and
\(\mathcal F_m\subseteq\mathcal F_t\) for \(m\le t\). Thus, at time \(t\),
it is already known whether \(T=t\). Hence
\(\mathbb{I}_{\{T=t\}}\) is measurable with respect to
\(\mathcal F_{t+h}\) for every \(h\ge0\). Moreover,
\eqref{eq:block-gamma-range} gives the same post-switch value of
\(\gamma\) at every iteration \(t+h\).
\begin{currentrevision}
We record one application
explicitly. Set \(X_{t+h}:=\Phi_{t+h}/A_{t+h}^{\gamma}\). The indicator
identity at the beginning of Step~5 and \eqref{eq:block-normalized} give
\begin{equation}
\label{eq:block-post-switch-one-step}
\begin{aligned}
 \E[X_{t+h+1}\mathbb I_{\{T=t\}}]
 &=\E\!\left[
 \mathbb I_{\{T=t\}}
 \E[X_{t+h+1}\mid\mathcal F_{t+h}]
 \right]\\
 &\le\E\!\left[
 \frac{\Phi_{t+h}}{(A_{t+h}+b_{t+h})^\gamma}
 \mathbb I_{\{T=t\}}
 \right]\\
 &=\E\!\left[
 X_{t+h}
 \left(\frac{A_{t+h}}{A_{t+h}+b_{t+h}}\right)^\gamma
 \mathbb I_{\{T=t\}}
 \right]\\
 &\le\left(\frac{s_h}{s_{h+1}}\right)^\gamma
 \E[X_{t+h}\mathbb I_{\{T=t\}}].
\end{aligned}
\end{equation}
The first equality uses the \(\mathcal F_{t+h}\)-measurability of the
indicator, the first inequality substitutes \eqref{eq:block-normalized}, the
second equality factors \(A_{t+h}^{-\gamma}\), and the last inequality is
\eqref{eq:block-envelope-ratio} on \(\{T=t\}\). Its deterministic factor may
be pulled outside the expectation. Applying this one-step relation for
\(h=0,\ldots,j-1\) gives
\end{currentrevision}
\begin{equation}
\label{eq:block-post-switch-contraction}
 \begin{aligned}
 &\E\!\left[
 \frac{\Phi_{t+j}}{A_{t+j}^{\gamma}}
 \mathbb{I}_{\{T=t\}}\right]\\
 &\overset{\circledOne}{\le}
 \left(\prod_{h=0}^{j-1}\frac{s_h}{s_{h+1}}\right)^\gamma
 \E\!\left[
 \frac{\Phi_t}{A_t^\gamma}\mathbb{I}_{\{T=t\}}\right]\\
 &\overset{\circledTwo}{=}
 \left(\frac{s_0}{s_j}\right)^\gamma
 \E\!\left[
 \frac{\Phi_t}{A_t^\gamma}\mathbb{I}_{\{T=t\}}\right]\\
 &\overset{\circledThree}{\le}
 \frac{2^\gamma}{s_j^\gamma}
 \E[\Phi_t\mathbb{I}_{\{T=t\}}],
 \end{aligned}
\end{equation}
where in \(\circledOne\), \newrev{we successively applied
\eqref{eq:block-post-switch-one-step}}; in
\(\circledTwo\), the product telescopes; in \(\circledThree\), we used
\(s_0=2A_{\rm sw}\) from \eqref{eq:block-envelope-sequence} and
\(A_t\ge A_{\rm sw}\) on \(\{T=t\}\), which follows from
\eqref{eq:block-switch-event}.
\end{greenrevision}

\noindent\fbox{\strut\textit{5.4. Paths that have switched by \(N_1\).}}
To specify how many post-switch iterations are sufficient, define the second
analytical horizon by
\begin{equation}
\label{eq:block-second-horizon}
 N_2:=\min\left\{j\ge0:s_j\ge\frac{4eR^2}{\varepsilon}\right\}.
\end{equation}
\begin{currentrevision}
Thus \(N_2\) is the first number of post-switch steps for which the
deterministic envelope reaches the target weight \(4eR^2/\varepsilon\).
This particular target is chosen because the substitution in
\eqref{eq:block-switched-reduction}--\eqref{eq:block-early} contains
\(s_j^{-\gamma}\). Replacing \(s_j\) by
\(4eR^2/\varepsilon\), together with the initial-potential bound
\(\Phi_0\le R^2\), makes the contribution of all switched paths smaller
than \(\varepsilon/2\). The definition of
\(N_2\) alone does not assert convergence.
\end{currentrevision}
Step~6 verifies that this index is finite and that the prescribed iteration
budget satisfies \(N\ge N_1+N_2\).

\begin{greenrevision}
Throughout this part, the stopped quantity on the switching event is
understood as the finite sum
\[
 \Phi_T\mathbb{I}_{\{T\le N_1\}}
 :=\sum_{t=0}^{N_1}\Phi_t\mathbb{I}_{\{T=t\}}.
\]
Thus no value \(\Phi_\infty\) is needed on paths that never switch.
\end{greenrevision}

\begin{currentrevision}
Assume \(N\ge N_1+N_2\). On \(\{T\le N_1\}\), we have
\(N-T\ge N_2\). Equation~\eqref{eq:block-envelope-sequence} gives
\[
 s_{j+1}-s_j=
 \frac{s_j}
 {1+128R\sqrt{K_1(1+\log(2\Delta_0/\varepsilon))}
  +64\sqrt{K_0s_j}}>0,
\]
because \(s_j>0\). Thus \((s_j)\) is increasing, and
\eqref{eq:block-second-horizon} gives
\(s_{N-T}\ge4eR^2/\varepsilon\). Moreover,
\end{currentrevision}
\begin{equation}
\label{eq:block-objective-from-normalized}
 \F(x_N)
 \overset{\circledOne}{\le}\frac{\Phi_N}{A_N}
 \overset{\circledTwo}{=}
 \frac{\Phi_N/A_N^\gamma}{A_N^{1-\gamma}}
 \overset{\circledThree}{\le}
 \frac{\Phi_N/A_N^\gamma}{A_0^{1-\gamma}},
\end{equation}
where in \(\circledOne\), we used \eqref{eq:block-potential}; in
\(\circledTwo\), we split the power of \(A_N\); in \(\circledThree\), we
used \(A_N\ge A_0>0\) from Line~\ref{line:block-A-update} and
\(1-\gamma>0\) from \eqref{eq:block-gamma-range}.

\begin{greenrevision}
Combining \eqref{eq:block-post-switch-contraction},
\eqref{eq:block-second-horizon}, and
\eqref{eq:block-objective-from-normalized} gives
\begin{equation}
\label{eq:block-switched-reduction}
 \begin{aligned}
 \E[\F(x_N)\mathbb{I}_{\{T\le N_1\}}]
 &\overset{\circledOne}{\le}
 \frac{1}{A_0^{1-\gamma}}
 \E\!\left[
 \frac{\Phi_N}{A_N^\gamma}\mathbb{I}_{\{T\le N_1\}}
 \right]\\
 &\overset{\circledTwo}{=}
 \frac{1}{A_0^{1-\gamma}}
 \sum_{t=0}^{N_1}
 \E\!\left[
 \frac{\Phi_N}{A_N^\gamma}
 \mathbb{I}_{\{T=t\}}\right]\\
 &\overset{\circledThree}{\le}
 \frac{2^\gamma}{A_0^{1-\gamma}}
 \sum_{t=0}^{N_1}
 \E\!\left[
 \frac{\Phi_t}{s_{N-t}^\gamma}
 \mathbb{I}_{\{T=t\}}
 \right]\\
 &\overset{\circledFour}{\le}
 \frac{2^\gamma}
 {A_0^{1-\gamma}(4eR^2/\varepsilon)^\gamma}
 \sum_{t=0}^{N_1}
 \E[\Phi_t\mathbb{I}_{\{T=t\}}]\\
 &\overset{\circledFive}{=}
 \frac{2^\gamma}
 {A_0^{1-\gamma}(4eR^2/\varepsilon)^\gamma}
 \E[\Phi_T\mathbb{I}_{\{T\le N_1\}}],
 \end{aligned}
\end{equation}
where in \(\circledOne\), we multiplied
\eqref{eq:block-objective-from-normalized} by the nonnegative indicator and
took expectations; in \(\circledTwo\), we used the disjoint decomposition
\(\{T\le N_1\}=\bigcup_{t=0}^{N_1}\{T=t\}\); in \(\circledThree\), for each
\(t\le N_1\), we applied \eqref{eq:block-post-switch-contraction} with
\(j=N-t\); in
\(\circledFour\), \(t\le N_1\) and \(N\ge N_1+N_2\) imply
\(N-t\ge N_2\); hence \eqref{eq:block-second-horizon} and monotonicity of
\((s_j)\) give \(s_{N-t}\ge4eR^2/\varepsilon\); in
\(\circledFive\), we used \(\Phi_t=\Phi_T\) on \(\{T=t\}\) and recombined
the disjoint events.
\end{greenrevision}

It remains to control the stopped potential. For \(k<N_1\),
\begin{equation}
\label{eq:block-stopped-potential}
 \begin{aligned}
 \E[\Phi_{(k+1)\wedge T}]
 &\overset{\circledOne}{=}
 \E[\Phi_T\mathbb{I}_{\{T\le k\}}]
 +\E[\Phi_{k+1}\mathbb{I}_{\{T>k\}}]\\
 &\overset{\circledTwo}{\le}
 \E[\Phi_T\mathbb{I}_{\{T\le k\}}]
 +\E[\Phi_k\mathbb{I}_{\{T>k\}}]
 \overset{\circledThree}{=}
 \E[\Phi_{k\wedge T}],
 \end{aligned}
\end{equation}
where in \(\circledOne\), we split according to whether the process has
stopped by time \(k\); in \(\circledTwo\), we used
\(\{T>k\}\in\mathcal F_k\), which follows from
\eqref{eq:block-switch-time} and \eqref{eq:block-adaptedness}, together
with the tower property and
\eqref{eq:block-unscaled}; in \(\circledThree\), we recombined the two
events.
\begin{currentrevision}
The conditional-expectation calculation in \(\circledTwo\) is
\[
\begin{aligned}
 \E[\Phi_{k+1}\mathbb I_{\{T>k\}}]
 &=\E\!\left[
 \mathbb I_{\{T>k\}}
 \E[\Phi_{k+1}\mid\mathcal F_k]\right]\\
 &\le\E[\Phi_k\mathbb I_{\{T>k\}}].
\end{aligned}
\]
The equality pulls the \(\mathcal F_k\)-measurable indicator outside the
conditional expectation. The inequality substitutes
\eqref{eq:block-unscaled} and discards its nonpositive term
\(-\frac14\E[\norm{\Delta v_k}^2\mid\mathcal F_k]\).
\end{currentrevision}
Iterating \eqref{eq:block-stopped-potential} and using
\eqref{eq:block-initial-potential} gives
\begin{equation}
\label{eq:block-stopped-potential-bound}
 \E[\Phi_T\mathbb{I}_{\{T\le N_1\}}]
 \overset{\circledOne}{\le}
 \E[\Phi_{N_1\wedge T}]
 \overset{\circledTwo}{\le}\Phi_0
 \overset{\circledThree}{\le}R^2,
\end{equation}
where in \(\circledOne\), nonnegativity follows from
\eqref{eq:block-potential}, \(A_k\ge A_0>0\) from
Line~\ref{line:block-A-update}, and \(\F\ge0\); in
\(\circledTwo\), we iterated \eqref{eq:block-stopped-potential}; in
\(\circledThree\), we used \eqref{eq:block-initial-potential}.

Substituting \eqref{eq:block-stopped-potential-bound} into
\eqref{eq:block-switched-reduction} yields
\begin{align}
 \E[\F(x_N)\mathbb{I}_{\{T\le N_1\}}]
 &\overset{\circledOne}{\le}\frac{2^\gamma R^2}
 {A_0^{1-\gamma}(4eR^2/\varepsilon)^\gamma}
 \overset{\circledTwo}{=}2^{1-2\gamma}e^{-\gamma}
 \left(\frac{\Delta_0}{\varepsilon}\right)^{1-\gamma}\varepsilon
 \nonumber\\
 &\overset{\circledThree}{=}\varepsilon e^{-1/2-\gamma\log 2}
 \overset{\circledFour}{\le}\frac{\varepsilon}{\sqrt{2e}}
 \overset{\circledFive}{<}\frac{\varepsilon}{2},
 \label{eq:block-early}
\end{align}
where in \(\circledOne\), we used
\eqref{eq:block-switched-reduction} and
\eqref{eq:block-stopped-potential-bound}; in \(\circledTwo\), we
substituted \(A_0=R^2/(2\Delta_0)\) and collected the powers; in
\(\circledThree\), we used \eqref{eq:block-gamma-range}, which gives
\((1-\gamma)(1+\log(2\Delta_0/\varepsilon))=1/2\); in
\(\circledFour\), we used \(\gamma\ge1/2\) from
\eqref{eq:block-gamma-range}; in \(\circledFive\), we used
\(e>2\).

Finally, the events \(\{T>N_1\}\) and \(\{T\le N_1\}\) partition the
outcomes. Therefore, \eqref{eq:block-late} and \eqref{eq:block-early} give
\begin{equation}
\label{eq:block-final-gap}
 \E[\F(x_N)]
 =\E[\F(x_N)\mathbb{I}_{\{T>N_1\}}]
 +\E[\F(x_N)\mathbb{I}_{\{T\le N_1\}}]
 <\varepsilon.
\end{equation}

\paragraph{Step 6: Iteration complexity.}
It remains to verify the two facts used in Step~5: the second horizon
\(N_2\) is finite, and the total number of iterations needed is at most the
stated complexity. We proceed in four parts.

\noindent\fbox{\strut\textit{6.1. Progress in one envelope step.}}
To measure the growth of \((s_j)\) without solving its recurrence explicitly,
define, for \(s\ge s_0\),
\begin{currentrevision}
the following function. Its form is obtained directly from the reciprocal of
one envelope increment. Indeed, if the current envelope value is \(s\), then
\eqref{eq:block-envelope-sequence} adds
\[
 \frac{s}
 {1+128R\sqrt{K_1(1+\log(2\Delta_0/\varepsilon))}
  +64\sqrt{K_0s}}.
\]
The reciprocal of this increment is
\[
 \frac{1+128R\sqrt{K_1(1+\log(2\Delta_0/\varepsilon))}}{s}
 +\frac{64\sqrt{K_0}}{\sqrt s}.
\]
We therefore choose \(P\) so that \(P'(s)\) equals this reciprocal. The
integral of the first term is a logarithm, while the integral of the second
is \(128\sqrt{K_0s}\). This gives
\end{currentrevision}
\begin{equation}
\label{eq:block-progress-function}
 P(s):=
 \left(1+128R\sqrt{K_1(1+\log(2\Delta_0/\varepsilon))}\right)
 \log\frac{s}{s_0}
 +128\sqrt{K_0}\left(\sqrt s-\sqrt{s_0}\right).
\end{equation}
Direct differentiation gives
\begin{equation}
\label{eq:block-progress-derivatives}
 \begin{aligned}
 P'(s)
 &=\frac{1+128R\sqrt{K_1(1+\log(2\Delta_0/\varepsilon))}}{s}
   +\frac{64\sqrt{K_0}}{\sqrt s}>0,\\
 P''(s)
 &=-\frac{1+128R\sqrt{K_1(1+\log(2\Delta_0/\varepsilon))}}{s^2}
   -\frac{32\sqrt{K_0}}{s^{3/2}}<0.
 \end{aligned}
\end{equation}
Thus \(P\) is increasing and \(P'\) is decreasing. We first bound the length
of one envelope step:
\begin{equation}
\label{eq:block-envelope-step-size}
 s_{j+1}-s_j
 \overset{\circledOne}{=}
 \frac{s_j}
 {1+128R\sqrt{K_1(1+\log(2\Delta_0/\varepsilon))}
  +64\sqrt{K_0s_j}}
 \overset{\circledTwo}{\le}s_j,
 \qquad s_j<s_{j+1}\le2s_j,
\end{equation}
where in \(\circledOne\), we used the recurrence
\eqref{eq:block-envelope-sequence}; in \(\circledTwo\), the denominator is
at least one. The fraction is positive because \(s_j>0\), which gives the
remaining two inequalities. We can now quantify the progress in one step:
\begin{equation}
\label{eq:block-progress-increment}
 \begin{aligned}
 P(s_{j+1})-P(s_j)
 &\overset{\circledOne}{=}
 \int_{s_j}^{s_{j+1}}P'(s)\,ds\\
 &\overset{\circledTwo}{\ge}
 (s_{j+1}-s_j)P'(2s_j)\\
 &\overset{\circledThree}{=}
 \frac{
 \frac12\left(1+128R\sqrt{K_1(1+\log(2\Delta_0/\varepsilon))}\right)
 +\frac{64\sqrt{K_0s_j}}{\sqrt2}}
 {1+128R\sqrt{K_1(1+\log(2\Delta_0/\varepsilon))}
 +64\sqrt{K_0s_j}}\\
 &\overset{\circledFour}{\ge}\frac12,
 \end{aligned}
\end{equation}
where in \(\circledOne\), we used the fundamental theorem of calculus; in
\(\circledTwo\), \eqref{eq:block-envelope-step-size} gives
\(s\le s_{j+1}\le2s_j\) throughout the integration interval, and the
monotonicity in \eqref{eq:block-progress-derivatives} gives
\(P'(s)\ge P'(2s_j)\); in \(\circledThree\), we substituted
\eqref{eq:block-envelope-sequence} and the formula for \(P'(2s_j)\); in
\(\circledFour\), we used \(1/\sqrt2\ge1/2\), so the numerator is at least
one half of the denominator. Hence every envelope step increases \(P\) by at
least \(1/2\).

\noindent\fbox{\strut\textit{6.2. Finiteness and a preliminary bound for \(N_2\).}}
Summing \eqref{eq:block-progress-increment} from \(0\) to \(j-1\) gives
\begin{equation}
\label{eq:block-progress-sum}
 P(s_j)
 \overset{\circledOne}{=}
 P(s_0)+\sum_{h=0}^{j-1}\bigl(P(s_{h+1})-P(s_h)\bigr)
 \overset{\circledTwo}{\ge}\frac j2,
\end{equation}
where in \(\circledOne\), the sum telescopes; in \(\circledTwo\), we used
\(P(s_0)=0\) from \eqref{eq:block-progress-function} and the lower bound
\eqref{eq:block-progress-increment}. If \(N_2\) were infinite, then
\eqref{eq:block-second-horizon} would give
\(s_j<4eR^2/\varepsilon\) for every \(j\). Since \(P\) is increasing,
this would imply
\(P(s_j)<P(4eR^2/\varepsilon)\), contradicting
\eqref{eq:block-progress-sum} for sufficiently large \(j\). Thus \(N_2\) is
finite.

If \(N_2\ge1\), its minimality gives
\(s_{N_2-1}<4eR^2/\varepsilon\). Therefore,
\begin{equation}
\label{eq:block-progress-before-target}
 \frac{N_2-1}{2}
 \overset{\circledOne}{\le}P(s_{N_2-1})
 \overset{\circledTwo}{<}P\!\left(\frac{4eR^2}{\varepsilon}\right),
\end{equation}
where in \(\circledOne\), we used \eqref{eq:block-progress-sum}; in
\(\circledTwo\), we used
\(s_{N_2-1}<4eR^2/\varepsilon\), which follows from the minimality in
\eqref{eq:block-second-horizon}, and the monotonicity of \(P\). Substituting
\eqref{eq:block-progress-function} into
\eqref{eq:block-progress-before-target} yields
\begin{equation}
\label{eq:block-N2-preliminary}
 \begin{aligned}
 N_2
 &\overset{\circledOne}{\le}
 2\left(1+128R\sqrt{K_1(1+\log(2\Delta_0/\varepsilon))}\right)
 \log\frac{4eR^2/\varepsilon}{s_0}\\
 &\quad
 +256\sqrt{K_0}
 \left(\sqrt{\frac{4eR^2}{\varepsilon}}-\sqrt{s_0}\right)+1\\
 &\overset{\circledTwo}{\le}
 2\left(1+128R\sqrt{K_1(1+\log(2\Delta_0/\varepsilon))}\right)
 \max\left\{0,
 \log\frac{4eR^2/\varepsilon}{2A_{\rm sw}}\right\}\\
 &\quad+256\sqrt{K_0}\sqrt{\frac{4eR^2}{\varepsilon}}+1,
 \end{aligned}
\end{equation}
where in \(\circledOne\), we multiplied
\eqref{eq:block-progress-before-target} by two and substituted the definition
of \(P\); in \(\circledTwo\), we used \(s_0=2A_{\rm sw}\) from
\eqref{eq:block-envelope-sequence} and removed the nonpositive term
\(-256\sqrt{K_0s_0}\). The positive part of the logarithm makes the same
bound valid when \(N_2=0\), in which case no post-switch envelope step is
needed.

\noindent\fbox{\strut\textit{6.3. Bounding the logarithmic factor.}}
\newrev{If \(N_2=0\), the required bound is immediate.} Suppose \(N_2\ge1\). Then the
minimality of \(N_2\) and \(s_0=2A_{\rm sw}\) give
\(2A_{\rm sw}<4eR^2/\varepsilon\). In particular,
\(A_{\rm sw}<4eR^2/\varepsilon\). The definition
\eqref{eq:block-switch-threshold} then implies
\begin{equation}
\label{eq:block-threshold-lower-bound}
 \begin{aligned}
 \frac{(1+128R\sqrt{K_1})^2}{4096K_0}
 &\overset{\circledOne}{<}\frac{4eR^2}{\varepsilon},\\
 A_{\rm sw}
 &\overset{\circledTwo}{=}
 \max\left\{A_0,
 \frac{(1+128R\sqrt{K_1})^2}{4096K_0}\right\}
 \ge\frac{(1+128R\sqrt{K_1})^2}{4096K_0},
 \end{aligned}
\end{equation}
where in \(\circledOne\), if the first inequality failed, the inner minimum in
\eqref{eq:block-switch-threshold} would equal
\(4eR^2/\varepsilon\), contradicting
\(A_{\rm sw}<4eR^2/\varepsilon\); in \(\circledTwo\), once the first
argument of the inner minimum is smaller than the target, the inner minimum
equals that argument.

The logarithm in \eqref{eq:block-N2-preliminary} can now be bounded without
the maximum:
\begin{equation}
\label{eq:block-log-ratio-bound}
 \begin{aligned}
 0<\log\frac{4eR^2/\varepsilon}{2A_{\rm sw}}
 &\overset{\circledOne}{\le}
 \log\frac{4eR^2/\varepsilon}
 {2(1+128R\sqrt{K_1})^2/(4096K_0)}\\
 &\overset{\circledTwo}{\le}
 \log\frac{16384eK_0R^2}
 {\varepsilon(1+128R\sqrt{K_1})^2},
 \end{aligned}
\end{equation}
where in \(\circledOne\), we used the lower bound on \(A_{\rm sw}\) from
\eqref{eq:block-threshold-lower-bound} and monotonicity of the logarithm; in
\(\circledTwo\), the denominator in the first ratio is twice the denominator
in the second ratio, so the first ratio is no larger; we then used
monotonicity of the logarithm. The ratio in the last logarithm is larger than one:
this follows from the first inequality in
\eqref{eq:block-threshold-lower-bound} after multiplying by
\(4096K_0\varepsilon/(1+128R\sqrt{K_1})^2\).

The coefficient multiplying this logarithm satisfies
\begin{equation}
\label{eq:block-progress-coefficient-bound}
 \begin{aligned}
 &1+128R\sqrt{K_1(1+\log(2\Delta_0/\varepsilon))}\\
 &\quad\overset{\circledOne}{\le}
 (1+128R\sqrt{K_1})
 \sqrt{1+\log\frac{2\Delta_0}{\varepsilon}},
 \end{aligned}
\end{equation}
where in \(\circledOne\), we used
\(\sqrt{1+\log(2\Delta_0/\varepsilon)}\ge1\): after expanding the right-hand
side, its first term is at least \(1\), while its second term equals the
second term on the left-hand side. Finally, for any \(a\ge1\) and \(u\ge1\),
\begin{equation}
\label{eq:block-elementary-log-bound}
 \sqrt a\log u
 \overset{\circledOne}{=}
 4\sqrt a\log(u^{1/4})
 \overset{\circledTwo}{\le}
 4\sqrt a\,u^{1/4}
 \overset{\circledThree}{\le}2a+2\sqrt u,
\end{equation}
where in \(\circledOne\), we used \(\log u=4\log(u^{1/4})\); in
\(\circledTwo\), we used
\(\log v\le v-1\le v\) with \(v=u^{1/4}\ge1\); in
\(\circledThree\), we used
\(4xy\le2x^2+2y^2\), which follows from \((x-y)^2\ge0\), with
\(x=\sqrt a\) and \(y=u^{1/4}\).
\begin{currentrevision}
The inequality \(\log v\le v-1\) used in \(\circledTwo\) follows by
considering \(g(v):=v-1-\log v\). Since
\(g(1)=0\) and \(g'(v)=1-1/v\ge0\) for \(v\ge1\), we have
\(g(v)\ge0\), which is exactly \(\log v\le v-1\).
\end{currentrevision}

\noindent\fbox{\strut\textit{6.4. Total iteration budget.}}
The square root of the ratio in \eqref{eq:block-log-ratio-bound} satisfies
\begin{equation}
\label{eq:block-log-ratio-square-root}
 \begin{aligned}
 &(1+128R\sqrt{K_1})
 \sqrt{\frac{16384eK_0R^2}
 {\varepsilon(1+128R\sqrt{K_1})^2}}\\
 &\quad\overset{\circledOne}{=}
 64\sqrt{K_0}\sqrt{\frac{4eR^2}{\varepsilon}},
 \end{aligned}
\end{equation}
where in \(\circledOne\), we used
\[
 \frac{1+128R\sqrt{K_1}}
 {\sqrt{(1+128R\sqrt{K_1})^2}}=1,
\]
which is valid because \(1+128R\sqrt{K_1}>0\), and then used
\(\sqrt{16384}=128=64\sqrt4\).

Suppose first that \(N_2\ge1\). Combining
\eqref{eq:block-log-ratio-bound},
\eqref{eq:block-progress-coefficient-bound}, and
\eqref{eq:block-elementary-log-bound} gives
\begin{equation}
\label{eq:block-mixed-term-bound}
 \begin{aligned}
 &2\left(1+128R\sqrt{K_1(1+\log(2\Delta_0/\varepsilon))}\right)
 \log\frac{4eR^2/\varepsilon}{2A_{\rm sw}}\\
 &\overset{\circledOne}{\le}
 2(1+128R\sqrt{K_1})
 \sqrt{1+\log\frac{2\Delta_0}{\varepsilon}}
 \log\frac{16384eK_0R^2}
 {\varepsilon(1+128R\sqrt{K_1})^2}\\
 &\overset{\circledTwo}{\le}
 4(1+128R\sqrt{K_1})
 \left(1+\log\frac{2\Delta_0}{\varepsilon}\right)
 +256\sqrt{K_0}\sqrt{\frac{4eR^2}{\varepsilon}},
 \end{aligned}
\end{equation}
where in \(\circledOne\), we applied
\eqref{eq:block-log-ratio-bound} and
\eqref{eq:block-progress-coefficient-bound}; in \(\circledTwo\), we applied
\eqref{eq:block-elementary-log-bound} with
\(a=1+\log(2\Delta_0/\varepsilon)\) and with \(u\) equal to the ratio in
\eqref{eq:block-log-ratio-bound}. Here \(a\ge1\) follows from
\(\Delta_0>\varepsilon\).
\newrev{Moreover, multiplying the first inequality in
\eqref{eq:block-threshold-lower-bound} by
\(4096K_0\varepsilon/(1+128R\sqrt{K_1})^2\) gives \(u>1\).}
We then used
\eqref{eq:block-log-ratio-square-root}.

Substituting \eqref{eq:block-mixed-term-bound} into
\eqref{eq:block-N2-preliminary} gives
\begin{equation}
\label{eq:block-N2-final}
 N_2
 \overset{\circledOne}{\le}
 4(1+128R\sqrt{K_1})
 \left(1+\log\frac{2\Delta_0}{\varepsilon}\right)
 +512\sqrt{K_0}\sqrt{\frac{4eR^2}{\varepsilon}}+1,
\end{equation}
where in \(\circledOne\), the two \(K_0\)-terms, each with coefficient \(256\),
sum to \(512\). The bound~\eqref{eq:block-N2-final} is immediate when
\(N_2=0\), because its
right-hand side is positive.

The definition \eqref{eq:block-first-horizon} and
\(\lceil x\rceil\le x+1\) give, in both branches of that definition,
\begin{equation}
\label{eq:block-N1-final}
 N_1\le
 12(1+128R\sqrt{K_1})
 \left(1+\log\frac{2\Delta_0}{\varepsilon}\right)+1.
\end{equation}
\begin{greenrevision}
The first horizon \(N_1\) covers the iterations needed to reach the switching
threshold, while \(N_2\) covers the subsequent growth to the target weight.
Thus, we choose \(N:=N_1+N_2\). Adding
\eqref{eq:block-N2-final} and \eqref{eq:block-N1-final} yields
\begin{equation}
\label{eq:block-total-horizon}
 \begin{aligned}
 N=N_1+N_2
 &\le16(1+128R\sqrt{K_1})
 \left(1+\log\frac{2\Delta_0}{\varepsilon}\right)\\
 &\quad+512\sqrt{K_0}\sqrt{\frac{4eR^2}{\varepsilon}}+2.
 \end{aligned}
\end{equation}
For this choice of \(N\), \eqref{eq:block-final-gap} gives
\(\E_{\xi_0,\ldots,\xi_{N-1}}[f(x_N)-f^*]\le\varepsilon\). Each iteration
uses exactly one block-gradient query, as shown in
Algorithm~\ref{alg:accelerated-block}.
\end{greenrevision}
\end{proof}

\begin{currentrevision}
\paragraph{Relation to subsequent work.}
The first version of this paper (arXiv:2608.04884v1, August~5, 2026) predates
and is cited by \citet{podlipnova2026smooth} (arXiv:2609.32558v1,
September~26, 2026).  This chronology supports our claim that, to our knowledge,
this paper provides the first accelerated full-gradient and coordinate
guarantees for the convex \((H_0,H_1)\)-smooth class.  The
unified block analysis in Theorem~\ref{thm:common-block} and its full-gradient
reduction were developed independently of their related construction.  Our
analysis applies to differentiable
convex objectives through local first-order block models rather than a global
twice-differentiable Hessian--gap condition, and one proof yields full-gradient,
uniform-coordinate, and non-uniform-coordinate guarantees.
Moreover, \orrev{for Algorithm~\ref{alg:full-gradient-inner} in the
full-gradient setting, we provide} a dedicated analysis
\orrev{that exploits} distinct query and update weights to attain the
\((L_1R)^{2/3}\) dependence on the \((L_0,L_1)\)-smooth subclass with one
full-gradient query per iteration, \orrev{as shown in}
Proposition~\ref{prop:l0l1-full-gradient-rate}.
Finally, for the general randomized block-coordinate
setting, Theorem~\ref{thm:common-block} directly gives
\(\E[F(x_N)]\le\varepsilon\) in one run, using one block-gradient query per
iteration and no function-value comparisons. In contrast, the expected-gap
guarantee of \citet[Theorem~4.3]{podlipnova2026smooth} converts their
high-probability guarantee in Theorem~4.2 using
\(O(\log(\Delta_0/\varepsilon))\) independent repetitions and objective
comparisons. In their published randomized expected-gap bound, this introduces,
in particular, an additional logarithmic factor in the baseline-curvature term
hidden by \(\widetilde O\)-notation.
\end{currentrevision}

\subsection{Proof of Theorem~\ref{thm:full-gradient-main}}
\label{app:full-gradient-proof}

\paragraph{Proof outline.}
We first reduce Algorithm~\ref{alg:accelerated-block} to a single block. We
then show that its projection is inactive, write the resulting full-gradient
updates explicitly, and substitute the specialized constants into
Theorem~\ref{thm:common-block}. Before the reduction, we record a gradient
estimate. In Part~2, this estimate is used to show that the point before
projection already belongs to \(B(x^0,3R)\).

\begin{lemmabox}
\begin{lemma}[Gradient bound]
\label{lem:full-gradient-self-bound}
Set
\(C:=1+(2\sqrt3+\sqrt6)^2=19+12\sqrt2<36\). Under
Assumptions~\ref{ass:convexity} and~\ref{ass:h0h1-smoothness}, every
\(x\in\mathbb R^d\) satisfies
\begin{equation}
\label{eq:full-gradient-self-bound}
 \norm{\nabla f(x)}^2
 \le2C\bigl(H_0+H_1\F(x)\bigr)\F(x).
\end{equation}
\end{lemma}
\end{lemmabox}

\begin{proof}[Proof of Lemma~\ref{lem:full-gradient-self-bound}]
The case \(\nabla f(x)=0\) is immediate. Suppose now that
\(\nabla f(x)\ne0\). Then \(H_0+H_1\F(x)>0\). Indeed, if this quantity were
zero and \(H_1>0\), then \(\F(x)=0\), so \(x\) would be a minimizer and its
gradient would vanish. If \(H_0=H_1=0\), convexity and
Assumption~\ref{ass:h0h1-smoothness} make \(f\) affine, and the existence of a
minimizer makes it constant.

We first bound \(\norm{\nabla f(x)}\). If \(H_1>0\), consider the boundary
point
\[
 y:=x-r_1\frac{\nabla f(x)}{\norm{\nabla f(x)}}.
\]
By construction, \(\norm{y-x}=r_1\), so
Assumption~\ref{ass:h0h1-smoothness} applies to the pair \(x,y\). Since
\(\F(y)\ge0\),
\[
 \begin{aligned}
 0\le\F(y)
 &\overset{\circledOne}{\le}
 \F(x)-r_1\norm{\nabla f(x)}
 +(H_0+H_1\F(x))r_1^2,\\
 \frac{\norm{\nabla f(x)}}{r_1}
 &\overset{\circledTwo}{\le}
 H_0+H_1\F(x)+\frac{\F(x)}{r_1^2},
 \end{aligned}
\]
where in \(\circledOne\), we applied
Assumption~\ref{ass:h0h1-smoothness} at \(x\) with endpoint \(y\); in
\(\circledTwo\), we rearranged the first line and divided by \(r_1^2\).
Moreover,
\[
 \frac{\F(x)}{r_1^2}
 =(C-1)H_1\F(x)
 \le(C-1)\bigl(H_0+H_1\F(x)\bigr),
\]
where we used the definition of \(r_1\) in
Assumption~\ref{ass:h0h1-smoothness}, the definition of \(C\) in
Lemma~\ref{lem:full-gradient-self-bound}, and \(H_0\ge0\). Therefore,
\begin{equation}
\label{eq:full-gradient-radius-bound}
 \frac{\norm{\nabla f(x)}}{r_1}
 \le C\bigl(H_0+H_1\F(x)\bigr).
\end{equation}

We now define the actual trial point
\[
 x^+:=x-\frac{\nabla f(x)}{C(H_0+H_1\F(x))}.
\]
If \(H_1>0\), then \eqref{eq:full-gradient-radius-bound} gives
\[
 \norm{x^+-x}
 =\frac{\norm{\nabla f(x)}}{C(H_0+H_1\F(x))}
 \le r_1.
\]
If \(H_1=0\), then \(r_1=\infty\), so the same radius condition is
automatic. Thus Assumption~\ref{ass:h0h1-smoothness} applies to the pair
\(x,x^+\), and gives
\begin{equation}
\label{eq:full-gradient-descent-bound}
 \begin{aligned}
 0
 &\le\F(x^+)\\
 &\overset{\circledOne}{\le}
 \F(x)-\frac{\norm{\nabla f(x)}^2}
 {C(H_0+H_1\F(x))}
 +\frac{\norm{\nabla f(x)}^2}
 {C^2(H_0+H_1\F(x))}\\
 &\overset{\circledTwo}{\le}
 \F(x)-\frac{\norm{\nabla f(x)}^2}
 {2C(H_0+H_1\F(x))},
 \end{aligned}
\end{equation}
where in \(\circledOne\), we applied
Assumption~\ref{ass:h0h1-smoothness} to \(x,x^+\) and substituted the
definition of \(x^+\); in
\(\circledTwo\), we used \(C\ge2\). Rearranging the last line proves
\eqref{eq:full-gradient-self-bound}.
\end{proof}

\begin{proof}[Proof of Theorem~\ref{thm:full-gradient-main}]
\noindent\fbox{\strut\textit{1. Single-block reduction.}}
Set
\[
 n=1,\qquad d_1=d,\qquad U_1=I_d,\qquad p_1=1,
 \qquad H_{0,1}=H_0,\qquad H_{1,1}=H_1,
\]
Then
\(\nabla_1f(x)=U_1^\top\nabla f(x)=\nabla f(x)\), and the block index equals
one with probability one. Moreover, \(x+U_1h=x+h\) and
\(\norm{U_1h}=\norm{h}\), so
Assumption~\ref{ass:block-h0h1-smoothness} coincides with
Assumption~\ref{ass:h0h1-smoothness}. Finally,
\begin{equation}
\label{eq:full-K-reduction}
 K_0
 \overset{\circledOne}{=}
 \frac{H_{0,1}}{p_1^2}
 \overset{\circledTwo}{=}H_0,
 \qquad
 K_1
 \overset{\circledThree}{=}
 \frac{H_{1,1}}{p_1^2}
 \overset{\circledFour}{=}H_1,
\end{equation}
where in \(\circledOne\) and \(\circledThree\), the maximum in the definition
of \(K_0\) and \(K_1\) contains only one block; in \(\circledTwo\) and
\(\circledFour\), we used \(p_1=1\) and the chosen block constants.
\begin{currentrevision}
Here and throughout this subsection, \(c_k\) denotes the curvature-valued
parameter of Algorithm~\ref{alg:accelerated-block}. Substituting
\(K_0=H_0\) and \(K_1=H_1\) into \eqref{eq:block-c-rule} gives
\(c_k=4H_1\) in the first regime and
\(c_k=4H_1(1+\log(2\Delta_0/\varepsilon))\) in the second. Thus the
specialized parameter rules are
\end{currentrevision}
\begin{equation}
\label{eq:full-parameter-rules}
\begin{aligned}
 A_0&=\frac{R^2}{2\Delta_0},\qquad
 b_k=\frac{A_k}
 {1+\newrev{64R\sqrt{c_k}}+64\sqrt{H_0A_k}}.
\end{aligned}
\end{equation}

\noindent\fbox{\strut\textit{2. The projection is inactive.}}
Under the single-block choice, \eqref{eq:block-unscaled} is deterministic.
Together with \eqref{eq:block-initial-potential}, it gives
\(\Phi_k\le\Phi_0\le R^2\) for every \(k\). Hence
\begin{equation}
\label{eq:full-potential-consequences}
 \begin{aligned}
 \F(x_k)
 &\overset{\circledOne}{\le}\frac{\Phi_k}{A_k}
 \le\frac{R^2}{A_k},\\
 \norm{v_k-x^*}^2
 &\overset{\circledTwo}{\le}2\Phi_k
 \le2R^2,
 \qquad\text{and hence}\qquad
 \norm{v_k-x^*}\le\sqrt2R,
 \end{aligned}
\end{equation}
where in \(\circledOne\), we discarded the nonnegative term
\(\frac12\norm{v_k-x^*}^2\) from \(\Phi_k\) and divided by \(A_k>0\); in
\(\circledTwo\), we discarded the nonnegative term \(A_k\F(x_k)\) and
multiplied by two. The last inequality follows by taking square roots of the
nonnegative quantities.

Fix an iteration \(k\), and define only for this calculation
\begin{equation}
\label{eq:full-G-w-definitions}
 G_k:=1+\newrev{64R\sqrt{c_k}}+64\sqrt{H_0A_k},
 \qquad
 w_k:=\frac{A_kx_k+b_kx^*}{A_k+b_k}.
\end{equation}
Combining \eqref{eq:full-parameter-rules} and
\eqref{eq:full-G-w-definitions} gives \(b_k=A_k/G_k\). Convexity and
\eqref{eq:full-G-w-definitions} give
\begin{equation}
\label{eq:full-w-gap}
 \F(w_k)
 \le\frac{A_k}{A_k+b_k}\F(x_k)
 \le\F(x_k)
 \le\frac{R^2}{A_k}.
\end{equation}
Here the first inequality uses convexity and \(\F(x^*)=0\), the second uses
\(A_k/(A_k+b_k)\le1\) and \(\F(x_k)\ge0\), and the third is
\eqref{eq:full-potential-consequences}.
Furthermore,
\begin{equation}
\label{eq:full-y-w-distance}
 \norm{y_k-w_k}
 \overset{\circledOne}{=}
 \frac{b_k}{A_k+b_k}\norm{v_k-x^*}
 \overset{\circledTwo}{\le}
 \frac{\sqrt2R}{G_k},
\end{equation}
where in \(\circledOne\), we subtracted the definitions of \(y_k\) and
\(w_k\) from Line~\ref{line:block-y} and
\eqref{eq:full-G-w-definitions}; in \(\circledTwo\), we used
\eqref{eq:full-potential-consequences}, \(b_k/(A_k+b_k)\le b_k/A_k\), and
\(b_k/A_k=1/G_k\), which follows from
\eqref{eq:full-parameter-rules} and \eqref{eq:full-G-w-definitions}. Since
\newrev{\(c_k\ge4H_1\) by the two specialized branches stated before
\eqref{eq:full-parameter-rules}}, \eqref{eq:full-G-w-definitions} gives
\begin{equation}
\label{eq:full-G-lower}
 G_k\ge64\bigl(\sqrt{H_0A_k}+R\sqrt{H_1}\bigr)
 \ge64\sqrt{H_0A_k+H_1R^2}.
\end{equation}
The first inequality uses \newrev{\(c_k\ge4H_1\)} and discards the nonnegative constant
term in \(G_k\); the second follows by squaring both nonnegative sides.
In particular, \(G_k\ge64R\sqrt{H_1}\), so the displacement in
\eqref{eq:full-y-w-distance} satisfies
\begin{equation}
\label{eq:full-y-w-radius}
 \norm{y_k-w_k}
 \le\frac{\sqrt2}{64\sqrt{H_1}}
 <\frac{1}{(2\sqrt3+\sqrt6)\sqrt{H_1}}
 =r_1
\end{equation}
whenever \(H_1>0\); the strict inequality follows from
\(\sqrt2(2\sqrt3+\sqrt6)<64\). When \(H_1=0\), the radius condition is
absent.

The radius bound~\eqref{eq:full-y-w-radius} when \(H_1>0\), together with
\(r_1=\infty\) when \(H_1=0\), therefore allows us to apply
Assumption~\ref{ass:h0h1-smoothness} at \(w_k\) with displacement
\(y_k-w_k\). We first bound the linear term in that assumption:
\begin{equation}
\label{eq:full-linear-term}
 \begin{aligned}
 \inner{\nabla f(w_k)}{y_k-w_k}
 &\overset{\circledOne}{\le}
 \norm{\nabla f(w_k)}\norm{y_k-w_k}\\
 &\overset{\circledTwo}{\le}
 \sqrt{2C\bigl(H_0+H_1\F(w_k)\bigr)\F(w_k)}
 \norm{y_k-w_k}\\
 &\overset{\circledThree}{\le}
 \sqrt{2C\left(H_0+H_1\frac{R^2}{A_k}\right)
 \frac{R^2}{A_k}}
 \norm{y_k-w_k}\\
 &\overset{\circledFour}{\le}
 \sqrt{2C\left(H_0+H_1\frac{R^2}{A_k}\right)
 \frac{R^2}{A_k}}
 \frac{\sqrt2R}{G_k}\\
 &\overset{\circledFive}{=}
 \frac{2\sqrt C R^2}{A_k}
 \frac{\sqrt{H_0A_k+H_1R^2}}{G_k}\\
 &\overset{\circledSix}{<}
 \frac{R^2}{A_k}
 \frac{12\sqrt{H_0A_k+H_1R^2}}{G_k}.
 \end{aligned}
\end{equation}
where in \(\circledOne\), we used Cauchy--Schwarz; in \(\circledTwo\), we
applied Lemma~\ref{lem:full-gradient-self-bound} at \(w_k\); in
\(\circledThree\), we used \eqref{eq:full-w-gap} and
\(H_0,H_1\ge0\), so
\(\bigl(H_0+H_1\F(w_k)\bigr)\F(w_k)
=H_0\F(w_k)+H_1\F(w_k)^2
\le H_0R^2/A_k+H_1R^4/A_k^2
=\bigl(H_0+H_1R^2/A_k\bigr)R^2/A_k\); in \(\circledFour\), we used
\eqref{eq:full-y-w-distance}; in
\(\circledFive\), we used
\(H_0+H_1R^2/A_k=(H_0A_k+H_1R^2)/A_k\) and simplified
\(\sqrt{R^2/A_k^2}=R/A_k\); in \(\circledSix\), we used \(C<36\), which gives
\(2\sqrt C<12\).

The quadratic term satisfies
\begin{equation}
\label{eq:full-quadratic-term}
 \begin{aligned}
 &(H_0+H_1\F(w_k))\norm{y_k-w_k}^2\\
 &\quad\overset{\circledOne}{\le}
 \left(H_0+H_1\frac{R^2}{A_k}\right)\frac{2R^2}{G_k^2}
 \overset{\circledTwo}{=}
 \frac{R^2}{A_k}
 \frac{2(H_0A_k+H_1R^2)}{G_k^2},
 \end{aligned}
\end{equation}
where in \(\circledOne\), we used \eqref{eq:full-w-gap} to bound
\(\F(w_k)\) and \eqref{eq:full-y-w-distance} to bound
\(\norm{y_k-w_k}\); in \(\circledTwo\), we multiplied the first
factor by \(A_k/A_k\). Combining Assumption~\ref{ass:h0h1-smoothness} at \(w_k\) with
\eqref{eq:full-linear-term} and \eqref{eq:full-quadratic-term} yields
\begin{equation}
\label{eq:full-y-gap}
 \begin{aligned}
 \F(y_k)
 &\overset{\circledOne}{\le}
 \frac{R^2}{A_k}\left(
 1+\frac{12\sqrt{H_0A_k+H_1R^2}}{G_k}
 +\frac{2(H_0A_k+H_1R^2)}{G_k^2}
 \right)\\
 &\overset{\circledTwo}{<}
 \frac{2R^2}{A_k},
 \end{aligned}
\end{equation}
where in \(\circledOne\), we applied Assumption~\ref{ass:h0h1-smoothness} at
\(w_k\) with displacement \(y_k-w_k\), and then used the bounds on
\(\F(w_k)\), its linear term, and its quadratic term from
\eqref{eq:full-w-gap}, \eqref{eq:full-linear-term}, and
\eqref{eq:full-quadratic-term}, respectively; in
\(\circledTwo\), we used
\eqref{eq:full-G-lower} and
\(1+12/64+2/64^2<2\).

We now bound the unprojected auxiliary displacement:
\begin{equation}
\label{eq:full-unprojected-displacement}
 b_k^2\norm{\nabla f(y_k)}^2
 \overset{\circledOne}{\le}
 \frac{144R^2(H_0A_k+2H_1R^2)}{G_k^2}
 \overset{\circledTwo}{\le}
 \frac9{128}R^2
 <\frac{R^2}{4},
\end{equation}
where in \(\circledOne\), we applied
\eqref{eq:full-gradient-self-bound} at \(y_k\), used
\eqref{eq:full-y-gap}, substituted \(b_k=A_k/G_k\) from
\eqref{eq:full-parameter-rules} and \eqref{eq:full-G-w-definitions}, and used
\(C<36\); in
\(\circledTwo\), we used
\(H_0A_k+2H_1R^2\le2(H_0A_k+H_1R^2)\) and
the squared form of \eqref{eq:full-G-lower}. Finally,
\begin{equation}
\label{eq:full-projection-inactive}
 \begin{aligned}
 &\norm{v_k-b_k\nabla f(y_k)-x^0}\\
 &\quad\overset{\circledOne}{\le}
 \norm{v_k-x^*}+\norm{x^*-x^0}+b_k\norm{\nabla f(y_k)}
 \overset{\circledTwo}{<}3R,
 \end{aligned}
\end{equation}
where in \(\circledOne\), we used the triangle inequality; in
\(\circledTwo\), we used \eqref{eq:full-potential-consequences},
\(\norm{x^0-x^*}\le R\), and \eqref{eq:full-unprojected-displacement}, which
give the sharper coefficient \(\sqrt2+1+1/2<3\). Thus the point before
projection lies in \(B(x^0,3R)\), and the projection in
Line~\ref{line:block-v-update} leaves it unchanged.

\noindent\fbox{\strut\textit{3. Resulting full-gradient updates.}}
Substituting \(p_1=1\), \(U_1=I_d\), and the inactive projection into
Algorithm~\ref{alg:accelerated-block} gives
\begin{equation}
\label{eq:full-specialized-updates}
 \begin{aligned}
 y_k&=\frac{A_kx_k+b_kv_k}{A_k+b_k},\\
 v_{k+1}&=v_k-b_k\nabla f(y_k),
 &\Delta v_k&=-b_k\nabla f(y_k),\\
 a_k&=A_k+b_k-
 \newrev{\frac{c_kb_k^2\norm{\Delta v_k}^2}{A_k+b_k}},\\
 z_k&=y_k+\frac{b_k}{a_k}\Delta v_k,\\
 A_{k+1}&=\max\{A_0,a_k\},
 &x_{k+1}&=\frac{a_kz_k+(A_{k+1}-a_k)x^0}{A_{k+1}}.
 \end{aligned}
\end{equation}
Thus \eqref{eq:full-specialized-updates}, together with the specialized rules
for \(b_k\) and \(c_k\), is the full-gradient version of the unified method.
\begin{currentrevision}
Algorithm~\ref{alg:full-gradient-inner} uses the same curvature-valued
parameter \(c_k\), so these updates and parameter rules coincide directly.
\end{currentrevision}

\noindent\fbox{\strut\textit{4. Reduction of the upper bound.}}
Substituting \eqref{eq:full-K-reduction} into
\eqref{eq:block-total-horizon} gives
\begin{equation}
\label{eq:full-exact-complexity}
 \begin{aligned}
 N=N_1+N_2
 &\le16(1+128R\sqrt{H_1})
 \left(1+\log\frac{2\Delta_0}{\varepsilon}\right)\\
 &\quad+512\sqrt{H_0}\sqrt{\frac{4eR^2}{\varepsilon}}+2.
 \end{aligned}
\end{equation}
Equivalently, after suppressing numerical constants in the final upper bound,
\[
 N=O\left(
 R\sqrt{\frac{H_0}{\varepsilon}}
 +\left(1+R\sqrt{H_1}\right)
 \left(1+\log\frac{\Delta_0}{\varepsilon}\right)
 \right).
\]
Since the block choice is deterministic, the expectation in
Theorem~\ref{thm:common-block} disappears. Hence
\(f(x_N)-f^*\le\varepsilon\). Each iteration of
\eqref{eq:full-specialized-updates} evaluates one full gradient and no function
values.
\end{proof}

\subsection{Proof of Theorem~\ref{thm:coordinate-main}}
\label{app:coordinate-proof}

\paragraph{Proof outline.}
We first represent every coordinate as a singleton block and verify that
Assumption~\ref{ass:h0h1-smoothness} implies the corresponding block
assumption. We then specialize every line of
Algorithm~\ref{alg:accelerated-block} to uniform sampling and substitute the
resulting values of \(K_0\) and \(K_1\) into
Theorem~\ref{thm:common-block}.

\begin{proof}[Proof of Theorem~\ref{thm:coordinate-main}]
\noindent\fbox{\strut\textit{1. Singleton blocks and uniform sampling.}}
Set
\begin{equation}
\label{eq:uniform-singleton-blocks}
 n=d,\qquad d_i=1,\qquad U_i=e_i,\qquad p_i=\frac1d,
 \qquad H_{0,i}=H_0,\qquad H_{1,i}=H_1.
\end{equation}
for \(i=1,\ldots,d\). For a scalar displacement \(h\), we have
\(U_i h=he_i\), \(\norm{U_i h}=|h|\), and
\(\nabla_i f(x)=e_i^\top\nabla f(x)\). If \(|h|\le r_1\), then
\(\norm{he_i}=|h|\le r_1\), so
Assumption~\ref{ass:h0h1-smoothness} gives
\begin{equation}
\label{eq:uniform-assumption-reduction}
 \begin{aligned}
 f(x+U_i h)
 &\overset{\circledOne}{=}f(x+he_i)\\
 &\overset{\circledTwo}{\le}
 f(x)+h\nabla_i f(x)
 +\bigl(H_0+H_1\F(x)\bigr)|h|^2\\
 &\overset{\circledThree}{=}
 f(x)+\inner{\nabla_i f(x)}{h}
 +\bigl(H_{0,i}+H_{1,i}\F(x)\bigr)\norm h^2.
 \end{aligned}
\end{equation}
where in \(\circledOne\), we used \(U_i=e_i\) from
\eqref{eq:uniform-singleton-blocks}; in \(\circledTwo\), we applied
Assumption~\ref{ass:h0h1-smoothness} at \(x\) with displacement \(he_i\);
in \(\circledThree\), we used that the singleton-block displacement is a
scalar and substituted \(H_{0,i}=H_0\) and \(H_{1,i}=H_1\) from
\eqref{eq:uniform-singleton-blocks}. Thus
Assumption~\ref{ass:block-h0h1-smoothness} holds for these singleton blocks.
Its radius is the same \(r_1\), because \(H_{1,i}=H_1\).

The effective constants reduce as follows:
\begin{equation}
\label{eq:uniform-K-reduction}
 \begin{aligned}
 K_0
 &\overset{\circledOne}{=}
 \max_{1\le i\le d}\frac{H_0}{(1/d)^2}
 \overset{\circledTwo}{=}d^2H_0,\\
 K_1
 &\overset{\circledThree}{=}
 \max_{1\le i\le d}\frac{H_1}{(1/d)^2}
 \overset{\circledFour}{=}d^2H_1,
 \end{aligned}
\end{equation}
where in \(\circledOne\) and \(\circledThree\), we substituted the uniform
probabilities and the common coordinate constants into the definitions of
\(K_0\) and \(K_1\) in Theorem~\ref{thm:common-block}; in \(\circledTwo\) and
\(\circledFour\), the quantity
inside each maximum is the same for every coordinate.
\begin{currentrevision}
Here \(c_k\) continues to denote the curvature-valued parameter of
Algorithm~\ref{alg:accelerated-block}. Substituting
\eqref{eq:uniform-K-reduction} into \eqref{eq:block-c-rule} gives
\(c_k=4d^2H_1\) in the first regime and
\(c_k=4d^2H_1(1+\log(2\Delta_0/\varepsilon))\) in the second. Therefore,
\end{currentrevision}
\begin{equation}
\label{eq:uniform-parameter-rules}
\begin{aligned}
 A_0&=\frac{R^2}{2\Delta_0},\qquad
 b_k=\frac{A_k}
 {1+\newrev{64R\sqrt{c_k}}+64d\sqrt{H_0A_k}}.
\end{aligned}
\end{equation}
where we used \(\sqrt{K_0}=d\sqrt{H_0}\) and
\(\sqrt{K_1}=d\sqrt{H_1}\), which follow from
\eqref{eq:uniform-K-reduction}.

\noindent\fbox{\strut\textit{2. Resulting uniform-coordinate updates.}}
Let \(i_k\) be the index sampled in
Line~\ref{line:block-sample}. Then
\(\Pr(i_k=i)=1/d\), \(p_{i_k}=1/d\), and \(1/p_{i_k}=d\). Define the selected
component of the projected displacement by
\begin{equation}
\label{eq:uniform-displacement}
 D_{k,i_k}:=e_{i_k}^\top(v_{k+1}-v_k)
 =e_{i_k}^\top\Delta v_k.
\end{equation}
Since the selected block is a singleton,
\((\Delta v_k)^{(i_k)}=D_{k,i_k}\) and
\(U_{i_k}(\Delta v_k)^{(i_k)}=D_{k,i_k}e_{i_k}\). Substituting these
identities and \(1/p_{i_k}=d\) into
Algorithm~\ref{alg:accelerated-block} gives
\begin{equation}
\label{eq:uniform-specialized-updates}
 \begin{aligned}
 y_k&=\frac{A_kx_k+b_kv_k}{A_k+b_k},\\
 v_{k+1}
 &=P_{B(x^0,3R)}\!\left(
 v_k-db_k\nabla_{i_k}f(y_k)e_{i_k}
 \right),\\
 a_k&=A_k+b_k-
 \newrev{\frac{c_kb_k^2D_{k,i_k}^2}{A_k+b_k}},\\
 z_k
 &\overset{\circledOne}{=}
 y_k+\frac{b_k}{(1/d)a_k}
 U_{i_k}(\Delta v_k)^{(i_k)}\\
 &\overset{\circledTwo}{=}
 y_k+\frac{b_k}{(1/d)a_k}D_{k,i_k}e_{i_k}
 \overset{\circledThree}{=}
 y_k+\frac{db_k}{a_k}D_{k,i_k}e_{i_k},\\
 A_{k+1}&=\max\{A_0,a_k\},\\
 x_{k+1}&=\frac{a_kz_k+(A_{k+1}-a_k)x^0}{A_{k+1}}.
 \end{aligned}
\end{equation}
The first two lines of \eqref{eq:uniform-specialized-updates} follow from
Lines~\ref{line:block-y} and~\ref{line:block-v-update}, using
\(U_{i_k}=e_{i_k}\) and \(1/p_{i_k}=d\). We retain the projection because,
unlike in the single-block specialization, no argument shows that the
unprojected coordinate update must lie in \(B(x^0,3R)\). For the
\(a_k\)-update, we used
\(\norm{(\Delta v_k)^{(i_k)}}^2=D_{k,i_k}^2\) from
\eqref{eq:uniform-displacement}; \newrev{the correction coefficient is
the curvature-valued \(c_k\) prescribed before
\eqref{eq:uniform-parameter-rules}.} In \(\circledOne\), we used
Line~\ref{line:block-z} and \(p_{i_k}=1/d\); in \(\circledTwo\), we used
\(U_{i_k}(\Delta v_k)^{(i_k)}=D_{k,i_k}e_{i_k}\); in
\(\circledThree\), we used \((1/d)^{-1}=d\). The updates of \(a_k\),
\(A_{k+1}\), and \(x_{k+1}\) are unchanged from
Lines~\ref{line:block-a}--\ref{line:block-x-update}. Together with
\eqref{eq:uniform-parameter-rules} and the corresponding rule for \(c_k\),
the update identities~\eqref{eq:uniform-specialized-updates} give the uniform-coordinate
specialization of Algorithm~\ref{alg:accelerated-block}.
\begin{currentrevision}
Algorithm~\ref{alg:coordinate-inner} uses the same curvature-valued parameter
\(c_k\), so these specialized updates coincide directly with that algorithm.
\end{currentrevision}

\noindent\fbox{\strut\textit{3. Reduction of the upper bound.}}
Substituting \eqref{eq:uniform-K-reduction} into
\eqref{eq:block-total-horizon} gives
\begin{equation}
\label{eq:uniform-exact-complexity}
 \begin{aligned}
 N=N_1+N_2
 &\le16(1+128dR\sqrt{H_1})
 \left(1+\log\frac{2\Delta_0}{\varepsilon}\right)\\
 &\quad+512d\sqrt{H_0}\sqrt{\frac{4eR^2}{\varepsilon}}+2.
 \end{aligned}
\end{equation}
Indeed, \(\sqrt{K_0}=d\sqrt{H_0}\) produces the second term in
\eqref{eq:uniform-exact-complexity}, while
\(\sqrt{K_1}=d\sqrt{H_1}\) produces its first term. Suppressing only numerical
constants gives
\[
 N=O\left(
 dR\sqrt{\frac{H_0}{\varepsilon}}
 +\left(1+dR\sqrt{H_1}\right)
 \left(1+\log\frac{\Delta_0}{\varepsilon}\right)
 \right).
\]
Theorem~\ref{thm:common-block} gives
\(\E[f(x_N)-f^*]\le\varepsilon\), where the expectation is over the
independent uniform indices \(i_0,\ldots,i_{N-1}\). Each iteration of
\eqref{eq:uniform-specialized-updates} evaluates one partial derivative and
no function values.
\end{proof}

\subsection{Proof of Theorem~\ref{thm:nonuniform-coordinate-main}}
\label{app:nonuniform-coordinate-proof}

\begin{remark}[Blocks with zero baseline curvature]
\label{rem:constant-block}
Suppose that a minimizer exists. If \(H_{0,\xi}=0\), then
\(f(x+U_\xi h)=f(x)\) for every \(x\in\mathbb R^d\) and
\(h\in\mathbb R^{d_\xi}\). Thus block \(\xi\) may be fixed at its initial
value. In particular, every singleton block with \(H_{0,i}=0\) is a constant
coordinate.
\end{remark}

\begin{proof}[Justification of Remark~\ref{rem:constant-block}]
Let \(z\) be a minimizer. Since \(f\) is differentiable,
\(\nabla f(z)=0\), and hence \(\nabla_\xi f(z)=0\); moreover,
\(\F(z)=0\). For every \(h\) with \(\norm h\le r_{1,\xi}\),
\begin{equation}
\label{eq:constant-block-local}
 f^*
 \overset{\circledOne}{\le}f(z+U_\xi h)
 \overset{\circledTwo}{\le}f(z)=f^*.
\end{equation}
where in \(\circledOne\), we used the minimality of \(z\); in
\(\circledTwo\), we applied
Assumption~\ref{ass:block-h0h1-smoothness} at \(z\) and used
\(H_{0,\xi}=0\), \(\nabla_\xi f(z)=0\), and \(\F(z)=0\). Therefore every
point in the admissible block neighborhood of \(z\) is also a minimizer.

Consider the set
\begin{equation}
\label{eq:constant-block-minimizers}
 \mathcal M_\xi:=
 \left\{h\in\mathbb R^{d_\xi}:f(z+U_\xi h)=f^*\right\}.
\end{equation}
It is nonempty because \(0\in\mathcal M_\xi\), and it is closed by the
continuity of \(f\). It is also open: if \(h_0\in\mathcal M_\xi\), then
\(z+U_\xi h_0\) is a minimizer, so the argument in
\eqref{eq:constant-block-local} applied at this point places a whole
neighborhood of \(h_0\) in \(\mathcal M_\xi\). Since
\(\mathbb R^{d_\xi}\) is connected, the nonempty set
\(\mathcal M_\xi\), which is both open and closed, must equal
\(\mathbb R^{d_\xi}\). Hence every point \(z+U_\xi h\) is a minimizer.

It remains to extend this invariance from a minimizer to an arbitrary point.
Fix \(x\in\mathbb R^d\), \(h\in\mathbb R^{d_\xi}\), and \(0<t<1\). Then
\begin{equation}
\label{eq:constant-block-convex-combination}
 x+U_\xi h
 \overset{\circledOne}{=}
 (1-t)\frac{x-tz}{1-t}
 +t\left(z+U_\xi\frac{h}{t}\right).
\end{equation}
where in \(\circledOne\), we expanded the right-hand side.
The second point on the right is a minimizer by
\(\mathcal M_\xi=\mathbb R^{d_\xi}\), established after
\eqref{eq:constant-block-minimizers}. Therefore convexity and
\eqref{eq:constant-block-convex-combination} give
\begin{equation}
\label{eq:constant-block-limit}
 f(x+U_\xi h)
 \overset{\circledTwo}{\le}
 (1-t)f\left(\frac{x-tz}{1-t}\right)+tf^*.
\end{equation}
where in \(\circledTwo\), we used convexity and
\(f(z+U_\xi h/t)=f^*\). As \(t\downarrow0\), the point
\((x-tz)/(1-t)\) converges to \(x\). Continuity of \(f\) in
\eqref{eq:constant-block-limit} therefore gives
\(f(x+U_\xi h)\le f(x)\). Applying the same inequality at
\(x+U_\xi h\) with displacement \(-h\) gives
\(f(x)\le f(x+U_\xi h)\). Thus \(f(x+U_\xi h)=f(x)\).
\end{proof}

\paragraph{Proof outline.}
We first remove the constant coordinates using
Remark~\ref{rem:constant-block}. On the remaining singleton blocks, we verify
the sampling probabilities separately for
\(S_{1/2}^{(1)}>0\) and \(S_{1/2}^{(1)}=0\). We then derive every coordinate
update from Algorithm~\ref{alg:accelerated-block} and substitute the resulting
bounds on \(K_0\) and \(K_1\) into Theorem~\ref{thm:common-block}.

\begin{proof}[Proof of Theorem~\ref{thm:nonuniform-coordinate-main}]
\noindent\fbox{\strut\textit{1. Active singleton blocks.}}
By Remark~\ref{rem:constant-block}, fixing every coordinate with
\(H_{0,i}=0\) at its initial value leaves the objective unchanged. Let
\(\mathcal I:=\{i:H_{0,i}>0\}\) be the set of remaining coordinates. If
\(\mathcal I\) is empty, the objective is constant and the initial point is
optimal. Otherwise, restricting \(x^0-x^*\) to \(\mathcal I\) cannot increase
its Euclidean norm because
\begin{equation}
\label{eq:nonuniform-reduced-distance}
 \sum_{i\in\mathcal I}(x_i^0-x_i^*)^2
 \overset{\circledOne}{\le}
 \sum_{i=1}^d(x_i^0-x_i^*)^2
 \overset{\circledTwo}{\le}R^2.
\end{equation}
where in \(\circledOne\), we removed nonnegative squared terms; in
\(\circledTwo\), we used the distance assumption in
Theorem~\ref{thm:common-block}. Thus the reduced problem has the same valid
upper bound \(R\). We now reindex the coordinates in \(\mathcal I\) by
\(i=1,\ldots,d\); after this reindexing, \(H_{0,i}>0\) for every retained
coordinate.

For the reduced problem, set \(n=d\), \(d_i=1\), \(U_i=e_i\), and use the
constants \(H_{0,i},H_{1,i}\) from
Assumption~\ref{ass:coordinate-h0h1-smoothness}. Then the block gradient is
\(U_i^\top\nabla f(x)=e_i^\top\nabla f(x)=\nabla_i f(x)\). For every scalar
\(h\) with \(|h|\le r_{1,i}\),
\begin{equation}
\label{eq:nonuniform-assumption-reduction}
 \begin{aligned}
 f(x+U_i h)
 &\overset{\circledOne}{=}f(x+he_i)\\
 &\overset{\circledTwo}{\le}
 f(x)+h\nabla_i f(x)
 +\bigl(H_{0,i}+H_{1,i}\F(x)\bigr)h^2.
 \end{aligned}
\end{equation}
where in \(\circledOne\), we used \(U_i=e_i\); in \(\circledTwo\), we
applied Assumption~\ref{ass:coordinate-h0h1-smoothness}. Since
\(\norm h=|h|\) for a singleton block,
\eqref{eq:nonuniform-assumption-reduction} is exactly
Assumption~\ref{ass:block-h0h1-smoothness}; the two assumptions also use the
same radius \(r_{1,i}\). Write
\begin{equation}
\label{eq:nonuniform-S-definitions}
 S_{1/2}^{(j)}:=\sum_{i=1}^d\sqrt{H_{j,i}},
 \qquad j=0,1.
\end{equation}

\noindent\fbox{\strut\textit{2. Sampling probabilities and effective
constants.}}
\par\noindent
We treat the cases \(S_{1/2}^{(1)}>0\) and \(S_{1/2}^{(1)}=0\) separately.
In the first case, we verify that the mixed probabilities are positive and
sum to one, and then bound both effective constants \(K_0\) and \(K_1\). In
the second case, all gap-dependent constants \(H_{1,i}\) vanish, so we use
the baseline-curvature probabilities and compute \(K_0\) exactly. After both
cases have been treated, we state the resulting parameter rules used by the
unified Algorithm~\ref{alg:accelerated-block}.

\noindent\fbox{\strut\textit{2.1. Mixed sampling when
\(S_{1/2}^{(1)}>0\).}}
Suppose first that \(S_{1/2}^{(1)}>0\), and choose
\begin{equation}
\label{eq:nonuniform-probabilities}
 p_i:=
 \frac12\frac{\sqrt{H_{0,i}}}{S_{1/2}^{(0)}}
 +\frac12\frac{\sqrt{H_{1,i}}}{S_{1/2}^{(1)}}.
\end{equation}
The sum \(S_{1/2}^{(0)}\) is positive because every retained coordinate has
\(H_{0,i}>0\). Hence each probability in
\eqref{eq:nonuniform-probabilities} is positive. Moreover,
\begin{equation}
\label{eq:nonuniform-probability-sum}
 \sum_{i=1}^d p_i
 \overset{\circledOne}{=}
 \frac{1}{2S_{1/2}^{(0)}}\sum_{i=1}^d\sqrt{H_{0,i}}
 +\frac{1}{2S_{1/2}^{(1)}}\sum_{i=1}^d\sqrt{H_{1,i}}
 \overset{\circledTwo}{=}1,
\end{equation}
where in \(\circledOne\), we summed \eqref{eq:nonuniform-probabilities}; in
\(\circledTwo\), we used \eqref{eq:nonuniform-S-definitions}.

For \(j=0,1\), retaining only the corresponding nonnegative term in
\eqref{eq:nonuniform-probabilities} gives
\begin{equation}
\label{eq:nonuniform-K-reduction}
 p_i
 \overset{\circledOne}{\ge}
 \frac{\sqrt{H_{j,i}}}{2S_{1/2}^{(j)}}
 \quad\Longrightarrow\quad
 \frac{\sqrt{H_{j,i}}}{p_i}
 \overset{\circledTwo}{\le}2S_{1/2}^{(j)}
 \quad\Longrightarrow\quad
 \sqrt{K_j}
 \overset{\circledThree}{\le}2S_{1/2}^{(j)},
\end{equation}
where in \(\circledOne\), we discarded the other nonnegative component of
\(p_i\); in \(\circledTwo\), we rearranged the positive quantities; in
\(\circledThree\), we used
\(\sqrt{K_j}=\max_i\sqrt{H_{j,i}}/p_i\) and took the maximum over \(i\).
Squaring the last inequality in \eqref{eq:nonuniform-K-reduction} gives
\begin{equation}
\label{eq:nonuniform-K-squared}
 K_0\le4\bigl(S_{1/2}^{(0)}\bigr)^2,
 \qquad
 K_1\le4\bigl(S_{1/2}^{(1)}\bigr)^2.
\end{equation}

\noindent\fbox{\strut\textit{2.2. Sampling when
\(S_{1/2}^{(1)}=0\).}}
If \(S_{1/2}^{(1)}=0\), then every \(H_{1,i}=0\), because every term
\(\sqrt{H_{1,i}}\) in \eqref{eq:nonuniform-S-definitions} is nonnegative.
In this case choose
\begin{equation}
\label{eq:nonuniform-probabilities-smooth}
 p_i:=\frac{\sqrt{H_{0,i}}}{S_{1/2}^{(0)}}.
\end{equation}
These probabilities are positive because \(H_{0,i}>0\). Their sum and the
effective constants satisfy
\begin{equation}
\label{eq:nonuniform-K-smooth}
 \begin{aligned}
 \sum_{i=1}^d p_i
 &\overset{\circledOne}{=}
 \frac{1}{S_{1/2}^{(0)}}\sum_{i=1}^d\sqrt{H_{0,i}}
 \overset{\circledTwo}{=}1,\\
 \sqrt{K_0}
 &\overset{\circledThree}{=}
 \max_i\frac{\sqrt{H_{0,i}}}{p_i}
 \overset{\circledFour}{=}S_{1/2}^{(0)},
 \qquad
 K_1\overset{\circledFive}{=}0.
 \end{aligned}
\end{equation}
where in \(\circledOne\), we summed
\eqref{eq:nonuniform-probabilities-smooth}; in \(\circledTwo\), we used
\eqref{eq:nonuniform-S-definitions}; in \(\circledThree\), we used the
definition of \(K_0\) from Theorem~\ref{thm:common-block}; in
\(\circledFour\), we substituted
\eqref{eq:nonuniform-probabilities-smooth}; in \(\circledFive\), we used
\(H_{1,i}=0\) for every retained coordinate.

\begin{currentrevision}
In both cases, \(c_k\) denotes the curvature-valued parameter of
Algorithm~\ref{alg:accelerated-block}. Its two values in
\eqref{eq:block-c-rule} are \(4K_1\) and
\(4K_1(1+\log(2\Delta_0/\varepsilon))\). The specialized parameter rules are
therefore
\end{currentrevision}
\begin{equation}
\label{eq:nonuniform-parameter-rules}
\begin{aligned}
 A_0&=\frac{R^2}{2\Delta_0},\qquad
 b_k=\frac{A_k}
 {1+\newrev{64R\sqrt{c_k}}+64\sqrt{K_0A_k}}.
\end{aligned}
\end{equation}

\noindent\fbox{\strut\textit{3. Resulting non-uniform coordinate updates.}}
Let \(i_k\) be sampled with \(\Pr(i_k=i)=p_i\), using
\eqref{eq:nonuniform-probabilities} when \(S_{1/2}^{(1)}>0\) and
\eqref{eq:nonuniform-probabilities-smooth} when \(S_{1/2}^{(1)}=0\).
Define
\begin{equation}
\label{eq:nonuniform-displacement}
 D_{k,i_k}:=e_{i_k}^\top(v_{k+1}-v_k)
 =e_{i_k}^\top\Delta v_k.
\end{equation}
Then \((\Delta v_k)^{(i_k)}=D_{k,i_k}\) and
\(U_{i_k}(\Delta v_k)^{(i_k)}=D_{k,i_k}e_{i_k}\). Substitution of these
identities into
Algorithm~\ref{alg:accelerated-block} gives
\begin{equation}
\label{eq:nonuniform-specialized-updates}
 \begin{aligned}
 y_k&=\frac{A_kx_k+b_kv_k}{A_k+b_k},\\
 v_{k+1}
 &=P_{B(x^0,3R)}\!\left(
 v_k-\frac{b_k}{p_{i_k}}\nabla_{i_k}f(y_k)e_{i_k}
 \right),\\
 a_k&=A_k+b_k-
 \newrev{\frac{c_kb_k^2D_{k,i_k}^2}{A_k+b_k}},\\
 z_k
 &\overset{\circledOne}{=}
 y_k+\frac{b_k}{p_{i_k}a_k}
 U_{i_k}(\Delta v_k)^{(i_k)}\\
 &\overset{\circledTwo}{=}
 y_k+\frac{b_k}{p_{i_k}a_k}D_{k,i_k}e_{i_k},\\
 A_{k+1}&=\max\{A_0,a_k\},\\
 x_{k+1}&=\frac{a_kz_k+(A_{k+1}-a_k)x^0}{A_{k+1}}.
 \end{aligned}
\end{equation}
The first two lines of \eqref{eq:nonuniform-specialized-updates} follow from
Lines~\ref{line:block-y} and~\ref{line:block-v-update} with
\(U_{i_k}=e_{i_k}\). We retain the projection because no argument shows that
the unprojected coordinate update must lie in \(B(x^0,3R)\). For the
\(a_k\)-update, we used
\(\norm{(\Delta v_k)^{(i_k)}}^2=D_{k,i_k}^2\) from
\eqref{eq:nonuniform-displacement}; \newrev{the correction coefficient is
the curvature-valued \(c_k\) prescribed before
\eqref{eq:nonuniform-parameter-rules}}. In \(\circledOne\), we used
Line~\ref{line:block-z}; in \(\circledTwo\), we used
\(U_{i_k}(\Delta v_k)^{(i_k)}=D_{k,i_k}e_{i_k}\). The updates of \(a_k\),
\(A_{k+1}\), and \(x_{k+1}\) are unchanged from
Lines~\ref{line:block-a}--\ref{line:block-x-update}. Together with
\eqref{eq:nonuniform-parameter-rules} and the corresponding rule for \(c_k\),
the update identities~\eqref{eq:nonuniform-specialized-updates} give the non-uniform
coordinate specialization of Algorithm~\ref{alg:accelerated-block}.
\begin{currentrevision}
Algorithm~\ref{alg:nonuniform-coordinate-inner} uses the same
curvature-valued parameter \(c_k\), so these specialized updates coincide
directly with that algorithm.
\end{currentrevision}

\noindent\fbox{\strut\textit{4. Reduction of the upper bound.}}
For \(S_{1/2}^{(1)}>0\), substitute
\eqref{eq:nonuniform-K-squared} into
\eqref{eq:block-total-horizon}. Since its right-hand side is nondecreasing in
\(\sqrt{K_0}\) and \(\sqrt{K_1}\), we obtain
\begin{equation}
\label{eq:nonuniform-exact-complexity}
 \begin{aligned}
 N=N_1+N_2
 &\le16(1+256R S_{1/2}^{(1)})
 \left(1+\log\frac{2\Delta_0}{\varepsilon}\right)\\
 &\quad+1024S_{1/2}^{(0)}
 \sqrt{\frac{4eR^2}{\varepsilon}}+2.
 \end{aligned}
\end{equation}
Here the first term follows from
\(\sqrt{K_1}\le2S_{1/2}^{(1)}\), and the second follows from
\(\sqrt{K_0}\le2S_{1/2}^{(0)}\), both from
\eqref{eq:nonuniform-K-reduction}.
The complexity bound~\eqref{eq:nonuniform-exact-complexity} remains valid when
\(S_{1/2}^{(1)}=0\), because
\eqref{eq:nonuniform-K-smooth} gives the sharper identities
\(\sqrt{K_0}=S_{1/2}^{(0)}\) and \(K_1=0\). Therefore,
\[
 N=O\left(
 \frac{R S_{1/2}^{(0)}}{\sqrt\varepsilon}
 +\left(1+R S_{1/2}^{(1)}\right)
 \left(1+\log\frac{\Delta_0}{\varepsilon}\right)
 \right).
\]
Theorem~\ref{thm:common-block} gives
\(\E[f(x_N)-f^*]\le\varepsilon\), where the expectation is over the
independent indices sampled from \(p\). Each iteration of
\eqref{eq:nonuniform-specialized-updates} evaluates one partial derivative
and no function values.
\end{proof}

\end{revision}

\begin{currentrevision}
\section{An Improved Full-Gradient Rate under
\texorpdfstring{\((L_0,L_1)\)}{(L0,L1)}-Smoothness}
\label{app:l0l1-full-gradient}

This section shows that the same full-gradient method as in
Algorithm~\ref{alg:full-gradient-inner} admits a sharper specialization to the
local first-order \((L_0,L_1)\)-smooth class from
Appendix~\ref{app:l0l1-inclusion}.  The only additional freedom is that the
query weight \(\widehat b_k\), which determines \(y_k\), may be reduced to the
update weight \(b_k\) after the gradient at \(y_k\) has been evaluated.  No
additional gradient or function-value query is needed.

\begin{mainresultbox}
\begin{proposition}[Full-gradient complexity under
\((L_0,L_1)\)-smoothness]
\label{prop:l0l1-full-gradient-rate}
Suppose that \(f:\mathbb R^d\to\mathbb R\) is differentiable and convex, has a
minimizer \(x^*\), and satisfies the local \((L_0,L_1)\)-model
\eqref{eq:local-l0l1-function-model} for some \(L_0>0\) and \(L_1\ge0\).  Let
\(\norm{x^0-x^*}\le R\), where \(R>0\), let
\(\Delta_0\ge f(x^0)-f^*\), and let
\(0<\varepsilon\le\Delta_0\), where \(\varepsilon\) denotes the desired
accuracy.  Set \(A_0=R^2/(2\Delta_0)\), \(c_k=0\), and choose the parameters
in Algorithm~\ref{alg:full-gradient-inner} as follows:
\(h:=\min\{1,(\log(2\Delta_0/\varepsilon)/(L_1^2R^2))^{1/3}\}\), and
\begin{equation}
\label{eq:l0l1-query-weight}
 \widehat b_k
 =
 \min\left\{
 \frac{1+\sqrt{1+32L_0A_k}}{16L_0},
 \frac{hA_k}{1-h}
 \right\},
 \qquad
 b_k
 =
 \min\left\{
 \widehat b_k,
 \frac{A_k+\widehat b_k}
 {4\widehat b_k\bigl(L_0+L_1\norm{\nabla f(y_k)}\bigr)}
 \right\}.
\end{equation}
Then Algorithm~\ref{alg:full-gradient-inner} returns a point \(x_N\)
satisfying \(f(x_N)-f^*\le\varepsilon\) after
\begin{equation}
\label{eq:l0l1-adaptive-complexity}
 \boxed{
 N
 =O\left(
 1+\sqrt{\frac{L_0R^2}{\varepsilon}}
 +(L_1R)^{2/3}
 \log^{2/3}\!\frac{2\Delta_0}{\varepsilon}
 \right),
 }
\end{equation}
iterations, using one full-gradient query and no function-value evaluations
per iteration.
\end{proposition}
\end{mainresultbox}

\begin{proof}[Proof of Proposition~\ref{prop:l0l1-full-gradient-rate}]
The argument is organized into six steps.  Step~1 derives both weight rules,
checks the radius in \eqref{eq:local-l0l1-function-model}, and establishes
descent for the resulting gradient step.  Here the query
weight \(\widehat b_k\) determines \(y_k\) before the gradient is known,
whereas the update weight \(b_k\) is chosen after observing the local
curvature through \(\nabla f(y_k)\).  Step~2 transfers the descent estimate to
the updates of \(x_k\) and \(v_k\), and Step~3 combines the resulting bounds
into a decreasing potential and proves the stopping guarantee.  Steps~4
and~5 count reduced and fully accepted weights, respectively.  Finally,
Step~6 chooses \(h\) to balance the two counts.

\noindent\fbox{\strut\textit{Proof setup.}}

For completeness, we first record every definition and update used below.
Write
\[
 \F(x):=f(x)-f^*,
 \qquad
 f^*:=f(x^*).
\]
When \(L_1=0\), the definition of \(h\) is understood by continuity as
\(h=1\); whenever \(h=1\), the candidate \(hA_k/(1-h)\) is understood as
\(+\infty\).  Since \(0<\varepsilon\le\Delta_0\),
\[
 \log\frac{2\Delta_0}{\varepsilon}\ge\log 2>0,
\]
and hence \(0<h\le1\).  Starting from
\[
 x_0=v_0=x^0,
 \qquad
 A_0=\frac{R^2}{2\Delta_0},
\]
one iteration consists of
\[
 y_k=\frac{A_kx_k+\widehat b_kv_k}{A_k+\widehat b_k},
 \qquad
 \nabla f(y_k),
\]
followed by
\begin{equation}
\label{eq:l0l1-specialized-updates}
 \begin{aligned}
 a_k&=A_k+b_k,
 &A_{k+1}&=A_k+b_k,\\
 v_{k+1}&=v_k-b_k\nabla f(y_k),
 &x_{k+1}&=
 \frac{A_kx_k+b_kv_{k+1}}{A_k+b_k}.
 \end{aligned}
\end{equation}
The method stops at the first index \(N\) for which
\(A_N\ge R^2/\varepsilon\).  These formulas also show that the only oracle
query in an iteration is the single gradient \(\nabla f(y_k)\): the accepted
weight \(b_k\) and every subsequent update reuse this gradient.

All denominators used below are positive.  Indeed,
\(A_0=R^2/(2\Delta_0)>0\).  If \(A_k>0\), the first candidate for
\(\widehat b_k\) in \eqref{eq:l0l1-query-weight} is positive; the second is
positive when \(h<1\) and is \(+\infty\) when \(h=1\).  Thus
\(\widehat b_k>0\).  Both candidates for \(b_k\) are then positive because
\(L_0>0\), so \(b_k>0\).  Finally,
\(A_{k+1}=A_k+b_k>A_k>0\).  Induction therefore proves
\[
 A_k>0,\qquad \widehat b_k>0,\qquad b_k>0
 \quad\text{for every }k.
\]

\paragraph{Step 1: Why the two weight rules give a valid local step.}
Fix an iteration \(k\).  We first derive the two restrictions imposed before
and after evaluating \(\nabla f(y_k)\).

\noindent\fbox{\strut\textit{1.1. Choice of the query weight.}}

The query weight must satisfy
\begin{equation}
\label{eq:l0l1-query-constraints}
 \frac{\widehat b_k^2}{A_k+\widehat b_k}
 \le\frac1{8L_0},
 \qquad
 \frac{\widehat b_k}{A_k+\widehat b_k}\le h.
\end{equation}
The first constraint limits the step associated with the baseline curvature
\(L_0\).  The second limits the fraction of the query point contributed by
\(v_k\); in Step~4, this cap will turn every reduced weight into a fixed
amount of potential decrease.

We now derive the rule for \(\widehat b_k\) from these requirements.  Since
\(A_k+\widehat b_k>0\), the first inequality in
\eqref{eq:l0l1-query-constraints} is equivalent to
\begin{align*}
 \frac{\widehat b_k^2}{A_k+\widehat b_k}
 \le\frac1{8L_0}
 &\quad\overset{\circledOne}{\Longleftrightarrow}\quad
 8L_0\widehat b_k^2\le A_k+\widehat b_k\\
 &\quad\overset{\circledTwo}{\Longleftrightarrow}\quad
 8L_0\widehat b_k^2-\widehat b_k-A_k\le0.
\end{align*}
In \(\circledOne\), we multiply by the positive quantity
\(8L_0(A_k+\widehat b_k)\); in \(\circledTwo\), we move
\(A_k+\widehat b_k\) to the left-hand side.
The corresponding quadratic equation has roots
\[
 \frac{1\pm\sqrt{1+32L_0A_k}}{16L_0}.
\]
The root with the minus sign is negative, while the quadratic has positive
leading coefficient.  Hence, among positive weights, the first constraint
holds exactly when
\[
 0<\widehat b_k
 \le
 \frac{1+\sqrt{1+32L_0A_k}}{16L_0}.
\]
This gives the first candidate in \eqref{eq:l0l1-query-weight}.  For \(h<1\),
the second constraint is similarly equivalent to
\[
 \frac{\widehat b_k}{A_k+\widehat b_k}\le h
 \quad\overset{\circledOne}{\Longleftrightarrow}\quad
 \widehat b_k\le hA_k+h\widehat b_k
 \quad\overset{\circledTwo}{\Longleftrightarrow}\quad
 \widehat b_k\le\frac{hA_k}{1-h},
\]
where in \(\circledOne\), we multiply by
\(A_k+\widehat b_k>0\); in \(\circledTwo\), we collect the terms containing
\(\widehat b_k\) and divide by \(1-h>0\).  This gives the second
candidate.  For \(h=1\), the second constraint is automatic because
\(A_k>0\) implies
\(\widehat b_k/(A_k+\widehat b_k)<1\).  Thus
the positive weights satisfying both constraints form the interval whose
right endpoint is the smaller candidate in
\eqref{eq:l0l1-query-weight}.  The stated rule chooses this right endpoint,
so it is the largest admissible query weight.  Consequently, at least one
constraint is active when \(h<1\), while the first one is active when \(h=1\).
Here and below, a constraint is called active when it holds with equality.

After evaluating \(\nabla f(y_k)\), we choose \(b_k\).  Its rule in
\eqref{eq:l0l1-query-weight} is obtained from two requirements, which we now
derive separately.

\noindent\fbox{\strut\textit{1.2. From the primal update to a convex combination.}}

The first requirement comes from the relation between the gradient step from
\(y_k\) and the actual primal update.  The definition of the query point gives
\[
 (A_k+\widehat b_k)y_k=A_kx_k+\widehat b_kv_k.
\]
Subtracting \(A_kx_k\) and multiplying by
\(b_k/\widehat b_k\) yields
\[
 b_kv_k
 =
 \frac{b_k}{\widehat b_k}
 \left(
 (A_k+\widehat b_k)y_k-A_kx_k
 \right).
\]
Using this relation in the primal update gives the complete chain
\begin{equation}
\label{eq:l0l1-convex-combination-identity}
\begin{aligned}
 (A_k+b_k)x_{k+1}
 &\overset{\circledOne}{=}
 A_kx_k+b_kv_{k+1}\\
 &\overset{\circledTwo}{=}
 A_kx_k+b_kv_k-b_k^2\nabla f(y_k)\\
 &\overset{\circledThree}{=}
 \left(1-\frac{b_k}{\widehat b_k}\right)A_kx_k
 +\frac{b_k}{\widehat b_k}(A_k+\widehat b_k)y_k
 -b_k^2\nabla f(y_k)\\
 &\overset{\circledFour}{=}
 \left(1-\frac{b_k}{\widehat b_k}\right)A_kx_k\\
 &\quad+
 \frac{b_k}{\widehat b_k}(A_k+\widehat b_k)
 \left(
 y_k-
 \frac{b_k\widehat b_k}{A_k+\widehat b_k}\nabla f(y_k)
 \right).
\end{aligned}
\end{equation}
In \(\circledOne\), we multiplied the definition of \(x_{k+1}\) by
\(A_k+b_k\); in \(\circledTwo\), we substituted
\(v_{k+1}=v_k-b_k\nabla f(y_k)\) and distributed \(b_k\); in
\(\circledThree\), we substituted the displayed expression for \(b_kv_k\)
and collected the two terms containing \(A_kx_k\); in \(\circledFour\), we
factored the \(y_k\)- and gradient terms using
\(b_k^2=(b_k/\widehat b_k)(A_k+\widehat b_k)
(b_k\widehat b_k/(A_k+\widehat b_k))\).

We now check when this identity becomes a convex combination after division
by \(A_k+b_k\).  Recall that this requires the resulting coefficients to be
nonnegative and to sum to one.  The two unnormalized coefficients are
\[
 \left(1-\frac{b_k}{\widehat b_k}\right)A_k
 \quad\text{and}\quad
 \frac{b_k}{\widehat b_k}(A_k+\widehat b_k).
\]
Their sum is
\begin{align*}
 &\left(1-\frac{b_k}{\widehat b_k}\right)A_k
 +\frac{b_k}{\widehat b_k}(A_k+\widehat b_k)\\
 &\quad\overset{\circledOne}{=}
 A_k-\frac{b_k}{\widehat b_k}A_k
 +\frac{b_k}{\widehat b_k}A_k+b_k
 \overset{\circledTwo}{=}A_k+b_k.
\end{align*}
In \(\circledOne\), we expanded both products; in \(\circledTwo\), the two
terms containing \((b_k/\widehat b_k)A_k\) cancel.  Thus division by
\(A_k+b_k\) makes the coefficients sum to one.  The second
coefficient is positive.  Since \(A_k>0\), the first coefficient is
nonnegative exactly when
\[
 1-\frac{b_k}{\widehat b_k}\ge0
 \quad\Longleftrightarrow\quad
 b_k\le\widehat b_k.
\]
Thus \(b_k\le\widehat b_k\) is precisely the condition that turns
\eqref{eq:l0l1-convex-combination-identity}, after division by
\(A_k+b_k\), into a convex combination of \(x_k\) and the displayed
gradient step from \(y_k\).  Convexity of \(f\) can then transfer a descent
estimate for this gradient step to \(x_{k+1}\).

\noindent\fbox{\strut\textit{1.3. Choice of the accepted weight.}}

The second requirement is that this gradient step lies inside the radius of
the local model and provides sufficient descent.  Its step size is
\[
 \frac{b_k\widehat b_k}{A_k+\widehat b_k},
\]
and the local quadratic coefficient at \(y_k\) is
\[
 L_0+L_1\norm{\nabla f(y_k)}.
\]
The single condition
\[
 \frac{b_k\widehat b_k}{A_k+\widehat b_k}
 \bigl(L_0+L_1\norm{\nabla f(y_k)}\bigr)
 \le\frac14.
\]
serves both purposes.  Step~1.4 shows that it places the endpoint within
distance \(1/(4L_1)\le1/L_1\) of \(y_k\) when \(L_1>0\); for \(L_1=0\), the
radius is infinite.  Step~1.5 can therefore apply the local model
\eqref{eq:local-l0l1-function-model} directly.  Its quadratic term consumes
at most one eighth of the linear decrease, so
\begin{align*}
 &f\left(
 y_k-\frac{b_k\widehat b_k}{A_k+\widehat b_k}\nabla f(y_k)
 \right)\\
 &\quad\overset{\circledOne}{\le}
 f(y_k)
 -\frac{b_k\widehat b_k}{A_k+\widehat b_k}
 \norm{\nabla f(y_k)}^2\\
 &\qquad+
 \frac18\frac{b_k\widehat b_k}{A_k+\widehat b_k}
 \norm{\nabla f(y_k)}^2\\
 &\quad\overset{\circledTwo}{=}
 f(y_k)
 -\frac{7b_k\widehat b_k}{8(A_k+\widehat b_k)}
 \norm{\nabla f(y_k)}^2\\
 &\quad\overset{\circledThree}{\le}
 f(y_k)
 -\frac{3b_k\widehat b_k}{4(A_k+\widehat b_k)}
 \norm{\nabla f(y_k)}^2.
\end{align*}
In \(\circledOne\), we use the local-model bound proved in Step~1.5; in
\(\circledTwo\), we combine \(-1+1/8=-7/8\); in \(\circledThree\), we use
\(7/8\ge3/4\).  We retain the weaker factor \(3/4\) because it is sufficient
for the potential estimate below.

We finally solve the sufficient condition for \(b_k\).  Since all
denominators and curvature terms are positive,
\begin{align*}
 &\frac{b_k\widehat b_k}{A_k+\widehat b_k}
 \bigl(L_0+L_1\norm{\nabla f(y_k)}\bigr)
 \le\frac14\\
 &\quad\overset{\circledOne}{\Longleftrightarrow}\quad
 b_k\widehat b_k
 \bigl(L_0+L_1\norm{\nabla f(y_k)}\bigr)
 \le\frac{A_k+\widehat b_k}{4}\\
 &\quad\overset{\circledTwo}{\Longleftrightarrow}\quad
 b_k\le
 \frac{A_k+\widehat b_k}
 {4\widehat b_k\bigl(L_0+L_1\norm{\nabla f(y_k)}\bigr)}.
\end{align*}
In \(\circledOne\), we multiply by \(A_k+\widehat b_k>0\); in
\(\circledTwo\), we divide by the positive quantity
\(\widehat b_k(L_0+L_1\norm{\nabla f(y_k)})\).
Together with the convex-combination requirement, the admissible weights
therefore satisfy
\[
 0<b_k\le
 \min\left\{
 \widehat b_k,\,
 \frac{A_k+\widehat b_k}
 {4\widehat b_k\bigl(L_0+L_1\norm{\nabla f(y_k)}\bigr)}
 \right\}.
\]
Because \(A_{k+1}=A_k+b_k\), a larger admissible \(b_k\) produces at least as
much growth of the weight \(A_k\).  We therefore choose the largest admissible
value, which is exactly the minimum in
\eqref{eq:l0l1-query-weight}.  Since both candidates are positive, this
choice has \(b_k>0\).  Since it does not exceed either candidate, it also
gives
\begin{equation}
\label{eq:l0l1-accepted-consequences}
 0<b_k\overset{\circledOne}{\le}\widehat b_k,
 \qquad
 \frac{b_k\widehat b_k}{A_k+\widehat b_k}
 \bigl(L_0+L_1\norm{\nabla f(y_k)}\bigr)
 \overset{\circledTwo}{\le}\frac14.
\end{equation}
In \(\circledOne\), the minimum defining \(b_k\) does not exceed
\(\widehat b_k\); in \(\circledTwo\), it does not exceed the
curvature-dependent candidate, and we multiply that bound by the positive
factor
\(\widehat b_k(L_0+L_1\norm{\nabla f(y_k)})/(A_k+\widehat b_k)\).

\noindent\fbox{\strut\textit{1.4. Verification of the local radius.}}

The local model~\eqref{eq:local-l0l1-function-model} may be applied at
\(y_k\) only when the gradient-step endpoint lies within distance \(1/L_1\).
When \(L_1=0\), this radius is infinite, so there is nothing to prove.  Suppose
that \(L_1>0\).  The distance from \(y_k\) to the endpoint satisfies
\begin{align*}
 &L_1\norm{
 -\frac{b_k\widehat b_k}{A_k+\widehat b_k}\nabla f(y_k)
 }\\
 &\quad\overset{\circledOne}{=}
 \frac{b_k\widehat b_k}{A_k+\widehat b_k}
 L_1\norm{\nabla f(y_k)}\\
 &\quad\overset{\circledTwo}{\le}
 \frac{b_k\widehat b_k}{A_k+\widehat b_k}
 \bigl(L_0+L_1\norm{\nabla f(y_k)}\bigr)
 \overset{\circledThree}{\le}\frac14.
\end{align*}
In \(\circledOne\), we use absolute homogeneity of the norm and positivity of
the weights.  In \(\circledTwo\), we use \(L_0>0\).  In \(\circledThree\), we
use \eqref{eq:l0l1-accepted-consequences}.  Dividing by \(L_1>0\) yields
\[
 \norm{
 -\frac{b_k\widehat b_k}{A_k+\widehat b_k}\nabla f(y_k)
 }
 \le\frac1{4L_1}
 \le\frac1{L_1}.
\]
Thus the pair
\[
 x=y_k,
 \qquad
 y=y_k-\frac{b_k\widehat b_k}{A_k+\widehat b_k}\nabla f(y_k)
\]
satisfies the radius condition in
\eqref{eq:local-l0l1-function-model} for every iteration.

\noindent\fbox{\strut\textit{1.5. Descent from the local model.}}

We now apply \eqref{eq:local-l0l1-function-model} to the two points verified
in Step~1.4.  Substitution gives the complete chain
\begin{align}
 &f\left(
 y_k-
 \frac{b_k\widehat b_k}{A_k+\widehat b_k}\nabla f(y_k)
 \right)
 \nonumber\\
 &\quad\overset{\circledOne}{\le}
 f(y_k)
 +\inner{\nabla f(y_k)}{
 -\frac{b_k\widehat b_k}{A_k+\widehat b_k}\nabla f(y_k)
 }
 \nonumber\\
 &\qquad+
 \frac{L_0+L_1\norm{\nabla f(y_k)}}{2}
 \norm{
 -\frac{b_k\widehat b_k}{A_k+\widehat b_k}\nabla f(y_k)
 }^2
 \nonumber\\
 &\quad\overset{\circledTwo}{=}
 f(y_k)
 -\frac{b_k\widehat b_k}{A_k+\widehat b_k}
 \norm{\nabla f(y_k)}^2
 \nonumber\\
 &\qquad+
 \frac12
 \left(\frac{b_k\widehat b_k}{A_k+\widehat b_k}\right)^2
 \bigl(L_0+L_1\norm{\nabla f(y_k)}\bigr)
 \norm{\nabla f(y_k)}^2
 \nonumber\\
 &\quad\overset{\circledThree}{=}
 f(y_k)
 -\frac{b_k\widehat b_k}{A_k+\widehat b_k}
 \left(
 1-
 \frac{b_k\widehat b_k}{2(A_k+\widehat b_k)}
 \bigl(L_0+L_1\norm{\nabla f(y_k)}\bigr)
 \right)
 \norm{\nabla f(y_k)}^2
 \nonumber\\
 &\quad\overset{\circledFour}{\le}
 f(y_k)
 -\frac{7b_k\widehat b_k}{8(A_k+\widehat b_k)}
 \norm{\nabla f(y_k)}^2
 \nonumber\\
 &\quad\overset{\circledFive}{\le}
 f(y_k)
 -\frac{3b_k\widehat b_k}{4(A_k+\widehat b_k)}
 \norm{\nabla f(y_k)}^2.
\label{eq:l0l1-local-descent}
\end{align}
In \(\circledOne\), we apply the local model, whose radius condition was
verified in Step~1.4.  In \(\circledTwo\), we evaluate the inner product and
use absolute homogeneity of the norm.  In \(\circledThree\), we factor the
linear decrease.  In \(\circledFour\),
\eqref{eq:l0l1-accepted-consequences} gives
\[
 \frac{b_k\widehat b_k}{2(A_k+\widehat b_k)}
 \bigl(L_0+L_1\norm{\nabla f(y_k)}\bigr)
 \le\frac18.
\]
Finally, \(\circledFive\) uses \(7/8\ge3/4\).  This proves the descent
estimate required in Step~2 using only differentiability and the local
first-order model.

\paragraph{Step 2: Primal and auxiliary estimates.}
The estimate in Step~1 concerns a gradient step from \(y_k\), whereas the
algorithm updates \(x_{k+1}\) through \(v_{k+1}\).  We now transfer the local
descent estimate to these two algorithmic updates.

\noindent\fbox{\strut\textit{2.1. Transfer to the actual primal update.}}

Step~1 derived the exact
identity \eqref{eq:l0l1-convex-combination-identity} connecting these points
and proved that its coefficients form a convex combination when
\(0<b_k\le\widehat b_k\).  Dividing that identity by \(A_k+b_k\) and applying
convexity of \(\F\) first gives
\begin{align}
 (A_k+b_k)\F(x_{k+1})
 &\overset{\circledOne}{\le}
 \left(1-\frac{b_k}{\widehat b_k}\right)A_k\F(x_k)
 \nonumber\\
 &\quad+
 \frac{b_k}{\widehat b_k}(A_k+\widehat b_k)
 \F\left(
 y_k-
 \frac{b_k\widehat b_k}{A_k+\widehat b_k}\nabla f(y_k)
 \right)
 \nonumber\\
 &\overset{\circledTwo}{\le}
 \left(1-\frac{b_k}{\widehat b_k}\right)A_k\F(x_k)
 +\frac{b_k}{\widehat b_k}(A_k+\widehat b_k)\F(y_k)
 \nonumber\\
 &\quad-
 \frac34b_k^2\norm{\nabla f(y_k)}^2.
\label{eq:l0l1-primal-before-coupling}
\end{align}
In \(\circledOne\), we divide
\eqref{eq:l0l1-convex-combination-identity} by \(A_k+b_k\) and apply
convexity of \(\F\); in \(\circledTwo\), we substitute
\eqref{eq:l0l1-local-descent} after subtracting \(f^*\) from both sides, and
the descent coefficient simplifies because
\((b_k/\widehat b_k)(A_k+\widehat b_k)\cdot
(3b_k\widehat b_k)/(4(A_k+\widehat b_k))=3b_k^2/4\).

\noindent\fbox{\strut\textit{2.2. Combining two convexity inequalities.}}

It remains to control the term containing \(\F(y_k)\).  Convexity at \(y_k\)
states that, for every comparison point \(u\),
\[
 f(u)\ge f(y_k)+\inner{\nabla f(y_k)}{u-y_k},
\]
or, equivalently,
\[
 f(y_k)-f(u)
 \le\inner{\nabla f(y_k)}{y_k-u}.
\]
Taking first \(u=x_k\) and then \(u=x^*\), and using
\(f(x^*)=f^*\), gives
\[
 \F(y_k)-\F(x_k)
 \overset{\circledOne}{\le}\inner{\nabla f(y_k)}{y_k-x_k},
 \qquad
 \F(y_k)
 \overset{\circledTwo}{\le}\inner{\nabla f(y_k)}{y_k-x^*}.
\]
Multiplying \(\circledOne\) by \(A_k\), \(\circledTwo\) by
\(\widehat b_k\), and adding gives
\begin{align*}
 &(A_k+\widehat b_k)\F(y_k)-A_k\F(x_k)\\
 &\quad\overset{\circledOne}{\le}
 \inner{\nabla f(y_k)}{
 A_k(y_k-x_k)+\widehat b_k(y_k-x^*)}\\
 &\quad\overset{\circledTwo}{=}
 \inner{\nabla f(y_k)}{
 (A_k+\widehat b_k)y_k-A_kx_k-\widehat b_kx^*}\\
 &\quad\overset{\circledThree}{=}
 \widehat b_k\inner{\nabla f(y_k)}{v_k-x^*}.
\end{align*}
In \(\circledOne\), we multiply the two preceding convexity inequalities by
\(A_k\) and \(\widehat b_k\) and add them; in \(\circledTwo\), we expand the
two weighted differences and collect terms; in \(\circledThree\), we use
\((A_k+\widehat b_k)y_k=A_kx_k+\widehat b_kv_k\).  Rearranging proves
\begin{equation}
\label{eq:l0l1-query-coupling}
 (A_k+\widehat b_k)\F(y_k)
 \le
 A_k\F(x_k)
 +\widehat b_k\inner{\nabla f(y_k)}{v_k-x^*}.
\end{equation}
Multiplying \eqref{eq:l0l1-query-coupling} by
\(b_k/\widehat b_k\) and substituting it into
\eqref{eq:l0l1-primal-before-coupling} yields
\begin{equation}
\label{eq:l0l1-primal-one-step}
 (A_k+b_k)\F(x_{k+1})
 \le
 A_k\F(x_k)
 +b_k\inner{\nabla f(y_k)}{v_k-x^*}
 -\frac34b_k^2\norm{\nabla f(y_k)}^2.
\end{equation}
Indeed, the two coefficients of \(A_k\F(x_k)\) add to
\[
 \left(1-\frac{b_k}{\widehat b_k}\right)A_k
 +\frac{b_k}{\widehat b_k}A_k
 =A_k,
\]
while the inner-product coefficient becomes
\((b_k/\widehat b_k)\widehat b_k=b_k\).

\noindent\fbox{\strut\textit{2.3. Auxiliary-distance update.}}

Finally, the auxiliary update gives
\[
 v_{k+1}-x^*
 =
 v_k-x^*-b_k\nabla f(y_k).
\]
For any vectors \(u,g\) and scalar \(b\),
\[
 \norm{u-bg}^2
 =
 \norm{u}^2-2b\inner{g}{u}+b^2\norm{g}^2.
\]
Applying this identity with
\(u=v_k-x^*\), \(g=\nabla f(y_k)\), and \(b=b_k\), and then dividing by two
yields
\begin{equation}
\label{eq:l0l1-auxiliary-one-step}
 \frac12\norm{v_{k+1}-x^*}^2
 =\frac12\norm{v_k-x^*}^2
 -b_k\inner{\nabla f(y_k)}{v_k-x^*}
 +\frac{b_k^2}{2}\norm{\nabla f(y_k)}^2.
\end{equation}

\paragraph{Step 3: Potential decrease and the stopping condition.}
The primal and auxiliary estimates contain the same inner product with
opposite signs.  We therefore combine them into a single potential.

\noindent\fbox{\strut\textit{3.1. One-step potential decrease.}}

Define the potential as the sum of the two quantities whose one-step
estimates contain these cancelling terms:
\begin{equation}
\label{eq:l0l1-adaptive-potential}
 \Phi_k:=A_k\F(x_k)+\frac12\norm{v_k-x^*}^2.
\end{equation}
Adding \eqref{eq:l0l1-primal-one-step} and
\eqref{eq:l0l1-auxiliary-one-step} cancels the inner-product terms.  Since
\(A_{k+1}=A_k+b_k\), the left-hand side becomes \(\Phi_{k+1}\), while the
two gradient-norm coefficients satisfy
\[
 -\frac34b_k^2+\frac12b_k^2=-\frac14b_k^2.
\]
Therefore,
\begin{equation}
\label{eq:l0l1-potential-decrease}
 \Phi_{k+1}
 \le
 \Phi_k-\frac14b_k^2\norm{\nabla f(y_k)}^2.
\end{equation}
This inequality has two roles.  It keeps the potential below its initial
value, which will yield the objective guarantee, and it assigns a definite
amount of potential decrease to every reduced iteration, which will be used
in Step~4.

\noindent\fbox{\strut\textit{3.2. Initialization and telescoping.}}

At initialization, \(x_0=v_0=x^0\).  Using
\(\F(x^0)\le\Delta_0\), \(\norm{x^0-x^*}\le R\), and
\(A_0=R^2/(2\Delta_0)\), we obtain
\begin{equation}
\label{eq:l0l1-initial-potential}
 \Phi_0
 \overset{\circledOne}{\le}
 \frac{R^2}{2\Delta_0}\Delta_0+\frac{R^2}{2}
 \overset{\circledTwo}{=}R^2.
\end{equation}
In \(\circledOne\), we use \(\F(x^0)\le\Delta_0\),
\(\norm{x^0-x^*}\le R\), and \(A_0=R^2/(2\Delta_0)\); in
\(\circledTwo\), we cancel \(\Delta_0\) and add the two terms.
Summing \eqref{eq:l0l1-potential-decrease} for
\(j=0,\ldots,k-1\) gives
\[
 \sum_{j=0}^{k-1}(\Phi_{j+1}-\Phi_j)
 \le
 -\frac14\sum_{j=0}^{k-1}
 b_j^2\norm{\nabla f(y_j)}^2.
\]
The sum on the left telescopes to \(\Phi_k-\Phi_0\).  Moving the
gradient-norm sum to the left therefore gives
\[
 \Phi_k
 +\frac14\sum_{j=0}^{k-1}
 b_j^2\norm{\nabla f(y_j)}^2
 \overset{\circledOne}{\le}\Phi_0
 \overset{\circledTwo}{\le}R^2.
\]
In \(\circledOne\), the sum on the left of the preceding display telescopes
to \(\Phi_k-\Phi_0\), after which we move the gradient-norm sum to the left;
in \(\circledTwo\), we use \eqref{eq:l0l1-initial-potential}.
Because \(x^*\) is a minimizer, \(f(x)\ge f(x^*)=f^*\) for every \(x\), so
\(\F(x)\ge0\).  Together with \(A_k>0\), this shows that \(\Phi_k\ge0\).
The gradient-norm sum is also nonnegative because every summand is a square.
Dropping either nonnegative term therefore proves
\begin{equation}
\label{eq:l0l1-potential-consequences}
 \Phi_k\overset{\circledOne}{\le} R^2,
 \qquad
 \sum_{j=0}^{k-1}b_j^2\norm{\nabla f(y_j)}^2
 \overset{\circledTwo}{\le}4R^2.
\end{equation}

\noindent\fbox{\strut\textit{3.3. From weight growth to the stopping guarantee.}}

Bound \(\circledOne\) in \eqref{eq:l0l1-potential-consequences} turns growth
of \(A_k\) into decay of the objective gap.
Indeed, the squared-distance term in \(\Phi_k\) is nonnegative, so
\begin{equation}
\label{eq:l0l1-gap-from-weight}
 \F(x_k)
 \overset{\circledOne}{\le}\frac{\Phi_k}{A_k}
 \overset{\circledTwo}{\le}\frac{R^2}{A_k}.
\end{equation}
In \(\circledOne\), we discard the nonnegative squared-distance term from
\(\Phi_k\) and divide by \(A_k>0\); in \(\circledTwo\), we use
\eqref{eq:l0l1-potential-consequences}.
Consequently, \(A_N\ge R^2/\varepsilon\) gives
\[
 \F(x_N)
 \overset{\circledOne}{\le}\frac{R^2}{A_N}
 \overset{\circledTwo}{\le}\varepsilon.
\]
In \(\circledOne\), we apply \eqref{eq:l0l1-gap-from-weight}; in
\(\circledTwo\), we use \(A_N\ge R^2/\varepsilon\) and positivity to take
reciprocals.
It remains only to count the iterations needed for \(A_k\) to reach this
threshold.

\paragraph{Step 4: Counting reduced query weights.}
We now consider iterations on which the curvature observed after the gradient
query forces \(b_k<\widehat b_k\).  Each such iteration consumes a fixed
amount of the finite descent budget in
\eqref{eq:l0l1-potential-consequences}.

\noindent\fbox{\strut\textit{4.1. What reduction says about the observed curvature.}}

Suppose that \(b_k<\widehat b_k\).  Then the curvature-dependent candidate in
the rule for \(b_k\) is strictly smaller than \(\widehat b_k\).  Hence
\[
 \frac{A_k+\widehat b_k}
 {4\widehat b_k\bigl(L_0+L_1\norm{\nabla f(y_k)}\bigr)}
 <\widehat b_k,
\]
and \(b_k\) equals this candidate.  Rearranging the strict inequality and
the equality gives, respectively,
\begin{align*}
 &\frac{A_k+\widehat b_k}
 {4\widehat b_k\bigl(L_0+L_1\norm{\nabla f(y_k)}\bigr)}
 <\widehat b_k\\
 &\quad\overset{\circledOne}{\Longleftrightarrow}\quad
 A_k+\widehat b_k
 <
 4\widehat b_k^2
 \bigl(L_0+L_1\norm{\nabla f(y_k)}\bigr)\\
 &\quad\overset{\circledTwo}{\Longleftrightarrow}\quad
 \frac{4\widehat b_k^2}{A_k+\widehat b_k}
 \bigl(L_0+L_1\norm{\nabla f(y_k)}\bigr)>1,
\end{align*}
and
\begin{align*}
 \frac{b_k\widehat b_k}{A_k+\widehat b_k}
 &\overset{\circledThree}{=}
 \frac{\widehat b_k}{A_k+\widehat b_k}
 \frac{A_k+\widehat b_k}
 {4\widehat b_k\bigl(L_0+L_1\norm{\nabla f(y_k)}\bigr)}\\
 &\overset{\circledFour}{=}
 \frac1{4\bigl(L_0+L_1\norm{\nabla f(y_k)}\bigr)}.
\end{align*}
In \(\circledOne\), we multiply by the positive denominator; in
\(\circledTwo\), we divide by \(A_k+\widehat b_k>0\); in
\(\circledThree\), we substitute the curvature-dependent value of \(b_k\);
in \(\circledFour\), we cancel \(A_k+\widehat b_k\) and \(\widehat b_k\).
We record these two consequences as
\begin{equation}
\label{eq:l0l1-reduced-identities}
 \frac{4\widehat b_k^2}{A_k+\widehat b_k}
 \bigl(L_0+L_1\norm{\nabla f(y_k)}\bigr)
 \overset{\circledOne}{>}1,
 \qquad
 \frac{b_k\widehat b_k}{A_k+\widehat b_k}
 \overset{\circledTwo}{=}
 \frac1{4\bigl(L_0+L_1\norm{\nabla f(y_k)}\bigr)}.
\end{equation}
Relation \(\circledOne\) says that the total observed curvature is large
enough to trigger a reduction, while \(\circledTwo\) records the resulting
effective step size.  We next show that the baseline part \(L_0\) cannot
cause this by itself.  The first constraint in
\eqref{eq:l0l1-query-constraints}, multiplied by \(4L_0>0\), gives
\[
 \frac{4L_0\widehat b_k^2}{A_k+\widehat b_k}\le\frac12.
\]
Subtracting this inequality from \(\circledOne\) in
\eqref{eq:l0l1-reduced-identities} isolates the \(L_1\)-dependent part:
\[
 \frac{4\widehat b_k^2}{A_k+\widehat b_k}
 \bigl(L_0+L_1\norm{\nabla f(y_k)}\bigr)
 -
 \frac{4L_0\widehat b_k^2}{A_k+\widehat b_k}
 >
 1-\frac12,
\]
which is \(\circledOne\) below:
\begin{equation}
\label{eq:l0l1-reduced-curvature}
 \frac{4L_1\widehat b_k^2\norm{\nabla f(y_k)}}
 {A_k+\widehat b_k}\overset{\circledOne}{>}\frac12,
 \qquad
 L_1\norm{\nabla f(y_k)}
 \overset{\circledTwo}{>}
 \frac{A_k+\widehat b_k}{8\widehat b_k^2}
 \overset{\circledThree}{\ge} L_0.
\end{equation}
In \(\circledOne\), we simplify the preceding subtraction; in
\(\circledTwo\), we multiply \(\circledOne\) by
\((A_k+\widehat b_k)/(4\widehat b_k^2)>0\); in \(\circledThree\), we rearrange
the first query constraint
\(\widehat b_k^2/(A_k+\widehat b_k)\le1/(8L_0)\).
Thus a reduced iteration is impossible when \(L_1=0\), because the left-hand
side of \eqref{eq:l0l1-reduced-curvature} would be zero.  When \(L_1>0\), the
last inequality gives
\[
 L_0\le L_1\norm{\nabla f(y_k)}.
\]
Adding \(L_1\norm{\nabla f(y_k)}\) to both sides yields
\[
 L_0+L_1\norm{\nabla f(y_k)}
 \le2L_1\norm{\nabla f(y_k)}.
\]
Using this bound in the second identity of
\eqref{eq:l0l1-reduced-identities} gives
\begin{equation}
\label{eq:l0l1-reduced-progress}
 \frac{b_k\widehat b_k}{A_k+\widehat b_k}
 \norm{\nabla f(y_k)}
 \ge\frac1{8L_1}.
\end{equation}
Indeed,
\[
 \frac{\norm{\nabla f(y_k)}}
 {4\bigl(L_0+L_1\norm{\nabla f(y_k)}\bigr)}
 \overset{\circledOne}{\ge}
 \frac{\norm{\nabla f(y_k)}}
 {8L_1\norm{\nabla f(y_k)}}
 \overset{\circledTwo}{=}\frac1{8L_1}.
\]
In \(\circledOne\), we use
\(L_0+L_1\norm{\nabla f(y_k)}\le2L_1\norm{\nabla f(y_k)}\); in
\(\circledTwo\), we cancel the positive factor
\(\norm{\nabla f(y_k)}\).

\noindent\fbox{\strut\textit{4.2. Energy spent by one reduced iteration.}}

It remains to remove the factor
\(\widehat b_k/(A_k+\widehat b_k)\) from the left-hand side of
\eqref{eq:l0l1-reduced-progress}.  The second query constraint gives
\[
 \frac{\widehat b_k}{A_k+\widehat b_k}\le h.
\]
Multiplying by the positive quantity
\(b_k\norm{\nabla f(y_k)}\) and combining with
\eqref{eq:l0l1-reduced-progress} yields
\[
 h b_k\norm{\nabla f(y_k)}
 \overset{\circledOne}{\ge}
 \frac{b_k\widehat b_k}{A_k+\widehat b_k}
 \norm{\nabla f(y_k)}
 \overset{\circledTwo}{\ge}\frac1{8L_1}.
\]
In \(\circledOne\), we multiply the second query constraint by the positive
quantity \(b_k\norm{\nabla f(y_k)}\); in \(\circledTwo\), we use
\eqref{eq:l0l1-reduced-progress}.
Dividing by \(h>0\) and squaring proves that every reduced iteration
contributes at least
\begin{equation}
\label{eq:l0l1-reduced-energy}
 b_k^2\norm{\nabla f(y_k)}^2
 \ge\frac1{64L_1^2h^2}
\end{equation}
to the descent budget in \eqref{eq:l0l1-potential-consequences}.  If there
are \(m\) reduced iterations among \(k=0,\ldots,N-1\), summing
\eqref{eq:l0l1-reduced-energy} only over those indices gives
\[
 \frac{m}{64L_1^2h^2}
 \le
 \sum_{\substack{0\le j<N:\\ b_j<\widehat b_j}}
 b_j^2\norm{\nabla f(y_j)}^2.
\]
Every summand is nonnegative, so this partial sum is no larger than the sum
over all iterations.  Applying
\eqref{eq:l0l1-potential-consequences} with \(k=N\) therefore gives
\[
 \frac{m}{64L_1^2h^2}
 \overset{\circledOne}{\le}
 \sum_{j=0}^{N-1}b_j^2\norm{\nabla f(y_j)}^2
 \overset{\circledTwo}{\le}4R^2.
\]
In \(\circledOne\), we enlarge the partial sum over reduced iterations to
all iterations; in \(\circledTwo\), we apply
\eqref{eq:l0l1-potential-consequences} with \(k=N\).

\noindent\fbox{\strut\textit{4.3. Number of reduced iterations.}}

Multiplying by \(64L_1^2h^2\) gives the count
\begin{equation}
\label{eq:l0l1-reduced-count}
 m\le256L_1^2R^2h^2.
\end{equation}

\paragraph{Step 5: Counting fully accepted query weights.}
The energy argument in Step~4 is only needed when the query weight is
reduced.  If \(b_k=\widehat b_k\), then
\(A_{k+1}=A_k+\widehat b_k\), and the maximal construction of
\(\widehat b_k\) ensures that at least one constraint in
\eqref{eq:l0l1-query-constraints} is active.  We use the active constraint to
obtain a direct lower bound on the growth of \(A_k\).

\noindent\fbox{\strut\textit{5.1. Baseline-curvature constraint is active.}}

If the first constraint is active, then
\(\widehat b_k^2/(A_k+\widehat b_k)=1/(8L_0)\), and
\begin{align}
 \sqrt{A_{k+1}}-\sqrt{A_k}
 &\overset{\circledOne}{=}\frac{\widehat b_k}
 {\sqrt{A_k+\widehat b_k}+\sqrt{A_k}}
 \nonumber\\
 &\overset{\circledTwo}{\ge}
 \frac{\widehat b_k}{2\sqrt{A_k+\widehat b_k}}
 \overset{\circledThree}{=}\frac1{4\sqrt{2L_0}}.
\label{eq:l0l1-curvature-active-growth}
\end{align}
In \(\circledOne\), we rationalize the difference of square roots and use
\(A_{k+1}=A_k+\widehat b_k\); in \(\circledTwo\), we use
\(\sqrt{A_k+\widehat b_k}+\sqrt{A_k}
\le2\sqrt{A_k+\widehat b_k}\); in \(\circledThree\), we rewrite the ratio
as \((1/2)\sqrt{\widehat b_k^2/(A_k+\widehat b_k)}\) and use the active
constraint \(\widehat b_k^2/(A_k+\widehat b_k)=1/(8L_0)\).
Thus every fully accepted iteration of this type increases
\(\sqrt{A_k}\) by at least \(1/(4\sqrt{2L_0})\).  All other iterations also
increase \(A_k\), so they also increase \(\sqrt{A_k}\) and cannot invalidate
this count.  Suppose there are \(m\) fully accepted iterations of this type.
If \(m=0\), the desired bound is automatic, so assume \(m\ge1\).
All but possibly the final iteration occur before the stopping threshold is
crossed.  Since \(\sqrt{A_0}\ge0\) and
\(\sqrt{A_k}<R/\sqrt{\varepsilon}\) before termination, summing
\eqref{eq:l0l1-curvature-active-growth} over the first \(m-1\) such
iterations gives
\[
 \frac{m-1}{4\sqrt{2L_0}}
 \overset{\circledOne}{\le}
 \sum_{k=0}^{N-2}
 \bigl(\sqrt{A_{k+1}}-\sqrt{A_k}\bigr)
 \overset{\circledTwo}{=}
 \sqrt{A_{N-1}}-\sqrt{A_0}
 \overset{\circledThree}{<}
 \frac{R}{\sqrt{\varepsilon}}.
\]
In \(\circledOne\), each of the first \(m-1\) iterations of this type
contributes at least the increment in
\eqref{eq:l0l1-curvature-active-growth}, while all remaining increments are
nonnegative; in \(\circledTwo\), the sum telescopes; in \(\circledThree\), we
use \(\sqrt{A_{N-1}}<R/\sqrt{\varepsilon}\) and \(\sqrt{A_0}\ge0\).
Therefore,
\[
 m
 <
 1+4\sqrt{\frac{2L_0R^2}{\varepsilon}},
\]
which is
\begin{equation}
\label{eq:l0l1-curvature-active-count}
 O\left(1+\sqrt{\frac{L_0R^2}{\varepsilon}}\right).
\end{equation}

\noindent\fbox{\strut\textit{5.2. Weight-fraction constraint is active.}}

If the second constraint is active, then \(h<1\) and
\[
 \frac{\widehat b_k}{A_k+\widehat b_k}=h,
 \qquad
 \frac{A_{k+1}}{A_k}
 \overset{\circledOne}{=}\frac{A_k+\widehat b_k}{A_k}
 \overset{\circledTwo}{=}\frac1{1-h}.
\]
In \(\circledOne\), we use
\(A_{k+1}=A_k+\widehat b_k\); in \(\circledTwo\), the active constraint
\(\widehat b_k/(A_k+\widehat b_k)=h\) rearranges to
\((1-h)\widehat b_k=hA_k\), and therefore
\(A_k+\widehat b_k=A_k/(1-h)\).  Every such iteration consequently increases
\(\log A_k\) by
\[
 \log A_{k+1}-\log A_k
 \overset{\circledOne}{=}-\log(1-h)
 \overset{\circledTwo}{\ge} h,
\]
where in \(\circledOne\), we substitute
\(A_{k+1}/A_k=1/(1-h)\); in \(\circledTwo\), we use
\(-\log(1-h)=\int_0^h(1-t)^{-1}dt\ge\int_0^hdt=h\).
All other iterations have \(A_{k+1}>A_k\), so they contribute a nonnegative
amount to \(\log A_{k+1}-\log A_k\).  Before termination, the total possible
logarithmic growth is
\[
 \log\frac{R^2/\varepsilon}{A_0}
 =\log\frac{2\Delta_0}{\varepsilon}.
\]
Thus, if there are \(m\) fully accepted iterations of this type, the case
\(m=0\) is automatic.  For \(m\ge1\), all but possibly the final
threshold-crossing iteration satisfy
\[
 (m-1)h
 \overset{\circledOne}{\le}
 \sum_{k=0}^{N-2}
 \bigl(\log A_{k+1}-\log A_k\bigr)
 \overset{\circledTwo}{=}
 \log\frac{A_{N-1}}{A_0}
 \overset{\circledThree}{<}
 \log\frac{2\Delta_0}{\varepsilon},
\]
where in \(\circledOne\), each of the first \(m-1\) iterations of this type
contributes at least \(h\), while all remaining logarithmic increments are
nonnegative; in \(\circledTwo\), the sum telescopes; in \(\circledThree\), we
use \(A_{N-1}<R^2/\varepsilon\) and
\(A_0=R^2/(2\Delta_0)\).  Thus
\[
 m\le
 1+\frac1h\log\frac{2\Delta_0}{\varepsilon}.
\]
Hence there are at most
\begin{equation}
\label{eq:l0l1-growth-active-count}
 1+\frac1h\log\frac{2\Delta_0}{\varepsilon}
\end{equation}
fully accepted iterations of this type.  If both constraints are active, the
iteration may be assigned to either count, so it is never counted twice.

\paragraph{Step 6: Balancing the two counts.}
The two \(h\)-dependent bounds move in opposite directions: Step~4 gives
\(256L_1^2R^2h^2\) reduced iterations, which decreases with \(h\), while
Step~5 gives \(1+h^{-1}\log(2\Delta_0/\varepsilon)\) fully accepted
iterations of the capped-growth type, which increases as \(h\) decreases.
The definition of \(h\) balances these two contributions.

\noindent\fbox{\strut\textit{6.1. Case \(h<1\).}}

Suppose first that \(h<1\).  Then
\[
 h^3
 =
 \frac{\log(2\Delta_0/\varepsilon)}{L_1^2R^2}.
\]
Substituting this identity into the reduced-iteration contribution gives
\[
 L_1^2R^2h^2
 \overset{\circledOne}{=}
 L_1^2R^2
 \left(
 \frac{\log(2\Delta_0/\varepsilon)}{L_1^2R^2}
 \right)^{2/3}
 \overset{\circledTwo}{=}
 (L_1R)^{2/3}
 \log^{2/3}\!\frac{2\Delta_0}{\varepsilon}.
\]
In \(\circledOne\), we substitute the displayed expression for \(h^3\) and
use \(h^2=(h^3)^{2/3}\); in \(\circledTwo\), we collect the powers of
\(L_1\), \(R\), and \(\log(2\Delta_0/\varepsilon)\).
The capped-growth contribution gives the same expression:
\[
 \frac1h\log\frac{2\Delta_0}{\varepsilon}
 \overset{\circledOne}{=}
 \left(
 \frac{L_1^2R^2}{\log(2\Delta_0/\varepsilon)}
 \right)^{1/3}
 \log\frac{2\Delta_0}{\varepsilon}
 \overset{\circledTwo}{=}
 (L_1R)^{2/3}
 \log^{2/3}\!\frac{2\Delta_0}{\varepsilon}.
\]
In \(\circledOne\), we take the reciprocal cube root of the expression for
\(h^3\); in \(\circledTwo\), we again collect the powers of the same three
positive factors.
Thus
\begin{equation}
\label{eq:l0l1-balance}
 L_1^2R^2h^2
 \overset{\circledOne}{=}\frac1h\log\frac{2\Delta_0}{\varepsilon}
 \overset{\circledTwo}{=}(L_1R)^{2/3}
 \log^{2/3}\!\frac{2\Delta_0}{\varepsilon}.
\end{equation}
Here \(\circledOne\) compares the two expressions just evaluated, while
\(\circledTwo\) substitutes their common value.
Every iteration is either reduced or fully accepted, and the fully accepted
iterations were divided into the two active-constraint cases in Step~5.
Thus the three counts in
\eqref{eq:l0l1-reduced-count},
\eqref{eq:l0l1-curvature-active-count}, and
\eqref{eq:l0l1-growth-active-count} first give
\[
 N
 =
 O\left(
 1+\sqrt{\frac{L_0R^2}{\varepsilon}}
 +L_1^2R^2h^2
 +\frac1h\log\frac{2\Delta_0}{\varepsilon}
 \right).
\]
Substituting the two identities in \eqref{eq:l0l1-balance} into this display
gives
\[
 N
 =
 O\left(
 1+\sqrt{\frac{L_0R^2}{\varepsilon}}
 +(L_1R)^{2/3}
 \log^{2/3}\!\frac{2\Delta_0}{\varepsilon}
 \right),
\]
which is \eqref{eq:l0l1-adaptive-complexity}.  Here and below, the
\(O(\cdot)\) notation hides only absolute numerical constants.

\noindent\fbox{\strut\textit{6.2. Case \(h=1\).}}

If \(h=1\), the second constraint in
\eqref{eq:l0l1-query-constraints} is never active because
\(\widehat b_k/(A_k+\widehat b_k)<1\).  Hence the capped-growth iterations
counted in \eqref{eq:l0l1-growth-active-count} do not occur.  If \(L_1>0\),
the definition of \(h\) gives
\[
 \left(
 \frac{\log(2\Delta_0/\varepsilon)}{L_1^2R^2}
 \right)^{1/3}\ge1,
 \qquad\overset{\circledOne}{\Longrightarrow}\qquad
 L_1^2R^2\le\log\frac{2\Delta_0}{\varepsilon}.
\]
In \(\circledOne\), we cube both nonnegative sides and multiply by
\(L_1^2R^2>0\).  Raising the resulting inequality to the power \(2/3\)
gives
\[
 (L_1^2R^2)^{2/3}
 \le
 \log^{2/3}\!\frac{2\Delta_0}{\varepsilon}.
\]
Multiplying both sides by
\((L_1^2R^2)^{1/3}=(L_1R)^{2/3}>0\) yields
\[
 L_1^2R^2
 \le
 (L_1R)^{2/3}
 \log^{2/3}\!\frac{2\Delta_0}{\varepsilon}.
\]
Since \(h=1\), this shows that the reduced-iteration count in
\eqref{eq:l0l1-reduced-count} is again bounded by the last term in
\eqref{eq:l0l1-adaptive-complexity}.  There are no capped-growth iterations,
and the remaining fully accepted iterations are bounded by
\eqref{eq:l0l1-curvature-active-count}.  Adding these two counts gives
\[
 N
 =
 O\left(
 1+\sqrt{\frac{L_0R^2}{\varepsilon}}
 +(L_1R)^{2/3}
 \log^{2/3}\!\frac{2\Delta_0}{\varepsilon}
 \right).
\]
When \(L_1=0\), a reduced iteration
would make the left-hand side of
\eqref{eq:l0l1-reduced-curvature} equal to zero while its right-hand side is
positive, which is impossible.  Thus only the first fully accepted case
remains, and \eqref{eq:l0l1-curvature-active-count} gives the required
bound.  This completes the iteration count and the proof.
\end{proof}

\paragraph{Comparison with recent work.}
Assuming a known bound \(\Delta_0\ge f(x^0)-f^*\), as in
Proposition~\ref{prop:l0l1-full-gradient-rate},
\orrev{\citet[Theorem~B.7(b)]{rudomanov2026descending}} independently obtain the
\((L_1R)^{2/3}\) power; moreover,
their lower bound \orrev{\citep[Theorem~B.3]{rudomanov2026descending}} shows
that this power is optimal up to logarithmic factors for arbitrary deterministic
first-order methods when \(f(x^0)-f^*\le4\varepsilon\)
\orrev{and \(L_0R^2\ge\varepsilon\)}. Their outer iteration count contains the same
\((L_1R)^{2/3}\log^{2/3}(2\Delta_0/\varepsilon)\) contribution as
Proposition~\ref{prop:l0l1-full-gradient-rate}.  However, each outer iteration
invokes ICEPICK, an inexact segment search implemented by bisection.  The
analyzed implementation uses no function values but makes up to \(m+2\)
full-gradient queries, where
\(m=O(\log(2+L_0R^2/\varepsilon+L_1R))\).  Although \(m\) grows only
logarithmically, these full-gradient evaluations are repeated at every outer
iteration and can therefore impose substantial computational overhead.  Thus,
compared with the exact
segment search of \citet{vankov2024l0l1}, their outer iteration bound improves
the gap-logarithm power, but their gradient-oracle bound retains an additional
multiplicative logarithm.  In contrast,
Algorithm~\ref{alg:full-gradient-inner} uses the single gradient queried at
\(y_k\) both to choose the accepted weight and to update the iterate, requires
no segment search, and achieves
\eqref{eq:l0l1-adaptive-complexity} directly in gradient-oracle complexity.
\end{currentrevision}

\begin{revision}
\section{Proof of the Lower Bounds for \texorpdfstring{\(H_1\)}{H1}}
\label{app:growth-lower-bounds}

\paragraph{Proof outline.}
Following the classical zero-chain template for oracle lower bounds
\citep{nesterov2018lectures,carmon2021thinking}, we proceed in five steps.
First, we explain the choice of the chain length
\(m\), construct the hard instance, and identify its minimizer. Second, we
prove convexity and verify both smoothness assumptions with the prescribed
constants \(H_0\) and \(H_1\). Third, we show that the initial gap lies
between \(3\varepsilon\) and \(4\varepsilon\), whereas every point with
\(x_m=0\) has gap at least \(5\varepsilon/2\). Fourth, we use the zero-chain
property to obtain the full-gradient lower bound. Fifth, we count how quickly
uniform coordinate queries reveal the chain and then state the resulting
\(\Omega\)-bounds. A concluding remark verifies that the certificates and
auxiliary operations used by the algorithms in the main text do not change
the argument.

\begin{proof}[Proof of Theorem~\ref{thm:growth-lower-bounds}]
Fix an integer \(d\ge1\), a radius \(R>0\), an accuracy
\(\varepsilon>0\), and a constant \(H_1>0\) such that
\(\kappa:=H_1R^2\ge1\). Set \(H_0:=\varepsilon/R^2\), and write
\(F(x):=f(x)-f^*\). We construct a convex function with
\(\norm{x^0-x^*}=R\) and
\(3\varepsilon\le F(x^0)\le4\varepsilon\), and then prove the
full-gradient and uniform-coordinate query lower bounds.

\noindent\fbox{\strut\textit{1. Construction and basic properties.}}
\par\noindent
We first explain the length of the chain and then define the hard instance.
\par\noindent\fbox{\strut\textit{1.1. Choice of the chain length.}}
\par\noindent
Define the chain length by
\[
 m=\min\left\{d,
 \left\lfloor
 \frac{\sqrt\kappa}{4[\log(8\kappa)]^{3/2}}
 \right\rfloor\right\}.
\]
We consider the nontrivial case \(m\ge1\).
Set \(x^0=0\), \(t=\log(8\kappa)\), \(\beta=R/\sqrt m\), and
\(\omega=t^2m/\kappa\).
Thus, among the positive integers not exceeding \(d\), \(m\) is the largest
one satisfying the following balance:
\begin{equation}
\label{eq:lower-chain-length-balance}
 m\le\frac{\sqrt\kappa}{4t^{3/2}}
 \overset{\circledOne}{\Longleftrightarrow}
 16m^2t^3\le\kappa
 \overset{\circledTwo}{\Longleftrightarrow}
 4m\omega\le\frac{1}{4t},
\end{equation}
where in \(\circledOne\), we squared the positive quantities and
rearranged; in \(\circledTwo\), we substituted \(\omega=t^2m/\kappa\).
The last inequality is the reason for this choice: in Step~3.2 it limits
the loss accumulated along the \(m\) links of the chain and guarantees that
every point with \(x_m=0\) still has objective gap larger than
\(\varepsilon\).

The balance relation~\eqref{eq:lower-chain-length-balance} also gives
\begin{equation}
\label{eq:lower-weight-bound}
 0<\omega=\frac{t^2m}{\kappa}
 \overset{\circledOne}{\le}\frac{1}{16mt}
 \overset{\circledTwo}{<}\frac14,
\end{equation}
where in \(\circledOne\), we used \(\kappa\ge16m^2t^3\); in
\(\circledTwo\), we used \(m\ge1\) and
\(t=\log(8\kappa)\ge\log8\).

\par\noindent\fbox{\strut\textit{1.2. Definition of the instance.}}
\par\noindent
Define
\begin{equation}
\label{eq:lower-chain-objective}
 \begin{aligned}
 \mathcal R(x)^2&=\omega(x_1-\beta)^2+\frac14\sum_{i=1}^{m-1}(x_i-x_{i+1})^2
                   +\sum_{i=m+1}^d x_i^2,\\
 f(x)&=\frac{H_0}{H_1}\bigl(\cosh(\sqrt{H_1}\,\mathcal R(x))-1\bigr).
 \end{aligned}
\end{equation}
Here \(H_1=\kappa/R^2>0\).
\par\noindent\fbox{\strut\textit{1.3. Minimizer and distance.}}
\par\noindent
Every term of \(\mathcal R(x)^2\) is nonnegative, and they all vanish precisely at
\(x^*=(\beta,\ldots,\beta,0,\ldots,0)\), where the first \(m\) entries
equal \(\beta\). Indeed, the first term fixes \(x_1=\beta\), the differences make the first
\(m\) entries equal, and the last sum fixes the remaining entries to zero.
Since \(\cosh s\ge1\), with equality only at \(s=0\), the minimizer is
unique and \(f^*=0\). Finally,
\begin{equation}
\label{eq:lower-minimizer-distance}
 \norm{x^0-x^*}^2
 \overset{\circledOne}{=}m\beta^2
 \overset{\circledTwo}{=}R^2,
\end{equation}
where in \(\circledOne\), we used \(x^0=0\) and the fact that the first
\(m\) entries of \(x^*\) equal \(\beta\); in \(\circledTwo\), we substituted
\(\beta=R/\sqrt m\).

\noindent\fbox{\strut\textit{2. Convexity and the smoothness assumptions.}}
\par\noindent
We verify this part in three stages: the quadratic chain is nonexpansive, the
radial profile has controlled Hessian, and the resulting Hessian bound implies
the two assumptions used in the paper.
\par\noindent\fbox{\strut\textit{2.1. Norm of the chain operator.}}
\par\noindent
Write \(\mathcal R(x)=\norm{\mathcal C(x-x^*)}\), with \(\mathcal C\) the linear map
given by the weighted differences and entries in
\eqref{eq:lower-chain-objective}.
For any vector \(z\), the inequality
\((z_i-z_{i+1})^2\le2z_i^2+2z_{i+1}^2\) follows by expanding
\((z_i+z_{i+1})^2\ge0\).
Suppose first that \(m>1\). Multiplying
\((z_i-z_{i+1})^2\le2z_i^2+2z_{i+1}^2\) by \(1/4\) and summing over the
chain gives
\begin{equation}
\label{eq:lower-chain-operator}
 \begin{aligned}
 \norm{\mathcal C z}^2
 &=\omega z_1^2+\frac14\sum_{i=1}^{m-1}(z_i-z_{i+1})^2
   +\sum_{i=m+1}^d z_i^2\\
 &\overset{\circledOne}{\le}
 (\omega+\tfrac12)z_1^2
 +\sum_{i=2}^{m-1}z_i^2+\frac12z_m^2
 +\sum_{i=m+1}^d z_i^2\\
 &\overset{\circledTwo}{\le}\norm z^2,
 \end{aligned}
\end{equation}
where in \(\circledOne\), each difference contributes at most
\(z_i^2/2+z_{i+1}^2/2\); in \(\circledTwo\), we used
\(\omega<1/4\) from \eqref{eq:lower-weight-bound}, so
\(\omega+1/2<1\), and compared
every coefficient with one. If \(m=1\), the difference sum is absent and
\(\norm{\mathcal C z}^2=\omega z_1^2+\sum_{i=2}^d z_i^2\le\norm z^2\)
using \(\omega<1/4\) from \eqref{eq:lower-weight-bound}. Thus
\(\norm{\mathcal C}\le1\) in both
cases.

\par\noindent\fbox{\strut\textit{2.2. Hessian of the radial profile.}}
\par\noindent
Define
\[
 \psi(z):=\frac{H_0}{H_1}
 \bigl(\cosh(\sqrt{H_1}\norm z)-1\bigr),
 \qquad \zeta:=\sqrt{H_1}\norm z.
\]
For \(z\ne0\), differentiating the radial profile gives
\begin{equation}
\label{eq:lower-radial-eigenvalues}
 \lambda_{\rm rad}(z)=H_0\cosh\zeta,
 \qquad
 \lambda_{\rm tan}(z)=H_0\frac{\sinh\zeta}{\zeta}.
\end{equation}
The first value is the second derivative in the radial direction. The second
is the first radial derivative divided by \(\norm z\), and it applies to every
direction orthogonal to \(z\). The even power series of \(\cosh\) shows that
\(\psi\) is analytic at \(z=0\) and that both eigenvalues have the continuous
value \(H_0\) there.

We next compare the two eigenvalues. For \(\zeta\ge0\), the function
\(\zeta\cosh\zeta-\sinh\zeta\) vanishes at zero and has derivative
\(\zeta\sinh\zeta\ge0\). Hence
\[
 0\le\sinh\zeta\le\zeta\cosh\zeta,
\]
so \eqref{eq:lower-radial-eigenvalues} implies
\(0\preceq\nabla^2 \psi(z)\preceq H_0\cosh(\zeta)I_d\).
Since \(f(x)=\psi(\mathcal C(x-x^*))\), the chain rule gives
\begin{equation}
\label{eq:lower-chain-hessian}
 \begin{aligned}
 0\preceq\nabla^2 f(x)
 &\overset{\circledOne}{=}
 \mathcal C^\top\nabla^2\psi(\mathcal C(x-x^*))\mathcal C\\
 &\overset{\circledTwo}{\preceq}
 H_0\cosh(\sqrt{H_1}\,\mathcal R(x))\mathcal C^\top\mathcal C\\
 &\overset{\circledThree}{\preceq}
 H_0\cosh(\sqrt{H_1}\,\mathcal R(x))I_d\\
 &\overset{\circledFour}{=}
 (H_0+H_1\F(x))I_d,
 \end{aligned}
\end{equation}
where in \(\circledOne\), we differentiated the composition with the linear
map \(\mathcal C\); in \(\circledTwo\), we used the Hessian bound for
\(\psi\); in
\(\circledThree\), we used \(\norm{\mathcal C}\le1\) from
\eqref{eq:lower-chain-operator}, which gives
\(\mathcal C^\top\mathcal C\preceq I_d\); in
\(\circledFour\), we used \(f^*=0\) and
\eqref{eq:lower-chain-objective}. The leftmost inequality proves convexity.

\par\noindent\fbox{\strut\textit{2.3. Verification of the local assumptions.}}
\par\noindent
We verify Assumption~\ref{ass:h0h1-smoothness} directly. Fix \(x\) and a
displacement \(h\) with \(\norm h\le r_1\). For \(0\le u\le1\),
\begin{equation}
\label{eq:lower-segment-radius}
 \begin{aligned}
 \mathcal R(x+uh)
 &\overset{\circledOne}{\le}\mathcal R(x)+u\norm{\mathcal C h}\\
 &\overset{\circledTwo}{\le}\mathcal R(x)+\norm h,
 \end{aligned}
\end{equation}
where in \(\circledOne\), we used the triangle inequality in the definition
\(\mathcal R(x)=\norm{\mathcal C(x-x^*)}\); in \(\circledTwo\), we used \(u\le1\)
and \(\norm{\mathcal C}\le1\) from \eqref{eq:lower-chain-operator}.

Set \(s:=\sqrt{H_1}\,\mathcal R(x)\) and
\(\tau:=\sqrt{H_1}\norm h\). Both are
nonnegative. The radius estimate~\eqref{eq:lower-segment-radius} gives
\begin{equation}
\label{eq:lower-cosh-segment}
 \begin{aligned}
 \cosh\bigl(\sqrt{H_1}\,\mathcal R(x+uh)\bigr)
 &\overset{\circledOne}{\le}\cosh(s+\tau)\\
 &\overset{\circledTwo}{\le}e^\tau\cosh s
 \overset{\circledThree}{\le}2\cosh s,
 \end{aligned}
\end{equation}
where in \(\circledOne\), we used that \(\cosh\) is nondecreasing on
\([0,\infty)\); in \(\circledTwo\), the addition formula and
\(\sinh s\le\cosh s\) give
\(\cosh(s+\tau)\le\cosh s(\cosh\tau+\sinh\tau)=e^\tau\cosh s\); in
\(\circledThree\), the radius condition gives
\(\tau\le1/(2\sqrt3+\sqrt6)<1/2<\log2\). Here the first strict inequality
uses \(2\sqrt3+\sqrt6>2\), while
\(\log2=\int_1^2u^{-1}du>1/2\).

Applying \eqref{eq:lower-chain-hessian} at every point \(x+uh\) and then
using \eqref{eq:lower-cosh-segment} gives
\begin{equation}
\label{eq:lower-segment-hessian}
 \nabla^2f(x+uh)
 \preceq2H_0\cosh(s)I_d
 \overset{\circledOne}{=}
 2\bigl(H_0+H_1\F(x)\bigr)I_d,
\end{equation}
where in \(\circledOne\), we used \eqref{eq:lower-chain-objective} and
\(f^*=0\).

Taylor's formula with integral remainder now yields
\begin{equation}
\label{eq:lower-local-assumption}
 \begin{aligned}
 &f(x+h)-f(x)-\inner{\nabla f(x)}{h}\\
 &\quad\overset{\circledOne}{=}
 \int_0^1(1-u)\,h^\top\nabla^2 f(x+uh)h\,du\\
 &\quad\overset{\circledTwo}{\le}
 2\bigl(H_0+H_1\F(x)\bigr)\norm h^2
 \int_0^1(1-u)\,du\\
 &\quad\overset{\circledThree}{=}
 \bigl(H_0+H_1\F(x)\bigr)\norm h^2,
 \end{aligned}
\end{equation}
where in \(\circledOne\), we used Taylor's formula; in \(\circledTwo\), we
applied \eqref{eq:lower-segment-hessian}; in \(\circledThree\), we used
\(\int_0^1(1-u)\,du=1/2\). This proves
Assumption~\ref{ass:h0h1-smoothness}. Taking \(h=\tau e_i\) gives
\(\norm h=|\tau|\) and
\(\inner{\nabla f(x)}{h}=\tau\nabla_i f(x)\), so the same inequality proves
Assumption~\ref{ass:coordinate-h0h1-smoothness} with
\(H_{0,i}=H_0\) and \(H_{1,i}=H_1\) for every coordinate.

\noindent\fbox{\strut\textit{3. Initial gap and the last coordinate.}}
\par\noindent
We first compute the gap at \(x^0\), and then show that accuracy
\(\varepsilon\) is impossible until the \(m\)-th coordinate becomes nonzero.
Since \(x^0=0\), we write \(\F(0)=\F(x^0)\) throughout this calculation.
\par\noindent\fbox{\strut\textit{3.1. Initial gap.}}
\par\noindent
At the origin, only the first term in \(\mathcal R(0)^2\) is nonzero. Hence
\begin{equation}
\label{eq:lower-initial-radius}
 \sqrt{H_1}\,\mathcal R(0)
 \overset{\circledOne}{=}\sqrt{H_1\omega}\,\beta
 \overset{\circledTwo}{=}t,
\end{equation}
where in \(\circledOne\), we used \(\mathcal R(0)^2=\omega\beta^2\); in
\(\circledTwo\), we substituted \(H_1=\kappa/R^2\),
\(\omega=t^2m/\kappa\), and \(\beta=R/\sqrt m\).

Since \(H_0/H_1=\varepsilon/\kappa\) and \(e^t=8\kappa\),
\begin{equation}
\label{eq:lower-initial-gap}
 \begin{aligned}
 \frac{\F(0)}{\varepsilon}
 &\overset{\circledOne}{=}\frac{\cosh t-1}{\kappa}\\
 &\overset{\circledTwo}{=}
 4-\frac1\kappa+\frac{1}{16\kappa^2}
 \overset{\circledThree}{\in}[3,4],
 \end{aligned}
\end{equation}
where in \(\circledOne\), we used \eqref{eq:lower-chain-objective},
\eqref{eq:lower-initial-radius}, and \(f^*=0\); in \(\circledTwo\), we used
\(\cosh t=(e^t+e^{-t})/2\) and \(e^t=8\kappa\); in
\(\circledThree\), we used \(\kappa\ge1\). The expression is at least
\(4-1=3\), while
\(-1/\kappa+1/(16\kappa^2)\le0\), so it is at most \(4\).
\par\noindent\fbox{\strut\textit{3.2. Gap when \(x_m=0\).}}
\par\noindent
If \(x_m=0\), telescoping gives
\[
 \beta=(\beta-x_1)+\sum_{i=1}^{m-1}(x_i-x_{i+1}).
\]
Applying Cauchy--Schwarz with weights \(\omega\) and \(1/4\) gives
\begin{equation}
\label{eq:lower-weighted-cauchy}
 \begin{aligned}
 \beta^2
 &\overset{\circledOne}{\le}
 \bigl(1/\omega+4(m-1)\bigr)
 \left(\omega(\beta-x_1)^2
 +\frac14\sum_{i=1}^{m-1}(x_i-x_{i+1})^2\right)\\
 &\overset{\circledTwo}{\le}
 \bigl(1/\omega+4(m-1)\bigr)\mathcal R(x)^2,
 \end{aligned}
\end{equation}
where in \(\circledOne\), the reciprocal weights are \(1/\omega\) for
\(\beta-x_1\) and \(4\) for each of the \(m-1\) differences; in
\(\circledTwo\), we used \eqref{eq:lower-chain-objective}, whose remaining
terms are nonnegative.

Set \(\delta:=4(m-1)\omega\). Using \eqref{eq:lower-weight-bound}, we obtain
\begin{equation}
\label{eq:lower-delta-bound}
 0\le\delta
 \overset{\circledOne}{\le}4m\omega
 \overset{\circledTwo}{\le}\frac{1}{4t},
\end{equation}
where in \(\circledOne\), we used \(m-1\le m\); in \(\circledTwo\), we
used \(\omega\le1/(16mt)\) from \eqref{eq:lower-weight-bound}.
Taking square roots in \eqref{eq:lower-weighted-cauchy} and using
\(\sqrt{H_1\omega}\,\beta=t\) gives
\begin{equation}
\label{eq:lower-terminal-radius}
 \begin{aligned}
 \sqrt{H_1}\,\mathcal R(x)
 &\overset{\circledOne}{\ge}
 \frac{t}{\sqrt{1+\delta}}\\
 &\overset{\circledTwo}{\ge}
 t\left(1-\frac{\delta}{2}\right)
 \overset{\circledThree}{\ge}t-\frac18,
 \end{aligned}
\end{equation}
where in \(\circledOne\), we divided the first line of
\eqref{eq:lower-weighted-cauchy} by its positive coefficient and used
\(\sqrt{H_1\omega}\,\beta=t\); in \(\circledTwo\), convexity of
\((1+s)^{-1/2}\) gives
\((1+\delta)^{-1/2}\ge1-\delta/2\); in \(\circledThree\), we used
\eqref{eq:lower-delta-bound}.

Since \(\cosh\) is nondecreasing on \([0,\infty)\),
\eqref{eq:lower-terminal-radius} implies
\begin{equation}
\label{eq:lower-terminal-cosh}
 \begin{aligned}
 \cosh\bigl(\sqrt{H_1}\,\mathcal R(x)\bigr)
 &\overset{\circledOne}{\ge}\cosh(t-\tfrac18)\\
 &\overset{\circledTwo}{\ge}\cosh t-\tfrac18\sinh t\\
 &\overset{\circledThree}{\ge}\frac78\cosh t
 \overset{\circledFour}{\ge}\frac72\kappa,
 \end{aligned}
\end{equation}
where in \(\circledOne\), we used
\eqref{eq:lower-terminal-radius}; in \(\circledTwo\), convexity of
\(\cosh\) gives its tangent lower bound at \(t\); in \(\circledThree\), we
used \(\sinh t\le\cosh t\); in \(\circledFour\), we used
\(\cosh t\ge e^t/2=4\kappa\).

Substitution into \eqref{eq:lower-chain-objective} now gives
\begin{equation}
\label{eq:lower-chain-last}
 x_m=0\quad\Longrightarrow\quad
 \frac{\F(x)}{\varepsilon}
 \overset{\circledOne}{\ge}
 \frac{(7/2)\kappa-1}{\kappa}
 \overset{\circledTwo}{\ge}\frac52,
\end{equation}
where in \(\circledOne\), we used \(H_0/H_1=\varepsilon/\kappa\) and
\eqref{eq:lower-terminal-cosh}; in \(\circledTwo\), we used
\(\kappa\ge1\).

\noindent\fbox{\strut\textit{4. Full-gradient lower bound.}}
\par\noindent
We now prove the full-gradient oracle claim. For vectors
\(u_0,\ldots,u_{k-1}\), their linear span is
\[
 \operatorname{span}\{u_0,\ldots,u_{k-1}\}
 :=\left\{\sum_{j=0}^{k-1}\alpha_j u_j:
 \alpha_0,\ldots,\alpha_{k-1}\in\mathbb R\right\}.
\]
Thus a full-gradient linear-span method chooses its \(k\)-th query point
\(y_k\) and its output after \(N\) queries so that
\begin{equation}
\label{eq:lower-full-gradient-span-rule}
 y_k\in x^0+\operatorname{span}\{g_0,\ldots,g_{k-1}\},
 \qquad
 x_{\rm out}\in x^0+\operatorname{span}\{g_0,\ldots,g_{N-1}\},
\end{equation}
where \(g_k:=\nabla f(y_k)\). We use the convention
\(\operatorname{span}\varnothing=\{0\}\), so the first query is made at
\(y_0=x^0\). This class contains gradient descent, momentum methods, and
standard accelerated gradient methods. We show that, on the constructed
instance, each full-gradient query reveals at most one additional coordinate
of the chain.

For \(j=0,\ldots,m\), let
\[
 E_j:=\operatorname{span}\{e_1,\ldots,e_j\},
 \qquad E_0:=\{0\}.
\]
The key zero-chain property is
\begin{equation}
\label{eq:lower-zero-chain}
 x\in E_j,\quad j<m
 \qquad\Longrightarrow\qquad
 \nabla f(x)\in E_{j+1}.
\end{equation}
To verify it, first differentiate \(\mathcal R(x)^2\) in
\eqref{eq:lower-chain-objective}. Every chain coordinate interacts only with
its two neighbors. Hence, when \(x_{j+1}=\cdots=x_d=0\), every partial
derivative of \(\mathcal R(x)^2\) beyond coordinate \(j+1\) is zero; outside the
chain, the derivative is simply \(2x_i=0\). The even power series of
\(\cosh\) also shows that \(f\) is a differentiable scalar function of
\(\mathcal R(x)^2\). By the chain rule, \(\nabla f(x)\) is therefore a
scalar multiple of \(\nabla(\mathcal R(x)^2)\). It has the same support
restriction, proving \eqref{eq:lower-zero-chain}.

Because \(x^0=0\), the affine spaces in
\eqref{eq:lower-full-gradient-span-rule} reduce to ordinary linear spans.
Combining the span rule with the zero-chain property
\eqref{eq:lower-zero-chain} gives the induction
\begin{equation}
\label{eq:lower-full-gradient-induction}
 \begin{aligned}
 y_0=x^0&\in E_0,
 &g_0&\in E_1,\\
 y_k\in\operatorname{span}\{g_0,\ldots,g_{k-1}\}
 &\subseteq E_k,
 &g_k&\in E_{k+1}.
 \end{aligned}
\end{equation}
The first row uses \(y_0=x^0=0\) and the zero-chain property
\eqref{eq:lower-zero-chain}. For the second row, the first inclusion follows
from \eqref{eq:lower-full-gradient-span-rule}; the induction hypothesis gives
\(g_0,\ldots,g_{k-1}\in E_k\), and the last inclusion follows from
\eqref{eq:lower-zero-chain}. The same induction
and the output rule in \eqref{eq:lower-full-gradient-span-rule} place
\(x_{\rm out}\) in \(E_N\). The argument is pathwise and therefore
remains valid when the scalar coefficients used by the method are random.
If \(N<m\), the output has \(x_m=0\), and
\eqref{eq:lower-chain-last} gives
\(\F(x_{\rm out})\ge(5/2)\varepsilon>\varepsilon\). Consequently at least
\(m\) full-gradient queries are necessary.

\noindent\fbox{\strut\textit{5. Uniform-coordinate lower bound.}}
\par\noindent
We next count how many partial-derivative queries are needed to reveal the
\(m\)-th coordinate. In this step, \(x_{\rm out}\) denotes the output after
\(Q\) scalar queries. We consider methods whose query points and output can
differ from \(x^0\) only in coordinates where a nonzero partial derivative
has already been observed. \orrev{If an optional initial sweep is used, it is
completed before the method forms an update or returns an output.} Apart from
an optional initial sweep through all
\(d\) coordinates at \(x^0\), the next queried coordinate is conditionally
uniform on \(\{1,\ldots,d\}\) given all information observed so far.
\par\noindent\fbox{\strut\textit{5.1. Number of revealed coordinates.}}
\par\noindent
After \(k\) scalar queries, let \(J_k\in\{0,\ldots,m\}\) be the largest chain
index at which a nonzero partial derivative has been observed, with \(J_k=0\)
if none has been observed. By the update rule just stated and \(x^0=0\), the
current point belongs to \(E_{J_k}\). The zero-chain property
\eqref{eq:lower-zero-chain} then shows that a nonzero partial derivative at
a new chain coordinate can occur only at coordinate \(J_k+1\). Thus
\(J_k\) can increase by at most one.

Let \(i_k\) be the coordinate selected by the next query, and let
\(\mathcal F_k\) be the sigma-algebra generated by all queries, oracle
answers, and internal randomness observed before this selection. Conditional
uniformity gives the following bounds, where \(\mathbb{I}_E\) denotes the
indicator of an event \(E\):
\begin{equation}
\label{eq:lower-coordinate-increment}
 \begin{aligned}
 J_{k+1}-J_k
 &\overset{\circledOne}{\le}
 \mathbb{I}_{\{i_k=J_k+1\}},\\
 \E[J_{k+1}-J_k\mid\mathcal F_k]
 &\overset{\circledTwo}{\le}\frac1d,
 \end{aligned}
\end{equation}
where in \(\circledOne\), we used the zero-chain property
\eqref{eq:lower-zero-chain}; in
\(\circledTwo\), we took the conditional expectation and used
\(\Pr(i_k=J_k+1\mid\mathcal F_k)\le1/d\).

Without an initial sweep, \(J_0=0\). Summing
\eqref{eq:lower-coordinate-increment} for \(k=0,\ldots,Q-1\) and using the
tower property gives
\begin{equation}
\label{eq:lower-coordinate-count}
 \E[J_Q]
 \overset{\circledOne}{=}
 \sum_{k=0}^{Q-1}\E[J_{k+1}-J_k]
 \overset{\circledTwo}{\le}\frac Qd,
\end{equation}
where in \(\circledOne\), the sum telescopes and \(J_0=0\); in
\(\circledTwo\), we applied the second line of
\eqref{eq:lower-coordinate-increment} and the tower property to each term.

An initial sweep at \(x^0=0\) costs \(d\) queries and reveals at most
coordinate one, so \(J_d\le1\). Applying the same argument to the remaining
\(Q-d\) queries gives
\begin{equation}
\label{eq:lower-coordinate-count-sweep}
 \E[J_Q]
 \overset{\circledOne}{\le}
 1+\frac{Q-d}{d}
 \overset{\circledTwo}{=}\frac Qd,
\end{equation}
where in \(\circledOne\), the sweep gives \(J_d\le1\) and
\eqref{eq:lower-coordinate-increment} bounds the expected increase during
the remaining \(Q-d\) queries by \((Q-d)/d\); in \(\circledTwo\), we
simplified the expression. If \(Q=d\), the sum over the remaining queries is
empty and the same formula holds. \orrev{If \(Q<d\), a method returning an
output cannot be inside the optional sweep by the method-class definition
above; hence it follows the no-sweep case, and
\eqref{eq:lower-coordinate-count} applies directly.}
Using a cached derivative from the sweep reveals no additional coordinate,
so it is already covered by this count.

\par\noindent\fbox{\strut\textit{5.2. From revealed coordinates to the gap.}}
\par\noindent
The coordinate-update rule stated at the beginning of Step~5 places the
output after \(Q\) queries in \(E_{J_Q}\). Hence, if \(J_Q<m\), the output
has \(x_m=0\).
By \eqref{eq:lower-chain-last}, nonnegativity of the gap, and
\(0\le J_Q\le m\), we have
\begin{equation}
\label{eq:lower-coordinate-gap}
 \F(x_{\rm out})
 \overset{\circledOne}{\ge}
 \frac52\varepsilon\,\mathbb{I}_{\{J_Q<m\}}
 \overset{\circledTwo}{\ge}
 \frac52\varepsilon\left(1-\frac{J_Q}m\right),
\end{equation}
where in \(\circledOne\), we used \eqref{eq:lower-chain-last} on the indicated
event and nonnegativity of \(\F\) on its complement; in \(\circledTwo\), if
\(J_Q<m\), then
\(1-J_Q/m\le1\), while for \(J_Q=m\) both sides after the
common factor are zero.

Taking expectations in \eqref{eq:lower-coordinate-gap} and using
\eqref{eq:lower-coordinate-count}, including its sweep variant, gives
\begin{equation}
\label{eq:lower-coordinate-query-bound}
 \E[\F(x_{\rm out})]
 \overset{\circledOne}{\ge}
 \frac52\varepsilon
 \left(1-\frac{\E[J_Q]}m\right)
 \overset{\circledTwo}{\ge}
 \frac52\varepsilon\left(1-\frac{Q}{dm}\right),
\end{equation}
where in \(\circledOne\), we used linearity of expectation; in
\(\circledTwo\), we used \(\E[J_Q]\le Q/d\), proved in
\eqref{eq:lower-coordinate-count} without a sweep and in
\eqref{eq:lower-coordinate-count-sweep} with a sweep.
If \(\E[\F(x_{\rm out})]\le\varepsilon\), then
\eqref{eq:lower-coordinate-query-bound} implies
\begin{equation}
\label{eq:lower-coordinate-final}
 1\ge\frac52\left(1-\frac{Q}{dm}\right)
 \overset{\circledOne}{\Longrightarrow}
 \frac{Q}{dm}\ge\frac35
 \overset{\circledTwo}{\Longrightarrow}
 Q\ge\frac35dm,
\end{equation}
where in \(\circledOne\), we divided by \(5/2\) and rearranged the
inequality; in \(\circledTwo\), we multiplied by \(dm>0\). This proves the
coordinate lower bound.

\par\noindent\fbox{\strut\textit{5.3. Final complexity lower bounds.}}
\par\noindent
For this construction,
\begin{equation}
\label{eq:lower-final-scaling}
 R\sqrt{\frac{H_0}{\varepsilon}}
 \overset{\circledOne}{=}1,
 \qquad
 \kappa
 \overset{\circledTwo}{=}H_1R^2,
\end{equation}
where in \(\circledOne\), we used \(H_0=\varepsilon/R^2\); in
\(\circledTwo\), we used \(\kappa=H_1R^2\). When the dimension does not
truncate the chain, the quantity
\(\sqrt\kappa/(4[\log(8\kappa)]^{3/2})\) is at least two for sufficiently
large \(\kappa\). Its integer part is then at least one half of the quantity
itself. Thus the definition of \(m\), the full-gradient bound in Step~4, and
the coordinate bound
\eqref{eq:lower-coordinate-final} give, respectively,
\begin{equation}
\label{eq:lower-explicit-final-orders}
 N\ge m
 \overset{\circledOne}{\ge}
 \frac{\sqrt\kappa}{8[\log(8\kappa)]^{3/2}},
 \qquad
 Q\ge\frac35dm
 \overset{\circledTwo}{\ge}
 \frac{3d\sqrt\kappa}{40[\log(8\kappa)]^{3/2}},
\end{equation}
where in \(\circledOne\), we used \(\lfloor z\rfloor\ge z/2\) for
\(z\ge2\), with
\(z=\sqrt\kappa/(4[\log(8\kappa)]^{3/2})\); in \(\circledTwo\), we
multiplied the same inequality by \(3d/5\).
Consequently, when the dimension does not truncate the chain, the standard
lower-bound notation gives
\begin{equation}
\label{eq:lower-omega-orders}
 N=\Omega\!\left(
 \frac{\sqrt{H_1R^2}}{[\log(8H_1R^2)]^{3/2}}
 \right),
 \qquad
 Q=\Omega\!\left(
 \frac{d\sqrt{H_1R^2}}{[\log(8H_1R^2)]^{3/2}}
 \right).
\end{equation}
For any fixed \(\alpha<1/2\) and any fixed logarithmic power \(c\),
\begin{equation}
\label{eq:lower-power-comparison}
 \frac{\kappa^{1/2}/[\log(8\kappa)]^{3/2}}
      {\kappa^\alpha[\log(8\kappa)]^c}
 \overset{\circledOne}{=}
 \frac{\kappa^{1/2-\alpha}}
        {[\log(8\kappa)]^{c+3/2}}
 \overset{\circledTwo}{\longrightarrow}\infty,
\end{equation}
where in \(\circledOne\), we collected the powers of \(\kappa\) and the
logarithm; in \(\circledTwo\), we used \(1/2-\alpha>0\), so a positive
power of \(\kappa\) grows faster than every fixed power of its logarithm.
Hence a smaller fixed power of \(H_1R^2\) cannot replace the square-root
dependence, even after multiplication by fixed logarithmic factors, for the
full-gradient and coordinate update classes defined in Steps~4 and~5.
\end{proof}

\begin{remark}[Compatibility with auxiliary operations and certificates]
\label{rem:lower-bound-compatibility}
The lower bounds in \eqref{eq:lower-omega-orders} remain valid if a method
knows \(R,H_0,H_1\), and \(f^*=0\), evaluates function values at admissible
query points, changes scalar parameters, averages an iterate with \(x^0\),
or applies the Euclidean projection onto \(B(x^0,3R)\). The
supplied certificate \(\Delta_0=4\varepsilon\) is valid. Moreover, the
default certificate \(\Delta_0=R\norm{\nabla f(x^0)}\) satisfies
\[
 \log\!\left(\frac{2\Delta_0}{\varepsilon}\right)
 =O\bigl(\log(2+\kappa)\bigr).
\]
\end{remark}

\begin{proof}[Justification of Remark~\ref{rem:lower-bound-compatibility}]
\noindent\fbox{\strut\textit{1. Preservation of support and span.}}
\par\noindent
A function value is a scalar and supplies no new vector direction, so it
does not change either oracle argument. For the auxiliary vector operations,
fix \(x\in E_j\) and scalars \(\alpha,\theta\). Then
\begin{equation}
\label{eq:lower-support-preservation}
 \begin{aligned}
 \alpha x&\overset{\circledOne}{\in}E_j,\qquad
 (1-\theta)x+\theta x^0
 \overset{\circledTwo}{\in}E_j,\qquad
 P_{B(x^0,3R)}(x)
 \overset{\circledThree}{\in}E_j,
 \end{aligned}
\end{equation}
where in \(\circledOne\), we used that \(E_j\) is a linear subspace; in
\(\circledTwo\), we used \(x^0=0\); in \(\circledThree\), the Euclidean
projection onto a ball centered at zero is
\[
 P_{B(0,3R)}(x)=
 \begin{cases}
 x, & \norm{x}\le3R,\\
 3R x/\norm{x}, & \norm{x}>3R,
 \end{cases}
\]
and is therefore a scalar multiple of \(x\). Thus changing scalar
parameters, averaging with \(x^0\), and applying the radial projection
preserve the support restriction in the coordinate argument and the
linear-span restriction in the full-gradient argument.

\par\noindent\fbox{\strut\textit{2. Supplied and default certificates.}}
\par\noindent
The supplied value \(\Delta_0=4\varepsilon\) is a valid certificate by
\eqref{eq:lower-initial-gap}, and
\(\log(2\Delta_0/\varepsilon)=\log8\). It remains to estimate the default choice
\(\Delta_0=R\norm{\nabla f(x^0)}\). At \(x^0=0\),
\begin{equation}
\label{eq:lower-initial-gradient-radius}
 \begin{aligned}
 \nabla\!\bigl(\mathcal R(x)^2\bigr)\big|_{x=0}
 &\overset{\circledOne}{=}-2\omega\beta e_1,\\
 \nabla \mathcal R(0)
 &\overset{\circledTwo}{=}
 \frac{\nabla(\mathcal R(x)^2)|_{x=0}}{2\mathcal R(0)}
 \overset{\circledThree}{=}-\sqrt\omega\,e_1,
 \end{aligned}
\end{equation}
where in \(\circledOne\), only the derivative of
\(\omega(x_1-\beta)^2\) in \eqref{eq:lower-chain-objective} is nonzero at the
origin; in \(\circledTwo\), we differentiated the square root; in
\(\circledThree\), we used \(\mathcal R(0)=\sqrt\omega\,\beta>0\).

Differentiating the second line of \eqref{eq:lower-chain-objective} and
using \eqref{eq:lower-initial-radius} and
\eqref{eq:lower-initial-gradient-radius} gives
\begin{equation}
\label{eq:lower-default-certificate}
 \begin{aligned}
 \frac{R\norm{\nabla f(0)}}{\F(0)}
 &\overset{\circledOne}{=}
 R\sqrt{H_1\omega}\,\frac{\sinh t}{\cosh t-1}\\
 &\overset{\circledTwo}{=}
 t\sqrt m\,\coth(t/2)
 \overset{\circledThree}{\le}
 \frac97t\sqrt m,
 \end{aligned}
\end{equation}
where in \(\circledOne\), differentiation gives
\(\norm{\nabla f(0)}=H_0\sqrt{\omega/H_1}\sinh t\), while
\(\F(0)=(H_0/H_1)(\cosh t-1)\); in \(\circledTwo\), we used
\(R\sqrt{H_1\omega}=t\sqrt m\) and
\(\sinh t/(\cosh t-1)=\coth(t/2)\); in \(\circledThree\), we used
\[
 \coth(t/2)=\frac{e^t+1}{e^t-1}\le\frac97,
 \qquad e^t=8\kappa\ge8.
\]
Indeed, \((u+1)/(u-1)\) is decreasing for \(u>1\), so its value at
\(u=e^t\ge8\) is at most \(9/7\).

Combining \eqref{eq:lower-initial-gap} and
\eqref{eq:lower-default-certificate} with \(m\le\sqrt\kappa\) and
\(t=\log(8\kappa)\) yields
\begin{equation}
\label{eq:lower-certificate-log}
 3
 \overset{\circledOne}{\le}
 \frac{\F(0)}{\varepsilon}
 \overset{\circledTwo}{\le}
 \frac{R\norm{\nabla f(0)}}{\varepsilon}
 \overset{\circledThree}{\le}
 \frac{36}{7}\log(8\kappa)\,\kappa^{1/4},
\end{equation}
where in \(\circledOne\), we used \eqref{eq:lower-initial-gap}; in
\(\circledTwo\), convexity gives
\(\F(0)\le\inner{\nabla f(0)}{x^0-x^*}\le
R\norm{\nabla f(0)}\); in \(\circledThree\), we multiplied the last
inequality in \eqref{eq:lower-default-certificate} by
\(\F(0)/\varepsilon\le4\) and used \(\sqrt m\le\kappa^{1/4}\).
Taking logarithms in \eqref{eq:lower-certificate-log} gives
\begin{equation}
\label{eq:lower-certificate-logarithm}
 0
 \overset{\circledOne}{\le}
 \log\!\left(\frac{2R\norm{\nabla f(0)}}{\varepsilon}\right)
 \overset{\circledTwo}{\le}
 \log\frac{72}{7}+\log\log(8\kappa)+\frac14\log\kappa
 \overset{\circledThree}{=}
 O\bigl(\log(2+\kappa)\bigr),
\end{equation}
where in \(\circledOne\), the middle ratio is at least three by
\eqref{eq:lower-certificate-log}; in \(\circledTwo\), we multiplied the
upper bound in \eqref{eq:lower-certificate-log} by two and took its
logarithm; in
\(\circledThree\), we used
\(\log\log(8\kappa)\le\log(8\kappa)\) for \(\kappa\ge1\).
This \(O(\cdot)\)-estimate concerns only the size of the default certificate;
the oracle-complexity lower bounds are the \(\Omega(\cdot)\)-estimates in
\eqref{eq:lower-omega-orders}.
\end{proof}
\end{revision}

\begin{experimentrevision}
\color{black}
\section{Additional Experiments}
\label{app:additional-experiments}
\color{experimentred}

\purrev{This section supplements Section~\ref{sec:experiments} with complete
results and diagnostic comparisons.
Appendix~\ref{app:real-data-full-results} reports the full-gradient results for
all initializations of the Bike Sharing and Household Power problems and an
additional comparison with Frozen-FGM on \texttt{covtype.binary}.
Appendix~\ref{app:coordinate-experiment} then compares uniform and non-uniform
projected coordinate methods on the controlled capped-cosh objective and on
Bike Sharing. Appendix~\ref{app:parameter-choices} compares theoretical and
practical AGM parameters, Appendix~\ref{app:search-tolerance} studies the
search tolerance of AGMsDR, and Appendix~\ref{app:experimental-implementation}
records the implementation settings.}

\begingroup
\experimentrevisionfalse
\color{revisionblue}
\captionsetup{labelfont=revisionblue,textfont=revisionblue}
\subsection{Complete real-data results}
\label{app:real-data-full-results}

Section~\ref{sec:experiments} presents representative convergence curves,
whereas Table~\ref{tab:real-data-full-results} reports all five initializations
\(b^0\in\{0,2,4,6,8\}\) for both real-data problems. Iteration and oracle
counts are deterministic; runtimes are medians over seven runs for Bike
Sharing and three runs for Household Power. For AGMsDR, each golden-section
evaluation is counted as one function-value query, and the total oracle
complexity is the sum of gradient and function-value queries.
\orrev{Times are measured through the final checkpoint and include gap
monitoring; monitoring values are excluded from oracle counts.}

\begin{table}[H]
\centering
\caption{\rev{Complete full-gradient results at relative gap \(10^{-6}\).
The column ``Grad./Val.'' reports gradient and function-value calls. Bold
entries mark the lowest total oracle count and runtime for each initialization.}}
\label{tab:real-data-full-results}
\small
\setlength{\tabcolsep}{2pt}
\renewcommand{\arraystretch}{1.02}
\begin{minipage}[t]{0.49\textwidth}
\centering
\textbf{(a) Bike Sharing}\par\smallskip
\resizebox{\linewidth}{!}{%
\begin{tabular}{@{}rlrrrr@{}}
\toprule
\(b^0\) & Method & Iter. & Grad./Val. & Total & Time (s) \\
\midrule
0 & AGM (ours) & \orrev{152} & \orrev{152/0} & \orrev{\textbf{152}} & \orrev{\textbf{0.383}} \\
   & AGMsDR & 101 & 101/5626 & 5727 & \orrev{1.568} \\
   & \(\ell\)-AGD & 653 & 654/1 & 655 & \orrev{1.751} \\
   & Accelerated GRAAL & 292 & 584/584 & 1168 & \orrev{1.044} \\
\cmidrule(lr){1-6}
2 & AGM (ours) & \orrev{110} & \orrev{110/0} & \orrev{\textbf{110}} & \orrev{\textbf{0.299}} \\
   & AGMsDR & 63 & 63/3362 & 3425 & \orrev{0.952} \\
   & \(\ell\)-AGD & 514 & 515/1 & 516 & \orrev{1.370} \\
   & Accelerated GRAAL & 277 & 554/554 & 1108 & \orrev{0.976} \\
\cmidrule(lr){1-6}
4 & AGM (ours) & \orrev{98} & \orrev{98/0} & \orrev{\textbf{98}} & \orrev{\textbf{0.255}} \\
   & AGMsDR & 37 & 37/1646 & 1683 & \orrev{0.476} \\
   & \(\ell\)-AGD & 612 & 613/1 & 614 & \orrev{1.659} \\
   & Accelerated GRAAL & 260 & 520/520 & 1040 & \orrev{0.937} \\
\cmidrule(lr){1-6}
6 & AGM (ours) & \orrev{93} & \orrev{93/0} & \orrev{\textbf{93}} & \orrev{\textbf{0.254}} \\
   & AGMsDR & 34 & 34/1265 & 1299 & \orrev{0.387} \\
   & \(\ell\)-AGD & 755 & 756/1 & 757 & \orrev{1.845} \\
   & Accelerated GRAAL & 249 & 498/498 & 996 & \orrev{0.802} \\
\cmidrule(lr){1-6}
8 & AGM (ours) & \orrev{91} & \orrev{91/0} & \orrev{\textbf{91}} & \orrev{\textbf{0.240}} \\
   & AGMsDR & 29 & 29/813 & 842 & \orrev{0.261} \\
   & \(\ell\)-AGD & 945 & 946/1 & 947 & \orrev{2.387} \\
   & Accelerated GRAAL & 278 & 556/556 & 1112 & \orrev{0.858} \\
\bottomrule
\end{tabular}
}
\end{minipage}\hfill
\begin{minipage}[t]{0.49\textwidth}
\centering
\textbf{(b) Household Power}\par\smallskip
\resizebox{\linewidth}{!}{%
\begin{tabular}{@{}rlrrrr@{}}
\toprule
\(b^0\) & Method & Iter. & Grad./Val. & Total & Time (s) \\
\midrule
0 & AGM (ours) & \orrev{113} & \orrev{113/0} & \orrev{\textbf{113}} & \orrev{\textbf{3.461}} \\
   & AGMsDR & 65 & 65/3738 & 3803 & \orrev{19.375} \\
   & \(\ell\)-AGD & 345 & 346/1 & 347 & \orrev{10.044} \\
   & Accelerated GRAAL & 249 & 498/498 & 996 & \orrev{8.842} \\
\cmidrule(lr){1-6}
2 & AGM (ours) & \orrev{81} & \orrev{81/0} & \orrev{\textbf{81}} & \orrev{\textbf{2.470}} \\
   & AGMsDR & 32 & 32/1316 & 1348 & \orrev{7.202} \\
   & \(\ell\)-AGD & 316 & 317/1 & 318 & \orrev{9.190} \\
   & Accelerated GRAAL & 234 & 468/468 & 936 & \orrev{8.299} \\
\cmidrule(lr){1-6}
4 & AGM (ours) & \orrev{76} & \orrev{76/0} & \orrev{\textbf{76}} & \orrev{\textbf{2.278}} \\
   & AGMsDR & 32 & 32/1316 & 1348 & \orrev{6.899} \\
   & \(\ell\)-AGD & 413 & 414/1 & 415 & \orrev{11.742} \\
   & Accelerated GRAAL & 205 & 410/410 & 820 & \orrev{7.526} \\
\cmidrule(lr){1-6}
6 & AGM (ours) & \orrev{74} & \orrev{74/0} & \orrev{\textbf{74}} & \orrev{\textbf{2.498}} \\
   & AGMsDR & 26 & 26/798 & 824 & \orrev{4.852} \\
   & \(\ell\)-AGD & 562 & 563/1 & 564 & \orrev{16.490} \\
   & Accelerated GRAAL & 234 & 468/468 & 936 & \orrev{8.795} \\
\cmidrule(lr){1-6}
8 & AGM (ours) & \orrev{84} & \orrev{84/0} & \orrev{\textbf{84}} & \orrev{2.403} \\
   & AGMsDR & 21 & 21/285 & 306 & \orrev{\textbf{1.899}} \\
   & \(\ell\)-AGD & 801 & 802/1 & 803 & \orrev{23.362} \\
   & Accelerated GRAAL & 264 & 528/528 & 1056 & \orrev{8.835} \\
\bottomrule
\end{tabular}
}
\end{minipage}
\end{table}

\paragraph{\texttt{covtype.binary}.}
We additionally compare AGM and Frozen-FGM on all \(n=581{,}012\) observations
and \(54\) features of the scaled
\href{https://www.csie.ntu.edu.tw/~cjlin/libsvmtools/datasets/binary.html\#covtype.binary}{\texttt{covtype.binary}}
data from LIBSVM \citep{chang2011libsvm}. We append an intercept and rescale
every row by the same factor so that \(\max_j\norm{a_j}=1\). With signed margin
\(m\), we use the regularized capped-exponential objective
\[
 \ell(m):=
 \begin{cases}
  e^{-m}, & m\ge0,\\
  1-m+m^2/2, & m<0,
 \end{cases}
 \qquad
 f(x):=\frac1n\sum_{j=1}^n\ell(y_ja_j^\top x)
       +\frac{10^{-2}}2\norm{x}^2.
\]
Both methods start from \(x^0=0\) and stop when
\(\F(x)/\F(x^0)\le10^{-4}\). AGM uses the rule from
Section~\ref{sec:experiments}. Frozen-FGM is the classical two-sequence
accelerated method with the fixed initial level-set curvature
\(M_{\rm frozen}:=H_0+H_1\F(x^0)\). Timings use FP64 on one NVIDIA A100
80\,GB GPU and report medians and interquartile ranges over ten interleaved runs
after one warm-up.

\begin{table}[H]
\centering
\caption{AGM and Frozen-FGM on \texttt{covtype.binary} at relative gap
\(10^{-4}\). Time is reported as median [IQR] over ten runs.
\orrev{Bold marks the smaller values.}}
\label{tab:covtype-agm-frozen}
\small
\setlength{\tabcolsep}{5pt}
\renewcommand{\arraystretch}{1.05}
\begin{tabular}{@{}lrrrrr@{}}
\toprule
Method & Iter. & Grad. & Val. & Total & Time (ms) \\
\midrule
AGM (ours) & \orrev{\textbf{50}} & \orrev{\textbf{50}} & \textbf{0} & \orrev{\textbf{50}} & \orrev{\textbf{55.6 [55.0--56.9]}} \\
Frozen-FGM & 110 & 110 & 1 & 111 & \orrev{99.8 [99.7--100.0]} \\
\bottomrule
\end{tabular}
\end{table}
\vspace{-1em}

Objective values used only to monitor the relative gap are included in the
timed measurements but excluded from the ``Val.'' column. This column counts
only algorithmic function-value queries, including the single query used by
Frozen-FGM to fix its initial curvature.
\endgroup

\begingroup
\experimentrevisionfalse
\color{revisionblue}
\captionsetup{labelfont=revisionblue,textfont=revisionblue}
\subsection{Uniform versus non-uniform coordinate sampling}
\label{app:coordinate-experiment}

We complement the full-gradient experiments with a controlled synthetic
objective and a real-data regression problem. The first experiment separates
the effect of global coordinate curvature from the gap-dependent constants,
whereas the second tests the same methods when the coordinates are coupled
through shared data features.

We compare the Accelerated Projected Coordinate Method (APCM) with uniform
sampling in Algorithm~\liveexperimentref{alg:coordinate-inner}, its
non-uniform counterpart in
Algorithm~\liveexperimentref{alg:nonuniform-coordinate-inner}, and the ACDM of
\citet{Nesterov_2017}. We report coordinate-oracle complexity (measured by
partial-derivative calls) and wall-clock time needed to reach the relative gap
\(10^{-6}\); all curves are
pointwise medians over seven sampling seeds and all timings use one OpenBLAS
thread. \orrev{Times include initialization and gap monitoring through the
final checkpoint. The gap is checked every \(d\) iterations; monitoring values
are excluded from oracle counts.} For ACDM, the constants \(L_i\) are the exact global maxima of the
coordinate curvatures. The constants \(H_{0,i}\) and \(H_{1,i}\) are derived
analytically; they are attained in the capped-cosh experiment and are valid,
but conservative, on Bike Sharing. The experiments therefore test whether the
effective gap-dependent curvature \(H_{0,i}+H_{1,i}(f(x)-f^*)\) can be
substantially smaller than the global constant \(L_i\) along the iterates.
\orrev{ACDM uses strong convexity \(0.02\) for capped cosh and \(0\) for Bike Sharing.}

\purrev{Both practical APCM variants require \emph{no prescribed
\(L_i,H_{0,i},H_{1,i}\)}. An initial pool contains two partial-derivative
queries per coordinate, at \(0\) and \(10^{-3}e_i\). Their secant differences
give empirical curvatures \(D_i\): uniform sampling uses \(p_i=1/d\), while
non-uniform sampling uses \(p_i\propto\sqrt{D_i}\), favoring coordinates with
larger observed curvature. The pool also calibrates the initial weight and
step scale and initializes derivative memory in both variants, explaining
its use even under uniform sampling. All \(2d\) queries are counted and
timed; full rules appear in Appendix~\ref{app:experimental-implementation}.}

\vspace{-1em}
\paragraph{Capped cosh.}
We minimize the convex separable objective
\(f(x)=\sum_{i=1}^{48}\phi_i(x_i)\). Six coordinates have
\((H_{0,i},H_{1,i})=(1,2)\), while the remaining 42 have
\((H_{0,i},H_{1,i})=(0.02,0.04)\). Here \(T\in\{6,8\}\) is a
dimensionless cutoff that determines where the hyperbolic-cosine branch is
replaced by its quadratic continuation. Coordinate \(i\) follows the
hyperbolic-cosine branch while
\(\sqrt{H_{1,i}}\lvert x_i\rvert\le T\), equivalently while
\(\lvert x_i\rvert\le T/\sqrt{H_{1,i}}\), and in this region
\[
  \phi_i(x_i)
  =\frac{H_{0,i}}{H_{1,i}}
  \left(\cosh\!\left(\sqrt{H_{1,i}}x_i\right)-1\right).
\]
Hence \(\phi_i''(x_i)=H_{0,i}+H_{1,i}\phi_i(x_i)\) inside the
threshold. Beyond the threshold, we continue the function quadratically as
\[
  \phi_i(x_i)
  =\frac{H_{0,i}}{H_{1,i}}
  \left[
    \cosh T-1
    +\sinh T\bigl(\sqrt{H_{1,i}}\lvert x_i\rvert-T\bigr)
    +\frac{\cosh T}{2}
      \bigl(\sqrt{H_{1,i}}\lvert x_i\rvert-T\bigr)^2
  \right].
\]
This continuation matches the value, first derivative, and second derivative
at the threshold. Inside the threshold,
\(\phi_i''(x_i)=H_{0,i}\cosh(\sqrt{H_{1,i}}x_i)\le
H_{0,i}\cosh T\), whereas outside the threshold the quadratic continuation
has the constant curvature \(H_{0,i}\cosh T\). Therefore, the exact global
coordinate smoothness constant is
\(L_i=\sup_{x_i\in\mathbb{R}}\phi_i''(x_i)=H_{0,i}\cosh T\). Moreover,
\[
  \partial_{ii}f(x)\le H_{0,i}+H_{1,i}\bigl(f(x)-f^*\bigr).
\]
The second relation is an equality inside the threshold when all other
coordinates vanish. Increasing \(T\) from 6 to 8 multiplies every \(L_i\) by
\(\cosh(8)/\cosh(6)\approx7.4\), while leaving the constants
\(H_{0,i},H_{1,i}\), the initial point, and the initial gap unchanged. This
provides a direct test of sensitivity to global curvature and of the benefit of
non-uniform sampling under coordinate heterogeneity.
\orrev{We use \(x^0=0.75z\), \(z\sim\mathcal N(0,I_{48})\) (seed \(31\)),
and \(R=\norm{x^0}\approx4.766\).}

\vspace{-1em}
\begin{figure}[H]
  \centering
  \includegraphics[width=0.96\linewidth]{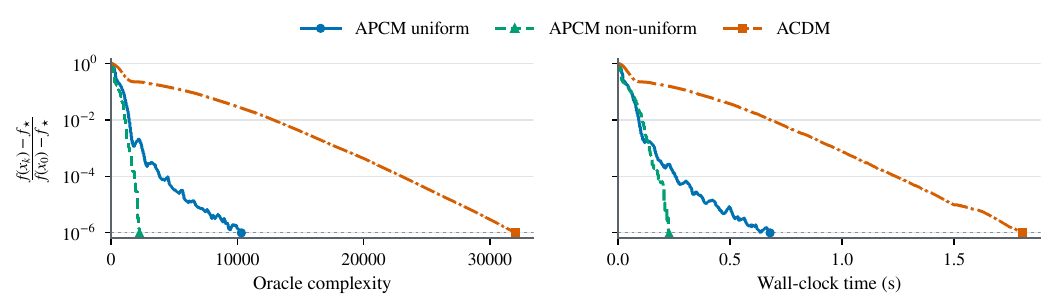}
  \caption{\small Capped-cosh experiment with \(T=8\). The panels show the
  relative objective gap against \orrev{oracle calls} and wall-clock time;
  the dashed line marks the target \(10^{-6}\). Curves are pointwise medians
  over seven seeds. \orrev{APCM counts include \(96\) initialization queries.}}
  \label{fig:coordinate-capped-cosh}
\end{figure}
\vspace{-1.5em}

Figure~\ref{fig:coordinate-capped-cosh} and
Table~\ref{tab:coordinate-capped-cosh} show that both APCM variants have the
same coordinate-oracle complexity at \(T=6\) and \(T=8\), whereas the ACDM
complexity increases from 12,144 to 32,016. This contrast is explained by the
\orrev{curvature scales: the empirical curvatures \(D_i\) used for APCM
initialization and sampling remain unchanged}, while the global constants \(L_i\)
used by ACDM increase by a factor of approximately 7.4. The non-uniform APCM
further improves on the uniform variant by exploiting the heterogeneity of
\orrev{the empirical curvatures \(D_i\)}.

\vspace{-0.5em}
\begin{table}[H]
\centering
\caption{Median capped-cosh results over seven seeds at relative gap
\(10^{-6}\). \orrev{Counts include APCM initialization.} Bold marks the best result for each \(T\).}
\label{tab:coordinate-capped-cosh}
\small
\renewcommand{\arraystretch}{0.92}
\begin{tabular}{@{}clrr@{}}
\toprule
\(T\) & Method & Oracle complexity & Time (s) \\
\midrule
6 & APCM uniform & \orrev{10,320} & \orrev{0.641} \\
   & APCM non-uniform & \orrev{\textbf{2,256}} & \orrev{\textbf{0.202}} \\
   & ACDM & 12,144 & 0.705 \\
\midrule
8 & APCM uniform & \orrev{10,320} & \orrev{0.678} \\
   & APCM non-uniform & \orrev{\textbf{2,256}} & \orrev{\textbf{0.228}} \\
   & ACDM & 32,016 & 1.804 \\
\bottomrule
\end{tabular}
\end{table}
\vspace{-1em}

\vspace{-1em}
\paragraph{Bike Sharing.}
We use the same 17,379 observations and 51-parameter regularized Poisson model
as in Section~\ref{sec:experiments}. The original loss contains the exponential
\(e^s\), whose curvature is unbounded. Consequently, the global coordinate
smoothness constants \(L_i\) required by ACDM are not finite, and ACDM cannot
be applied directly. \orrev{The full-gradient baselines requiring generalized-smoothness
constants have finite certificates for the original exponential loss.} We therefore modify the objective
only for the coordinate comparison.
\orrev{We use \(x^0=(b^0,0)\), \(b^0\in\{0,2,4,6,8\}\), and a certified
upper bound on \(\norm{x^0-x^*}\) for \(R\).}

We choose \(T=6\) such that \(e^T\ge\max_j y_j\) and every predictor at the
reference solution is smaller than \(T\). The first property will ensure that
the shifted scalar losses introduced below are nonnegative. The second implies
that the modified objective and its gradient agree with the original objective
at the reference solution. Since the modified objective remains convex, the
reference solution is still a minimizer. Thus \(T\) is an experimental cutoff,
not a parameter of either APCM algorithm. For a scalar predictor \(s\), define
\[
  q_T(s):=
  \begin{cases}
    e^s, & s\le T,\\[2pt]
    e^T\left(1+(s-T)+\frac12(s-T)^2\right), & s>T.
  \end{cases}
\]
The two branches have the same value, first derivative, and second derivative
at \(s=T\). Hence \(q_T\) is twice continuously differentiable and
\[
  q_T''(s)=\exp\!\bigl(\min\{s,T\}\bigr).
\]

We next derive \orrev{the global constants required by ACDM and the
gap-dependent certificates of the objective}. Let \(a_j\)
be the augmented feature vector, including the intercept, let \(a_{ji}\) be its
\(i\)-th entry, and set \(s_j=a_j^\top x\). We write
\(\overline{r_j}:=n^{-1}\sum_{j=1}^n r_j\) for the empirical average. The
coordinate curvature is
\[
  \partial_{ii}f(x)
  =\lambda_i+\frac1n\sum_{j=1}^n a_{ji}^2q_T''(s_j),
\]
where \(\lambda_i=0\) for the intercept and \(\lambda_i=0.01\) otherwise.
Since \(q_T''(s)\le e^T\), this identity gives
\[
  \partial_{ii}f(x)
  \le \lambda_i+e^T\,\overline{a_i^2}.
\]
If \(w=0\) and \(b\ge T\), then every predictor satisfies \(s_j\ge T\), so
\(q_T''(s_j)=e^T\) for every observation and the inequality becomes an
equality. Hence the exact global coordinate smoothness constant is
\[
  L_i=\lambda_i+e^T\,\overline{a_i^2}.
\]

To derive the gap-dependent constants, consider the shifted scalar loss
\[
  \ell_y(s):=q_T(s)-ys+y\log y-y,
\]
with \(0\log0:=0\). Because \(y\le e^T\), the convex function
\(q_T(s)-ys\) is, for \(y>0\), minimized at \(s=\log y\le T\), where its
value is \(y-y\log y\). For \(y=0\), nonnegativity follows from
\(\ell_0(s)=q_T(s)\ge0\). Therefore, \(\ell_y(s)\ge0\) in all cases. We will
use the following pointwise curvature bound:
\[
  q_T''(s)\le2\ell_y(s)+2y\log2.
\]
For \(y>0\) and \(s\le T\), let \(u=e^s/y\). Since \(q_T(s)=q_T''(s)=e^s\)
in this region, the difference between the right- and left-hand sides is
\[
  2\ell_y(s)+2y\log2-q_T''(s)
  =y\bigl(u-2\log u-2+2\log2\bigr)\ge0.
\]
The inequality follows because the expression in parentheses is minimized at
\(u=2\), where it equals zero. For \(s>T\), the curvature remains equal to
\(e^T\), while
\(\ell_y'(s)=q_T'(s)-y\ge e^T-y\ge0\). Hence the right-hand side is
nondecreasing beyond \(T\), and the inequality at \(s=T\) remains valid for
all \(s>T\). If \(y=0\), it follows directly from
\(q_T''(s)\le q_T(s)=\ell_0(s)\).

Multiplying the pointwise bound by \(a_{ji}^2\), averaging over the data, and
using \(a_{ji}^2\le\max_j a_{ji}^2\) gives
\[
\begin{aligned}
  \partial_{ii}f(x)
  &\le \lambda_i+2\,\overline{a_i^2\ell_y(s)}
       +2\log2\,\overline{y a_i^2} \\
  &\le \lambda_i
       +2\max_j a_{ji}^2\,\overline{\ell_y(s)}
       +2\log2\,\overline{y a_i^2}.
\end{aligned}
\]
It remains to relate the average shifted loss to the common objective gap.
Define
\[
  \widetilde f(x)
  :=\overline{\ell_y(s)}+\frac{0.01}{2}\lVert w\rVert^2.
\]
This function differs from \(f\) only by an additive constant, so
\(\widetilde f(x)-\widetilde f^*=F(x)\). Moreover,
\(\widetilde f^*\le\widetilde f(0)=\overline{y\log y}\), where we used
\(q_T(0)=1\) and the normalization \(\overline y=1\). Since the regularizer is
nonnegative,
\[
  \overline{\ell_y(s)}
  \le\widetilde f(x)
  =F(x)+\widetilde f^*
  \le F(x)+\overline{y\log y}.
\]
Substituting this estimate into the curvature bound yields
\[
  \partial_{ii}f(x)
  \le H_{0,i}+H_{1,i}F(x),
\]
with
\[
  H_{1,i}=2\max_j a_{ji}^2,
  \qquad
  H_{0,i}=\lambda_i+2\log2\,\overline{y a_i^2}
  +H_{1,i}\,\overline{y\log y}.
\]
This proves the required coordinate \((H_0,H_1)\)-smoothness certificate. The
full-gradient constants follow from the same argument with
\(\lVert a_j\rVert^2\) in place of \(a_{ji}^2\). The separation between the
global and gap-dependent descriptions is already substantial for the
intercept: its curvature at the solution equals \(1\),
\(H_{0,i}\approx2.2645\), and \(L_i\approx403.429\).

\vspace{-1em}
\begin{figure}[H]
  \centering
  \includegraphics[width=0.96\linewidth]{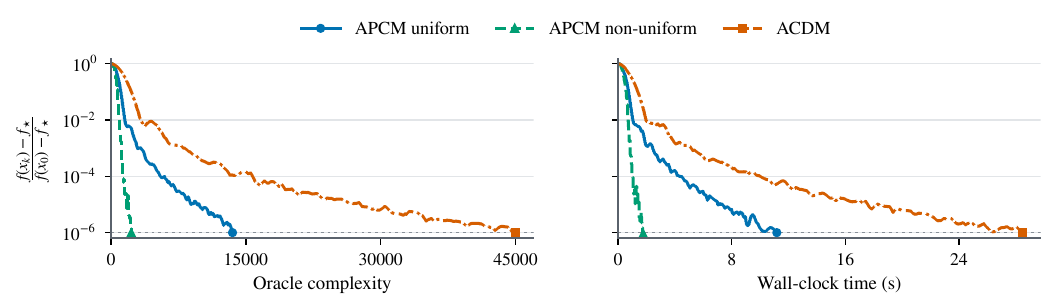}
  \caption{\small Projected coordinate methods on Bike Sharing for \(b^0=0\) and
  \(T=6\). The panels show the relative objective gap against
  \orrev{oracle calls} and wall-clock time; the dashed line marks the
  target \(10^{-6}\). Curves are pointwise medians over seven seeds.
  \orrev{APCM counts include \(102\) initialization queries.}}
  \label{fig:coordinate-bike-sharing}
\end{figure}
\vspace{-1.5em}

Figure~\ref{fig:coordinate-bike-sharing} and
Table~\ref{tab:coordinate-bike-sharing} show that the non-uniform APCM has the
lowest coordinate-oracle complexity and runtime for every initialization. At
\(b^0=8\), \orrev{ACDM is faster than uniform APCM, but uses more calls
(\(3213\) versus \(2805\))}. \orrev{ACDM also uses} non-uniform importance sampling with probabilities
proportional to \(\sqrt{L_i}\) and can therefore exploit coordinate
heterogeneity. \orrev{Non-uniform APCM combines empirical-curvature sampling
with derivative memory.} Its
advantage is consistent with the effective
quantities \(H_{0,i}+H_{1,i}(f(x)-f^*)\) being substantially smaller than the
global \(L_i\) \orrev{near the solution}; for the intercept, for example,
\(H_{0,i}\approx2.2645\), whereas \(L_i\approx403.429\). Thus the
capped-cosh experiment isolates this separation, while Bike Sharing shows that
it persists for coupled real-data features.

\vspace{-0.5em}
\begin{table}[H]
\centering
\caption{Median Bike Sharing results over seven seeds at relative gap
\(10^{-6}\). \orrev{Counts include APCM initialization.} Bold marks the best result for each \(b^0\).}
\label{tab:coordinate-bike-sharing}
\small
\renewcommand{\arraystretch}{0.88}
\begin{tabular}{@{}clrr@{}}
\toprule
\(b^0\) & Method & Oracle complexity & Time (s) \\
\midrule
0 & APCM uniform & \orrev{13,515} & \orrev{11.174} \\
   & APCM non-uniform & \orrev{\textbf{2,244}} & \orrev{\textbf{1.770}} \\
   & ACDM & 44,982 & 28.462 \\
\midrule
2 & APCM uniform & \orrev{7,803} & \orrev{5.667} \\
   & APCM non-uniform & \orrev{\textbf{1,530}} & \orrev{\textbf{1.152}} \\
   & ACDM & 20,502 & 13.493 \\
\midrule
4 & APCM uniform & \orrev{4,998} & \orrev{4.404} \\
   & APCM non-uniform & \orrev{\textbf{1,326}} & \orrev{\textbf{1.072}} \\
   & ACDM & 10,608 & 7.425 \\
\midrule
6 & APCM uniform & \orrev{3,774} & \orrev{2.874} \\
   & APCM non-uniform & \orrev{\textbf{1,173}} & \orrev{\textbf{0.951}} \\
   & ACDM & 5,967 & 4.045 \\
\midrule
8 & APCM uniform & \orrev{2,805} & \orrev{2.479} \\
   & APCM non-uniform & \orrev{\textbf{1,071}} & \orrev{\textbf{1.015}} \\
   & ACDM & 3,213 & 2.208 \\
\bottomrule
\end{tabular}
\end{table}
\endgroup

\end{experimentrevision}

\begin{purplerevision}
\experimentrevisionfalse
\raggedbottom
\captionsetup{font+=revisionpurple,labelfont=revisionpurple,textfont=revisionpurple}

\subsection{Theoretical and practical parameter choices}
\label{app:parameter-choices}

We examine how the sufficient theoretical parameters translate into observed
convergence, and motivate the practical schedule used in
Section~\ref{sec:experiments}. We compare three parameterizations of the same
Algorithm~\ref{alg:full-gradient-inner}: AGM (Theorem~\ref{thm:full-gradient-main}),
AGM (Proposition~\ref{prop:l0l1-full-gradient-rate}), and AGM (practical).
The last is exactly the method denoted AGM (ours) in the main comparisons,
with the same fixed coefficients reused here.

We use the \(48\)-dimensional capped-cosh objective of
Appendix~\ref{app:coordinate-experiment} with \(T=8\) and the same initial
point \(x^0=0.75z\), \(z\sim\mathcal N(0,I_{48})\), generated with seed
\(31\). Here \(x^*=0\), \(f^*=0\), \(R=\norm{x^0}\approx4.7661\), and
\(\Delta_0=f(x^0)\approx0.87107\). Valid certificates are
\((H_0,H_1)=(1,2)\) for Assumption~\ref{ass:h0h1-smoothness} and
\((L_0,L_1)=(2,2\sqrt2)\) for the local model
\eqref{eq:local-l0l1-function-model}; their justification is given in
Appendix~\ref{app:experimental-implementation}. The theoretical variants use
their full prescribed rules and \(A_0=R^2/(2\Delta_0)\approx13.0389\), while
the practical variant retains its own \(A_0\approx0.44325\).

For each relative target in \(\{10^{-2},10^{-4},10^{-6}\}\), we start a
separate run with that target and a common budget of \(10^5\) gradient
queries. Each iteration uses one gradient and no function values in the
update. We stop at the first monitored target crossing, following the
monitoring convention of Appendix~\ref{app:real-data-full-results}.
This tests the complete parameter rules, including their initialization.

\begin{figure}[H]
  \centering
  \includegraphics[width=0.96\linewidth]{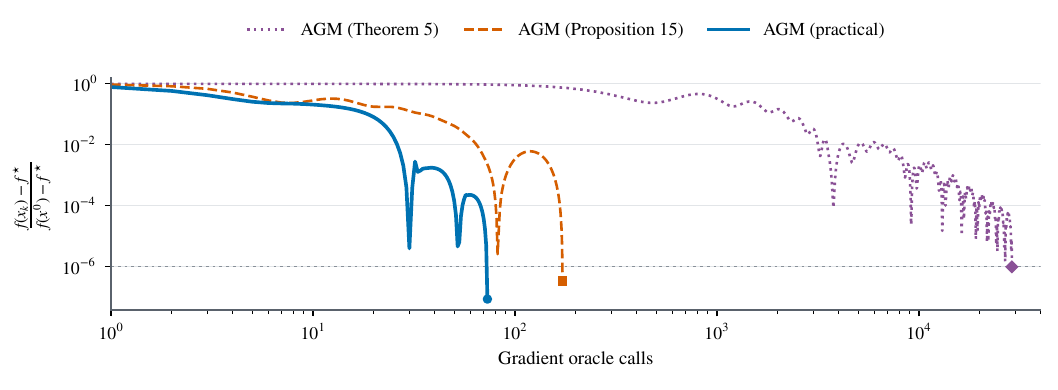}
  \caption{\small Convergence of three parameter choices of
  Algorithm~\ref{alg:full-gradient-inner} on capped cosh with \(T=8\).
  All three curves are configured for relative target \(10^{-6}\).
  The horizontal axis shows full-gradient oracle calls on a logarithmic
  scale; the dashed line marks the target.}
  \label{fig:agm-parameter-choices}
\end{figure}
\vspace{-0.7em}

\begin{table}[H]
\color{revisionpurple}
\centering
\caption{Full-gradient oracle calls required to reach the target relative gap
on capped cosh with \(T=8\). Each entry is a separate run configured for its
target. All runs reach the target within \(10^5\) gradient queries.
Bold marks the lowest count for each target.}
\label{tab:agm-parameter-choices}
\small
\renewcommand{\arraystretch}{0.98}
\begin{tabular}{@{}lrrr@{}}
\toprule
& \multicolumn{3}{c}{Target relative gap} \\
\cmidrule(l){2-4}
Parameter choice & \(10^{-2}\) & \(10^{-4}\) & \(10^{-6}\) \\
\midrule
AGM (Theorem~\ref{thm:full-gradient-main}) & 3,163 & 3,730 & 28,862 \\
AGM (Proposition~\ref{prop:l0l1-full-gradient-rate}) & 65 & 81 & 172 \\
AGM (practical) & \textbf{27} & \textbf{30} & \textbf{73} \\
\bottomrule
\end{tabular}
\end{table}

\vspace{-0.5em}

Figure~\ref{fig:agm-parameter-choices} and
Table~\ref{tab:agm-parameter-choices} show that the practical schedule needs
the fewest gradient queries at all three targets. At relative gap
\(10^{-6}\), Theorem~\ref{thm:full-gradient-main} and
Proposition~\ref{prop:l0l1-full-gradient-rate} require \(28{,}862\) and
\(172\) queries, respectively, compared with \(73\) for the practical
schedule. All nine runs reach their prescribed targets. The specialized
parameter choice reduces the conservatism on this instance, and the
practical rule improves the observed convergence further. This supports its
use in the main experiments; the comparison concerns finite-instance
behavior of the full parameterizations.

\subsection{Sensitivity to the segment-search tolerance}
\label{app:search-tolerance}

The advantage of AGM in total oracle calls partly reflects the inner
segment search of AGMsDR. We test how much of that difference depends on
the search tolerance, using
\(\tau\in\{10^{-4},10^{-6},10^{-8},10^{-12}\}\) on Bike Sharing and
Household Power at \(b^0=4\) and relative target \(10^{-6}\).
The tolerance controls both interval width and endpoint probes, so this
experiment varies the complete implemented search rule; details appear in
Appendix~\ref{app:experimental-implementation}.

All runs use one OpenBLAS thread. After one warm-up per variant, we record
seven timing repetitions on Bike Sharing and three on Household Power.
We remeasure all tolerance settings and the practical AGM reference in a
separate batch with rotating execution order. Thus these runtimes describe
this common batch, whereas Table~\ref{tab:real-data-full-results} reports
the original timing batch. Oracle accounting follows
Appendix~\ref{app:real-data-full-results}. Every run reaches the target.

\begin{figure}[H]
  \centering
  \includegraphics[width=0.96\linewidth]{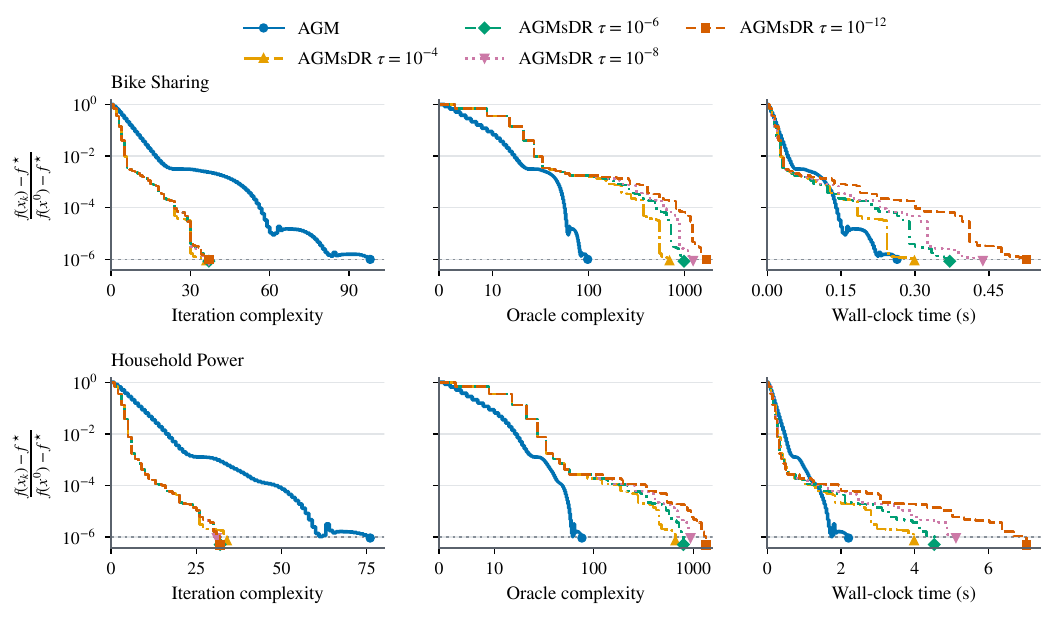}
  \caption{\small Sensitivity of AGMsDR to the search tolerance on Bike
  Sharing (top) and Household Power (bottom), \(b^0=4\).
  Relative objective gap versus iterations, oracle calls, and wall-clock
  time. The dashed line marks the target \(10^{-6}\).}
  \label{fig:agmsdr-tolerance}
\end{figure}

\begin{table}[H]
\color{revisionpurple}
\centering
\caption{Sensitivity to the AGMsDR search tolerance at \(b^0=4\) and
relative gap \(10^{-6}\). Times are medians [IQR] over seven repetitions
for Bike Sharing and three for Household Power. Counts are deterministic.
Bold marks the lowest total oracle count and median time for each dataset.}
\label{tab:agmsdr-tolerance}
\small
\setlength{\tabcolsep}{5pt}
\renewcommand{\arraystretch}{0.95}
\begin{tabular}{@{}lrrrr@{}}
\toprule
Method / tolerance & Iter. & Grad./Val. & Total & Time (s) \\
\midrule
\multicolumn{5}{@{}l}{\textbf{Bike Sharing}} \\
\midrule
AGM (ours) & 98 & 98/0 & \textbf{98} & \textbf{0.264 [0.237--0.310]} \\
AGMsDR (\(\tau=10^{-4}\)) & 36 & 36/661 & 697 & 0.299 [0.292--0.329] \\
AGMsDR (\(\tau=10^{-6}\)) & 37 & 37/950 & 987 & 0.371 [0.337--0.399] \\
AGMsDR (\(\tau=10^{-8}\)) & 37 & 37/1190 & 1227 & 0.438 [0.383--0.449] \\
AGMsDR (\(\tau=10^{-12}\)) & 37 & 37/1646 & 1683 & 0.526 [0.483--0.600] \\
\midrule
\multicolumn{5}{@{}l}{\textbf{Household Power}} \\
\midrule
AGM (ours) & 76 & 76/0 & \textbf{76} & \textbf{2.205 [2.168--2.213]} \\
AGMsDR (\(\tau=10^{-4}\)) & 34 & 34/628 & 662 & 3.981 [3.967--4.311] \\
AGMsDR (\(\tau=10^{-6}\)) & 32 & 32/765 & 797 & 4.528 [4.468--4.536] \\
AGMsDR (\(\tau=10^{-8}\)) & 31 & 31/908 & 939 & 5.117 [5.108--5.162] \\
AGMsDR (\(\tau=10^{-12}\)) & 32 & 32/1316 & 1348 & 7.025 [7.001--7.031] \\
\bottomrule
\end{tabular}
\end{table}

Relaxing \(\tau\) from \(10^{-12}\) to \(10^{-4}\) reduces the total
oracle count of AGMsDR from \(1683\) to \(697\) on Bike Sharing and from
\(1348\) to \(662\) on Household Power. AGM uses \(98\) and \(76\)
queries, respectively, and has the lowest total count at every tested
tolerance, despite using more outer iterations. On Bike Sharing, the
runtimes at \(\tau=10^{-4}\) are comparable: the medians are \(0.264\)
seconds for AGM and \(0.299\) seconds for AGMsDR, with overlapping IQRs.
On Household Power, AGM has a lower observed median time, \(2.205\)
versus \(3.981\) seconds at this tolerance. The clearest advantage on Bike
Sharing is therefore fewer oracle calls at comparable runtime; on
Household Power the lower count is accompanied by a lower median time.

\subsection{Implementation details}
\label{app:experimental-implementation}

\paragraph{Practical AGM.}
The schedule uses gradient memory and secant curvature, with
\emph{no prescribed \(L_0,L_1,H_0,H_1\)}.
Write \(B_A(t):=(t+\sqrt{t^2+4At})/2\),
\(g_k:=\nabla f(y_k)\), and \(t_k:=b_k^2/(A_k+b_k)\).
Set \(c_k=4\), \(A_0=(1+\norm{g_0})^{-1}\),
\(\widehat t_0=\min\{A_0,0.15/\norm{g_0}\}\), and, for \(k\ge1\),
\(\widehat t_k=\min\{\tfrac98t_{k-1},s_k,0.15/\norm{g_{k-1}}\}\).
Here \(s_1=+\infty\), and, for \(k\ge2\),
\(s_k=\inner{g_{k-1}-g_{k-2}}{y_{k-1}-y_{k-2}}/
(2\norm{g_{k-1}-g_{k-2}}^2)\) if the numerator is positive;
otherwise \(s_k=s_{k-1}\). For every \(k\ge0\), the query weight and
accepted weight are
\[
 \widehat b_k=B_{A_k}(\widehat t_k),
 \qquad
 b_k=\min\left\{\widehat b_k,B_{A_k}\!\left(\frac{0.15}{\norm{g_k}}\right)\right\}.
\]
The query uses \(\widehat b_k\); after evaluating \(g_k\), the update uses
\(b_k\). Step caps take \(1/0=+\infty\).
Since \(y_0=x^0\), \(g_0\) serves both initialization and the first iteration.
Thus there is one gradient query per iteration and no function values or
inner search. The cap gives
\[
 a_k=(A_k+b_k)(1-4t_k^2\norm{g_k}^2)\ge0.91(A_k+b_k)>0.
\]
The convergence rates in the paper concern their prescribed theoretical
schedules.

\paragraph{Practical APCM.}
Both practical variants require \emph{no prescribed \(L_i,H_{0,i},H_{1,i}\)}.
For each \(i\), query \(\partial_i f(0)\) and
\(\partial_i f(10^{-3}e_i)\) once. This initial pool costs \(2d\) counted and
timed coordinate-oracle calls and yields the empirical curvatures
\(D_i=\max\{10^{-12},|\partial_i f(10^{-3}e_i)-\partial_i f(0)|/10^{-3}\}\)
and the stored values \(\partial_i f(0)\) for derivative memory.
Fix \(p_i=1/d\) for uniform sampling or
\(p_i=\sqrt{D_i}/\sum_j\sqrt{D_j}\) for empirical importance sampling;
these probabilities remain fixed after initialization. Both variants use
the same pool to calibrate the initial weight and step scale. Set
\(c_k=4/p_{i_k}^2\), \(\bar t=(\max_i D_i/p_i^2)^{-1}\),
\(m_i=|\partial_i f(0)|/p_i^2\), and
\(t_{-1}=A_0=\min\{\bar t/[4(1+R)^2],0.10/\max_i m_i\}\).
At \(k\ge0\), use
\(\widehat b_k=B_{A_k}(\min\{\tfrac98t_{k-1},\bar t,0.10/\max_i m_i\})\).
After querying \(h_k=\partial_{i_k}f(y_k)\), set
\(b_k=\min\{\widehat b_k,B_{A_k}(0.15p_{i_k}^2/|h_k|)\}\),
\(t_k=b_k^2/(A_k+b_k)\), and \(m_{i_k}=|h_k|/p_{i_k}^2\), leaving other
entries unchanged. Exact projection onto \(\mathcal B(x^0,3R)\) ensures
\(a_k\ge0.91(A_k+b_k)>0\). These rules use one partial derivative per iteration after initialization,
with no function values or inner search.

\paragraph{Baseline settings and radii.}
Table~\ref{tab:baseline-settings} records the fixed choices used in the
comparisons. On Bike Sharing and Household Power, the full-gradient
baselines use \((L_0,L_1)=(2.01,2)\); all methods share the same initial
point and reference objective value.

\begin{table}[H]
\color{revisionpurple}
\centering
\caption{Settings of the baseline implementations. The radius interval
\([R_-,R_+]\) bounds the initial distance to the minimizer.}
\label{tab:baseline-settings}
\small
\renewcommand{\arraystretch}{1.10}
\begin{tabular}{@{}p{0.20\linewidth}p{0.74\linewidth}@{}}
\toprule
Method & Fixed parameters and initialization \\
\midrule
AGMsDR \citep{vankov2024l0l1}
 & Initial accumulated weight \(0\), \(x_0=v_0=x^0\);
 golden-section search with \(\tau=10^{-12}\), at most \(64\) contractions
 (varied in Appendix~\ref{app:search-tolerance}). \\
\(\ell\)-AGD \citep{tyurin2025near}
 & \(\ell(s)=L_0+L_1s\), \(\gamma_0=2\Delta_0/R_-^2\);
 the update uses the upper bound \(R_+\). \\
Accelerated GRAAL \citep{borodich2026nesterov}
 & \(\theta=4\), \(\gamma=1/8\), \(\nu=1/162\),
 \(\eta_0=\exp(-L_1R_+)/L_0\). \\
ACDM \citep{Nesterov_2017}
 & \(p_i=\sqrt{L_i}/\sum_j\sqrt{L_j}\), with exact global \(L_i\)
 from Appendix~\ref{app:coordinate-experiment}; strong convexity
 \(0.02\) for capped cosh and \(0\) for Bike Sharing. \\
Frozen-FGM
 & \(M_{\rm frozen}=H_0+H_1\Delta_0\approx1.01\) on
 \texttt{covtype.binary}, step \(1/M_{\rm frozen}\), and initial
 momentum parameter \(1\). \\
\bottomrule
\end{tabular}
\end{table}

For the two generalized linear models, we eliminate the intercept
analytically to obtain a \(0.01\)-strongly convex objective in the slopes.
Under our feature normalization its optimal intercept is \(1\)-Lipschitz
in the slopes. If \(r_w\) is the profiled gradient residual at the reference
slopes and \(\bar r\) is the mean exponential rate at the reference point,
we use the reference-error bound
\(\delta_R=|\log\bar r|+\sqrt2\norm{r_w}/0.01+10^{-10}\), including a
numerical margin. With \(R_{\rm ref}=\norm{x^0-x_{\rm ref}}\), set
\(R_-=R_{\rm ref}-\delta_R\) and \(R_+=R_{\rm ref}+\delta_R\).
At \(b^0=4\), the resulting \((\gamma_0,\eta_0)\) are approximately
\((3.23735,7.59989\cdot10^{-6})\) on Bike Sharing and
\((5.33227,8.72776\cdot10^{-5})\) on Household Power.

\paragraph{Segment search and counting.}
The AGMsDR search parameterizes the segment by \(\beta\in[0,1]\).
It contracts an interior bracket to width \(\tau\) and probes endpoints
at distance \(\max\{\sqrt\tau,\sqrt{\varepsilon_{\rm mach}}\}\);
endpoint returns can bypass interval contraction. The same rule and the
limit of \(64\) contractions are used for every \(\tau\).
For the real-data models we cache two design projections per nondegenerate segment,
then evaluate scalar function values without additional gradients.
The step from the query point uses
\[
 \alpha_k=\frac{1}{L_1\norm{g_k}}
 \log\!\left(1+\frac{L_1\norm{g_k}}{L_0+L_1\norm{g_k}}\right),
 \qquad g_k=\nabla f(y_k),
\]
with the continuous value \(1/L_0\) when \(L_1\norm{g_k}=0\).
All endpoint probes, golden-section values, and values required by the
weight update are included in the oracle count. External gap monitoring
is excluded from this count and included in the reported checkpoint times.
Counts are deterministic for all full-gradient runs. The timing curves in
Figure~\ref{fig:agmsdr-tolerance} use median cumulative times at each
iteration; table times are the corresponding final-checkpoint medians.

\paragraph{Smoothness certificates.}
For the capped-cosh function,
\(\norm{\nabla^2f(x)}\le1+2f(x)\) and
\(\norm{\nabla^2f(x)}\le1+\sqrt2\norm{\nabla f(x)}\).
Moreover, \(\log\phi_i''\) is \(\sqrt{H_{1,i}}\)-Lipschitz.
Consequently the curvature increases by less than a factor of two along
either prescribed local radius, and Taylor's integral formula gives the
certificates used in Appendix~\ref{app:parameter-choices}.

\end{purplerevision}

\begingroup
\section{Full-Gradient Rates without an Initial Gap Bound}
\label{app:radius-full-gradient}

Theorem~\ref{thm:full-gradient-main} initializes the accelerated gradient
method using an upper bound on \(\F(x^0)=f(x^0)-f^*\). Suppose instead
that the initial point is known to satisfy \(\norm{x^0-x^*}\le R\).
After \(J\) normalized steps, Algorithm~\ref{alg:radius-normalized}
either returns \(w_J\) with the requested accuracy or passes \(w_J\)
and its gradient to Algorithm~\ref{alg:radius-full-gradient}. Together the
two algorithms yield a dependence of order
\((H_1R^2)^{3/5}\) under Assumption~\ref{ass:h0h1-smoothness}. For the
gradient-dependent model in \eqref{eq:local-l0l1-function-model}, a different
choice of step sizes yields a dependence of order \(L_1R\). Both bounds count
full-gradient queries. For deterministic first-order methods in growing
dimension, \(3/5\) is the optimal worst-case exponent of \(H_1\) in a
bound independent of the initial gap. As in Theorem~\ref{thm:growth-lower-bounds},
set \(H_0:=\varepsilon/R^2\) to isolate the dependence on \(H_1\).
The lower bound in Theorem~\ref{radlb:theorem} then
matches the \((H_1R^2)^{3/5}\) term in Theorem~\ref{thm:rad-upper}
as \(H_1R^2\) grows.

\subsection{Accelerated Gradient Method from a Distance Bound}

Algorithm~\ref{alg:radius-normalized} takes steps of length \(r/2\)
until the requested accuracy is certified or the gradient norm is at most
\(G_{\mathrm{switch}}\), the threshold for switching to acceleration.
Theorem~\ref{thm:rad-upper} and Proposition~\ref{prop:rad-l-upper} give
\(r\), \(G_{\mathrm{switch}}\), and the accelerated weights for each smoothness model.

\begingroup
\color{black}
\setlength{\intextsep}{4pt}
\begin{algorithm}[H]
\caption{Normalized Gradient Initialization}
\label{alg:radius-normalized}
\begin{algorithmic}[1]
\renewcommand{\algorithmicrequire}{\textbf{Input:}}
\renewcommand{\algorithmicensure}{\textbf{Return:}}
\REQUIRE Objective \(f\), initial point \(x^0\), radius \(R\), accuracy \(\varepsilon\),
 step length \(r/2\), and gradient-norm threshold \(G_{\mathrm{switch}}\)
\STATE Initialize \(w_0\gets x^0\)
\FOR{\(j=0,1,\ldots\)}
  \STATE Evaluate \(\nabla f(w_j)\)
  \STATE If \(R\norm{\nabla f(w_j)}\le\varepsilon\), \textbf{return} \(w_j\) as the final point
  \STATE If \(\norm{\nabla f(w_j)}\le G_{\mathrm{switch}}\), \textbf{break}
  \STATE Update \(w_{j+1}\gets w_j-\dfrac r2\,\dfrac{\nabla f(w_j)}{\norm{\nabla f(w_j)}}\)
\ENDFOR
\ENSURE On switching, \(w_J\) and \(\nabla f(w_J)\) for Algorithm~\ref{alg:radius-full-gradient}, with \(J=j\)
\end{algorithmic}
\end{algorithm}
\endgroup

\citet[Theorem~4.1]{vankov2024l0l1} analyze normalized gradient descent
with a fixed step length for a prescribed iteration budget. The
\href{https://arxiv.org/html/2410.10800v1#S6}{two-stage method in their original preprint}
uses gradient-dependent step lengths until the gap is at most \(L_0/(5L_1^2)\),
then starts AGMsDR. \citet[Algorithm~1]{tyurin2025near} also initializes
acceleration by running gradient descent to a prescribed gap. In the
stochastic setting, \citet[Corollary~1]{elbakkali2026acceleration} use clipped
stochastic gradients before acceleration, rescaling samples whose norm
exceeds the clipping threshold.

In Algorithm~\ref{alg:radius-normalized}, each step has length \(r/2\),
and the observed gradient norm determines when to switch. Each step
decreases the squared distance to \(x^*\) by at least \(r^2/4\), so
\(\norm{w_J-x^*}\le R\). Convexity then gives the computable bound
\(\F(w_J)\le R\norm{\nabla f(w_J)}\). If acceleration is needed,
Algorithm~\ref{alg:radius-full-gradient} uses this bound to set its
initial weight \(A_0=R/(2\norm{\nabla f(w_J)})\).

\begingroup
\color{black}
\setlength{\intextsep}{4pt}
\begin{algorithm}[H]
\caption{Accelerated Gradient Method Initialized at \(w_J\)}
\label{alg:radius-full-gradient}
\begin{algorithmic}[1]
\renewcommand{\algorithmicrequire}{\textbf{Input:}}
\renewcommand{\algorithmicensure}{\textbf{Return:}}
\REQUIRE Objective \(f\), point \(w_J\) and gradient \(\nabla f(w_J)\) from
 Algorithm~\ref{alg:radius-normalized}, radius \(R\), accuracy \(\varepsilon\),
 and rules for the query and update weights \(\widehat b_k,b_k\)
\STATE Initialize \(x_0\gets w_J\), \(v_0\gets w_J\)
\STATE Set \(A_0\gets\dfrac{R}{2\norm{\nabla f(w_J)}}\) \label{eq:rad-start}
\FOR{\(k=0,1,\ldots\)}
  \STATE Set \(y_k\gets\dfrac{A_kx_k+\widehat b_k v_k}{A_k+\widehat b_k}\)
  \STATE At \(k=0\), use the supplied \(\nabla f(w_J)\). At \(k>0\), evaluate \(\nabla f(y_k)\)
  \STATE Update \(v_{k+1}\gets v_k-b_k\nabla f(y_k)\)
  \STATE Update \(A_{k+1}\gets A_k+b_k/2\)
  \STATE Update \(x_{k+1}\gets\dfrac{A_kx_k+b_kv_k}{A_k+b_k}
       -\dfrac{b_k^2}{A_{k+1}}\nabla f(y_k)\) \label{eq:rad-update}
  \STATE If \(A_{k+1}\ge R^2/\varepsilon\), \textbf{break}
\ENDFOR
\ENSURE \(x_{k+1}\)
\end{algorithmic}
\end{algorithm}
\endgroup
\vspace{-0.6em}

\paragraph*{Comparison with Algorithm~\ref{alg:full-gradient-inner}.}
Both methods use the same formulas for \(y_k\) and \(v_{k+1}\).
Algorithm~\ref{alg:radius-full-gradient} initializes \(A_0\) using
\(R\norm{\nabla f(w_J)}\) in place of \(\Delta_0\) and reuses the gradient
at \(w_J=y_0\) for its first step. Theorem~\ref{thm:full-gradient-main}
sets \(b_k=\widehat b_k\) before the query. Here \(\widehat b_k\) fixes
\(y_k\), and the observed gradient norm determines whether \(b_k\) must
be reduced to satisfy the local model bound. Taking \(A_{k+1}=A_k+b_k/2\)
leaves the term \(-b_k\F(y_k)/2\) in \eqref{eq:rad-coupled} to absorb
the gap-dependent curvature error. Since \(\F(x_k)\le R^2/A_k\), the method
stops when \(A_k\ge R^2/\varepsilon\).

\subsection{Analysis under \texorpdfstring{\((H_0,H_1)\)}{(H0,H1)}-Smoothness}
\label{sec:rad-h-rate}

\begin{mainresultbox}
\begin{theorem}[Full-gradient complexity under \((H_0,H_1)\)-smoothness]
\label{thm:rad-upper}
Suppose Assumptions~\ref{ass:convexity} and~\ref{ass:h0h1-smoothness}
hold with \(H_0>0\) and \(H_1\ge0\). Let \(x^*\) be a minimizer with
\(\norm{x^0-x^*}\le R\), where \(R>0\), and let \(\varepsilon>0\).
In Algorithm~\ref{alg:radius-normalized}, set \(r=R(1+R\sqrt{H_1})^{-3/5}\)
and \(G_{\mathrm{switch}}=2H_0r\exp(2\sqrt{H_1}\,r)\).
In Algorithm~\ref{alg:radius-full-gradient}, set
\(h=(1+R\sqrt{H_1})^{-4/5}/256\) and choose the weights
\begin{equation}
\label{eq:rad-h-weights}
 \widehat b_k=\min\!\left\{\frac{hA_k}{1-h},
                         \frac{1+\sqrt{1+64H_0A_k}}{32H_0}\right\},
 \qquad
 b_k=\min\!\left\{\widehat b_k,
              \frac{1}{32\sqrt{H_1\eta_k}\,\norm{\nabla f(y_k)}}\right\},
\end{equation}
where \(\eta_k=\widehat b_k/(A_k+\widehat b_k)\). If \(H_1=0\) or \(\nabla f(y_k)=0\), take
\(b_k=\widehat b_k\).
Algorithm~\ref{alg:radius-normalized}, followed upon switching by
Algorithm~\ref{alg:radius-full-gradient}, returns an \(\varepsilon\)-optimal point.
The total number \(N\) of full-gradient queries across both algorithms satisfies
\begin{equation}
 \label{eq:rad-rates}
 \boxed{\displaystyle
 N=O\!\left(1+\sqrt{\frac{H_0R^2}{\varepsilon}}+(H_1R^2)^{3/5}\right)} .
\end{equation}
\end{theorem}
\end{mainresultbox}

\begin{proof}[Proof of Theorem~\ref{thm:rad-upper}]
We keep the step length \(r/2\) and the upper bound \(h\) on \(\eta_k\)
free while counting queries, with \(0<r\le R\) and \(0<h<1\).
Lemma~\ref{lem:rad-curvature-remainder} will require
\(64\sqrt2\,R\sqrt{H_1}\,h^{3/2}\le1/8\). We will determine \(r\) and
\(h\) from the query counts for the two algorithms and then verify this condition.

\paragraph*{Normalized initialization.}
Suppose Algorithm~\ref{alg:radius-normalized} moves from \(w_j\) to \(w_{j+1}\). The accuracy test
has already excluded \(\nabla f(w_j)=0\), so the algorithm moves a distance
\(r/2\) in direction \(-\nabla f(w_j)/\norm{\nabla f(w_j)}\). To see how
much this decreases the distance to \(x^*\), we need the component of
\(w_j-x^*\) along the gradient direction. Evaluate the supporting
hyperplane at
\[
 x^*+r\,\frac{\nabla f(w_j)}{\norm{\nabla f(w_j)}} .
\]
Compared with its value at \(x^*\), the hyperplane gains exactly
\(r\norm{\nabla f(w_j)}\). The function value at the comparison point can
also be bounded, since the point lies at distance \(r\) from a minimizer.
This gives
\[
\begin{aligned}
 \inner{\frac{\nabla f(w_j)}{\norm{\nabla f(w_j)}}}{w_j-x^*}
 &\overset{\circledOne}{\ge}
    r+\frac{\F(w_j)-
     \F\bigl(x^*+r\,\nabla f(w_j)/\norm{\nabla f(w_j)}\bigr)}
     {\norm{\nabla f(w_j)}}\\
 &\overset{\circledTwo}{\ge}
    r-\frac{H_0r^2\exp(2\sqrt{H_1}\,r)}{\norm{\nabla f(w_j)}}\\
 &\overset{\circledThree}{>}r/2 ,
\end{aligned}
\]
where \(\circledOne\) uses the tangent-plane bound
\(\F(z)\ge\F(w_j)+\inner{\nabla f(w_j)}{z-w_j}\) for a convex function.
At the displayed comparison point the inner product equals
\(\inner{\nabla f(w_j)}{x^*-w_j}+r\norm{\nabla f(w_j)}\).
Rearranging and dividing by \(\norm{\nabla f(w_j)}\) gives \(\circledOne\).
Next, \(\circledTwo\) uses \(\F(w_j)\ge0\) and Lemma~\ref{lem:rad-envelope}
at the comparison point. For \(\circledThree\), the subtracted fraction
is less than \(r/2\) precisely when
\(\norm{\nabla f(w_j)}>2H_0r\exp(2\sqrt{H_1}\,r)\). This gives
\(G_{\mathrm{switch}}=2H_0r\exp(2\sqrt{H_1}\,r)\).

Expanding the squared distance after this update yields
\begin{equation}
\label{eq:rad-normalized-distance}
\begin{aligned}
 \norm{w_{j+1}-x^*}^2
 &\overset{\circledOne}{=}\norm{w_j-x^*}^2
   -r\inner{\frac{\nabla f(w_j)}{\norm{\nabla f(w_j)}}}{w_j-x^*}+r^2/4\\
 &\overset{\circledTwo}{\le}\norm{w_j-x^*}^2-r^2/4 ,
\end{aligned}
\end{equation}
where \(\circledOne\) expands the square after a step of length \(r/2\).
For \(\circledTwo\), the inner product exceeds \(r/2\), so the cross
term subtracts at least \(r^2/2\). The squared step length adds back
only \(r^2/4\). Thus each step reduces the squared distance by at least
\(r^2/4\).

The first step starts at \(w_0=x^0\). After two steps the total
decrease is at least \(2(r^2/4)\), since the second starts where the
first ended. When we add the decreases over \(j\) steps, each intermediate
squared distance occurs once with each sign. Only the distances at \(w_0\)
and \(w_j\) remain, giving
\begin{equation}
\label{eq:rad-normalized-radius}
\begin{aligned}
 \norm{w_j-x^*}^2
 &\overset{\circledOne}{\le}\norm{x^0-x^*}^2-\frac{jr^2}{4}\\
 &\overset{\circledTwo}{\le}R^2-\frac{jr^2}{4},
\end{aligned}
\end{equation}
where \(\circledOne\) sums \eqref{eq:rad-normalized-distance} and
\(\circledTwo\) uses the initial radius. For \(j=0\) this is just the
initial distance bound. Since \(jr^2/4\ge0\), taking square roots gives
\(\norm{w_j-x^*}\le R\) at every tested point. Also, the left-hand
side is nonnegative, so \(R^2-jr^2/4\) cannot be negative. Thus at
most \(4R^2/r^2\) steps can be taken. At any tested point,
\[
 \F(w_j)
 \overset{\circledOne}{\le}\inner{\nabla f(w_j)}{w_j-x^*}
 \overset{\circledTwo}{\le}R\norm{\nabla f(w_j)} ,
\]
where \(\circledOne\) is convexity at \(w_j\) evaluated at \(x^*\), and
\(\circledTwo\) is Cauchy--Schwarz and the radius bound. Taking \(j=J\)
in \eqref{eq:rad-normalized-radius} gives \(\norm{w_J-x^*}\le R\).
If \(R\norm{\nabla f(w_J)}\le\varepsilon\), the point returned by
Algorithm~\ref{alg:radius-normalized} therefore meets the accuracy
target. Otherwise the normalized phase must have stopped
because \(\norm{\nabla f(w_J)}\le G_{\mathrm{switch}}\), giving
\begin{equation}
 \label{eq:rad-switch-gap}
 \F(w_J)\overset{\circledOne}{\le}R\norm{\nabla f(w_J)}
       \overset{\circledTwo}{\le} 2H_0Rr\exp(2\sqrt{H_1}\,r),
\end{equation}
where \(\circledOne\) is the gap bound just proved and \(\circledTwo\)
is the stopping test multiplied by \(R\).

In either case, the total number \(J\) of normalized steps satisfies
\begin{equation}
 \label{eq:rad-h-init-count}
 J\le\frac{4R^2}{r^2},
\end{equation}
by \eqref{eq:rad-normalized-radius}.
Algorithm~\ref{alg:radius-normalized} makes exactly \(J+1\) gradient queries, one at each of
\(w_0,\ldots,w_J\), including the point \(w_J\) where it stops.
Increasing \(r\) shortens this loop but increases the exponent
\(2\sqrt{H_1}\,r\) in the gap bound \eqref{eq:rad-switch-gap}.
The balance in \eqref{eq:rad-h-balance} below gives
\(J\le4(1+R\sqrt{H_1})^{6/5}\).

\paragraph*{Decrease of the accelerated potential.}
For the accelerated iterates use the potential from the proof of
Theorem~\ref{thm:full-gradient-main},
\[
 \Phi_k=A_k\F(x_k)+\tfrac12\norm{v_k-x^*}^2 .
\]
Both terms are nonnegative. Keeping their sum below \(R^2\) bounds the
gap at \(x_k\) by \(R^2/A_k\) and keeps \(v_k\) within \(\sqrt2 R\) of
\(x^*\). Thus it is enough to decrease this sum while increasing \(A_k\).
At \(x_0=v_0=w_J\), the distance term can already be \(R^2/2\).
To guarantee \(\Phi_0\le R^2\), we can therefore spend \(R^2/2\) on
the gap term. The largest initial weight certified by
\eqref{eq:rad-switch-gap} is \(A_0=R/(2\norm{\nabla f(w_J)})\), giving
\[
 \Phi_0\overset{\circledOne}{\le}A_0 R\norm{\nabla f(w_J)}+R^2/2
       \overset{\circledTwo}{=}R^2 ,
\]
where \(\circledOne\) uses the gap bound \eqref{eq:rad-switch-gap} and
the radius bound \eqref{eq:rad-normalized-radius} at \(w_J\), and \(\circledTwo\) substitutes
\(A_0=R/(2\norm{\nabla f(w_J)})\). We prove \(\Phi_k\le R^2\) and
\(\norm{x_k-x^*}\le3\sqrt2 R\) together by induction. At \(k=0\),
\(x_0=w_J\), so the radius just proved gives
\(\norm{x_0-x^*}\le R\le3\sqrt2 R\). Now assume both bounds hold at
iteration \(k\). The distance in \(\Phi_{k+1}\) is measured at
\(v_{k+1}\). To bound the distance at \(x_{k+1}\) as well, we will
first prove the decrease of \(\Phi_{k+1}\), then use the relation
between the two updates.

Fix an accelerated iteration.
The smoothness bound is available at \(y_k\), where we have just queried
the gradient. To use it, we express \(x_{k+1}\) as a weighted average
of \(x_k\), \(y_k\), and a step along this gradient. Before subtracting
the gradient, the average in the update is
\[
 \frac{A_k x_k+b_k v_k}{A_k+b_k}
 =x_k+\frac{b_k}{(A_k+b_k)\eta_k}(y_k-x_k),
\]
since \(y_k-x_k=\eta_k(v_k-x_k)\). Of the weight on \(y_k\), move
\(b_k/(2A_{k+1}\eta_k)\) to the point
\(y_k-2b_k\eta_k\nabla f(y_k)\). This subtracts exactly
\((b_k^2/A_{k+1})\nabla f(y_k)\), as required by the update. The three
weights give the identity
\begin{equation}
\label{eq:rad-interpolation}
\begin{aligned}
 x_{k+1}
 &=\frac{A_k}{A_k+b_k}\left(1-\frac{b_k}{\widehat b_k}\right)x_k\\
 &\quad+\frac{b_k}{\eta_k}
     \left(\frac1{A_k+b_k}-\frac1{2A_{k+1}}\right)y_k\\
 &\quad+\frac{b_k}{2A_{k+1}\eta_k}
     \bigl(y_k-2b_k\eta_k\nabla f(y_k)\bigr).
\end{aligned}
\end{equation}
Indeed, the residual weight on \(x_k\) is
\(1-b_k/((A_k+b_k)\eta_k)=A_k(1-b_k/\widehat b_k)/(A_k+b_k)\).
It is nonnegative because \(b_k\le\widehat b_k\).
The residual weight on \(y_k\) is nonnegative because
\(2A_{k+1}=2A_k+b_k\ge A_k+b_k\). The three weights sum to one.
Convexity therefore gives
\[
\begin{aligned}
 \F(x_{k+1})-\F(y_k)
 &\overset{\circledOne}{\le}
 \frac{A_k}{A_k+b_k}\left(1-\frac{b_k}{\widehat b_k}\right)
       \bigl(\F(x_k)-\F(y_k)\bigr)\\
 &\quad+\frac{b_k}{2A_{k+1}\eta_k}
       \bigl(f(y_k-2b_k\eta_k\nabla f(y_k))-f(y_k)\bigr),
\end{aligned}
\]
where \(\circledOne\) applies convexity to \eqref{eq:rad-interpolation}
and subtracts \(\F(y_k)\) from the resulting weighted average.

We next expand the distance term in the potential. Its change is
\(-b_k\inner{\nabla f(y_k)}{v_k-x^*}
+\tfrac12 b_k^2\norm{\nabla f(y_k)}^2\).
To bound the inner product, write
\(v_k-x^*=(y_k-x^*)+(A_k/\widehat b_k)(y_k-x_k)\). Then
\[
 b_k\inner{\nabla f(y_k)}{v_k-x^*}
 \overset{\circledOne}{\ge}
 b_k\F(y_k)+\frac{b_k A_k}{\widehat b_k}\bigl(\F(y_k)-\F(x_k)\bigr),
\]
where \(\circledOne\) uses the supporting hyperplane at \(y_k\) evaluated
first at \(x^*\) and then at \(x_k\). Multiplying the convexity bound
by \(A_{k+1}\) and adding the change in the distance term now gives
\begin{equation}
\label{eq:rad-coupled}
\begin{aligned}
 \Phi_{k+1}-\Phi_k
 &\overset{\circledOne}{\le}
 \frac{b_k}{2\eta_k}\bigl(f(y_k-2b_k\eta_k\nabla f(y_k))-f(y_k)\bigr)\\
 &\quad+\tfrac12 b_k^2\norm{\nabla f(y_k)}^2-\tfrac{b_k}{2}\F(y_k)\\
 &\quad+\frac{b_k}{2}\frac{A_k}{A_k+b_k}
       \left(1-\frac{b_k}{\widehat b_k}\right)
       \bigl(\F(y_k)-\F(x_k)\bigr).
\end{aligned}
\end{equation}
For \(\circledOne\), write the old gap as
\(\F(x_k)=\F(y_k)+(\F(x_k)-\F(y_k))\) before collecting terms.
The coefficient of \(\F(y_k)\) is \(A_{k+1}-A_k-b_k=-b_k/2\).
The coefficient of \(\F(x_k)-\F(y_k)\) is
\(A_k(1-b_k/\widehat b_k)(A_{k+1}/(A_k+b_k)-1)\), which gives the last line.

It remains to bound the function difference in the first line of
\eqref{eq:rad-coupled}. The step from \(y_k\) has length
\(2b_k\eta_k\norm{\nabla f(y_k)}\). Its quadratic error has one part
with coefficient \(H_0\) and another with coefficient \(H_1\F(y_k)\).
To control the first, impose \(\widehat b_k\eta_k\le1/(16H_0)\), so
\(2H_0b_k\eta_k\le1/8\). Together with \(\eta_k\le h\), this requires
\(16H_0\widehat b_k^2\le A_k+\widehat b_k\) and
\((1-h)\widehat b_k\le hA_k\). Their largest common solution gives
the query weight in \eqref{eq:rad-h-weights}.
For the second part we have a negative term \(-b_k\F(y_k)/2\) in
\eqref{eq:rad-coupled}. After multiplication by \(b_k/(2\eta_k)\),
the coefficient of \(\F(y_k)\) in the quadratic error is
\(2b_k H_1\eta_k(b_k\norm{\nabla f(y_k)})^2\).
Requiring \(\sqrt{H_1\eta_k}\,b_k\norm{\nabla f(y_k)}\le1/32\)
keeps this coefficient at most \(b_k/512\). The largest
\(b_k\le\widehat b_k\) that meets this condition gives the update
weight in \eqref{eq:rad-h-weights}. For \(H_1>0\), it also ensures
that the step is inside the model radius, since
\[
\begin{aligned}
 2b_k\eta_k\norm{\nabla f(y_k)}
 &\overset{\circledOne}{\le}\frac{\sqrt{\eta_k}}{16\sqrt{H_1}}\\
 &\overset{\circledTwo}{\le}\frac{1}{(2\sqrt3+\sqrt6)\sqrt{H_1}},
\end{aligned}
\]
where \(\circledOne\) uses the constraint on \(b_k\) and \(\circledTwo\)
uses \(\eta_k\le1\) and \(2\sqrt3+\sqrt6<16\). For \(H_1=0\) the
model holds globally. Applying it to the first line of
\eqref{eq:rad-coupled} gives
\[
\begin{aligned}
 \frac{b_k}{2\eta_k}\bigl(f(y_k-2b_k\eta_k\nabla f(y_k))-f(y_k)\bigr)
 &\overset{\circledOne}{\le}-b_k^2\norm{\nabla f(y_k)}^2\\
 &\quad+2b_k^3\eta_k\bigl(H_0+H_1\F(y_k)\bigr)\norm{\nabla f(y_k)}^2\\
 &\overset{\circledTwo}{\le}-\tfrac78 b_k^2\norm{\nabla f(y_k)}^2
       +\tfrac{b_k}{2}\F(y_k),
\end{aligned}
\]
where \(\circledOne\) uses Assumption~\ref{ass:h0h1-smoothness} at \(y_k\),
and \(\circledTwo\) uses the two coefficient bounds just proved.
The last gap term cancels against \(-b_k\F(y_k)/2\) in \eqref{eq:rad-coupled}.
Thus
\begin{equation}
\label{eq:rad-potential-decrease}
\begin{aligned}
 \Phi_{k+1}-\Phi_k
 &\overset{\circledOne}{\le}-\tfrac38 b_k^2\norm{\nabla f(y_k)}^2\\
 &\quad+\frac{b_k}{2}\frac{A_k}{A_k+b_k}
       \left(1-\frac{b_k}{\widehat b_k}\right)
       \bigl(\F(y_k)-\F(x_k)\bigr)\\
 &\overset{\circledTwo}{\le}-\tfrac14 b_k^2\norm{\nabla f(y_k)}^2 ,
\end{aligned}
\end{equation}
where \(\circledOne\) combines the coefficients \(-7/8\) and \(1/2\).
For \(\circledTwo\), the gap difference has coefficient zero when
\(b_k=\widehat b_k\). When \(b_k<\widehat b_k\),
Lemma~\ref{lem:rad-curvature-remainder} bounds the difference by
\(b_k\norm{\nabla f(y_k)}^2/4\). Its coefficient lies between zero and \(b_k/2\),
so the contribution is at most \(b_k^2\norm{\nabla f(y_k)}^2/8\).
The lemma uses the induction bound on \(x_k\) and
\(\norm{v_k-x^*}\le\sqrt2 R\), which follows from \(\Phi_k\le R^2\).

It remains to show that \(x_{k+1}\) stays within \(3\sqrt2 R\) of \(x^*\).
We first locate the endpoints of the step from \(v_k\) to \(v_{k+1}\).
For \(i=k\) and \(i=k+1\),
\[
\begin{aligned}
 \tfrac12\norm{v_i-x^*}^2
 &\overset{\circledOne}{\le}\Phi_i\\
 &\overset{\circledTwo}{\le}R^2,
\end{aligned}
\]
where \(\circledOne\) drops the nonnegative gap term \(A_i\F(x_i)\)
from the potential, and \(\circledTwo\) uses the induction hypothesis
at \(k\) and the decrease \eqref{eq:rad-potential-decrease} at \(k+1\).
Multiplying by two and taking square roots puts each endpoint within
\(\sqrt2 R\) of \(x^*\). The distance between the endpoints is at most
the sum of their distances to \(x^*\), so
\[
\begin{aligned}
 \norm{v_{k+1}-v_k}
 &\overset{\circledOne}{\le}\norm{v_{k+1}-x^*}+\norm{v_k-x^*}\\
 &\overset{\circledTwo}{\le}2\sqrt2 R ,
\end{aligned}
\]
where \(\circledOne\) applies the triangle inequality through \(x^*\)
and \(\circledTwo\) uses the two endpoint bounds.

To see where \(x_{k+1}\) lies, rewrite its update using these endpoints.
\[
\begin{aligned}
 x_{k+1}
 &\overset{\circledOne}{=}\frac{A_k x_k+b_k v_k}{A_k+b_k}
       +\frac{b_k}{A_{k+1}}(v_{k+1}-v_k)\\
 &\overset{\circledTwo}{=}\frac{A_k}{A_k+b_k}x_k+
 \frac{b_k}{A_k+b_k}
 \left(v_{k+1}+\frac{b_k}{2A_{k+1}}(v_{k+1}-v_k)\right),
\end{aligned}
\]
where \(\circledOne\) substitutes
\(-b_k\nabla f(y_k)=v_{k+1}-v_k\) in the update for \(x_{k+1}\).
For \(\circledTwo\), factor out \(b_k/(A_k+b_k)\) from the terms
involving \(v_k,v_{k+1}\), then use
\((A_k+b_k)/A_{k+1}=1+b_k/(2A_{k+1})\).
To reach the point in parentheses, start at \(v_{k+1}\) and continue in
the direction \(v_{k+1}-v_k\). The extra distance is at most one more
step of the same length, since \(b_k/(2A_{k+1})=b_k/(2A_k+b_k)\le1\).
Its distance from \(x^*\) therefore satisfies
\[
\begin{aligned}
 \left\|v_{k+1}+\frac{b_k}{2A_{k+1}}(v_{k+1}-v_k)-x^*\right\|
 &\overset{\circledOne}{\le}
   \norm{v_{k+1}-x^*}+\frac{b_k}{2A_{k+1}}\norm{v_{k+1}-v_k}\\
 &\overset{\circledTwo}{\le}\sqrt2 R+2\sqrt2 R\\
 &=3\sqrt2 R ,
\end{aligned}
\]
where \(\circledOne\) is the triangle inequality. In \(\circledTwo\),
\(\sqrt2 R\) bounds the distance from \(x^*\) to \(v_{k+1}\) and
\(2\sqrt2 R\) bounds the extra distance from \(v_{k+1}\).

Finally, \(x_{k+1}\) lies on the segment from \(x_k\) to the point
in parentheses. The induction hypothesis bounds the distance to \(x^*\)
at the first endpoint, and the preceding calculation bounds it at the
second. Subtracting \(x^*\) from the weighted average gives
\[
\begin{aligned}
 \norm{x_{k+1}-x^*}
 &\overset{\circledOne}{\le}
   \frac{A_k}{A_k+b_k}\,3\sqrt2 R+\frac{b_k}{A_k+b_k}\,3\sqrt2 R\\
 &\overset{\circledTwo}{=}3\sqrt2 R ,
\end{aligned}
\]
where \(\circledOne\) applies the triangle inequality with the two
nonnegative weights and \(\circledTwo\) uses that they sum to one.
This proves the radius bound at \(k+1\) and completes the induction.
In particular, \(\F(x_k)\le R^2/A_k\), since the squared-distance term in
\(\Phi_k\) is nonnegative. Thus \(A_k\ge R^2/\varepsilon\) certifies the
requested accuracy.

\paragraph*{Number of gradient queries.}
Consider any prefix of \(n\) completed
accelerated updates with \(A_n<R^2/\varepsilon\). Summing the decrease gives
\[
 \tfrac14\sum_{k=0}^{n-1}b_k^2\norm{\nabla f(y_k)}^2
 \overset{\circledOne}{\le}\Phi_0-\Phi_n
 \overset{\circledTwo}{\le}R^2 ,
\]
where \(\circledOne\) telescopes \eqref{eq:rad-potential-decrease} and
\(\circledTwo\) uses \(\Phi_0\le R^2\) and \(\Phi_n\ge0\). On a step
where \(b_k<\widehat b_k\), we have \(H_1>0\) and
\[
 b_k^2\norm{\nabla f(y_k)}^2
 \overset{\circledOne}{=}\frac{1}{32^2H_1\eta_k}
 \overset{\circledTwo}{\ge}\frac{1}{32^2H_1h},
\]
where \(\circledOne\) uses the second candidate for \(b_k\) in
\eqref{eq:rad-h-weights}, and \(\circledTwo\) uses \(\eta_k\le h\).
The prefix therefore contains at most \(4R^2\cdot32^2H_1h=4096H_1R^2h\)
reduced weights.

Now suppose \(b_k=\widehat b_k\). If the square-root expression in
\eqref{eq:rad-h-weights} determines \(\widehat b_k\), it is the positive
root of \(16H_0\widehat b_k^2=A_k+\widehat b_k\). Consequently
\begin{equation}
 \label{eq:rad-root-growth}
\begin{aligned}
 \sqrt{A_{k+1}}-\sqrt{A_k}
 &\overset{\circledOne}{=}\frac{\widehat b_k}
               {2(\sqrt{A_{k+1}}+\sqrt{A_k})}\\
 &\overset{\circledTwo}{\ge}\frac{\widehat b_k}{4\sqrt{A_k+\widehat b_k}}\\
 &\overset{\circledThree}{=}\frac{1}{16\sqrt{H_0}},
\end{aligned}
\end{equation}
where \(\circledOne\) uses \(A_{k+1}-A_k=\widehat b_k/2\),
\(\circledTwo\) bounds both square roots in the denominator by
\(\sqrt{A_k+\widehat b_k}\), and \(\circledThree\) substitutes the
quadratic equation for \(\widehat b_k\). The weight increases on every
update, so the sum of these increments is at most
\(\sqrt{A_n}-\sqrt{A_0}\). This gives at most \(16R\sqrt{H_0/\varepsilon}\) such
steps in the prefix.

On each of the remaining fully accepted steps,
\(\widehat b_k=hA_k/(1-h)\), and hence
\begin{equation}
 \label{eq:rad-log-growth}
\begin{aligned}
 \log(A_{k+1}/A_k)
 &\overset{\circledOne}{\ge}1-A_k/A_{k+1}\\
 &\overset{\circledTwo}{=}\frac{h}{2-h}\\
 &\overset{\circledThree}{\ge}h/2 ,
\end{aligned}
\end{equation}
where \(\circledOne\) uses \(\log u\le u-1\) at \(u=A_k/A_{k+1}\),
\(\circledTwo\) substitutes \(A_{k+1}=A_k+\widehat b_k/2\), and
\(\circledThree\) uses \(0<h<1\). Summing the logarithmic increments
gives at most \(2\log(R^2/(\varepsilon A_0))/h\) such steps. A tie between the two
query-weight candidates is included only in the square-root count.

The three counts bound every finite prefix with \(A_n<R^2/\varepsilon\),
so the weight must eventually reach \(R^2/\varepsilon\).
Let \(n\) be the index of the first update that reaches this threshold.
The accelerated loop uses gradients at \(y_0,\ldots,y_n\).
Since \(y_0=w_J\), only the last \(n\) of these require new queries. Thus
\begin{equation}
 \label{eq:rad-count}
 \begin{aligned}
 N&=\underbrace{(J+1)}_{\text{queries in Algorithm~\ref{alg:radius-normalized}}}
    +\underbrace{n}_{\text{new queries in Algorithm~\ref{alg:radius-full-gradient}}}\\
  &\overset{\circledOne}{\le}1+\frac{4R^2}{r^2}+4096 H_1R^2h
       +16R\sqrt{H_0/\varepsilon}+\frac2h\log\frac{R^2}{\varepsilon A_0},
 \end{aligned}
\end{equation}
where \(\circledOne\) uses \eqref{eq:rad-h-init-count} for \(J\) and the
three preceding bounds for \(n\). If \(R\norm{\nabla f(w_J)}\le\varepsilon\),
the procedure ends in Algorithm~\ref{alg:radius-normalized} with \(N=J+1\).

It remains to bound the logarithm in \eqref{eq:rad-count}. We have
\[
\begin{aligned}
 \log(R^2/(\varepsilon A_0))
 &\overset{\circledOne}{=}\log(2R\norm{\nabla f(w_J)}/\varepsilon)\\
 &\overset{\circledTwo}{\le}2\sqrt{H_1}\,r
       +2\log\!\bigl(2R\sqrt{H_0/\varepsilon}\bigr)\\
 &\overset{\circledThree}{\le}2\sqrt{H_1}\,r
       +2hR\sqrt{H_0/\varepsilon}+2\log(2/h),
\end{aligned}
\]
where \(\circledOne\) substitutes \(A_0=R/(2\norm{\nabla f(w_J)})\), \(\circledTwo\) uses the gap bound
\eqref{eq:rad-switch-gap} and \(r\le R\), and \(\circledThree\) factors
the logarithm's argument as \((2/h)(hR\sqrt{H_0/\varepsilon})\) and
uses \(\log u\le u\) on the second factor.

We can now choose \(r\) and \(h\). In the upper bound \eqref{eq:rad-count},
the term for normalized steps is proportional to \(R^2/r^2\), the term
for reduced weights to \(H_1R^2h\), and the contribution from the
exponent in \eqref{eq:rad-switch-gap} to \(\sqrt{H_1}\,r/h\).
For \(R\sqrt{H_1}\ge1\), we make the three costs of the same order.
Equating the first with the second, and then with the third, gives
\begin{equation}
\label{eq:rad-h-balance}
\begin{aligned}
 h&=\frac{1}{H_1r^2},\\
 h&=\frac{\sqrt{H_1}\,r^3}{R^2}.
\end{aligned}
\end{equation}
Eliminating \(h\) yields \((r/R)^5=(R\sqrt{H_1})^{-3}\). Thus the
step length scales as \(r/R=(R\sqrt{H_1})^{-3/5}\), and substitution
in the first equality gives \(h=(R\sqrt{H_1})^{-4/5}\) up to an
absolute factor. Using \(1+R\sqrt{H_1}\) extends these scales to
\(H_1=0\) while keeping \(r\le R\). Take
\(r=R(1+R\sqrt{H_1})^{-3/5}\). For \(h\), a factor of \(1/256\)
suffices for Lemma~\ref{lem:rad-curvature-remainder}. Indeed,
\(h=(1+R\sqrt{H_1})^{-4/5}/256\) gives
\[
\begin{aligned}
 64\sqrt2\,R\sqrt{H_1}\,h^{3/2}
 &=\frac{64\sqrt2}{256^{3/2}}\,
   \frac{R\sqrt{H_1}}{(1+R\sqrt{H_1})^{6/5}}
 \overset{\circledOne}{\le}\frac18 ,
\end{aligned}
\]
where \(\circledOne\) uses \(R\sqrt{H_1}\le(1+R\sqrt{H_1})^{6/5}\)
and \(64\sqrt2/256^{3/2}\le1/8\). The three costs are now each
\(O((1+R\sqrt{H_1})^{6/5})\). To check the remaining logarithm, write
\[
\begin{aligned}
 \log(2/h)&=\log512+\tfrac45\log(1+R\sqrt{H_1})\\
 &\overset{\circledOne}{\le}(\log512+2)(1+R\sqrt{H_1})^{2/5},
\end{aligned}
\]
where \(\circledOne\) uses \(\log u\le(5/2)u^{2/5}\) for \(u\ge1\).
Multiplication by \(1/h=256(1+R\sqrt{H_1})^{4/5}\) gives the same
order for \(h^{-1}\log(2/h)\). Hence \eqref{eq:rad-count} gives
\eqref{eq:rad-rates}, since
\((1+R\sqrt{H_1})^{6/5}=O(1+(H_1R^2)^{3/5})\).
\end{proof}

\subsection{Analysis under \texorpdfstring{\((L_0,L_1)\)}{(L0,L1)}-Smoothness}
\label{sec:rad-l-rate}

Under the local \((L_0,L_1)\)-model, the gradient at \(y_k\) determines the
quadratic coefficient in the upper bound. We can therefore estimate the
decrease at \(y_k\) directly when choosing the update weight \(b_k\).

\begin{mainresultbox}
\begin{proposition}[Full-gradient complexity under \((L_0,L_1)\)-smoothness]
\label{prop:rad-l-upper}
Suppose \(f\) is differentiable and convex and satisfies the local model in
\eqref{eq:local-l0l1-function-model} with \(L_0>0\) and \(L_1\ge0\).
Let \(x^*\) be a minimizer with \(\norm{x^0-x^*}\le R\), where \(R>0\),
and let \(\varepsilon>0\). In Algorithm~\ref{alg:radius-normalized},
set \(r=R(1+RL_1)^{-1/2}\) and \(G_{\mathrm{switch}}=2L_0r\exp(2L_1r)\).
In Algorithm~\ref{alg:radius-full-gradient}, set \(h=(1+RL_1)^{-1/2}/256\)
and choose the weights
\begin{equation}
\label{eq:rad-l-weights}
 \widehat b_k=\min\!\left\{\frac{hA_k}{1-h},
                         \frac{1+\sqrt{1+64L_0A_k}}{32L_0}\right\},
 \qquad
 b_k=\min\!\left\{\widehat b_k,
              \frac{1}{32L_1\eta_k\norm{\nabla f(y_k)}}\right\},
\end{equation}
where \(\eta_k=\widehat b_k/(A_k+\widehat b_k)\). If \(L_1=0\) or \(\nabla f(y_k)=0\), take
\(b_k=\widehat b_k\).
Algorithm~\ref{alg:radius-normalized}, followed upon switching by
Algorithm~\ref{alg:radius-full-gradient}, returns an \(\varepsilon\)-optimal point.
The total number \(N\) of full-gradient queries across both algorithms satisfies
\begin{equation}
 \label{eq:rad-l-rate}
 \boxed{\displaystyle
 N=O\!\left(1+\sqrt{\frac{L_0R^2}{\varepsilon}}+L_1R\right)} .
\end{equation}
\end{proposition}
\end{mainresultbox}

\begin{proof}[Proof of Proposition~\ref{prop:rad-l-upper}]
Keep \(0<r\le R\) and \(0<h<1\) free until the two query counts are known.
Consider a point \(w_j\) at which Algorithm~\ref{alg:radius-normalized} takes a step.
We need to know how much of the displacement from \(x^*\) to \(w_j\)
lies along its gradient.
For this purpose, evaluate the supporting hyperplane at \(w_j\) at the point
\[
 x^*+r\,\frac{\nabla f(w_j)}{\norm{\nabla f(w_j)}} .
\]
This point is at distance \(r\) from \(x^*\), so Lemma~\ref{lem:rad-envelope}
bounds its gap by \(L_0 r^2\exp(2L_1r)\), independently of the current
gap at \(w_j\). Consequently,
\[
\begin{aligned}
 \inner{\frac{\nabla f(w_j)}{\norm{\nabla f(w_j)}}}{w_j-x^*}
 &\overset{\circledOne}{\ge}
 r+\frac{\F(w_j)
       -\F\!\left(x^*+r\,\frac{\nabla f(w_j)}{\norm{\nabla f(w_j)}}\right)}
         {\norm{\nabla f(w_j)}}\\
 &\overset{\circledTwo}{\ge}
 r-\frac{L_0r^2\exp(2L_1r)}{\norm{\nabla f(w_j)}}\\
 &\overset{\circledThree}{>} r/2 ,
\end{aligned}
\]
where \(\circledOne\) rearranges the supporting-hyperplane inequality,
\(\circledTwo\) uses the lemma and \(\F(w_j)\ge0\). For \(\circledThree\),
the subtracted fraction is less than \(r/2\) precisely when
\(\norm{\nabla f(w_j)}>2L_0r\exp(2L_1r)\), giving
\(G_{\mathrm{switch}}=2L_0r\exp(2L_1r)\). Thus the component towards
the minimizer exceeds the step length \(r/2\). In squared distance,
\[
\begin{aligned}
 \norm{w_{j+1}-x^*}^2
 &\overset{\circledOne}{=}\norm{w_j-x^*}^2
    -r\inner{\frac{\nabla f(w_j)}{\norm{\nabla f(w_j)}}}{w_j-x^*}
    +\frac{r^2}{4}\\
 &\overset{\circledTwo}{\le}\norm{w_j-x^*}^2-\frac{r^2}{4},
\end{aligned}
\]
where \(\circledOne\) expands the square of the normalized update and
\(\circledTwo\) substitutes the projection bound just proved. For any
\(J\) successive steps from \(w_0=x^0\),
\[
\begin{aligned}
 \norm{w_J-x^*}^2+Jr^2/4
 &\overset{\circledOne}{\le}\norm{x^0-x^*}^2\\
 &\overset{\circledTwo}{\le}R^2,
\end{aligned}
\]
where \(\circledOne\) sums the successive decreases. As in
\eqref{eq:rad-normalized-radius}, the squared distances at the intermediate
iterates cancel, leaving only those at \(x^0\) and \(w_J\).
Inequality \(\circledTwo\) uses the initial radius. Both terms on the
left are nonnegative, so each is at most \(R^2\). The term \(Jr^2/4\)
therefore bounds the number of steps by \(4R^2/r^2\). The other term
gives \(\norm{w_J-x^*}\le R\) upon taking square roots. At this last point,
\[
\begin{aligned}
 \F(w_J)&\overset{\circledOne}{\le}\inner{\nabla f(w_J)}{w_J-x^*}\\
 &\overset{\circledTwo}{\le}R\norm{\nabla f(w_J)},
\end{aligned}
\]
where \(\circledOne\) uses convexity and \(\circledTwo\) uses the distance
bound and Cauchy--Schwarz. Thus Algorithm~\ref{alg:radius-normalized}
returns an \(\varepsilon\)-optimal point if \(R\norm{\nabla f(w_J)}\le\varepsilon\).
Otherwise Algorithm~\ref{alg:radius-normalized} stopped because
\(\norm{\nabla f(w_J)}\le G_{\mathrm{switch}}\), giving
\begin{equation}
 \label{eq:rad-l-switch}
 R\norm{\nabla f(w_J)}\overset{\circledOne}{\le}2L_0Rr\exp(2L_1r),
\end{equation}
where \(\circledOne\) multiplies the gradient threshold by \(R\).

In either case, the total number \(J\) of steps in
Algorithm~\ref{alg:radius-normalized} satisfies
\begin{equation}
 \label{eq:rad-l-init-count}
 J\le\frac{4R^2}{r^2},
\end{equation}
by the summed distance bound above. Testing \(w_0,\ldots,w_J\) costs
exactly \(J+1\) gradient queries. A larger \(r\) reduces this count and
increases the exponent \(2L_1r\) in \eqref{eq:rad-l-switch}. The balance
in \eqref{eq:rad-l-balance} below gives \(J\le4(1+L_1R)\).

For the accelerated iteration use the potential
\[
 \Phi_k=A_k\F(x_k)+\tfrac12\norm{v_k-x^*}^2 .
\]
At the switch the distance term is at most \(R^2/2\). Taking the
largest weight with \(A_0 R\norm{\nabla f(w_J)}\le R^2/2\) gives
\(A_0=R/(2\norm{\nabla f(w_J)})\), so
\[
\begin{aligned}
 \Phi_0
 &\overset{\circledOne}{\le} A_0 R\norm{\nabla f(w_J)}+\tfrac12 R^2\\
 &\overset{\circledTwo}{=}R^2,
\end{aligned}
\]
where \(\circledOne\) uses \(x_0=v_0=w_J\) and the two bounds at \(w_J\),
and \(\circledTwo\) uses \(A_0=R/(2\norm{\nabla f(w_J)})\).

\paragraph*{Decrease under the local gradient-dependent model.}
The representation \eqref{eq:rad-interpolation} uses the point
\(y_k-2b_k\eta_k\nabla f(y_k)\). The coefficient of the quadratic
error in the local model at \(y_k\) is
\((L_0+L_1\norm{\nabla f(y_k)})/2\). We can control its two terms with
the query weight and the update weight, respectively. Requiring
\(L_0\widehat b_k\eta_k\le1/16\) and \(\eta_k\le h\) gives the first
minimum in \eqref{eq:rad-l-weights}, as for \(H_0\) above. After the
query at \(y_k\), taking the largest \(b_k\le\widehat b_k\) such that
\(L_1b_k\eta_k\norm{\nabla f(y_k)}\le1/32\) gives the second minimum.
In particular, when \(L_1>0\), the displacement
\(2b_k\eta_k\norm{\nabla f(y_k)}\) is at most \(1/(16L_1)\), within
the radius \(1/L_1\) of the model. The model holds globally when \(L_1=0\).

Apply the local bound at \(y_k\) and multiply by the coefficient of
the function difference in \eqref{eq:rad-coupled}. This yields
\[
\begin{aligned}
 \frac{b_k}{2\eta_k}\bigl(f(y_k-2b_k\eta_k\nabla f(y_k))-f(y_k)\bigr)
 &\overset{\circledOne}{\le}-b_k^2\norm{\nabla f(y_k)}^2\\
 &\quad+b_k^3\eta_k\bigl(L_0+L_1\norm{\nabla f(y_k)}\bigr)
                         \norm{\nabla f(y_k)}^2\\
 &\overset{\circledTwo}{\le}-\tfrac{29}{32}b_k^2\norm{\nabla f(y_k)}^2,
\end{aligned}
\]
where \(\circledOne\) substitutes the displacement
\(-2b_k\eta_k\nabla f(y_k)\) into \eqref{eq:local-l0l1-function-model}.
For \(\circledTwo\), the two terms in the quadratic error take up
at most \(1/16+1/32=3/32\) of the linear decrease.

We can now use \eqref{eq:rad-coupled}, which follows from convexity and
the updates shared by the two models. The coefficient of
\(\F(y_k)-\F(x_k)\) lies between zero and \(b_k/2\). If
\(\F(y_k)\le\F(x_k)\), the two gap terms are both nonpositive.
If \(\F(y_k)>\F(x_k)\), their sum is at most \(-b_k\F(x_k)/2\).
Consequently,
\[
\begin{aligned}
 \Phi_{k+1}-\Phi_k
 &\overset{\circledOne}{\le}
 -\tfrac{29}{32}b_k^2\norm{\nabla f(y_k)}^2
 +\tfrac12 b_k^2\norm{\nabla f(y_k)}^2\\
 &=-\tfrac{13}{32}b_k^2\norm{\nabla f(y_k)}^2,
\end{aligned}
\]
where \(\circledOne\) inserts the local decrease and drops the two gap
terms whose sum we have just bounded.

For any \(n\) completed updates this gives
\[
\begin{aligned}
 \tfrac{13}{32}\sum_{k=0}^{n-1}b_k^2\norm{\nabla f(y_k)}^2
 &\overset{\circledOne}{\le}\Phi_0-\Phi_n\\
 &\overset{\circledTwo}{\le}R^2,
\end{aligned}
\]
where \(\circledOne\) sums the decreases and \(\circledTwo\) uses
\(\Phi_0\le R^2\) and \(\Phi_n\ge0\). The same decrease proves the stopping
guarantee,
\[
\begin{aligned}
 \F(x_n)&\overset{\circledOne}{\le}\frac{\Phi_n}{A_n}\\
        &\overset{\circledTwo}{\le}\frac{R^2}{A_n},
\end{aligned}
\]
where \(\circledOne\) drops the nonnegative distance term and
\(\circledTwo\) uses monotonicity of the potential and its initial bound.
Thus \(A_n\ge R^2/\varepsilon\) certifies the requested accuracy.

\paragraph*{Number of gradient queries.}
Consider any prefix of updates ending with weight below \(R^2/\varepsilon\).
If \(b_k<\widehat b_k\), the
second candidate in the rule for \(b_k\) is active, which requires \(L_1>0\).
The contribution of this update to the preceding sum is bounded below by
\[
\begin{aligned}
 b_k^2\norm{\nabla f(y_k)}^2
 &\overset{\circledOne}{=}\frac{1}{32^2L_1^2\eta_k^2}\\
 &\overset{\circledTwo}{\ge}\frac{1}{32^2L_1^2h^2},
\end{aligned}
\]
where \(\circledOne\) substitutes the active candidate and \(\circledTwo\)
uses \(\eta_k\le h\). The sum over all updates is at most \(32R^2/13\),
hence at most \(4R^2\). Dividing this budget by the lower bound for one
update gives at most \(4096(L_1R)^2h^2\) reduced weights.

For a fully accepted weight \(b_k=\widehat b_k\), one of the two constraints
on \(\widehat b_k\) is tight. If \(\widehat b_k\eta_k=1/(16L_0)\), use
the square root of the weight as in \eqref{eq:rad-root-growth}. We have
\[
\begin{aligned}
 \sqrt{A_{k+1}}-\sqrt{A_k}
 &\overset{\circledOne}{=}
 \frac{\widehat b_k}{2(\sqrt{A_k+\widehat b_k/2}+\sqrt{A_k})}\\
 &\overset{\circledTwo}{\ge}
 \frac{\widehat b_k}{4\sqrt{A_k+\widehat b_k}}\\
 &\overset{\circledThree}{=}\frac{1}{16\sqrt{L_0}},
\end{aligned}
\]
where \(\circledOne\) uses \(A_{k+1}=A_k+\widehat b_k/2\) and rationalizes
the difference, \(\circledTwo\) bounds both square roots in the denominator
by \(\sqrt{A_k+\widehat b_k}\), and \(\circledThree\) uses
\(\widehat b_k^2/(A_k+\widehat b_k)=1/(16L_0)\). As \(A_k\) increases,
the total increase of its square root before the crossing is at most
\(R/\sqrt{\varepsilon}\), allowing at most \(16R\sqrt{L_0/\varepsilon}\) such updates. For
all other fully accepted weights \(\eta_k=h\), and
\eqref{eq:rad-log-growth} gives
\[
\begin{aligned}
 \log(A_{k+1}/A_k)
 &\overset{\circledOne}{\ge}1-A_k/A_{k+1}\\
 &\overset{\circledTwo}{=}\frac{h}{2-h}\\
 &\overset{\circledThree}{\ge}h/2,
\end{aligned}
\]
where \(\circledOne\) uses \(\log u\ge1-1/u\), \(\circledTwo\) substitutes
\(\widehat b_k=hA_k/(1-h)\) into the weight update, and \(\circledThree\)
uses \(0<h<1\). Summing these increases gives at most
\(2\log(R^2/(\varepsilon A_0))/h\) further updates. The three counts apply to every
finite prefix before termination, which also proves that the algorithm stops.

If the accelerated loop is entered, let \(n\) be the index of the first
update that reaches the threshold. Of the gradients at
\(y_0,\ldots,y_n\), the first is already available from \(w_J=y_0\).
Adding the two query counts gives
\[
 \begin{aligned}
 N&=\underbrace{(J+1)}_{\text{queries in Algorithm~\ref{alg:radius-normalized}}}
    +\underbrace{n}_{\text{new queries in Algorithm~\ref{alg:radius-full-gradient}}}\\
  &\overset{\circledOne}{\le}1+\frac{4R^2}{r^2}+4096(L_1R)^2h^2
       +16R\sqrt{L_0/\varepsilon}+\frac2h\log\frac{R^2}{\varepsilon A_0},
 \end{aligned}
\]
where \(\circledOne\) uses \eqref{eq:rad-l-init-count} for \(J\) and the
three preceding bounds for \(n\). If \(R\norm{\nabla f(w_J)}\le\varepsilon\),
the procedure ends in Algorithm~\ref{alg:radius-normalized} with \(N=J+1\).

To bound the number of updates controlled by the logarithm, use the
gap at the switch. We have
\[
\begin{aligned}
 \frac1h\log\frac{R^2}{\varepsilon A_0}
 &\overset{\circledOne}{=}\frac1h\log\frac{2R\norm{\nabla f(w_J)}}{\varepsilon}\\
 &\overset{\circledTwo}{\le}
 \frac{2L_1r}{h}+\frac2h\log\!\left(2R\sqrt{L_0/\varepsilon}\right)\\
 &\overset{\circledThree}{\le}
 \frac{2L_1r}{h}+2R\sqrt{L_0/\varepsilon}+\frac2h\log(2/h),
\end{aligned}
\]
where \(\circledOne\) substitutes \(A_0=R/(2\norm{\nabla f(w_J)})\), and \(\circledTwo\)
applies \eqref{eq:rad-l-switch} and \(r\le R\). For \(\circledThree\),
factor \(2R\sqrt{L_0/\varepsilon}\) as
\((2/h)(hR\sqrt{L_0/\varepsilon})\) and use \(\log u\le u\) for the second
factor.

The three costs that determine the powers of \(r\) and \(h\) are now
\(R^2/r^2\) for initialization, \((L_1R)^2h^2\) for reduced weights,
and \(L_1r/h\) from the gap at the switch. For \(L_1R\ge1\), equating
the first term with each of the other two after dropping absolute factors gives
\begin{equation}
\label{eq:rad-l-balance}
\begin{aligned}
 h&=\frac{1}{L_1r},\\
 h&=\frac{L_1r^3}{R^2}.
\end{aligned}
\end{equation}
Eliminating \(h\) gives \((r/R)^4=(L_1R)^{-2}\), so both \(r/R\)
and \(h\) scale as \((L_1R)^{-1/2}\). Replacing \(L_1R\) by \(1+L_1R\)
gives \(r=R(1+L_1R)^{-1/2}\) for all \(L_1\ge0\). Any fixed multiplier
strictly between zero and one in \(h\) preserves the balance and ensures
\(h<1\). Using the same factor \(1/256\) as for the first model gives
\[
\begin{aligned}
 R^2/r^2 &\overset{\circledOne}{=}1+L_1R,\\
 (L_1R)^2h^2
 &\overset{\circledTwo}{=}\frac{(L_1R)^2}{256^2(1+L_1R)}\\
 &\overset{\circledThree}{\le}\frac{L_1R}{256^2},\\
 L_1r/h &\overset{\circledFour}{=}256L_1R,
\end{aligned}
\]
where \(\circledOne\) and \(\circledTwo\) substitute \(r\) and \(h\),
respectively, \(\circledThree\) uses \(L_1R/(1+L_1R)\le1\), and
\(\circledFour\) uses \(r/h=256R\). Finally,
\[
\begin{aligned}
 \frac{\log(2/h)}h
 &\overset{\circledOne}{=}256\sqrt{1+L_1R}
 \left(\log512+\tfrac12\log(1+L_1R)\right)\\
 &\overset{\circledTwo}{\le}256(\log512+1)(1+L_1R),
\end{aligned}
\]
where \(\circledOne\) substitutes \(1/h=256\sqrt{1+L_1R}\), and
\(\circledTwo\) uses \(\sqrt u\le u\) and \(\log u\le2\sqrt u\) for
\(u\ge1\). Substituting these bounds into the query count gives
\eqref{eq:rad-l-rate}.
\end{proof}

\paragraph*{Comparison with recent work.}
The same normalized initialization and accelerated updates, with parameters
chosen for each smoothness model, attain the optimal worst-case powers of
\(H_1\) and \(L_1\) for deterministic first-order methods in growing
dimension, without a bound on the initial gap. The \((H_1R^2)^{3/5}\)
term in Theorem~\ref{thm:rad-upper} matches the lower bound in
Theorem~\ref{radlb:theorem} up to absolute constants when
\(H_0R^2=\varepsilon\) and \(R\sqrt{H_1}\) is sufficiently large.
For the gradient-dependent class,
Proposition~\ref{prop:rad-l-upper} removes the logarithmic factor in the
\((1+L_1R)\log(2+L_1R)\) term of
\citet[Theorem~3.7]{rudomanov2026descending} and the inexact segment
search that incurs it.\footnote{Their Lemma~A.1(c) implies the local
model \eqref{eq:local-l0l1-function-model} with constants \(2L_0,2L_1\),
since \(e^t-1-t\le t^2\) for \(0\le t\le 1/2\).}
In particular, our method reaches their first-stage accuracy
\(\varepsilon=L_0/(128L_1^2)\) in \(O(1+L_1R)\) gradient queries.
Their Theorem~4.10 gives a lower
bound of order \(L_1R/\sqrt{\log(2+L_1R)}\) when \(L_0R^2\ge\varepsilon\),
so the exponent \(1\) is sharp and the remaining gap in the \(L_1R\)
term is a factor of order \(\sqrt{\log(2+L_1R)}\).

\subsection{A Matching Lower Bound for the \texorpdfstring{\(H_1\)}{H1}-Dependence}
\label{radlb:sec}

The initial gap on the hard instances in Theorem~\ref{thm:growth-lower-bounds}
is at most \(4\varepsilon\). Allowing the initial gap to grow changes the
worst-case power of \(H_1\).

\paragraph*{Method class.}
We consider deterministic first-order methods with known \(H_0,H_1,R,\varepsilon\).
At each query point \(x\), the oracle returns \(f(x)\) and \(\nabla f(x)\).

\begin{mainresultbox}
\begin{theorem}[Full-gradient lower bound for the \(H_1\)-dependent term]
\label{radlb:theorem}
Fix \(R>0\), \(\varepsilon>0\), and \(H_1>0\), and set \(H_0:=\varepsilon/R^2\).
Suppose that \(H_1R^2\) is sufficiently large. There is an integer
\(m=\Theta((H_1R^2)^{3/5})\) such that the following holds in dimension
\(d=2m+1\). For every deterministic first-order method making fewer than
\(m\) queries, there exist a convex \(C^\infty\) function
\(f:\mathbb R^d\to\mathbb R\) satisfying Assumptions~\ref{ass:convexity}
and~\ref{ass:h0h1-smoothness} with constants \(H_0,H_1\), and a minimizer
\(x^*\) such that, with \(x^0=0\),
\[
 \norm{x^0-x^*}\le R,
 \qquad
 \log\!\left(\frac{f(x^0)-f^*}{\varepsilon}\right)
 =\Theta\!\left((H_1R^2)^{1/5}\right).
\]
On this instance the output satisfies \(f(x_{\rm out})-f^*>\varepsilon\).
Thus, to guarantee \(f(x_{\rm out})-f^*\le\varepsilon\), a deterministic
first-order method must use
\[
 N=\Omega\!\left((H_1R^2)^{3/5}\right)
\]
full-gradient queries in the worst case, even if it knows \(f(x^0)-f^*\).
\end{theorem}
\end{mainresultbox}

\begin{proof}
We construct a function in \(m\) variables, with a minimizer at
\(-s(1,\ldots,1)\) before rescaling. The objective is a smoothed maximum
of \(m\) coordinate functions and a constant. Their weights decrease with
the coordinate index, so that a query supported on the first \(k-1\)
directions can reveal at most the \(k\)-th direction. With \(C=10^4\), set
\begin{equation}
\label{radlb:parameters}
\begin{aligned}
 s   &= (500H_1R^2/C)^{1/5},\\
 m   &= \lfloor s^3/1000\rfloor,\\
 a_i &= s^5(1+16/s^2)^{m-i}\qquad(1\le i\le m).
\end{aligned}
\end{equation}
The factor \(1+16/s^2\) separates adjacent weights. Their total ratio
is at most \(\exp(16m/s^2)\). At the proposed minimizer each coordinate
function decays by a factor of order \(\exp(-s/2)\), so we can take
\(m\) proportional to \(s^3\) with a small constant. Rescaling the
coordinates below multiplies distances by \(\sqrt{2C/H_1}\).
The minimizer will be within radius \(R\) if
\(s\sqrt m\le R\sqrt{H_1/(2C)}\). Since \(m\le s^3/1000\), it suffices
that \(s^5/1000=H_1R^2/(2C)\). This gives the fifth root in
\eqref{radlb:parameters} and \(m=\Theta((R\sqrt{H_1})^{6/5})\).
Thus taking \(R\sqrt{H_1}\)
sufficiently large ensures the bounds on \(s\) used below.

Each coordinate function has an increasing convex part \(h\).
An independent perturbation proportional to \(j\) smooths changes
in the maximizing index. Write \(\phi_u=h+uj\), where
\begin{equation}
\label{radlb:profiles}
\begin{aligned}
 h(t)&=\begin{cases} e^t,&t\le0,\\1+2t-\log(1+t),&t>0,\end{cases}\\[2pt]
 j(t)&=\begin{cases}e^t,&t\le0,\\(1+3t)/(1+t)^2,&t>0.\end{cases}
\end{aligned}
\end{equation}
For \(1\le i\le m\), let \(U_i\) be independent and uniform on \([0,1/8]\).
At \(x\in\mathbb R^m\), branch \(i\) has value
\(B_i(x)=a_i\phi_{U_i}(x_i+s/2)\). Include the constant branch \(B_0(x)=1\)
and define the deterministic objective
\begin{equation}
\label{radlb:max}
 J(x)=\mathbb E\max_{0\le i\le m} B_i(x).
\end{equation}

\paragraph*{Curvature of the maximum.}
Both profiles
and their first derivatives agree at zero with the exponential branch.
For \(w>0\) we have
\[
\begin{aligned}
 \phi_u''(w)&\overset{\circledOne}{=}\frac{1}{(1+w)^2}
                   +\frac{6u(w-1)}{(1+w)^4}\\
 &\overset{\circledTwo}{\ge}\frac{1-6/8}{(1+w)^2},
\end{aligned}
\]
where \(\circledOne\) differentiates the positive branches in
\eqref{radlb:profiles}, and \(\circledTwo\) uses
\((w-1)/(1+w)^2\ge-1\) and \(0\le u\le1/8\).
On \(w<0\), \(\phi_u''(w)=(1+u)e^w>0\), so \(J\) is convex. Since
\(j(x_i+s/2)>0\), the nonconstant branch values have independent continuous
densities. At fixed \(x\) there is almost surely a unique maximizing
index \(I\). On a neighborhood of \(x\) the derivatives of the finitely
many branches are uniformly bounded. Dominated convergence for directional
difference quotients therefore gives
\(\nabla J(x)=\mathbb E[\nabla B_I(x)]\).

Consider a displacement \(v\) with \(\|v\|\le1\). It can
change both the value of the old winning branch and the identity of
the winner. The Taylor remainder therefore satisfies
\begin{equation}
\label{radlb:remainder}
\begin{aligned}
 &J(x+v)-J(x)-\langle\nabla J(x),v\rangle\\
 &\quad\overset{\circledOne}{=}
        \mathbb E\bigl[B_I(x+v)-B_I(x)-\langle\nabla B_I(x),v\rangle\bigr]\\
 &\qquad+\mathbb E\bigl[\max_i B_i(x+v)-B_I(x+v)\bigr],
\end{aligned}
\end{equation}
where \(\circledOne\) adds and subtracts \(B_I(x+v)\) inside the
expectation and uses the gradient formula above.
For the first term we need the curvature of one branch. For the second
we need to bound its change of value before it can overtake another.
For \(0\le u\le1/8\) and \(|w-t|\le1\), the respective bounds are
\begin{equation}
\label{radlb:scalarbounds}
\begin{aligned}
 |\phi_u''(w)|&\overset{\circledOne}{\le}6h(t)\qquad(w\ne0),\\
 |\phi_u'(w)|^2&\overset{\circledTwo}{\le}9e^2j(t)h(t),
\end{aligned}
\end{equation}
where for \(t\le-1\) both bounds follow from the exponential formula
and \(e^w\le e e^t\). For \(t>-1\), \(\circledOne\) uses
\(|\phi_u''(w)|\le2\) and \(h(t)\ge e^{-1}\), and \(\circledTwo\)
uses \(|\phi_u'(w)|\le3\) and \(j(t)h(t)\ge e^{-2}\).
The derivative bounds follow from the exponential formula at negative
arguments, and from \(j'(w)=(1-3w)/(1+w)^3\) and the displayed second
derivative at positive arguments. The lower bounds on \(h\) and \(jh\)
follow from \(h(t)=j(t)=e^t\) on \((-1,0)\) and, for \(t\ge0\), from
\(h(t)\ge1+t\) and \(j(t)h(t)\ge(1+3t)/(1+t)\ge1\).

Taylor's formula and continuity of \(\phi_u'\) at zero bound the remainder
of the old winning branch by \(3a_I h(x_I+s/2)v_I^2\) when \(I\ne0\).
The constant branch has zero remainder. Since \(B_i(x)\ge a_i h(x_i+s/2)\),
the expected maximum \(J(x)\) is at least \(a_i h(x_i+s/2)\) for every \(1\le i\le m\).
The first expectation in \eqref{radlb:remainder} is therefore at most
\(3J(x)\|v\|^2\). For the second expectation, the interval supporting
\(B_i(x)\) has length \(a_i j(x_i+s/2)/8\). By the mean value theorem
and \eqref{radlb:scalarbounds}, the squared change
\(\lvert B_i(x+v)-B_i(x)\rvert^2\) is at most
\(9e^2a_i^2j(x_i+s/2)h(x_i+s/2)v_i^2\). It follows that
\[
\begin{aligned}
 \mathbb E[\max_i B_i(x+v)-B_I(x+v)]
 &\overset{\circledOne}{\le}576e^2\sum_{i=1}^m a_i h(x_i+s/2)v_i^2\\
 &\overset{\circledTwo}{\le}576e^2 J(x)\|v\|^2,
\end{aligned}
\]
where \(\circledOne\) applies Lemma~\ref{radlb:curvature} to the changes
\(B_i(x+v)-B_i(x)\). Dividing the bound on each squared change by its
interval width cancels \(a_i j(x_i+s/2)\). Inequality \(\circledTwo\)
uses \(a_i h(x_i+s/2)\le J(x)\) and \(\sum_i v_i^2=\|v\|^2\).
Adding the two contributions in \eqref{radlb:remainder} gives
\begin{equation}
\label{radlb:local}
 J(x+v)-J(x)-\langle\nabla J(x),v\rangle
 \overset{\circledOne}{\le} C J(x)\|v\|^2 ,
\end{equation}
where \(\circledOne\) uses \(3+576e^2<C\) to absorb the two coefficients.
Let \(\rho=1/(100s)\) and convolve \(J\) with a nonnegative \(C^\infty\)
kernel of integral one supported in the Euclidean ball of radius \(\rho\).
The resulting \(J_\rho\) is convex and smooth. Integrating
\eqref{radlb:local} against the kernel preserves its right-hand side
with \(J_\rho(x)\) in place of \(J(x)\). Hence
\[
 0\overset{\circledOne}{\preceq}\nabla^2 J_\rho(x)
 \overset{\circledTwo}{\preceq}2C J_\rho(x) I,
\]
where \(\circledOne\) follows from convexity, and \(\circledTwo\) follows
by dividing the smoothed remainder bound by half the squared displacement
and letting the displacement tend to zero.

At \(x=-s(1,\ldots,1)\) every branch argument in the convolution is at
most \(-s/2+\rho\). The largest nonconstant branch value is bounded by
\[
\begin{aligned}
 \tfrac98 a_1 e^{-s/2+\rho}
 &\overset{\circledOne}{\le}\tfrac98 s^5\exp(16s/1000-s/2+\rho)\\
 &\overset{\circledTwo}{<}1,
\end{aligned}
\]
where \(\circledOne\) uses \(\log(1+16/s^2)\le16/s^2\) and
\(m\le s^3/1000\), and \(\circledTwo\) holds for all sufficiently large
\(s\) because \(16/1000-1/2<0\) and \(\rho=1/(100s)\). The
constant branch equals one everywhere, so this point attains
\(\min J_\rho=1\). Define the scaled objective by
\begin{equation}
\label{radlb:scaled}
 g(y)=\frac{\varepsilon}{H_1R^2}
       \left(J_\rho\!\left(\sqrt{\frac{H_1}{2C}}\,y\right)-1\right).
\end{equation}
It has minimum zero at \(y^*=-s\sqrt{2C/H_1}(1,\ldots,1)\). The radius
bound follows from \(s\sqrt m\le R\sqrt{H_1/(2C)}\).
The Hessian is positive semidefinite by convexity. For its upper bound,
\[
\begin{aligned}
 \nabla^2g(y)
 &\overset{\circledOne}{=}\frac{\varepsilon}{2CR^2}
                    \nabla^2J_\rho\!\left(\sqrt{\frac{H_1}{2C}}\,y\right)\\
 &\overset{\circledTwo}{\preceq}\frac{\varepsilon}{R^2}
                    J_\rho\!\left(\sqrt{\frac{H_1}{2C}}\,y\right) I\\
 &\overset{\circledThree}{=}\left(\frac{\varepsilon}{R^2}+H_1g(y)\right)I,
\end{aligned}
\]
where \(\circledOne\) differentiates \eqref{radlb:scaled} twice,
\(\circledTwo\) uses the Hessian bound for the convolution, and
\(\circledThree\) expresses the value of \(J_\rho\) in terms of \(g(y)\)
using \eqref{radlb:scaled}.
Every displacement allowed by Assumption~\ref{ass:h0h1-smoothness} has
norm at most \(1/((2\sqrt3+\sqrt6)\sqrt{2C})<1\) after scaling by \(\sqrt{H_1/(2C)}\).
Thus the integrated form of \eqref{radlb:local} gives the required local
model, with coefficient \((\varepsilon/R^2+H_1g(y))/2\).

\paragraph*{Information revealed by a query.}
Suppose \(y_k=\cdots=y_m=0\). For each \(i\ge k\), the argument
of branch \(i\) before convolution is exactly \(s/2\). Convolution
changes this argument by at most \(\rho\).
Lemma~\ref{radlb:chain} shows that branch \(k\) exceeds every branch
with index \(i>k\) for every perturbation in the kernel support.
The gradient of the winning branch has no component past \(k\).
Taking the expectation and the convolution preserves this property.
Consequently
\begin{equation}
\label{radlb:chainfacts}
 y_k=\cdots=y_m=0\quad\Longrightarrow\quad
       \partial_i g(y)=0\ \text{for every }i>k
 \qquad(1\le k\le m).
\end{equation}
If \(y_m=0\), the maximum is at least the value of its last branch
regardless of the first \(m-1\) coordinates. This gives
\begin{equation}
\label{radlb:last-gap}
 \begin{aligned}
 g(y)&\overset{\circledOne}{\ge}
         \frac{\varepsilon}{H_1R^2}\bigl(a_m h(s/2-\rho)-1\bigr)\\
     &\overset{\circledTwo}{\ge}\frac{\varepsilon}{H_1R^2}(s^6/3-1)\\
     &\overset{\circledThree}{>}3\varepsilon,
 \end{aligned}
\end{equation}
where \(\circledOne\) bounds that branch throughout the kernel support,
\(\circledTwo\) uses \(a_m=s^5\) and \(h(s/2-\rho)\ge s/3\), and
\(\circledThree\) uses \(H_1R^2=Cs^5/500\), so the lower bound divided by
\(\varepsilon\) grows linearly in \(s\). At the origin, Lemma~\ref{radlb:chain} also
places branch one above all other nonconstant branches, and
\(a_1h(s/2-\rho)>1\) puts it above the constant. Throughout the
kernel support its value lies between \(a_1h(s/2-\rho)\) and
\(a_1(h(s/2+\rho)+3/[8(1+s/2-\rho)])\), both of order \(a_1s\).
Since \(H_1R^2=Cs^5/500\), rescaling by \eqref{radlb:scaled} gives
\[
 g(0)/\varepsilon=\Theta\bigl(s(1+16/s^2)^{m-1}\bigr).
\]
Taking the logarithm and using \(m=\Theta(s^3)\) gives
\(\log(g(0)/\varepsilon)=\Theta(s)\).

To pass from the coordinate structure in \eqref{radlb:chainfacts} to an
arbitrary deterministic method, use the orthogonal-invariance argument of \citet[Proposition~4.3]{rudomanov2026descending}. Write
\(f(x)=g(V^\top x)\) with \(m\) orthonormal columns \(v_j\) in
\(\mathbb R^{2m+1}\). At query \(x^t\), choose \(v_t\) orthogonal to
all preceding columns and all queries \(x^1,\ldots,x^t\). Evaluate
\(g\) and its gradient at \((v_1^\top x^t,\ldots,v_{t-1}^\top x^t,0,\ldots,0)\).
By \eqref{radlb:chainfacts}, every gradient component past \(t\) is
zero. The oracle can therefore express its answer in the original
coordinates using only \(v_1,\ldots,v_t\).
After this answer we can choose the later columns in the
orthogonal complement of the queries and columns used so far.
After the method has made \(T<m\) queries and produced \(x_{\rm out}\),
choose the remaining columns orthogonal also to its output, all queries
and the preceding columns. The dimension still available is at least
\(2m+1-(2T+1)=2(m-T)\). For this fixed \(V\), every value and gradient
returned agrees with \(f\), and \((V^\top x_{\rm out})_m=0\).
Equation~\ref{radlb:last-gap} gives \(f(x_{\rm out})>3\varepsilon\).
The minimizer \(Vy^*\) has norm at most \(R\), and \(\|V^\top\|=1\)
preserves both smoothness bounds. Every orientation gives the same
initial gap \(g(0)\), so supplying it leaves the oracle argument valid.
\end{proof}

\begin{corollary}[Uniform coordinate queries]
\label{radlb:coordinate}
For the same \(m\) and any \(d\ge m\), the uniformly sampled coordinate
methods of Theorem~\ref{thm:growth-lower-bounds} require
\(\Omega(dm)\) scalar queries in the worst case to ensure expected
gap at most \(\varepsilon\).
\end{corollary}

\begin{proof}
Define \(f(x)=g(x_1,\ldots,x_m)\) on \(\mathbb R^d\).
In the stated model, query points and
the output can change only in coordinates where a nonzero derivative
has already been observed. Let \(K_T\) be the largest index revealed
after \(T\) queries, with \(K_0=0\). If \(K_T=k<m\),
\eqref{radlb:chainfacts} permits only \(k+1\) to be newly revealed,
with conditional probability at most \(1/d\). Summing expected
increments gives \(\mathbb E K_T\le T/d\). An optional initial sweep
at the origin costs \(d\) queries and reveals at most index one, so
the same estimate holds. On \(K_T<m\), the output has last chain
coordinate zero. By \eqref{radlb:last-gap} and Markov's inequality,
\[
\begin{aligned}
 \mathbb E[f(\widehat x)-f^*]
 &\overset{\circledOne}{\ge}3\varepsilon\,\Pr(K_T<m)\\
 &\overset{\circledTwo}{\ge}3\varepsilon\bigl(1-T/(dm)\bigr),
\end{aligned}
\]
where \(\circledOne\) uses \eqref{radlb:last-gap} on \(K_T<m\) and
nonnegativity of the gap on its complement, and \(\circledTwo\) bounds
\(\Pr(K_T\ge m)\) by \(\mathbb E K_T/m\le T/(dm)\).
Thus \(T\le dm/2\) gives expected gap greater than \(\varepsilon\).
\end{proof}

At \(H_0R^2=\varepsilon\), the term \(\sqrt{H_0R^2/\varepsilon}\)
in Theorem~\ref{thm:rad-upper} equals one. Thus, as \(H_1R^2\) grows,
Theorem~\ref{radlb:theorem} matches the entire full-gradient bound in this case.
A matching upper bound when just one uniform coordinate
can be queried from the first step remains open.

\subsection{Auxiliary Estimates for the Upper Bounds}

The proof of the normalized step compares \(w_j\) with a point at distance
\(r\) from \(x^*\). This distance may exceed the radius in the smoothness
assumption. The next lemma bounds the gap at any point by dividing the
segment from \(x^*\) to that point into shorter pieces.

\begin{lemma}[Growth with distance from a minimizer]
\label{lem:rad-envelope}
Let \(f\) be convex and differentiable with a minimizer \(x^*\). Under
Assumption~\ref{ass:h0h1-smoothness}, every point \(z\) satisfies
\[
 \F(z)\le H_0\norm{z-x^*}^2\exp\!\left(2\sqrt{H_1}\,\norm{z-x^*}\right).
\]
Under \eqref{eq:local-l0l1-function-model} with \(L_0>0\),
\[
 \F(z)\le L_0\norm{z-x^*}^2\exp\!\left(2L_1\norm{z-x^*}\right).
\]
\end{lemma}

\begin{proof}
We first prove the estimate under Assumption~\ref{ass:h0h1-smoothness}.
For \(z=x^*\) both sides vanish. Otherwise divide the segment from \(x^*\)
to \(z\) into \(n\) equal pieces of length \(\delta=\norm{z-x^*}/n\), choosing \(n\)
large enough for each piece to lie within the model radius. Write
\(\varphi_i=\F(x^*+(i/n)(z-x^*))\). At the first point the model gives
\(\varphi_1\le H_0\delta^2\), since \(\varphi_0=0\) and
\(\nabla f(x^*)=0\). At every interior point,
\begin{equation}
 \label{eq:rad-envel-diff}
 \varphi_{i+1}-\varphi_i
 \overset{\circledOne}{\le}
 \varphi_i-\varphi_{i-1}+2\delta^2(H_0+H_1\varphi_i),
\end{equation}
where \(\circledOne\) follows by adding the local upper bounds at the
two neighbors. Their displacements have opposite signs, so the two
first-order terms cancel.

To iterate \eqref{eq:rad-envel-diff}, first collect the two terms in
\(\varphi_i\). Their coefficient is \(2+2H_1\delta^2\).
The inequality \(2+t^2\le e^t+e^{-t}\), which follows by adding the
power series, bounds this coefficient by
\(e^{\delta\sqrt{2H_1}}+e^{-\delta\sqrt{2H_1}}\).
The gap \(\varphi_i\) is nonnegative, so replacing its coefficient by
this upper bound and regrouping gives
\[
\begin{aligned}
 \varphi_{i+1}-e^{-\delta\sqrt{2H_1}}\varphi_i
 &\le e^{\delta\sqrt{2H_1}}
      \bigl(\varphi_i-e^{-\delta\sqrt{2H_1}}\varphi_{i-1}\bigr)
       +2H_0\delta^2 .
\end{aligned}
\]
The expression in parentheses is the left-hand side one step earlier.
Its initial value is \(\varphi_1\le H_0\delta^2\). After \(i\) steps
there are \(i\) further contributions of \(2H_0\delta^2\), and each
contribution has been multiplied by \(e^{\delta\sqrt{2H_1}}\) at most
\(i\) times. Hence, for \(0\le i<n\),
\[
\begin{aligned}
 \varphi_{i+1}-\varphi_i
 &\overset{\circledOne}{\le}\varphi_{i+1}-e^{-\delta\sqrt{2H_1}}\varphi_i\\
 &\overset{\circledTwo}{\le}(2i+1)H_0\delta^2 e^{i\delta\sqrt{2H_1}},
\end{aligned}
\]
where \(\circledOne\) uses \(\varphi_i\ge0\) and
\(e^{-\delta\sqrt{2H_1}}\le1\), and \(\circledTwo\) sums the contributions
just counted. Summing the increments from \(x^*\) to \(z\) now gives
\[
\begin{aligned}
 \F(z)&\overset{\circledOne}{\le}
 H_0\delta^2 e^{n\delta\sqrt{2H_1}}\sum_{i=0}^{n-1}(2i+1)\\
 &=H_0\norm{z-x^*}^2\exp\!\left(\sqrt{2H_1}\,\norm{z-x^*}\right)\\
 &\overset{\circledTwo}{\le}
 H_0\norm{z-x^*}^2\exp\!\left(2\sqrt{H_1}\,\norm{z-x^*}\right),
\end{aligned}
\]
where \(\circledOne\) bounds every exponential by its value at \(n\).
The equality uses \(1+3+\cdots+(2n-1)=n^2\) and \(n\delta=\norm{z-x^*}\),
and \(\circledTwo\) uses \(\sqrt2\le2\).

Under \eqref{eq:local-l0l1-function-model},
Proposition~\ref{prop:l0l1-implies-local-h0h1} gives
Assumption~\ref{ass:h0h1-smoothness} with \(H_0=L_0\) and \(H_1=L_1^2\).
Substituting these constants into the first bound gives the second.
\end{proof}

\begin{lemma}[Gap difference when the update weight is reduced]
\label{lem:rad-curvature-remainder}
Consider an accelerated iteration of Algorithm~\ref{alg:radius-full-gradient}
for a differentiable convex \(f\), with the weights in \eqref{eq:rad-h-weights}.
Suppose \(0<h<1\), \(64\sqrt2\,R\sqrt{H_1}\,h^{3/2}\le1/8\),
\(\norm{x_k-x^*}\le3\sqrt2 R\), and \(\norm{v_k-x^*}\le\sqrt2 R\).
If \(b_k<\widehat b_k\), then
\begin{equation}
 \label{eq:rad-curv-gap-difference}
 \F(y_k)-\F(x_k)\le\tfrac14 b_k\norm{\nabla f(y_k)}^2 .
\end{equation}
\end{lemma}

\begin{proof}
The query point lies on the segment from \(x_k\) to \(v_k\), with
\(y_k-x_k=\eta_k(v_k-x_k)\). Convexity at \(y_k\) therefore gives
\[
\begin{aligned}
 \F(y_k)-\F(x_k)
 &\overset{\circledOne}{\le}\inner{\nabla f(y_k)}{y_k-x_k}\\
 &\overset{\circledTwo}{\le}\eta_k\norm{\nabla f(y_k)}\norm{v_k-x_k}\\
 &\overset{\circledThree}{\le}4\sqrt2 R\eta_k\norm{\nabla f(y_k)},
\end{aligned}
\]
where \(\circledOne\) evaluates the supporting hyperplane at \(x_k\),
\(\circledTwo\) uses the segment identity and Cauchy--Schwarz, and
\(\circledThree\) bounds \(\norm{v_k-x_k}\) by the sum of the two radii
\(\sqrt2 R\) and \(3\sqrt2 R\).

Since \(b_k<\widehat b_k\), the second candidate in \eqref{eq:rad-h-weights}
is active. Hence \(H_1>0\) and
\(b_k\norm{\nabla f(y_k)}=1/(32\sqrt{H_1\eta_k})\).
Dividing the preceding bound by \(b_k\norm{\nabla f(y_k)}^2\) now gives
\[
\begin{aligned}
 \frac{\F(y_k)-\F(x_k)}{b_k\norm{\nabla f(y_k)}^2}
 &\overset{\circledOne}{\le}\frac{4\sqrt2 R\eta_k}{b_k\norm{\nabla f(y_k)}}\\
 &=128\sqrt2\,R\sqrt{H_1}\,\eta_k^{3/2}\\
 &\overset{\circledTwo}{\le}128\sqrt2\,R\sqrt{H_1}\,h^{3/2}\\
 &\overset{\circledThree}{\le}\tfrac14 ,
\end{aligned}
\]
where \(\circledOne\) uses the gap bound just proved, \(\circledTwo\) uses
\(\eta_k\le h\), and \(\circledThree\) applies the hypothesis on \(h\).
\end{proof}

\subsection{Auxiliary Estimates for the Lower Bound}

\begin{lemma}[Expected gain from changing the maximizing index]
\label{radlb:curvature}
Let \(B_0\) be constant and \(B_1,\ldots,B_m\) be independent random
variables, each uniform on an interval of length \(w_i>0\). Let \(I\)
be the index of the largest of these \(m+1\) values. For arbitrary
random changes \(d_i\) satisfying \(d_0=0\) and deterministic bounds
\(|d_i|\le M_i\), we have
\begin{equation}
\label{radlb:switch}
 \mathbb E\bigl[\max_i(B_i+d_i)-(B_I+d_I)\bigr]
 \le 8\sum_{i=1}^m \frac{M_i^2}{w_i}.
\end{equation}
\end{lemma}

\begin{proof}
For a change of maximizer to increase the maximum beyond \(B_I+d_I\),
the new winner must overcome its initial deficit. Indeed, if index
\(K\ne I\) attains the new maximum with a positive gain, then
\[
 0\overset{\circledOne}{\le} B_I-B_K
  \overset{\circledTwo}{\le}d_K-d_I ,
\]
where \(\circledOne\) follows because \(I\) was the old maximizer and
\(\circledTwo\) rearranges \(B_K+d_K\ge B_I+d_I\).
Choose the index \(i\in\{I,K\}\) with the larger \(|d_i|\). It is
nonzero because the gain is positive and \(d_0=0\). Both the gain and
\(B_I-B_K\) are at most \(2M_i\). If \(i=K\), its strongest competitor
has value \(B_I\). If \(i=I\), its strongest competitor has value
between \(B_K\) and \(B_I\). Thus in either case \(B_i\) is within
\(2M_i\) of the largest of the other values. The gain is bounded
pointwise by
\[
 \max_i(B_i+d_i)-(B_I+d_I)
 \overset{\circledOne}{\le}
 \sum_{i=1}^m 2M_i\mathbf 1_{\{|B_i-\max_{j\ne i}B_j|\le2M_i\}} ,
\]
where \(\circledOne\) holds because every positive gain is bounded by
at least one of the summands, and every summand is nonnegative.
Condition on all branches other than \(i\). Their maximum is now fixed.
The interval of values within \(2M_i\) of this maximum has length \(4M_i\).
The conditional probability of the indicated event is at most \(4M_i/w_i\).
Multiplication by \(2M_i\), summation, and expectation give
\eqref{radlb:switch}.
\end{proof}

\begin{lemma}[Separation of two hidden branches]
\label{radlb:chain}
Let \(a_i\) be as in \eqref{radlb:parameters} and
\(\rho=1/(100s)\). For all sufficiently large \(s\), all \(i<k\),
all \(|z_i|,|z_k|\le\rho\), and all \(u_i,u_k\in[0,1/8]\),
\[
 a_i\phi_{u_i}(s/2+z_i)>a_k\phi_{u_k}(s/2+z_k).
\]
\end{lemma}

\begin{proof}
On \([s/2-\rho,s/2+\rho]\), \(h\) is increasing with slope at most two,
and \(j(w)\le3/(1+w)\). Thus the variation of \(h\) is at most
\(4\rho\), and the additional term \(u j(w)\) is at most
\(3/[8(1+s/2-\rho)]\). On the other hand \(a_i/a_k\ge 1+16/s^2\).
The excess weight of branch \(i\) covers both possible changes, since
\[
 \frac{16}{s^2}h(s/2-\rho)
 \overset{\circledOne}{\ge}\frac8s
 \overset{\circledTwo}{>}4\rho+\frac{3}{8(1+s/2-\rho)},
\]
where \(\circledOne\) uses \(h(s/2-\rho)\ge1+s/2-\rho\ge s/2\), and
\(\circledTwo\) uses \(4\rho=1/(25s)\) and \(1+s/2-\rho\ge s/2\)
for all sufficiently large \(s\). It follows that
\[
\begin{aligned}
 a_i\phi_{u_i}(s/2+z_i)
 &\overset{\circledOne}{\ge} a_k(1+16/s^2)h(s/2-\rho)\\
 &\overset{\circledTwo}{>} a_k\bigl(h(s/2+\rho)+3/[8(1+s/2-\rho)]\bigr)\\
 &\overset{\circledThree}{\ge} a_k\phi_{u_k}(s/2+z_k),
\end{aligned}
\]
where \(\circledOne\) uses the weight ratio, monotonicity of \(h\), and
\(u_i j(s/2+z_i)\ge0\), \(\circledTwo\) uses the excess just bounded and
\(h(s/2+\rho)-h(s/2-\rho)\le4\rho\), and \(\circledThree\) uses the upper
bound on the perturbation of branch \(k\).
\end{proof}

\par
\endgroup

\section{Comparison with General Gradient-Dependent
\texorpdfstring{\(\ell\)}{ell}-Smoothness}
\label{app:ell-smoothness-comparison}

Following \citet{tyurin2025near}, a twice differentiable function is
\(\ell\)-smooth if there is a positive, nondecreasing, locally Lipschitz
function \(\ell:[0,\infty)\to(0,\infty)\) such that
\begin{equation}
\label{eq:ell-smoothness-definition}
  \norm{\nabla^2 f(x)}
  \le
  \ell\bigl(\norm{\nabla f(x)}\bigr)
  \qquad\text{for every }x.
\end{equation}
For comparison, within the family of convex twice differentiable functions
with an attained minimum, let \(\mathcal E\) denote the finite-valued,
positive, nondecreasing, locally Lipschitz functions on \([0,\infty)\), and set
\[
\begin{aligned}
  \mathcal H
  &:=\left\{f:\ \exists\ 0\le H_0,H_1<\infty,\quad
  \norm{\nabla^2 f(x)}
  \le H_0+H_1\bigl(f(x)-f^*\bigr)\ \text{for all }x\right\}, \\
  \mathcal L_{\rm gen}
  &:=\left\{f:\ \eqref{eq:ell-smoothness-definition}\ \text{holds for some }
  \ell\in\mathcal E\right\}.
\end{aligned}
\]
The affine choice \(\ell(s)=L_0+L_1s\) defines a subclass
\(\mathcal L_{\rm aff}\).  As discussed in the main text and witnessed by
Examples~\ref{ex:one-dimensional-separation} and
\ref{ex:multidimensional-separation},
\(\mathcal L_{\rm aff}\subsetneq\mathcal H\).  The relation with the full
class \(\mathcal L_{\rm gen}\) is different: the two general classes are
incomparable.  We first clarify why the existing examples do not establish
this stronger statement and then prove both non-inclusions by explicit convex
examples.  For twice differentiable convex functions, the Hessian definition
of \(\mathcal H\) corresponds, up to absolute constants, to
Assumption~\ref{ass:h0h1-smoothness}; see the discussion following that
assumption.

\subsection{The existing separation examples are
\texorpdfstring{\(\ell\)}{ell}-smooth}

\begin{proposition}[Examples~\ref{ex:one-dimensional-separation} and
\ref{ex:multidimensional-separation} under \(\ell\)-smoothness]
\label{prop:existing-examples-ell-smooth}
For every \(p\in(1,2]\), both \(\phi_p\) and \(\Phi_p\) are
\(\ell_p\)-smooth with
\[
  \ell_p(s):=2+s^p.
\]
Moreover, the auxiliary function of \citet{tyurin2025near},
\(\psi_p(r):=r^2/[2\ell_p(4r)]\), is strictly increasing for \(r>0\).
Consequently, these examples satisfy the general \(\ell\)-smooth assumptions
of \citet[Theorems~3.2--3.3]{tyurin2025near}, even though they do not satisfy
any \((L_0,L_1)\)-smoothness condition with finite constants.
\end{proposition}

\begin{proof}
The definition of \(q_p\) and the inequality \(\exp(-u)\le1\) for \(u\ge0\)
give
\begin{equation}
\label{eq:existing-example-ell-hessian}
  \phi_p''(x)
  =q_p(x)
  \le
  1+(1+x^2)^{p/2}
  \le
  2+\lvert x\rvert^p.
\end{equation}
The last inequality follows from
\((1+x^2)^{p/2}\le1+\lvert x\rvert^p\), valid because
\(p/2\in(0,1]\).  On the other hand, \(q_p(t)\ge1\), and therefore
\begin{equation}
\label{eq:existing-example-ell-gradient}
  \lvert\phi_p'(x)\rvert
  =
  \left\lvert\int_0^x q_p(t)\,dt\right\rvert
  \ge
  \lvert x\rvert.
\end{equation}
Combining \eqref{eq:existing-example-ell-hessian} and
\eqref{eq:existing-example-ell-gradient} yields
\[
  \lvert\phi_p''(x)\rvert
  \le
  2+\lvert\phi_p'(x)\rvert^p
  =
  \ell_p\bigl(\lvert\phi_p'(x)\rvert\bigr),
\]
which proves \(\ell_p\)-smoothness of the one-dimensional example.

For the multidimensional example,
\(\nabla^2\Phi_p(x)=\operatorname{diag}(q_p(x_1),1,\ldots,1)\).  Since
\(q_p\ge1\), its spectral norm is \(q_p(x_1)\).  Also
\(\norm{\nabla\Phi_p(x)}\ge\lvert\phi_p'(x_1)\rvert\).  Hence
\[
\begin{aligned}
  \norm{\nabla^2\Phi_p(x)}
  &=q_p(x_1) \\
  &\le 2+\lvert\phi_p'(x_1)\rvert^p \\
  &\le 2+\norm{\nabla\Phi_p(x)}^p
  =\ell_p\bigl(\norm{\nabla\Phi_p(x)}\bigr).
\end{aligned}
\]
This proves \(\ell_p\)-smoothness of \(\Phi_p\).

It remains to check the auxiliary condition used by
\citet{tyurin2025near}.  Direct substitution gives
\[
  \psi_p(r)
  =
  \frac{r^2}{2\ell_p(4r)}
  =
  \frac{r^2}{4+2\cdot4^p r^p}.
\]
Differentiating this quotient gives
\[
  \psi_p'(r)
  =
  \frac{
    8r+2\cdot4^p(2-p)r^{p+1}
  }{
    (4+2\cdot4^p r^p)^2
  }
  >0,
  \qquad r>0,
\]
where positivity uses \(p\le2\).  Thus \(\psi_p\) is strictly increasing on
\((0,\infty)\), completing the proof.
\end{proof}

\subsection{Incomparability of the general classes}
\label{app:ell-incomparability}

The preceding proposition shows that failure of \((L_0,L_1)\)-smoothness in
Examples~\ref{ex:one-dimensional-separation} and
\ref{ex:multidimensional-separation} does not imply failure of general
\(\ell\)-smoothness.  We now give one example for each direction of the
non-inclusion.

\begin{proposition}[An \(\ell\)-smooth objective outside
\(\mathcal H\)]
\label{prop:ell-outside-h0h1}
The function \(u:\mathbb R\to\mathbb R\) defined by
\[
  u(x):=e^{x^2}-1
\]
is convex, coercive, infinitely differentiable, and \(\ell_u\)-smooth with
\(\ell_u(s):=2+\tfrac32s^2\), but \(u\notin\mathcal H\).  Hence
\(\mathcal L_{\rm gen}\nsubseteq\mathcal H\).
\end{proposition}

\begin{proof}
The derivatives of \(u\) are
\[
  u'(x)=2xe^{x^2},
  \qquad
  u''(x)=(2+4x^2)e^{x^2}>0.
\]
Thus \(u\) is infinitely differentiable and strictly convex.  Moreover,
\(u(0)=u'(0)=0\) and \(u(x)\to\infty\) as \(\lvert x\rvert\to\infty\), so
it is coercive and has the unique minimizer \(x^*=0\).

Set \(t=x^2\).  Since \(e^t-1\le te^t\) for \(t\ge0\), we obtain
\[
\begin{aligned}
  u''(x)
  &=2+2(e^t-1)+4te^t \\
  &\le 2+6te^t \\
  &\le 2+6te^{2t}
  =2+\frac32\lvert u'(x)\rvert^2
  =\ell_u\bigl(\lvert u'(x)\rvert\bigr).
\end{aligned}
\]
The function \(\ell_u\) is finite-valued, positive, nondecreasing, and
locally Lipschitz on \([0,\infty)\).  Therefore
\(u\in\mathcal L_{\rm gen}\).

Suppose instead that \(u\in\mathcal H\).  Then some finite
\(H_0,H_1\ge0\) would satisfy
\[
  (2+4t)e^t
  \le
  H_0+H_1(e^t-1),
  \qquad t\ge0.
\]
Dividing by \(e^t\) gives
\[
  2+4t
  \le
  H_0e^{-t}+H_1(1-e^{-t}).
\]
As \(t\to\infty\), the left-hand side diverges while the right-hand side
converges to \(H_1\), a contradiction.  Thus
\(u\notin\mathcal H\).
\end{proof}

The reverse non-inclusion requires curvature that becomes unbounded while the
gradient remains bounded.  The following construction provides it.

\begin{proposition}[An \((H_0,H_1)\)-smooth objective outside
\(\mathcal L_{\rm gen}\)]
\label{prop:h0h1-outside-ell}
There exists a convex, coercive, infinitely differentiable function
\(g:\mathbb R\to\mathbb R\) with the unique minimizer \(x^*=0\) such that
\begin{equation}
\label{eq:strict-ell-h0h1-bound}
  g''(x)
  \le
  3+g(x)-g^*,
  \qquad x\in\mathbb R,
\end{equation}
but \(g\) is not \(\ell\)-smooth for any finite-valued nondecreasing function
\(\ell:[0,\infty)\to(0,\infty)\).  Consequently,
\(g\in\mathcal H\setminus\mathcal L_{\rm gen}\), and hence
\(\mathcal H\nsubseteq\mathcal L_{\rm gen}\).
\end{proposition}

\begin{proof}
Define
\begin{equation}
\label{eq:strict-ell-construction}
  a(t):=\sqrt{1+t^2},
  \qquad
  q(t):=\operatorname{sech}^2(t)
  +a(t)\exp\bigl(-a(t)^6\sin^2(\pi t)\bigr),
\end{equation}
and let
\begin{equation}
\label{eq:strict-ell-function}
  g(x)
  :=
  \int_0^x\int_0^s q(t)\,dt\,ds.
\end{equation}
The function \(q\) is infinitely differentiable, even, and strictly positive.
The fundamental theorem of calculus therefore gives
\begin{equation}
\label{eq:strict-ell-derivatives}
  g'(x)=\int_0^x q(t)\,dt,
  \qquad
  g''(x)=q(x)>0.
\end{equation}
Thus \(g\) is infinitely differentiable and strictly convex.  Moreover,
\(g(0)=g'(0)=0\), so its unique minimizer is \(x^*=0\) and \(g^*=0\).

We next verify coercivity and the \((H_0,H_1)\) bound.  Since
\(q(t)\ge\operatorname{sech}^2(t)\), while
\((\log\cosh t)''=\operatorname{sech}^2(t)\), and both functions and their
first derivatives vanish at zero, convexity of their difference gives
\begin{equation}
\label{eq:strict-ell-logcosh-lower}
  g(x)
  \ge
  \log\cosh x.
\end{equation}
Because \(\log\cosh x\to\infty\) as \(\lvert x\rvert\to\infty\), the
function \(g\) is coercive.

The inequalities \(\exp(-u)\le1\), \(\operatorname{sech}^2x\le1\), and
\(\sqrt{1+x^2}\le1+\lvert x\rvert\) imply
\begin{equation}
\label{eq:strict-ell-hessian-linear-x}
  q(x)
  \le
  1+\sqrt{1+x^2}
  \le
  2+\lvert x\rvert.
\end{equation}
Also \(\cosh x\ge e^{\lvert x\rvert}/2\), and hence
\begin{equation}
\label{eq:strict-ell-x-logcosh}
  \lvert x\rvert
  \le
  \log\cosh x+\log 2.
\end{equation}
Combining \eqref{eq:strict-ell-logcosh-lower}--
\eqref{eq:strict-ell-x-logcosh}, and using \(\log 2<1\), gives
\[
  g''(x)
  =q(x)
  \le
  2+\log 2+\log\cosh x
  \le
  3+g(x).
\]
This proves \eqref{eq:strict-ell-h0h1-bound}.  Applying the same
Hessian-to-local-model implication used in the proofs of
Examples~\ref{ex:one-dimensional-separation} and
\ref{ex:multidimensional-separation}, namely
\citet[Lemma~2]{liu2025warmup}, shows that \(g\) satisfies
Assumption~\ref{ass:h0h1-smoothness} with \(H_0=3\) and \(H_1=1\).

It remains to show that the gradient is bounded although the Hessian is not.
For \(m\ge1\), set
\[
  I_m:=\left[m-\frac12,m+\frac12\right],
  \qquad
  A_m:=\sqrt{1+m^2}.
\]
If \(t\in I_m\), then elementary comparison of \(1+t^2\) and \(1+m^2\)
gives
\begin{equation}
\label{eq:strict-ell-am-bounds}
  \frac{A_m}{2}
  \le
  a(t)
  \le
  \frac{3A_m}{2}.
\end{equation}
Writing \(t=m+u\), where \(\lvert u\rvert\le1/2\), concavity of sine on
\([0,1/2]\) and symmetry give
\begin{equation}
\label{eq:strict-ell-sine-bound}
  \lvert\sin(\pi t)\rvert
  =
  \lvert\sin(\pi u)\rvert
  \ge
  2\lvert u\rvert.
\end{equation}
Equations \eqref{eq:strict-ell-am-bounds} and
\eqref{eq:strict-ell-sine-bound} imply
\(a(t)^6\sin^2(\pi t)\ge A_m^6u^2/16\).  Therefore
\begin{equation}
\label{eq:strict-ell-spike-integral}
\begin{aligned}
  &\int_{I_m}
  a(t)\exp\bigl(-a(t)^6\sin^2(\pi t)\bigr)\,dt \\
  &\quad\le
  \frac{3A_m}{2}
  \int_{-\infty}^{\infty}
  \exp\left(-\frac{A_m^6u^2}{16}\right)du
  =
  \frac{6\sqrt\pi}{A_m^2}.
\end{aligned}
\end{equation}
The inequality uses the upper bound on \(a(t)\), the lower bound on the
exponent, and an enlargement of the integration interval; the equality is the
Gaussian integral.  The spike term is continuous and hence integrable on
\([0,1/2]\).  Since
\(\sum_{m=1}^{\infty}A_m^{-2}=\sum_{m=1}^{\infty}(1+m^2)^{-1}<\infty\),
equation \eqref{eq:strict-ell-spike-integral} yields
\[
  \int_0^\infty
  a(t)\exp\bigl(-a(t)^6\sin^2(\pi t)\bigr)\,dt
  <\infty.
\]
Together with \(\int_0^\infty\operatorname{sech}^2(t)\,dt=1\), this proves
that there is a finite constant \(G\) such that
\begin{equation}
\label{eq:strict-ell-gradient-bound}
  \lvert g'(x)\rvert
  \le
  G,
  \qquad x\in\mathbb R.
\end{equation}
Here evenness of \(q\) extends the estimate from \(x\ge0\) to all real \(x\).

At every integer \(n\ge1\), \(\sin(\pi n)=0\).  Hence
\begin{equation}
\label{eq:strict-ell-hessian-integers}
  g''(n)
  =q(n)
  =\operatorname{sech}^2(n)+\sqrt{1+n^2}
  \ge
  n.
\end{equation}
Thus \(g''(n)\to\infty\), whereas \(\lvert g'(n)\rvert\le G\).  If \(g\)
were \(\ell\)-smooth for a finite-valued nondecreasing \(\ell\), then
\[
  g''(n)
  \le
  \ell\bigl(\lvert g'(n)\rvert\bigr)
  \le
  \ell(G)
  <\infty
\]
for every \(n\), contradicting \eqref{eq:strict-ell-hessian-integers}.
Therefore no such \(\ell\) exists.
\end{proof}

Propositions~\ref{prop:ell-outside-h0h1} and
\ref{prop:h0h1-outside-ell} establish the precise relationship
\[
  \mathcal L_{\rm aff}
  \subsetneq
  \mathcal H,
  \qquad
  \mathcal H\setminus\mathcal L_{\rm gen}\ne\varnothing,
  \qquad
  \mathcal L_{\rm gen}\setminus\mathcal H\ne\varnothing.
\]
Thus \((H_0,H_1)\)-smoothness is strictly weaker than affine
\((L_0,L_1)\)-smoothness, whereas the full \((H_0,H_1)\)-smooth and general
\(\ell\)-smooth classes are incomparable.  General \(\ell\)-smooth
acceleration and our native
\((H_0,H_1)\) guarantees are therefore complementary rather than mutually
subsuming.  Our results cover the full gap-dependent class, including
objectives that cannot be controlled by any finite-valued function of the
gradient norm alone, and provide uniform and non-uniform coordinate
extensions in the native \((H_0,H_1)\) parameters.

\end{document}